\documentclass[12pt]{amsart}

\usepackage{times}
\usepackage{latexsym,  amsmath,amscd, amssymb, amsthm,graphicx}
\usepackage[left=1in, right=1in, top=1in, bottom=1in, bindingoffset=0cm, a4paper]{geometry}
\usepackage[T1]{fontenc}
\usepackage{lmodern}
\usepackage[applemac]{inputenc}
\usepackage{calrsfs}
\usepackage{enumitem}
\usepackage[mathscr]{euscript}

\makeatletter
\def\thm@space@setup{%
  \thm@preskip=\parskip \thm@postskip=0pt 
}
\renewenvironment{proof}[1][\proofname]{\par
  \vspace{-\topsep}
  \pushQED{\qed}%
  \normalfont
  \topsep0pt \partopsep0pt 
  \trivlist
  \item[\hskip\labelsep
        \itshape
    #1\@addpunct{.}]\ignorespaces
}{%
  \popQED\endtrivlist\@endpefalse
}
\makeatother

\makeatletter
\def\subsection{\@startsection{subsection}{3}%
  \z@{.4\linespacing\@plus.0\linespacing}{.1\linespacing}%
  {\normalfont\bfseries}}
\def\subsubsection{\@startsection{subsubsection}{3}%
  \z@{.4\linespacing\@plus.0\linespacing}{.1\linespacing}%
  {\normalfont\bfseries}}
  \makeatother

\makeatletter
\renewcommand{\paragraph}{\@startsection{paragraph}{4}{0ex}%
    {-3.25ex plus -1ex minus -0.2ex}%
    {1.5ex plus 0.2ex}%
    {\normalfont\normalsize\bfseries}}
\makeatother
 
\stepcounter{secnumdepth}

\begin{document}
\bibliographystyle{plainnat}
\linespread{1.6}

\abovedisplayskip=0pt
\abovedisplayshortskip=0pt
\belowdisplayskip=0pt
\belowdisplayshortskip=0pt
\abovecaptionskip=0pt
\belowcaptionskip=0pt

\newcounter{Lcount}
\newcommand{\squishlisttwo}{
\begin{list}{\alph{Lcount}. }
 { \usecounter{Lcount}
 \setlength{\itemsep}{0pt}
 \setlength{\parsep}{0pt}
 \setlength{\topsep}{0pt}
 \setlength{\partopsep}{0pt}
 \setlength{\leftmargin}{2em}
 \setlength{\labelwidth}{1.5em}
 \setlength{\labelsep}{0.5em} } }

\newcommand{\squishend}{
 \end{list} }

\newtheorem{lemma}{Lemma}[subsection]
\newtheorem{definition}{Definition}[section]
\newtheorem{remarks}{Remark}[section]
\newtheorem{remark}{Remark}[subsection]
\newtheorem{theorem}{Theorem}[subsection]
\newtheorem{corollary}[lemma]{Corollary}
\newtheorem{proposition}[lemma]{Proposition}
\numberwithin{equation}{subsection}
\newtheorem{axiom}[lemma]{Axiom}
\newtheorem*{acknowledgements}{Acknowledgments}

\newcommand{\R}{\mbox{$\Bbb R$}} \newcommand{\C}{\mbox{$\Bbb C$}}\newcommand{\F}{\mbox{$\Bbb F$}}
\newcommand{\N}{\mbox{$\Bbb N$}}
\newcommand{\Z}{\mbox{$\Bbb Z$}} \def\g{\mathfrak{g}} \def\q{\mathfrak{q}}
\def\h{\mathfrak{h}} \def\c{\mathfrak{c}} \def\d{\mathfrak{d}}
\def\m{\mathfrak{m}} \def\n{\mathfrak{n}} \def\i{\mathfrak{i}}
\def\l{\mathfrak{l}} \def\s{\mathfrak{s}}
\def\L{\mathscr{L}}
\def\r{\mathscr{R}}

\mbox{}

\vskip 1cm

\title[Solvable extensions of quasi-filiform algebras]{Solvable extensions of the naturally graded quasi-filiform Leibniz algebra
of second type $\mathcal{L}^4$}\maketitle
\begin{center}
{A. Shabanskaya}

Department of Mathematics, Adrian College, 110 S Madison St.,
Adrian, MI 49221, USA

ashabanskaya@adrian.edu
\end{center}
\begin{abstract} For a sequence of the naturally graded quasi-filiform Leibniz algebra
of second type $\mathcal{L}^4$ introduced by Camacho, G\'{o}mez, Gonz\'{a}lez and Omirov, all the possible right solvable indecomposable
extensions over the field $\C$ are constructed.
 \end{abstract}



\vskip 6cm

\newpage

\section{Introduction} 
Leibniz algebras were discovered by Bloch in 1965 \cite{B} who called them $D-$ algebras. Later on they were considered by Loday and Cuvier \cite{ C, Lo, L, LP} as
a non-antisymmetric analogue of Lie algebras. It makes every Lie algebra be a Leibniz algebra, but the converse is not true.
 Exactly Loday named them Leibniz algebras after Gottfried Wilhelm Leibniz.
 
Since then many analogs of important theorems in Lie theory were found to be true for Leibniz algebras, such as
the analogue of Levi's theorem which was proved by Barnes \cite{Ba}. He showed that any finite-dimensional complex Leibniz algebra is decomposed into a semidirect sum of the solvable radical and a semisimple Lie algebra. 
Therefore the biggest challenge in the classification problem of finite-dimensional complex Leibniz algebras is to study the solvable part. And to classify solvable Leibniz algebras, we need nilpotent Leibniz algebras as their nilradicals, same as in the case of Lie algebras \cite{Mub3}.

Every Leibniz algebra satisfies a generalized version of the Jacobi identity called the Leibniz identity.
 There are two Leibniz identities: the left and the right.
 We call Leibniz algebras right Leibniz algebras if they satisfy the right Leibniz identity
 and left, if they satisfy the left. A left Leibniz algebra is not necessarily
 a right Leibniz algebra \cite{D}.
  
Leibniz algebras inherit an important property of Lie algebras which is that the right (left)
multiplication operator of a right (left) Leibniz algebra is a derivation \cite{CL}. Besides the algebra of right (left) multiplication operators
is endowed with a structure of a Lie algebra by means of the commutator \cite{CL}.  
 Also the quotient
algebra by the two-sided ideal generated by the square elements of a Leibniz algebra is a Lie algebra \cite{O},
where such ideal is the minimal, abelian and in the case of
non-Lie Leibniz algebras it is always non trivial.

It is possible to find solvable Leibniz algebras in any finite dimension working with the sequence of nilpotent Leibniz algebras in any finite dimension and their ``nil-independent''
derivations. This method for Lie algebras is based on what was shown by
Mubarakzyanov in \cite{Mub3}: the dimension of the complimentary vector space to the nilradical does not exceed the number of 
nil-independent derivations of the nilradical. This result 
was extended to Leibniz algebras by Casas, Ladra, Omirov and Karimjanov \cite{CLOK} with the help
of \cite{AO}. Besides, similarly to the case of Lie algebras, for a solvable Leibniz
algebra $L$ we also have the inequality $\dim\n\i\l(L)\geq\frac{1}{2}\dim L$ \cite{Mub3}. 
There is the following work performed over the field of characteristic zero using 
this method: Casas, Ladra, Omirov and Karimjanov classified solvable Leibniz algebras with null-filiform nilradical \cite{CLOK};
Omirov and his colleagues Casas, Khudoyberdiyev, Ladra, Karimjanov, Camacho and Masutova classified solvable Leibniz algebras whose nilradicals are a direct sum of null-filiform algebras \cite{KL}, naturally graded filiform \cite{CL, LadraMO}, triangular \cite{KK} and finally filiform \cite{COM}.
Bosko-Dunbar, Dunbar, Hird and Stagg attempted to classify left solvable Leibniz algebras with Heisenberg nilradical \cite{BDHS}.
Left and right solvable extensions of $\mathcal{R}_{18}$ \cite{AOR}, $\mathcal{L}^1,\mathcal{L}^2$ and $\mathcal{L}^3$ \cite{CGGO} over the field of real numbers were 
found by Shabanskaya in \cite{ShA, Sha, ShabA}.

The starting point of the present article is a naturally graded quasi-filiform non Lie Leibniz algebra of the second type
$\mathcal{L}^4,(n\geq4)$ in the notation of \cite{CGGO}.
This algebra is left and right at the same time and an associative when $n=4$.

Naturally graded quasi-filiform Leibniz algebras in any finite dimension over $\C$ were studied by Camacho, G\'{o}mez, Gonz\'{a}lez,
Omirov \cite{CGGO}. They found six such algebras of the first type, where two of them depend on a parameter and eight algebras of the second type with one of them
depending on a parameter. 

This paper continues a work on finding all solvable extensions of quasi-filiform Leibniz algebras over the field of complex numbers. For
a sequence $\mathcal{L}^4$ such extensions of codimension at most two are
possible.

The paper is organized as follows: in Section \ref{Pr} we give some basic definitions,
in Section \ref{Con} we show what involves constructing solvable Leibniz algebras with a given nilradical.
In Section \ref{NilL4} we describe the nilpotent sequence $\mathcal{L}^4$ and give the summary of the results stated in theorems, which could be found in
the remaining Section \ref{Classification}.
 
As regards notation, we use $\langle e_1,e_2,...,e_r \rangle$ to
denote the $r$-dimensional subspace generated by
$e_1,e_2,...,e_r,$ where $r\in\N$. Besides $\g$
and $l$
are used to denote solvable right and left Leibniz algebras, respectively.

Throughout the paper all the algebras are finite dimensional over the field of complex numbers and
if the bracket is not given, then it is assumed to be zero, except the brackets for the
nilradical, which most of the time are not given (see Remark \ref{remark5.1}) to save space. 

In the tables an ordered triple is a shorthand notation
for a derivation property of the multiplication operators, which is either $\r_z\left([x,y]\right)=
[\r_z(x),y]+[x,\r_z(y)]$ or $\L_z\left([x,y]\right)=
[\L_z(x),y]+[x,\L_z(y)]$. We also assign $\r_{e_{n+1}}:=\r$ and $\L_{e_{n+1}}:=\L.$

We use Maple software to compute the Leibniz identity, the ``absorption'' (see \cite{Shab, ST} and
Section \ref{Solvable left Leibniz algebras}), the change
of basis for solvable Leibniz algebras in some particular dimensions, which are generalized and proved in an arbitrary finite dimension.

\section{Preliminaries}\label{Pr}
We give some Basic definitions encountered working with Leibniz algebras.
\begin{definition}
\begin{enumerate}[noitemsep, topsep=0pt]
\item[1.] A vector space L over a field F with a bilinear operation\\
$[-,-]:\,L\rightarrow L$
is called a Leibniz algebra if for any $x,\,y,\,z\in\,L$ the Leibniz identity
\begin{equation}\nonumber [[x,y],z]=[[x,z],y]+[x,[y,z]]\end{equation}
holds. This Leibniz identity is known as the right Leibniz identity and we call $L$
in this case a right Leibniz algebra.
\item[2.] There exists the version corresponding to the left Leibniz identity
\begin{equation}\nonumber
[[x,y],z]=[x,[y,z]]-[y,[x,z]],
\end{equation}
and a Leibniz algebra L is called a left Leibniz algebra.
\end{enumerate}
\end{definition}
\begin{remarks}
In addition, if $L$ satisfies $[x,x]=0$ for every $x\in L$, then it is a Lie algebra.
Therefore every Lie algebra is a Leibniz algebra, but the converse is not true.
\end{remarks}
\begin{definition}
The two-sided ideal $C(L)=\{x\in L:\,[x,y]=[y,x]=0\}$ is said to be the center of $L.$
\end{definition}
\begin{definition}
A linear map $d:\,L\rightarrow L$ of a Leibniz algebra $L$ is a derivation
if for all $x,\,y\in L$
\begin{equation}
\nonumber d([x,y])=[d(x),y]+[x,d(y)].
\end{equation}
\end{definition}
If $L$ is a right Leibniz algebra and $x\in L,$ then the right multiplication operator $\r_x:\,L\rightarrow L$ defined as $\r_x(y)=[y,x],\,y\in L$
is a derivation (for a left Leibniz algebra $L$ with $x\in L,$ the left multiplication operator $\L_x:\,L\rightarrow L,\,\L_x(y)=[x,y],\,y\in L$ is a derivation).

Any right Leibniz algebra $L$ is associated with the algebra of right multiplications $\r(L)=\{\r_x\,|\,x\in L\}$
endowed with the structure of a Lie algebra by means of the commutator $[\r_x,\r_y]=\r_x\r_y-\r_y\r_x=\r_{[y,x]},$
which defines an antihomomorphism between $L$ and $\r(L).$

For a left Leibniz algebra $L,$ the corresponding algebra of left multiplications $\L(L)=\{\L_x\,|\,x\in L\}$
is endowed with the structure of a Lie algebra by means of the commutator as well $[\L_x,\L_y]=\L_x\L_y-\L_y\L_x=\L_{[x,y]}.$
In this case we have a homomorphism between $L$ and $\L(L)$.
\begin{definition}
Let $d_1,d_2,...,d_n$ be derivations of a Leibniz algebra $L.$ The derivations $d_1,d_2,...,d_n$ are said to be ``nil-independent''
if $\alpha_1d_1+\alpha_2d_2+\alpha_3d_3+\ldots+\alpha_nd_n$ is not nilpotent for any scalars $\alpha_1,\alpha_2,...,\alpha_n\in F,$
otherwise they are ``nil-dependent''.
\end{definition}
\begin{definition}
For a given Leibniz algebra $L,$ we define the sequence of two-sided ideals as follows:
\begin{equation}
\nonumber L^0=L,\,L^{k+1}=[L^k,L],\,(k\geq0)\qquad\qquad L^{(0)}=L,\,L^{(k+1)}=[L^{(k)},L^{(k)}],\,(k\geq0),
\end{equation}
which are the lower central series and the derived series of $L,$ respectively.

A Leibniz algebra $L$ is said to be nilpotent (solvable) if there exists $m\in\N$ such that $L^m=0$ ($L^{(m)}=0$).
The minimal such number $m$ is said to be the index of nilpotency (solvability). 
\end{definition}
\begin{definition}
A Leibniz algebra $L$ is called a quasi-filiform if $L^{n-3}\neq0$ and $L^{n-2}=0,$ where $\dim(L)=n.$ 
\end{definition}

\section{Constructing solvable Leibniz algebras with a given nilradical}\label{Con}

Every solvable Leibniz algebra $L$ contains a unique maximal nilpotent
ideal called the nilradical and denoted $\n\i\l(L)$ such that
$\dim\n\i\l(L)\geq\frac{1}{2}\dim(L)$ \cite{Mub3}. Let us consider
the problem of constructing solvable Leibniz algebras
$L$ with a given nilradical $N=\n\i\l(L)$. Suppose
$\{e_1,e_2,e_3,e_4,...,e_n\}$ is a basis for the nilradical and
$\{e_{n+1},...,e_{p}\}$ is a basis for a subspace complementary to
the nilradical.

If $L$ is a solvable Leibniz algebra \cite{AO}, then
\begin{equation}\label{Nil}[L,L]\subseteq N \end{equation} and we have
the following structure equations
\begin{equation}\label{Algebra}[e_i, e_j ] =C_{ij}^ke_k, [e_a, e_i]
=A^k_{ai}e_k, [e_i,e_a]=A^k_{ia}e_k, [e_a, e_b] =B^k_{ab} e_k, \end{equation} where $1\leq i,
j, k, m\leq n$ and $n+1\leq a, b\leq p$.

\subsection{Solvable right Leibniz algebras}  Calculations show that satisfying the right Leibniz identity is equivalent to the
following conditions: 
\begin{equation}\label{RLeibniz}
A_{ai}^kC_{kj}^m=A_{aj}^kC_{ki}^m+C_{ij}^kA_{ak}^m,\,A_{ia}^kC_{kj}^m=C_{ij}^kA_{ka}^m+A_{aj}^kC_{ik}^m,\,
C_{ij}^kA_{ka}^m=A_{ia}^kC_{kj}^m+A_{ja}^kC_{ik}^m,
\end{equation}
 \begin{equation}\label{RRLeibniz}
 B_{ab}^kC_{ki}^m=A_{ai}^kA_{kb}^m+A_{bi}^kA_{ak}^m,\,A_{ai}^kA_{kb}^m=B_{ab}^kC_{ki}^m+A_{ib}^kA_{ak}^m,\,
 B_{ab}^kC_{ik}^m=A_{kb}^mA_{ia}^k-A_{ka}^mA_{ib}^k.
\end{equation} 
 Then the entries of the matrices
$A_a,$ which are $(A_i^k)_a$, must satisfy the equations $(\ref{RLeibniz})$ obtained from
all the possible Leibniz identities between the triples $\{e_a,e_i,e_j\}.$

Since $N$ is the nilradical of $L$, no nontrivial linear
combination of the matrices $A_a,\,(n+1\leq a\leq p)$ is nilpotent
i.e. the matrices $A_a$ must be ``nil-independent'' \cite{CLOK, Mub3}.

Let us now consider the right multiplication operator $\r_{e_a}$ and restrict it
to $N$, $(n+1\leq a\leq p).$ We shall
get outer derivations of the nilradical $N=\n\i\l(L)$
\cite{CLOK}. Then finding the matrices $A_a$ is the same as
finding outer derivations $\r_{e_a}$ of $N.$ Further the commutators $[\r_{e_b},\r_{e_a}]=\r_{[e_a,e_b]},\,(n+1\leq a,b\leq
p)$ due to $(\ref{Nil})$ consist of inner derivations of
$N$. So those commutators give the structure constants $B_{ab}^k$ as shown in the last equation of $(\ref{RRLeibniz})$ but only up to the elements in the center of the nilradical
 $N$, because if $e_i,\,(1\leq i\leq n)$ is in the center of $N,$ then $\left(\r_{e_i}\right)_{|_N}=0,$ where $\left(\r_{e_i}\right)_{|_{N}}$
 is an inner derivation of the nilradical. 
\subsection{Solvable left Leibniz algebras}\label{Solvable left Leibniz algebras} Satisfying the left Leibniz identity, we have
 \begin{equation}\label{LLeibniz}
A_{ai}^kC_{jk}^m=A_{ja}^kC_{ki}^m+C_{ji}^kA_{ak}^m,\,A_{ia}^kC_{jk}^m=C_{ji}^kA_{ka}^m+A_{ja}^kC_{ik}^m,\,
C_{ij}^kA_{ak}^m=A_{ai}^kC_{kj}^m+A_{aj}^kC_{ik}^m,
\end{equation}
 \begin{equation}\label{LLLeibniz}
 B_{ab}^kC_{ik}^m=A_{ia}^kA_{kb}^m+A_{ib}^kA_{ak}^m,\,A_{ib}^kA_{ak}^m=B_{ab}^kC_{ik}^m+A_{ai}^kA_{kb}^m,\,
 B_{ab}^kC_{ki}^m=A_{ak}^mA_{bi}^k-A_{bk}^mA_{ai}^k,
\end{equation} 
and the entries of the matrices
$A_a,$ which are $(A_a)_i^k$, must satisfy the equations $(\ref{LLeibniz})$.
Similarly $A_a=\left(\L_{e_a}\right)_{|_{N}},\,(n+1\leq a\leq p)$ are outer derivations and
the commutators $[\L_{e_a},\L_{e_b}]=\L_{[e_a,e_b]}$ give the structure constants $B_{ab}^k,$ but only up to the elements in the center of the nilradical
 $N$.

Once the left or right Leibniz identities are satisfied in the most general possible way and the outer derivations are found: 

\begin{enumerate}[noitemsep, topsep=0pt]
 \item[(i)]
We can carry out the technique of ``absorption'' \cite{Shab, ST}, which means we can
simplify a solvable Leibniz algebra without affecting the nilradical in $(\ref{Algebra})$ applying the transformation
\begin{equation}
\nonumber e^{\prime}_i=e_i,\,(1\leq i\leq n),\,e^{\prime}_a=e_a+\sum_{k=1}^nd^ke_k,\,(n+1\leq a\leq p).
\end{equation}
\item[(ii)] We change basis without affecting the
nilradical in $(\ref{Algebra})$ to remove all the possible parameters and simplify the algebra. \end{enumerate}

\section{The nilpotent sequence $\mathcal{L}^4$}\label{NilL4}
In $\mathcal{L}^4,(n\geq4)$ the positive
integer $n$ denotes the dimension of the algebra. The center
of this algebra is $C(\mathcal{L}^4)=\langle e_{2},e_n \rangle$. $\mathcal{L}^4$ can be described explicitly as follows: in the
basis $\{e_1,e_2,e_3,e_4,\ldots,e_n\}$ it has only the following
non-zero brackets:  
\begin{equation}
\begin{array}{l}
\displaystyle  [e_1,e_1]=e_{2},[e_i,e_1]=e_{i+1},(3\leq i\leq n-1),[e_1,e_3]=2e_2-e_4,[e_3,e_3]=e_2,\\
\displaystyle [e_1,e_j]=-e_{j+1},(4\leq j\leq n-1,n\geq4).
\end{array} 
\label{L4}
\end{equation}
 The
dimensions of the ideals in the characteristic
series are
\begin{equation}\nonumber DS=[n,n-2,0],
LS=[n,n-2,n-4,n-5,n-6,...,0],(n\geq4).\end{equation}

A quasi-filiform Leibniz algebra  $\mathcal{L}^4$ was introduced by Camacho, G\'{o}mez, Gonz\'{a}lez and Omirov
in \cite{CGGO}.
This algebra is served as the nilradical for the
left and right solvable indecomposable extensions we construct in this
paper.

 It is shown below that solvable right (left) Leibniz algebras
 with the nilradical $\mathcal{L}^4$ only exist for $\dim\g=n+1$ and $\dim\g=n+2$($\dim l=n+1$ and $\dim l=n+2$).
 
 Right solvable extensions with a codimension
one nilradical $\mathcal{L}^4$ are found following the steps in
Theorems \ref{TheoremRL4}, \ref{TheoremRL4Absorption} and \ref{RL4(Change of Basis)}  with the main result summarized in Theorem \ref{RL4(Change of Basis)}, where
it is shown there are eight such algebras: $\g_{n+1,i},(1\leq i\leq 4,n\geq4),$
$\g_{5,5},\g_{5,6},\g_{5,7}$ and $\g_{5,8}$. It is noticed that $\g_{n+1,1}$ is left as well when $a=0,$
$\g_{5,5}$ is left when $b=-1$ and $\g_{n+1,4},\g_{5,7},\g_{5,8}$ are right and left Leibniz algebras
at the same time. 
There are four solvable indecomposable right Leibniz algebras with a codimension
two nilradical: $\g_{n+2,1},(n\geq5),\g_{6,2},\g_{6,3}$ and $\g_{6,4}$ stated in Theorem \ref{RCodim2L4}, where none of them is left.

We follow the steps in
Theorems \ref{TheoremLL4}, \ref{TheoremLL4Absorption} and \ref{LL4(Change of Basis)} to find codimension one left solvable extensions. 
We notice in Theorem \ref{LL4(Change of Basis)}, that we have eight of them as well: $l_{n+1,1},l_{n+1,2},l_{n+1,3},$
 $\g_{n+1,4},l_{5,5},l_{5,6},\g_{5,7}$ and $\g_{5,8},$ such that $l_{n+1,1}$ is right when $a=0$ and
$l_{5,5}$ is right when $b=-1.$
We find four solvable indecomposable left Leibniz algebras with a codimension two nilradical as well: $l_{n+2,1},(n\geq5),l_{6,2},l_{6,3}$ and $l_{6,4}$ stated in Theorem \ref{LCodim2L4}, where none of them is right.

 \section{Classification of solvable indecomposable Leibniz algebras with
a nilradical $\mathcal{L}^4$}\label{Classification}

Our goal in this Section is to find all possible right and left solvable indecomposable extensions
of the nilpotent Leibniz algebra $\mathcal{L}^4,$ which serves
as the nilradical of the extended algebra.
\begin{remarks}\label{remark5.1} It is assumed throughout this section that
solvable indecomposable right and left Leibniz algebras have the nilradical $\mathcal{L}^4$; however, most of the time, the brackets of the nilradical will be omitted.
\end{remarks}
\subsection{Solvable indecomposable right Leibniz algebras with a nilradical $\mathcal{L}^4$}
\subsubsection{Codimension one solvable extensions of $\mathcal{L}^4$}
 The nilpotent Leibniz algebra $\mathcal{L}^4$ is defined in $(\ref{L4})$. Suppose $\{e_{n+1}\}$
 is in the complementary subspace to the nilradical $\mathcal{L}^4$ and $\g$ is the corresponding solvable right Leibniz algebra.
 Since $[\g,\g]\subseteq \mathcal{L}^4,$ we have the following:
\begin{equation}
\left\{
\begin{array}{l}
\displaystyle [e_1,e_1]=e_{2},[e_i,e_1]=e_{i+1},(3\leq i\leq n-1),[e_1,e_3]=2e_2-e_4,[e_3,e_3]=e_2,\\
\displaystyle [e_1,e_j]=-e_{j+1},(4\leq j\leq n-1),[e_r,e_{n+1}]=\sum_{s=1}^na_{s,r}e_s,[e_{n+1},e_k]=\sum_{s=1}^nb_{s,k}e_s,\\
\displaystyle (1\leq k\leq n,1\leq r\leq n+1,n\geq4).
\end{array} 
\right.
\label{BRLeibniz}
\end{equation}
\begin{theorem}\label{TheoremRL4} We set $a_{1,1}:=a$ and $a_{3,3}:=b$ in $(\ref{BRLeibniz})$. To satisfy the right Leibniz identity, there are the following cases based on the conditions involving parameters,
each gives a continuous family of solvable Leibniz algebras:
\begin{enumerate}[noitemsep, topsep=0pt]
\item[(1)] If $a_{1,3}=0, b\neq-a,a\neq0,b\neq0,(n=4)$ or $b\neq(3-n)a,a\neq0,b\neq0,(n\geq5),$ then we have the following brackets for the algebra:
\begin{equation}
\left\{
\begin{array}{l}
\displaystyle  \nonumber [e_1,e_{n+1}]=ae_1+A_{2,1}e_2+(b-a)e_3+\sum_{k=4}^n{a_{k,1}e_k},[e_2,e_{n+1}]=2be_2,
[e_3,e_{n+1}]=a_{2,3}e_2+be_3+\\
\displaystyle \sum_{k=4}^n{a_{k,3}e_k},[e_4,e_{n+1}]=(b-a)e_2+(a+b)e_4+\sum_{k=5}^n{a_{k-1,3}e_k},
[e_{i},e_{n+1}]=\left((i-3)a+b\right)e_{i}+\\
\displaystyle \sum_{k=i+1}^n{a_{k-i+3,3}e_k},
[e_{n+1},e_{n+1}]=a_{2,n+1}e_2,[e_{n+1},e_1]=-ae_1+b_{2,1}e_2+(a-b)e_3-\sum_{k=4}^n{a_{k,1}e_k},\\
\displaystyle [e_{n+1},e_3]=(a_{2,3}+2a_{4,3})e_2-be_3-\sum_{k=4}^n{a_{k,3}e_k},[e_{n+1},e_4]=(a+b)(e_2-e_4)-\sum_{k=5}^n{a_{k-1,3}e_k},\\
\displaystyle [e_{n+1},e_i]=\left((3-i)a-b\right)e_i-\sum_{k=i+1}^n{a_{k-i+3,3}e_k},(5\leq i\leq n),\\
\displaystyle where\,\,A_{2,1}:=\frac{(2b-a)b_{2,1}-2b\cdot a_{4,1}+2(a-b)(a_{2,3}+a_{4,3})}{a}.
\end{array} 
\right.
\end{equation} 
\item[(2)] If $a_{1,3}=0,b:=-a,a\neq0,(n=4)$ or $b:=(3-n)a,a\neq0,(n\geq5),$
then
\begin{equation}
\left\{
\begin{array}{l}
\displaystyle  \nonumber [e_1,e_{n+1}]=ae_1+A_{2,1}e_2+(2-n)ae_3+\sum_{k=4}^n{a_{k,1}e_k},[e_2,e_{n+1}]=2(3-n)ae_2,\\
\displaystyle [e_3,e_{n+1}]=a_{2,3}e_2+(3-n)ae_3+\sum_{k=4}^n{a_{k,3}e_k},[e_4,e_{n+1}]=(2-n)ae_2+(4-n)ae_4+\\
\displaystyle \sum_{k=5}^n{a_{k-1,3}e_k},[e_{i},e_{n+1}]=(i-n)ae_{i}+\sum_{k=i+1}^n{a_{k-i+3,3}e_k},
[e_{n+1},e_{n+1}]=a_{2,n+1}e_2+a_{n,n+1}e_n,\\
\displaystyle [e_{n+1},e_1]=-ae_1+b_{2,1}e_2+(n-2)ae_3-\sum_{k=4}^n{a_{k,1}e_k},[e_{n+1},e_3]=(a_{2,3}+2a_{4,3})e_2+\\
\displaystyle (n-3)ae_3-\sum_{k=4}^n{a_{k,3}e_k},[e_{n+1},e_4]=(4-n)ae_2+(n-4)ae_4-\sum_{k=5}^n{a_{k-1,3}e_k},\\
\displaystyle [e_{n+1},e_i]=(n-i)ae_i-\sum_{k=i+1}^n{a_{k-i+3,3}e_k},(5\leq i\leq n-1),\\
\displaystyle where\,\,A_{2,1}:=(5-2n)b_{2,1}+2(n-3)a_{4,1}+2(n-2)(a_{2,3}+a_{4,3}).
\end{array} 
\right.
\end{equation} 
 \item[(3)] If $a_{1,3}=0,a=0,b\neq0,(n=4)$ or $a=0$ and $b\neq0,(n\geq5),$ then
\begin{equation}
\left\{
\begin{array}{l}
\displaystyle  \nonumber [e_1,e_{n+1}]=a_{2,1}e_2+be_3+\sum_{k=4}^n{a_{k,1}e_k},[e_2,e_{n+1}]=2be_2,
[e_3,e_{n+1}]=a_{2,3}e_2+be_3+\sum_{k=4}^n{a_{k,3}e_k},\\
\displaystyle [e_4,e_{n+1}]=b(e_2+e_4)+\sum_{k=5}^n{a_{k-1,3}e_k},
[e_{i},e_{n+1}]=be_{i}+\sum_{k=i+1}^n{a_{k-i+3,3}e_k},\\
\displaystyle [e_{n+1},e_{n+1}]=a_{2,n+1}e_2,
[e_{n+1},e_1]=\left(a_{2,3}+a_{4,1}+a_{4,3}\right)e_2-be_3-\sum_{k=4}^n{a_{k,1}e_k},\\
\displaystyle [e_{n+1},e_3]=(a_{2,3}+2a_{4,3})e_2-be_3-\sum_{k=4}^n{a_{k,3}e_k},[e_{n+1},e_4]=b(e_2-e_4)-\sum_{k=5}^n{a_{k-1,3}e_k},\\
\displaystyle [e_{n+1},e_i]=-be_i-\sum_{k=i+1}^n{a_{k-i+3,3}e_k},(5\leq i\leq n).
\end{array} 
\right.
\end{equation} 
   \allowdisplaybreaks
\item[(4)] If $a_{1,3}=0,b=0,a\neq0,(n=4)$ or $b=0,a\neq0,(n\geq5),$ then
\begin{equation}
\left\{
\begin{array}{l}
\displaystyle  \nonumber [e_1,e_{n+1}]=ae_1+\left(a_{2,3}-b_{2,1}+b_{2,3}\right)e_2-ae_3+\sum_{k=4}^n{a_{k,1}e_k},
[e_3,e_{n+1}]=a_{2,3}e_2+\sum_{k=4}^n{a_{k,3}e_k},\\
\displaystyle [e_4,e_{n+1}]=-ae_2+ae_4+\sum_{k=5}^n{a_{k-1,3}e_k},
[e_{i},e_{n+1}]=(i-3)ae_{i}+\sum_{k=i+1}^n{a_{k-i+3,3}e_k},\\
\displaystyle 
[e_{n+1},e_{n+1}]=a_{2,n+1}e_2,[e_{n+1},e_1]=-ae_1+b_{2,1}e_2+ae_3-\sum_{k=4}^n{a_{k,1}e_k},[e_{n+1},e_3]=b_{2,3}e_2-\\
\displaystyle \sum_{k=4}^n{a_{k,3}e_k},[e_{n+1},e_4]=ae_2-ae_4-\sum_{k=5}^n{a_{k-1,3}e_k},[e_{n+1},e_i]=(3-i)ae_i-\sum_{k=i+1}^n{a_{k-i+3,3}e_k},\\
\displaystyle (5\leq i\leq n).
\end{array} 
\right.
\end{equation} 
In the remaining cases $a_{1,3}:=c.$
\item[(5)] If $b\neq-a,a\neq0,b\neq-c,c\neq0,$ then
\begin{equation}
\left\{
\begin{array}{l}
\displaystyle  \nonumber [e_1,e_{5}]=ae_1+A_{2,1}e_2+(b+c-a)e_3+a_{4,1}e_4,[e_2,e_{5}]=2(b+c)e_2,
[e_3,e_{5}]=ce_1+a_{2,3}e_2+\\
\displaystyle be_3+A_{4,3}e_4,[e_4,e_{5}]=(2c+b-a)e_2+(a+b)e_4,[e_{5},e_{5}]=a_{2,5}e_2,[e_{5},e_1]=-ae_1+b_{2,1}e_2+\\
\displaystyle (a-b-c)e_3-a_{4,1}e_4,[e_{5},e_3]=-ce_1+b_{2,3}e_2-be_3-A_{4,3}e_4,[e_{5},e_4]=(a+b)e_2-(a+b)e_4,\\
\displaystyle where\,\,A_{2,1}:=a_{2,3}-b_{2,1}+b_{2,3}-\frac{(b+c)(a_{2,3}-2b_{2,1}+2a_{4,1}+b_{2,3})}{a}\,\,and \\
\displaystyle 
A_{4,3}:=\frac{(c-a)a_{2,3}+(a+c)b_{2,3}-2c(b_{2,1}-a_{4,1})}{2a},
\end{array} 
\right.
\end{equation} 
$\r_{e_{5}}=\left[\begin{smallmatrix}
 a & 0 & c & 0\\
  A_{2,1} & 2(b+c) & a_{2,3}& 2c+b-a \\
  b+c-a & 0 & b & 0\\
  a_{4,1} & 0 &  A_{4,3} & a+b
\end{smallmatrix}\right].$
\item[(6)] If $b:=-a,a\neq0,a\neq c,c\neq0,$ then
\begin{equation}
\left\{
\begin{array}{l}
\displaystyle  \nonumber [e_1,e_{5}]=ae_1+A_{2,1}e_2+(c-2a)e_3+a_{4,1}e_4,[e_2,e_{5}]=2(c-a)e_2,
[e_3,e_{5}]=ce_1+a_{2,3}e_2-\\
\displaystyle ae_3+A_{4,3}e_4,[e_4,e_{5}]=2(c-a)e_2,[e_{5},e_{5}]=a_{2,5}e_2+a_{4,5}e_4,[e_{5},e_1]=-ae_1+b_{2,1}e_2+\\
\displaystyle (2a-c)e_3-a_{4,1}e_4,[e_{5},e_3]=-ce_1+b_{2,3}e_2+ae_3-A_{4,3}e_4,\\
\displaystyle where\,\,A_{2,1}:=2a_{2,3}-3b_{2,1}+2b_{2,3}+2a_{4,1}-\frac{c(a_{2,3}-2b_{2,1}+2a_{4,1}+b_{2,3})}{a}\,\,and \\
\displaystyle 
A_{4,3}:=\frac{(c-a)a_{2,3}+(a+c)b_{2,3}-2c(b_{2,1}-a_{4,1})}{2a},
\end{array} 
\right.
\end{equation} 
$\r_{e_{5}}=\left[\begin{smallmatrix}
 a & 0 & c & 0\\
  A_{2,1} & 2(c-a) & a_{2,3}& 2(c-a) \\
 c-2a & 0 &-a & 0\\
  a_{4,1} & 0 &  A_{4,3} &0
\end{smallmatrix}\right].$
\item[(7)] If $b:=-c,c\neq0,a\neq c,a\neq0,$ then
\begin{equation}
\left\{
\begin{array}{l}
\displaystyle  \nonumber [e_1,e_{5}]=ae_1+\left(a_{2,3}-b_{2,1}+b_{2,3}\right)e_2-ae_3+a_{4,1}e_4,
[e_3,e_{5}]=ce_1+a_{2,3}e_2-ce_3+a_{4,3}e_4,\\
\displaystyle [e_4,e_{5}]=(c-a)e_2+(a-c)e_4,[e_{5},e_{5}]=a_{2,5}e_2,[e_{5},e_1]=-ae_1+b_{2,1}e_2+ae_3-a_{4,1}e_4,\\
\displaystyle [e_{5},e_3]=-ce_1+b_{2,3}e_2+ce_3-a_{4,3}e_4,[e_{5},e_4]=(a-c)e_2+(c-a)e_4,
\end{array} 
\right.
\end{equation} 
$\r_{e_{5}}=\left[\begin{smallmatrix}
 a & 0 & c & 0\\
  a_{2,3}-b_{2,1}+b_{2,3} & 0 & a_{2,3}& c-a \\
  -a & 0 & -c & 0\\
  a_{4,1} & 0 &  a_{4,3} & a-c
\end{smallmatrix}\right].$
\item[(8)] If $c:=a,b:=-a,a\neq0,$ then\footnote{The outer derivation $\r_{e_{5}}$ is nilpotent, so we do not consider this case any further.}
\begin{equation}
\left\{
\begin{array}{l}
\displaystyle  \nonumber [e_1,e_{5}]=ae_1+\left(a_{2,3}-b_{2,1}+b_{2,3}\right)e_2-ae_3+a_{4,1}e_4,
[e_3,e_{5}]=ae_1+a_{2,3}e_2-ae_3+a_{4,3}e_4,\\
\displaystyle [e_{5},e_{5}]=a_{2,5}e_2+a_{4,5}e_4,[e_{5},e_1]=-ae_1+b_{2,1}e_2+ae_3-\left(a_{4,1}-a_{4,3}-b_{4,3}\right)e_4,\\
\displaystyle [e_{5},e_3]=-ae_1+b_{2,3}e_2+ae_3+b_{4,3}e_4,
\end{array} 
\right.
\end{equation} 
$\r_{e_{5}}=\left[\begin{smallmatrix}
 a & 0 & a & 0\\
 a_{2,3}-b_{2,1}+b_{2,3}& 0& a_{2,3}& 0 \\
  -a & 0 & -a & 0\\
  a_{4,1} & 0 &  a_{4,3} & 0
\end{smallmatrix}\right].$
\item[(9)] If $a=0,b=0,c\neq0,$ then
\begin{equation}
\left\{
\begin{array}{l}
\displaystyle  \nonumber [e_1,e_{5}]=\left(3b_{2,1}-4a_{4,1}-2a_{2,3}-2a_{4,3}\right)e_2+ce_3+a_{4,1}e_4,[e_2,e_{5}]=2ce_2,
[e_3,e_{5}]=ce_1+\\
\displaystyle a_{2,3}e_2+a_{4,3}e_4,[e_4,e_{5}]=2ce_2,[e_{5},e_{5}]=a_{2,5}e_2+a_{4,5}e_4,[e_{5},e_1]=b_{2,1}e_2-ce_3-a_{4,1}e_4,\\
\displaystyle [e_{5},e_3]=-ce_1+\left(2b_{2,1}-2a_{4,1}-a_{2,3}\right)e_2-a_{4,3}e_4,
\end{array} 
\right.
\end{equation} 
$\r_{e_{5}}=\left[\begin{smallmatrix}
 0 & 0 & c & 0\\
  3b_{2,1}-4a_{4,1}-2a_{2,3}-2a_{4,3} & 2c & a_{2,3}& 2c\\
  c & 0 & 0 & 0\\
  a_{4,1} & 0 &  a_{4,3} & 0
\end{smallmatrix}\right].$
\item[(10)] If $a=0,b\neq0,b\neq-c,c\neq0,$ then
\begin{equation}
\left\{
\begin{array}{l}
\displaystyle  \nonumber [e_1,e_{5}]=a_{2,1}e_2+(b+c)e_3+a_{4,1}e_4,[e_2,e_{5}]=2(b+c)e_2,
[e_3,e_{5}]=ce_1+a_{2,3}e_2+be_3+\\
\displaystyle A_{4,3}e_4,[e_4,e_{5}]=(2c+b)e_2+be_4,[e_{5},e_{5}]=a_{2,5}e_2,[e_{5},e_1]=\left(a_{4,1}+\frac{a_{2,3}+b_{2,3}}{2}\right)e_2-\\
\displaystyle (b+c)e_3-a_{4,1}e_4,[e_{5},e_3]=-ce_1+b_{2,3}e_2-be_3-A_{4,3}e_4,[e_{5},e_4]=b(e_2-e_4),\\
\displaystyle where\,\,A_{4,3}:=\frac{(3c+2b)b_{2,3}-(2b+c)a_{2,3}-2c(a_{2,1}+a_{4,1})}{4(b+c)},
\end{array} 
\right.
\end{equation} 
$\r_{e_{5}}=\left[\begin{smallmatrix}
 0 & 0 & c & 0\\
  a_{2,1} & 2(b+c) & a_{2,3}& 2c+b \\
  b+c & 0 & b & 0\\
  a_{4,1} & 0 &  A_{4,3} & b
\end{smallmatrix}\right].$
\item[(11)] If $a=0,b:=-c,c\neq0,$ then
\begin{equation}
\left\{
\begin{array}{l}
\displaystyle  \nonumber [e_1,e_{5}]=a_{2,1}e_2+a_{4,1}e_4,
[e_3,e_{5}]=ce_1+a_{2,3}e_2-ce_3+a_{4,3}e_4,[e_4,e_{5}]=ce_2-ce_4,\\
\displaystyle [e_{5},e_{5}]=a_{2,5}e_2,[e_{5},e_1]=\left(a_{2,3}+b_{2,3}-a_{2,1}\right)e_2-a_{4,1}e_4,[e_{5},e_3]=-ce_1+b_{2,3}e_2+ce_3-\\
\displaystyle a_{4,3}e_4,[e_{5},e_4]=-ce_2+ce_4,
\end{array} 
\right.
\end{equation} 
$\r_{e_{5}}=\left[\begin{smallmatrix}
 0 & 0 & c & 0\\
  a_{2,1} & 0 & a_{2,3}& c \\
  0 & 0 &-c & 0\\
  a_{4,1} & 0 &  a_{4,3} & -c
\end{smallmatrix}\right].$
\end{enumerate}
\end{theorem}
\vskip 5pt
\begin{proof}
\begin{enumerate}[noitemsep, topsep=0pt]
\item[(1)] For $(n\geq5)$ 
 the proof is given in Table \ref{Right(L4)}. For $(n=4)$ we work out the following identities: $1.,3.-5.,7.,8.,10.,12.,13.$(or $15.$), $16.-19.$
 \item[(2)] For $(n\geq5),$ we apply the same identities
given in Table \ref{Right(L4)}, except $17.$ For $(n=4)$,
the identities are as follows: $1.,3.-5.,7.,8.,10.,12.,15.,16.-19.$
 \item[(3)] For $(n\geq5),$ same identities except $18.$ and $19.$
applying instead
$\r_{e_3}\left([e_{n+1},e_{n+1}]\right)=[\r_{e_3}(e_{n+1}),e_{n+1}]+[e_{n+1},\r_{e_3}(e_{n+1})]$ and
$\r[e_{n+1},e_{1}]=[\r(e_{n+1}),e_{1}]+[e_{n+1},\r(e_{1})]$.
For $(n=4)$, same identities as in case $(1),$ except $18.$ and $19.$
applying two identities given above.
\item[(4)] Same identities as in case $(1).$
\item[(5)] We apply the following identities:
$1.,3.-5.,7.,8.,10.,12.,13.,16.,17.,
\r_{e_3}\left([e_5,e_5]\right)=[\r_{e_3}(e_5),e_5]+[e_5,\r_{e_3}(e_5)],18.,19.$
\item[(6)] We apply the same identities as in $(5)$,
except $17.$ and $18.$
\item[(7)] Same as in $(5),$ except $19.$
\item[(8)] We apply the same identities as in $(5),$ except $17.,18.$ and $19.$
\item[(9)] Same identities as in $(5),$ except $17.$ and $18.$
\item[(10)] We have the identities:
$1.,3.-5.,7.,8.,10.,12.,13.,16.,17.,19., \r_{e_5}\left([e_5,e_1]\right)=[\r_{e_5}(e_5),e_1]+[e_5,\r_{e_5}(e_1)],
\r_{e_3}\left([e_5,e_5]\right)=[\r_{e_3}(e_5),e_5]+[e_5,\r_{e_3}(e_5)].$
\item[(11)] Same identities as in $(5),$ except $19.$
\end{enumerate}
\end{proof} 
\begin{table}[h!]
\caption{Right Leibniz identities in case $(1)$ in Theorem \ref{TheoremRL4}, ($n\geq5$).}
\label{Right(L4)}
\begin{tabular}{lp{2.4cm}p{12cm}}
\hline
\scriptsize Steps &\scriptsize Ordered triple &\scriptsize
Result\\ \hline
\scriptsize $1.$ &\scriptsize $\r[e_1,e_{1}]$ &\scriptsize
$a_{1,2}=0,a_{3,1}:=\frac{1}{2}a_{2,2}-a,a_{k,2}=0,(3\leq k\leq n)$
$\implies$ $[e_1,e_{n+1}]=ae_1+a_{2,1}e_2+\left(\frac{1}{2}a_{2,2}-a\right)e_3+\sum_{k=4}^n{a_{k,1}e_k},[e_2,e_{n+1}]=a_{2,2}e_2.$\\ \hline
\scriptsize $2.$ &\scriptsize $\r[e_{3},e_{i}]$ &\scriptsize
$a_{1,3}=a_{1,i}=a_{3,i}=0$
$\implies$ $[e_3,e_{n+1}]=a_{2,3}e_2+be_3+\sum_{k=4}^n{a_{k,3}e_k},[e_i,e_{n+1}]=a_{2,i}e_2+\sum_{k=4}^n{a_{k,i}e_k},(4\leq i\leq n-1).$\\ \hline
\scriptsize $3.$ &\scriptsize $\r[e_{3},e_{n}]$ &\scriptsize
$a_{1,n}=a_{3,n}=0$
$\implies$ $[e_{n},e_{n+1}]=a_{2,n}e_2+\sum_{k=4}^n{a_{k,n}e_k}.$\\ \hline
\scriptsize $4.$ &\scriptsize $\r[e_{3},e_{3}]$ &\scriptsize
$a_{2,2}:=2b$
$\implies$ $[e_{1},e_{n+1}]=ae_1+a_{2,1}e_2+(b-a)e_3+\sum_{k=4}^n{a_{k,1}e_k},
[e_2,e_{n+1}]=2be_2.$\\ \hline
\scriptsize $5.$ &\scriptsize $\r[e_1,e_{3}]$ &\scriptsize
$a_{2,4}:=b-a,a_{4,4}:=a+b,a_{k,4}:=a_{k-1,3},(5\leq k\leq n)$
$\implies$ $[e_4,e_{n+1}]=(b-a)e_2+(a+b)e_4+\sum_{k=5}^n{a_{k-1,3}e_k}.$\\ \hline
\scriptsize $6.$ &\scriptsize $\r[e_{1},e_{i}]$ &\scriptsize
$a_{2,i+1}=a_{4,i+1}=0,a_{i+1,i+1}:=(i-2)a+b,a_{k,i+1}:=a_{k-1,i},(5\leq k\leq n,k\neq i+1,4\leq i\leq n-1),$
where $i$ is fixed
$\implies$ $[e_{j},e_{n+1}]=\left((j-3)a+b\right)e_j+\sum_{k=j+1}^n{a_{k-j+3,3}e_k},(5\leq j\leq n).$\\ \hline
\scriptsize $7.$ &\scriptsize $\r_{e_1}\left([e_{n+1},e_{1}]\right)$ &\scriptsize
$[e_{n+1},e_2]=0$
$\implies$ $b_{k,2}=0,(1\leq k\leq n).$\\ \hline
\scriptsize $8.$ &\scriptsize $\r_{e_1}\left([e_3,e_{n+1}]\right)$ &\scriptsize
$b_{1,1}:=-a,b_{3,1}:=a-b$
$\implies$ $[e_{n+1},e_1]=-ae_1+b_{2,1}e_2+(a-b)e_3+\sum_{k=4}^{n}b_{k,1}e_k.$\\ \hline
\scriptsize $9.$ &\scriptsize $\r_{e_1}\left([e_{1},e_{n+1}]\right)$ &\scriptsize
$b_{k-1,1}:=-a_{k-1,1},(5\leq k\leq n,n\geq5)$
$\implies$ $[e_{n+1},e_1]=-ae_1+b_{2,1}e_2+(a-b)e_3-\sum_{k=4}^{n-1}{a_{k,1}e_k}+b_{n,1}e_n.$ \\ \hline
\scriptsize $10.$ &\scriptsize $\r_{e_3}\left([e_{1},e_{n+1}]\right)$ &\scriptsize
$b_{1,3}=0,b_{3,3}:=-b,b_{k-1,3}:=-a_{k-1,3},(5\leq k\leq n)$
$\implies$ $[e_{n+1},e_{3}]=b_{2,3}e_2-be_3-\sum_{k=4}^{n-1}{a_{k,3}e_k}+b_{n,3}e_n.$\\ \hline
\scriptsize $11.$ &\scriptsize $\r_{e_i}\left([e_{1},e_{n+1}]\right)$ &\scriptsize
$b_{3,i}=0\implies b_{1,i}=0;b_{i,i}:=(3-i)a-b, b_{k-1,i}=0,(5\leq k\leq i),$ $b_{k-1,i}:=-a_{k-i+2,3},(i+2\leq k\leq n),$
where $i$ is fixed.
$\implies$  $[e_{n+1},e_i]=b_{2,i}e_2+\left((3-i)a-b\right)e_i-\sum_{k=i+1}^{n-1}{a_{k-i+3,3}e_k}+b_{n,i}e_n,(4\leq i\leq n-1).$ \\ \hline
\scriptsize $12.$ &\scriptsize $\r_{e_{n}}\left([e_{1},e_{n+1}]\right)$ &\scriptsize
$b_{3,n}=0\implies b_{1,n}=0;b_{k-1,n}=0,(5\leq k\leq n)$
$\implies$ $[e_{n+1},e_{n}]=b_{2,n}e_2+b_{n,n}e_n.$\\ \hline
\scriptsize $13.$ &\scriptsize $\r_{e_{1}}\left([e_{n+1},e_{3}]\right)$ &\scriptsize
$b_{2,4}:=a+b,b_{n,4}:=-a_{n-1,3}$
$\implies$ $[e_{n+1},e_{4}]=(a+b)e_2-(a+b)e_4-\sum_{k=5}^n{a_{k-1,3}e_k}.$\\ \hline
\scriptsize $14.$ &\scriptsize $\r_{e_{1}}\left([e_{n+1},e_{i}]\right)$ &\scriptsize
$b_{2,i+1}=0,b_{n,i+1}:=-a_{n-i+2,3},(4\leq i\leq n-2)$
$\implies$ $[e_{n+1},e_{j}]=\left((3-j)a-b\right)e_j-\sum_{k=j+1}^n{a_{k-j+3,3}e_k},(5\leq j\leq n-1).$\\ \hline
\scriptsize $15.$ &\scriptsize $\r_{e_{1}}\left([e_{n+1},e_{n-1}]\right)$ &\scriptsize
$b_{2,n}=0,b_{n,n}:=(3-n)a-b$
$\implies$ $[e_{n+1},e_{n}]=\left((3-n)a-b\right)e_n.$ Combining with $14.,$
$[e_{n+1},e_{i}]=\left((3-i)a-b\right)e_i-\sum_{k=i+1}^n{a_{k-i+3,3}e_k},(5\leq i\leq n).$\\ \hline
\scriptsize $16.$ &\scriptsize $\r[e_{1},e_{n+1}]$ &\scriptsize
$a_{3,n+1}=0\implies a_{1,n+1}=0; a_{k-1,n+1}=0,(5\leq k\leq n)$
$\implies$ $[e_{n+1},e_{n+1}]=a_{2,n+1}e_2+a_{n,n+1}e_n.$\\ \hline
\scriptsize $17.$ &\scriptsize $\r[e_{n+1},e_{n+1}]$ &\scriptsize
$a_{n,n+1}=0,(b\neq(3-n)a)$
$\implies$ $[e_{n+1},e_{n+1}]=a_{2,n+1}e_2.$\\ \hline
\scriptsize $18.$ &\scriptsize $\r[e_{n+1},e_{3}]$ &\scriptsize
$b_{n,3}:=-a_{n,3},(a\neq0),b_{2,3}:=a_{2,3}+2a_{4,3},(b\neq0)$
$\implies$ $[e_{n+1},e_{3}]=\left(a_{2,3}+2a_{4,3}\right)e_2-be_3-\sum_{k=4}^{n}{a_{k,3}e_k}.$\\ \hline
\scriptsize $19.$ &\scriptsize $\r_{e_1}\left([e_{n+1},e_{n+1}]\right)$ &\scriptsize
$b_{n,1}:=-a_{n,1},(a\neq0),A_{2,1}:=\frac{(2b-a)b_{2,1}-2b\cdot a_{4,1}+2(a-b)(a_{2,3}+a_{4,3})}{a}\implies$
$[e_{n+1},e_1]=-ae_1+b_{2,1}e_2+(a-b)e_3-\sum_{k=4}^{n}{a_{k,1}e_k},
[e_{1},e_{n+1}]=ae_1+A_{2,1}e_2+(b-a)e_3+\sum_{k=4}^n{a_{k,1}e_k}.$
\\ \hline
\end{tabular}
\end{table}
\begin{theorem}\label{TheoremRL4Absorption} Applying the technique of ``absorption'' (see Section \ref{Solvable left Leibniz algebras}), we can further simplify the algebras 
in each of the cases in Theorem \ref{TheoremRL4} as follows:
\begin{enumerate}[noitemsep, topsep=0pt]
\allowdisplaybreaks
\item[(1)] If $a_{1,3}=0, b\neq-a,a\neq0,b\neq0,(n=4)$ or $b\neq(3-n)a,a\neq0,b\neq0,(n\geq5),$ then we have the following brackets for the algebra:
\begin{equation}
\left\{
\begin{array}{l}
\displaystyle  \nonumber [e_1,e_{n+1}]=ae_1+\mathcal{A}_{2,1}e_2+(b-a)e_3,[e_2,e_{n+1}]=2be_2,[e_3,e_{n+1}]=a_{2,3}e_2+be_3+\sum_{k=5}^n{a_{k,3}e_k},\\
\displaystyle [e_4,e_{n+1}]=(b-a)e_2+(a+b)e_4+\sum_{k=6}^n{a_{k-1,3}e_k},[e_{i},e_{n+1}]=\left((i-3)a+b\right)e_{i}+\\
\displaystyle \sum_{k=i+2}^n{a_{k-i+3,3}e_k},[e_{n+1},e_1]=-ae_1+b_{2,1}e_2+(a-b)e_3,[e_{n+1},e_3]=a_{2,3}e_2-be_3-\sum_{k=5}^n{a_{k,3}e_k},\\
\displaystyle [e_{n+1},e_4]=(a+b)\left(e_2-e_4\right)-\sum_{k=6}^n{a_{k-1,3}e_k},[e_{n+1},e_i]=\left((3-i)a-b\right)e_i-\sum_{k=i+2}^n{a_{k-i+3,3}e_k},\\
\displaystyle (5\leq i\leq n); where\,\,\mathcal{A}_{2,1}:=\frac{(2b-a)b_{2,1}+2(a-b)a_{2,3}}{a}.
\end{array} 
\right.
\end{equation} 
\item[(2)] If $a_{1,3}=0,b:=-a,a\neq0,(n=4)$ or $b:=(3-n)a,a\neq0,(n\geq5),$
then the brackets for the algebra are  
\begin{equation}
\left\{
\begin{array}{l}
\displaystyle  \nonumber [e_1,e_{n+1}]=ae_1+\mathcal{A}_{2,1}e_2+(2-n)ae_3,[e_2,e_{n+1}]=2(3-n)ae_2,[e_3,e_{n+1}]=a_{2,3}e_2+\\
\displaystyle (3-n)ae_3+\sum_{k=5}^n{a_{k,3}e_k},[e_4,e_{n+1}]=(2-n)ae_2+(4-n)ae_4+\sum_{k=6}^n{a_{k-1,3}e_k},\\
\displaystyle [e_{i},e_{n+1}]=(i-n)ae_{i}+\sum_{k=i+2}^n{a_{k-i+3,3}e_k},[e_{n+1},e_{n+1}]=a_{n,n+1}e_n,\\
\displaystyle [e_{n+1},e_1]=-ae_1+b_{2,1}e_2+(n-2)ae_3,[e_{n+1},e_3]=a_{2,3}e_2+(n-3)ae_3-\sum_{k=5}^n{a_{k,3}e_k},\\
\displaystyle[e_{n+1},e_4]=(4-n)a(e_2-e_4)-\sum_{k=6}^n{a_{k-1,3}e_k},[e_{n+1},e_i]=(n-i)ae_i-\sum_{k=i+2}^n{a_{k-i+3,3}e_k},\\
\displaystyle(5\leq i\leq n);where\,\,\mathcal{A}_{2,1}:=(5-2n)b_{2,1}+2(n-2)a_{2,3}.
\end{array} 
\right.
\end{equation} 
\item[(3)] If $a_{1,3}=0,a=0,b\neq0,(n=4)$ or $a=0$ and $b\neq0,(n\geq5),$ then
\begin{equation}
\left\{
\begin{array}{l}
\displaystyle  \nonumber [e_1,e_{n+1}]=a_{2,1}e_2+be_3,[e_2,e_{n+1}]=2be_2,
[e_3,e_{n+1}]=a_{2,3}e_2+be_3+\sum_{k=5}^n{a_{k,3}e_k},\\
\displaystyle [e_4,e_{n+1}]=b\left(e_2+e_4\right)+\sum_{k=6}^n{a_{k-1,3}e_k},
[e_{i},e_{n+1}]=be_{i}+\sum_{k=i+2}^n{a_{k-i+3,3}e_k},\\
\displaystyle 
[e_{n+1},e_1]=a_{2,3}e_2-be_3,[e_{n+1},e_3]=a_{2,3}e_2-be_3-\sum_{k=5}^n{a_{k,3}e_k},\\
\displaystyle [e_{n+1},e_4]=b\left(e_2-e_4\right)-\sum_{k=6}^n{a_{k-1,3}e_k},[e_{n+1},e_i]=-be_i-\sum_{k=i+2}^n{a_{k-i+3,3}e_k},(5\leq i\leq n).
\end{array} 
\right.
\end{equation} 
\allowdisplaybreaks
\item[(4)] If $a_{1,3}=0,b=0,a\neq0,(n=4)$ or $b=0,a\neq0,(n\geq5),$ then
\begin{equation}
\left\{
\begin{array}{l}
\displaystyle  \nonumber [e_1,e_{n+1}]=ae_1+\left(a_{2,3}-b_{2,1}+b_{2,3}\right)e_2-ae_3,
[e_3,e_{n+1}]=a_{2,3}e_2+\sum_{k=5}^n{a_{k,3}e_k},\\
\displaystyle [e_4,e_{n+1}]=-a(e_2-e_4)+\sum_{k=6}^n{a_{k-1,3}e_k},
[e_{i},e_{n+1}]=(i-3)ae_{i}+\sum_{k=i+2}^n{a_{k-i+3,3}e_k},\\
\displaystyle 
[e_{n+1},e_{n+1}]=a_{2,n+1}e_2,[e_{n+1},e_1]=-ae_1+b_{2,1}e_2+ae_3,[e_{n+1},e_3]=b_{2,3}e_2-\sum_{k=5}^n{a_{k,3}e_k},\\
\displaystyle [e_{n+1},e_4]=a(e_2-e_4)-\sum_{k=6}^n{a_{k-1,3}e_k},[e_{n+1},e_i]=(3-i)ae_i-\sum_{k=i+2}^n{a_{k-i+3,3}e_k},(5\leq i\leq n).
\end{array} 
\right.
\end{equation} 
In the remaining cases $a_{1,3}:=c.$
\item[(5)] If $b\neq-a,a\neq0,b\neq-c,c\neq0,$ then
\begin{equation}
\left\{
\begin{array}{l}
\displaystyle  \nonumber [e_1,e_{5}]=ae_1+\mathcal{A}_{2,1}e_2+(b+c-a)e_3,[e_2,e_{5}]=2(b+c)e_2,
[e_3,e_{5}]=ce_1+a_{2,3}e_2+be_3,\\
\displaystyle [e_4,e_{5}]=(2c+b-a)e_2+(a+b)e_4,[e_{5},e_1]=-ae_1+\mathcal{B}_{2,1}e_2+(a-b-c)e_3,\\
\displaystyle [e_{5},e_3]=-ce_1+\mathcal{B}_{2,3}e_2-be_3,[e_{5},e_4]=(a+b)(e_2-e_4),\\
\displaystyle where\,\,\mathcal{A}_{2,1}:=\frac{(3a-2b-3c)a_{2,3}+(a-2b-3c)b_{2,3}}{2a},\,\,\mathcal{B}_{2,1}:=\frac{(a-c)a_{2,3}-(a+c)b_{2,3}}{2a},\\
\displaystyle 
\mathcal{B}_{2,3}:=\frac{(a-c)a_{2,3}-c\cdot b_{2,3}}{a},
\end{array} 
\right.
\end{equation} 
$\r_{e_{5}}=\left[\begin{smallmatrix}
 a & 0 & c & 0\\
 \mathcal{A}_{2,1} & 2(b+c) & a_{2,3}& 2c+b-a \\
  b+c-a & 0 & b & 0\\
 0 & 0 &  0 & a+b
\end{smallmatrix}\right].$
\item[(6)] If $b:=-a,a\neq0,a\neq c,c\neq0,$ then
\begin{equation}
\left\{
\begin{array}{l}
\displaystyle  \nonumber [e_1,e_{5}]=ae_1+\mathcal{A}_{2,1}e_2+(c-2a)e_3,[e_2,e_{5}]=2(c-a)e_2,
[e_3,e_{5}]=ce_1+a_{2,3}e_2-ae_3,\\
\displaystyle [e_4,e_{5}]=2(c-a)e_2,[e_{5},e_{5}]=a_{4,5}e_4,[e_{5},e_1]=-ae_1+\mathcal{B}_{2,1}e_2+(2a-c)e_3,\\
\displaystyle [e_{5},e_3]=-ce_1+\mathcal{B}_{2,3}e_2+ae_3;where\,\,\mathcal{A}_{2,1}:=\frac{(5a-3c)a_{2,3}+(3a-3c)b_{2,3}}{2a},\\
\displaystyle \mathcal{B}_{2,1}:=\frac{(a-c)a_{2,3}-(a+c)b_{2,3}}{2a},\mathcal{B}_{2,3}:=\frac{(a-c)a_{2,3}-c\cdot b_{2,3}}{a},
\end{array} 
\right.
\end{equation} 
$\r_{e_{5}}=\left[\begin{smallmatrix}
 a & 0 & c & 0\\
  \mathcal{A}_{2,1} & 2(c-a) &a_{2,3}& 2(c-a) \\
 c-2a & 0 &-a & 0\\
  0 & 0 &  0 &0
\end{smallmatrix}\right].$
\item[(7)] If $b:=-c,c\neq0,a\neq c,a\neq0,$ then
\begin{equation}
\left\{
\begin{array}{l}
\displaystyle  \nonumber [e_1,e_{5}]=ae_1+\left(a_{2,3}-b_{2,1}+b_{2,3}\right)e_2-ae_3,
[e_3,e_{5}]=ce_1+a_{2,3}e_2-ce_3,\\
\displaystyle [e_4,e_{5}]=(c-a)\left(e_2-e_4\right),[e_{5},e_{5}]=a_{2,5}e_2,[e_{5},e_1]=-ae_1+b_{2,1}e_2+ae_3,\\
\displaystyle [e_{5},e_3]=-ce_1+b_{2,3}e_2+ce_3,[e_{5},e_4]=(a-c)\left(e_2-e_4\right),
\end{array} 
\right.
\end{equation} 
$\r_{e_{5}}=\left[\begin{smallmatrix}
 a & 0 & c & 0\\
  a_{2,3}-b_{2,1}+b_{2,3} & 0 &a_{2,3}& c-a \\
  -a & 0 & -c & 0\\
 0 & 0 & 0 & a-c
\end{smallmatrix}\right].$
\item[(8)] If $a=0,b=0,c\neq0,$ then
\begin{equation}
\left\{
\begin{array}{l}
\displaystyle  \nonumber [e_1,e_{5}]=\left(3b_{2,1}-2a_{2,3}\right)e_2+ce_3,[e_2,e_{5}]=2ce_2,
[e_3,e_{5}]=ce_1+a_{2,3}e_2,[e_4,e_{5}]=2ce_2,\\
\displaystyle [e_{5},e_{5}]=a_{4,5}e_4,[e_{5},e_1]=b_{2,1}e_2-ce_3, [e_{5},e_3]=-ce_1+\left(2b_{2,1}-a_{2,3}\right)e_2,
\end{array} 
\right.
\end{equation} 
$\r_{e_{5}}=\left[\begin{smallmatrix}
 0 & 0 & c & 0\\
  3b_{2,1}-2a_{2,3}& 2c & a_{2,3}& 2c\\
  c & 0 & 0 & 0\\
  0 & 0 &  0 & 0
\end{smallmatrix}\right].$
\item[(9)] If $a=0,b\neq0,b\neq-c,c\neq0,$ then
\begin{equation}
\left\{
\begin{array}{l}
\displaystyle  \nonumber [e_1,e_{5}]=\mathcal{A}_{2,1}e_2+(b+c)e_3,[e_2,e_{5}]=2(b+c)e_2,
[e_3,e_{5}]=ce_1+a_{2,3}e_2+be_3,\\
\displaystyle [e_4,e_{5}]=(2c+b)e_2+be_4,[e_{5},e_1]=\mathcal{B}_{2,1}e_2-(b+c)e_3,[e_{5},e_3]=-ce_1+\mathcal{B}_{2,3}e_2-be_3,\\
\displaystyle [e_{5},e_4]=b(e_2-e_4); where\,\,\mathcal{A}_{2,1}:=\frac{(2b+c)a_{2,3}-(2b+3c)b_{2,3}}{4(b+c)},\\
\displaystyle 
\mathcal{B}_{2,1}:=\frac{(4b+3c)a_{2,3}-c\cdot b_{2,3}}{4(b+c)},\mathcal{B}_{2,3}:=\frac{(2b+c)a_{2,3}-c\cdot b_{2,3}}{2(b+c)},
\end{array} 
\right.
\end{equation} 
$\r_{e_{5}}=\left[\begin{smallmatrix}
 0 & 0 & c & 0\\
  \mathcal{A}_{2,1}& 2(b+c) & a_{2,3}& 2c+b \\
  b+c & 0 & b & 0\\
  0 & 0 & 0 & b
\end{smallmatrix}\right].$
\item[(10)] If $a=0,b:=-c,c\neq0,$ then
\begin{equation}
\left\{
\begin{array}{l}
\displaystyle  \nonumber [e_1,e_{5}]=a_{2,1}e_2,
[e_3,e_{5}]=ce_1+a_{2,3}e_2-ce_3,[e_4,e_{5}]=c(e_2-e_4),[e_{5},e_{5}]=a_{2,5}e_2,\\
\displaystyle [e_{5},e_1]=\left(a_{2,3}+b_{2,3}-a_{2,1}\right)e_2,[e_{5},e_3]=-ce_1+b_{2,3}e_2+ce_3,[e_{5},e_4]=-c(e_2-e_4),
\end{array} 
\right.
\end{equation} 
$\r_{e_{5}}=\left[\begin{smallmatrix}
 0 & 0 & c & 0\\
  a_{2,1} & 0 & a_{2,3}& c \\
  0 & 0 &-c & 0\\
 0 & 0 &  0 & -c
\end{smallmatrix}\right].$

\end{enumerate}
\end{theorem}
\begin{proof}
\begin{enumerate}[noitemsep, topsep=0pt]
\item[(1)] The right (a derivation)
and left (not a derivation) multiplication operators  restricted to the nilradical are given below:
$$\r_{e_{n+1}}=\left[\begin{smallmatrix}
 a & 0 & 0 & 0&0&&\cdots &0& \cdots&0 & 0&0 \\
  A_{2,1} & 2b & a_{2,3}& b-a &0& & \cdots &0&\cdots  & 0& 0&0\\
  b-a & 0 & b & 0 & 0& &\cdots &0&\cdots &0& 0&0 \\
  a_{4,1} & 0 &  a_{4,3} & a+b &0 & &\cdots&0&\cdots &0 & 0&0\\
  a_{5,1} & 0 & a_{5,3} & a_{4,3} & 2a+b  &&\cdots &0 &\cdots&0 & 0&0 \\
 \boldsymbol{\cdot} & \boldsymbol{\cdot} & \boldsymbol{\cdot} & a_{5,3} & a_{4,3}  &\ddots& &\vdots &&\vdots & \vdots&\vdots \\
  \vdots & \vdots & \vdots &\vdots &\vdots  &\ddots& \ddots&\vdots &&\vdots & \vdots&\vdots \\
  a_{i,1} & 0 & a_{i,3} & a_{i-1,3} & a_{i-2,3}  &\cdots&a_{4,3}& (i-3)a+b&\cdots&0 & 0&0\\
   \vdots  & \vdots  & \vdots &\vdots &\vdots&&\vdots &\vdots &&\vdots  & \vdots&\vdots \\
 a_{n-1,1} & 0 & a_{n-1,3}& a_{n-2,3}& a_{n-3,3}&\cdots &a_{n-i+3,3} &a_{n-i+2,3}&\cdots&a_{4,3} &(n-4)a+b& 0\\
 a_{n,1} & 0 & a_{n,3}& a_{n-1,3}& a_{n-2,3}&\cdots &a_{n-i+4,3}  &a_{n-i+3,3}&\cdots&a_{5,3} &a_{4,3}& (n-3)a+b
\end{smallmatrix}\right],$$
$$\L_{e_{n+1}}=\left[\begin{smallmatrix}
 -a & 0 & 0 & 0&0&&\cdots &0& \cdots&0 & 0&0 \\
  b_{2,1} & 0 & a_{2,3}+2a_{4,3}& a+b &0& & \cdots &0&\cdots  & 0& 0&0\\
  a-b & 0 & -b & 0 & 0& &\cdots &0&\cdots &0& 0&0 \\
  -a_{4,1} & 0 &  -a_{4,3} & -a-b &0 & &\cdots&0&\cdots &0 & 0&0\\
  -a_{5,1} & 0 & -a_{5,3} & -a_{4,3} & -2a-b  &&\cdots &0 &\cdots&0 & 0&0 \\
 \boldsymbol{\cdot} & \boldsymbol{\cdot} & \boldsymbol{\cdot} & -a_{5,3} & -a_{4,3}  &\ddots& &\vdots &&\vdots & \vdots&\vdots \\
  \vdots & \vdots & \vdots &\vdots &\vdots  &\ddots& \ddots&\vdots &&\vdots & \vdots&\vdots \\
  -a_{i,1} & 0 & -a_{i,3} & -a_{i-1,3} & -a_{i-2,3}  &\cdots&-a_{4,3}& (3-i)a-b&\cdots&0 & 0&0\\
   \vdots  & \vdots  & \vdots &\vdots &\vdots&&\vdots &\vdots &&\vdots  & \vdots&\vdots \\
 -a_{n-1,1} & 0 & -a_{n-1,3}& -a_{n-2,3}& -a_{n-3,3}&\cdots &-a_{n-i+3,3} &-a_{n-i+2,3}&\cdots&-a_{4,3} &(4-n)a-b& 0\\
 -a_{n,1} & 0 & -a_{n,3}& -a_{n-1,3}& -a_{n-2,3}&\cdots &-a_{n-i+4,3}  &-a_{n-i+3,3}&\cdots&-a_{5,3} &-a_{4,3}& (3-n)a-b
\end{smallmatrix}\right].$$
\allowdisplaybreaks
\begin{itemize}
\item The transformation $e^{\prime}_k=e_k,(1\leq k\leq n),e^{\prime}_{n+1}=e_{n+1}-a_{4,3}e_1$
removes $a_{4,3}$ in $\r_{e_{n+1}}$ and $-a_{4,3}$ in $\L_{e_{n+1}}$ from the $(i,i-1)^{st}$ positions, where $(4\leq i\leq n),$ 
 but it affects other entries as well,
such as
the entry in the $(2,1)^{st}$ position in $\r_{e_{n+1}}$ and $\L_{e_{n+1}},$
which we change to $A_{2,1}-a_{4,3}$ and $b_{2,1}-a_{4,3},$ respectively.
It also changes the entry in the $(2,3)^{rd}$ position in $\L_{e_{n+1}}$ to 
$a_{2,3}.$
At the same time, it affects the coefficient in front of $e_2$ in the bracket $[e_{n+1},e_{n+1}],$ which we change back to $a_{2,n+1}$.

\item Applying the transformation $e^{\prime}_i=e_i,(1\leq i\leq n),e^{\prime}_{n+1}=e_{n+1}+\sum_{k=3}^{n-1}a_{k+1,1}e_{k},$
we remove $a_{k+1,1}$ in $\r_{e_{n+1}}$ and $-a_{k+1,1}$ in $\L_{e_{n+1}}$ from the entries in the $(k+1,1)^{st}$
positions, where $(3\leq k\leq n-1).$ It changes the entry in the $(2,1)^{st}$ position in
$\r_{e_{n+1}}$ to $A_{2,1}+2a_{4,1}-a_{4,3},$ the entries in the $(2,3)^{rd}$ positions in $\r_{e_{n+1}}$
and $\L_{e_{n+1}}$ to $a_{2,3}+a_{4,1}.$ 
It also affects the coefficient in front of $e_2$ in $[e_{n+1},e_{n+1}],$ which
we rename back by $a_{2,n+1}.$ We assign $a_{2,3}+a_{4,1}:=a_{2,3}$ and $b_{2,1}-a_{4,3}:=b_{2,1}.$ Then
$A_{2,1}+2a_{4,1}-a_{4,3}:=\frac{(2b-a)b_{2,1}+2(a-b)a_{2,3}}{a}.$
\item The transformation $e^{\prime}_j=e_j,(1\leq j\leq n),e^{\prime}_{n+1}=e_{n+1}-\frac{a_{2,n+1}}{2b}e_2$ 
removes the coefficient $a_{2,n+1}$ in front of $e_2$ in $[e_{n+1},e_{n+1}]$ and we prove the result.
\end{itemize}
\item[(2)] The right (a derivation) and left (not a derivation) multiplication operators restricted to the nilradical are as follows:
$$\r_{e_{n+1}}=\left[\begin{smallmatrix}
 a & 0 & 0 & 0&0&&\cdots &0& \cdots&0 & 0&0 \\
  A_{2,1} & 2(3-n)a & a_{2,3}& (2-n)a &0& & \cdots &0&\cdots  & 0& 0&0\\
  (2-n)a & 0 & (3-n)a & 0 & 0& &\cdots &0&\cdots &0& 0&0 \\
  a_{4,1} & 0 &  a_{4,3} & (4-n)a &0 & &\cdots&0&\cdots &0 & 0&0\\
  a_{5,1} & 0 & a_{5,3} & a_{4,3} & (5-n)a  &&\cdots &0 &\cdots&0 & 0&0 \\
 \boldsymbol{\cdot} & \boldsymbol{\cdot} & \boldsymbol{\cdot} & a_{5,3} & a_{4,3}  &\ddots& &\vdots &&\vdots & \vdots&\vdots \\
  \vdots & \vdots & \vdots &\vdots &\vdots  &\ddots& \ddots&\vdots &&\vdots & \vdots&\vdots \\
  a_{i,1} & 0 & a_{i,3} & a_{i-1,3} & a_{i-2,3}  &\cdots&a_{4,3}& (i-n)a&\cdots&0 & 0&0\\
   \vdots  & \vdots  & \vdots &\vdots &\vdots&&\vdots &\vdots &&\vdots  & \vdots&\vdots \\
 a_{n-1,1} & 0 & a_{n-1,3}& a_{n-2,3}& a_{n-3,3}&\cdots &a_{n-i+3,3} &a_{n-i+2,3}&\cdots&a_{4,3} &-a& 0\\
 a_{n,1} & 0 & a_{n,3}& a_{n-1,3}& a_{n-2,3}&\cdots &a_{n-i+4,3}  &a_{n-i+3,3}&\cdots&a_{5,3} &a_{4,3}& 0
\end{smallmatrix}\right],$$
$$\L_{e_{n+1}}=\left[\begin{smallmatrix}
 -a & 0 & 0 & 0&0&&\cdots &0& \cdots&0 & 0&0 \\
  b_{2,1} & 0 & a_{2,3}+2a_{4,3}& (4-n)a &0& & \cdots &0&\cdots  & 0& 0&0\\
  (n-2)a & 0 & (n-3)a & 0 & 0& &\cdots &0&\cdots &0& 0&0 \\
  -a_{4,1} & 0 &  -a_{4,3} & (n-4)a &0 & &\cdots&0&\cdots &0 & 0&0\\
  -a_{5,1} & 0 & -a_{5,3} & -a_{4,3} & (n-5)a  &&\cdots &0 &\cdots&0 & 0&0 \\
 \boldsymbol{\cdot} & \boldsymbol{\cdot} & \boldsymbol{\cdot} & -a_{5,3} & -a_{4,3}  &\ddots& &\vdots &&\vdots & \vdots&\vdots \\
  \vdots & \vdots & \vdots &\vdots &\vdots  &\ddots& \ddots&\vdots &&\vdots & \vdots&\vdots \\
  -a_{i,1} & 0 & -a_{i,3} & -a_{i-1,3} & -a_{i-2,3}  &\cdots&-a_{4,3}& (n-i)a&\cdots&0 & 0&0\\
   \vdots  & \vdots  & \vdots &\vdots &\vdots&&\vdots &\vdots &&\vdots  & \vdots&\vdots \\
 -a_{n-1,1} & 0 & -a_{n-1,3}& -a_{n-2,3}& -a_{n-3,3}&\cdots &-a_{n-i+3,3} &-a_{n-i+2,3}&\cdots&-a_{4,3} &a& 0\\
 -a_{n,1} & 0 & -a_{n,3}& -a_{n-1,3}& -a_{n-2,3}&\cdots &-a_{n-i+4,3}  &-a_{n-i+3,3}&\cdots&-a_{5,3} &-a_{4,3}& 0
\end{smallmatrix}\right].$$
\begin{itemize}
\item The transformation $e^{\prime}_k=e_k,(1\leq k\leq n),e^{\prime}_{n+1}=e_{n+1}-a_{4,3}e_1$
removes $a_{4,3}$ in $\r_{e_{n+1}}$ and $-a_{4,3}$ in $\L_{e_{n+1}}$ from the $(i,i-1)^{st}$ positions, where $(4\leq i\leq n),$ 
 but it affects other entries as well,
such as
the entry in the $(2,1)^{st}$ position in $\r_{e_{n+1}}$ and $\L_{e_{n+1}},$
which we change to $A_{2,1}-a_{4,3}$ and $b_{2,1}-a_{4,3},$ respectively.
It also changes the entry in the $(2,3)^{rd}$ position in $\L_{e_{n+1}}$ to 
$a_{2,3}.$
At the same time, it affects the coefficient in front of $e_2$ in the bracket $[e_{n+1},e_{n+1}],$ which we change back to $a_{2,n+1}$.

\item Then we apply the transformation $e^{\prime}_i=e_i,(1\leq i\leq n),e^{\prime}_{n+1}=e_{n+1}+\sum_{k=3}^{n-1}a_{k+1,1}e_{k}$
to remove $a_{k+1,1}$ in $\r_{e_{n+1}}$ and $-a_{k+1,1}$ in $\L_{e_{n+1}}$ from the entries in the $(k+1,1)^{st}$
positions, where $(3\leq k\leq n-1).$ It changes the entry in the $(2,1)^{st}$ position in
$\r_{e_{n+1}}$ to $A_{2,1}+2a_{4,1}-a_{4,3},$ the entries in the $(2,3)^{rd}$ positions in $\r_{e_{n+1}}$
and $\L_{e_{n+1}}$ to $a_{2,3}+a_{4,1}.$ 
It also affects the coefficient in front of $e_2$ in $[e_{n+1},e_{n+1}],$ which
we rename back by $a_{2,n+1}.$ We assign $a_{2,3}+a_{4,1}:=a_{2,3}$ and $b_{2,1}-a_{4,3}:=b_{2,1}.$ Then $A_{2,1}+2a_{4,1}-a_{4,3}:=
(5-2n)b_{2,1}+2(n-2)a_{2,3},$ which we set to be $\mathcal{A}_{2,1}.$
\item Applying the transformation $e^{\prime}_j=e_j,(1\leq j\leq n),e^{\prime}_{n+1}=e_{n+1}+\frac{a_{2,n+1}}{2(n-3)a}e_2,$
we remove the coefficient $a_{2,n+1}$ in front of $e_2$ in $[e_{n+1},e_{n+1}]$ and prove the result.
\end{itemize}
\item[(3)] The right (a derivation) and left (not a derivation) multiplication operators 
restricted to the nilradical are
$$\r_{e_{n+1}}=\left[\begin{smallmatrix}
 0 & 0 & 0 & 0&0&&\cdots &0& \cdots&0 & 0&0 \\
  a_{2,1} & 2b & a_{2,3}& b &0& & \cdots &0&\cdots  & 0& 0&0\\
  b & 0 & b & 0 & 0& &\cdots &0&\cdots &0& 0&0 \\
  a_{4,1} & 0 &  a_{4,3} & b &0 & &\cdots&0&\cdots &0 & 0&0\\
  a_{5,1} & 0 & a_{5,3} & a_{4,3} & b  &&\cdots &0 &\cdots&0 & 0&0 \\
 \boldsymbol{\cdot} & \boldsymbol{\cdot} & \boldsymbol{\cdot} & a_{5,3} & a_{4,3}  &\ddots& &\vdots &&\vdots & \vdots&\vdots \\
  \vdots & \vdots & \vdots &\vdots &\vdots  &\ddots& \ddots&\vdots &&\vdots & \vdots&\vdots \\
  a_{i,1} & 0 & a_{i,3} & a_{i-1,3} & a_{i-2,3}  &\cdots&a_{4,3}& b&\cdots&0 & 0&0\\
   \vdots  & \vdots  & \vdots &\vdots &\vdots&&\vdots &\vdots &&\vdots  & \vdots&\vdots \\
 a_{n-1,1} & 0 & a_{n-1,3}& a_{n-2,3}& a_{n-3,3}&\cdots &a_{n-i+3,3} &a_{n-i+2,3}&\cdots&a_{4,3} &b& 0\\
 a_{n,1} & 0 & a_{n,3}& a_{n-1,3}& a_{n-2,3}&\cdots &a_{n-i+4,3}  &a_{n-i+3,3}&\cdots&a_{5,3} &a_{4,3}& b
\end{smallmatrix}\right],$$
$$\L_{e_{n+1}}=\left[\begin{smallmatrix}
 0 & 0 & 0 & 0&0&&\cdots &0& \cdots&0 & 0&0 \\
  a_{2,3}+a_{4,1}+a_{4,3} & 0 & a_{2,3}+2a_{4,3}& b &0& & \cdots &0&\cdots  & 0& 0&0\\
  -b & 0 & -b & 0 & 0& &\cdots &0&\cdots &0& 0&0 \\
  -a_{4,1} & 0 &  -a_{4,3} & -b &0 & &\cdots&0&\cdots &0 & 0&0\\
  -a_{5,1} & 0 & -a_{5,3} & -a_{4,3} & -b  &&\cdots &0 &\cdots&0 & 0&0 \\
 \boldsymbol{\cdot} & \boldsymbol{\cdot} & \boldsymbol{\cdot} & -a_{5,3} & -a_{4,3}  &\ddots& &\vdots &&\vdots & \vdots&\vdots \\
  \vdots & \vdots & \vdots &\vdots &\vdots  &\ddots& \ddots&\vdots &&\vdots & \vdots&\vdots \\
  -a_{i,1} & 0 & -a_{i,3} & -a_{i-1,3} & -a_{i-2,3}  &\cdots&-a_{4,3}& -b&\cdots&0 & 0&0\\
   \vdots  & \vdots  & \vdots &\vdots &\vdots&&\vdots &\vdots &&\vdots  & \vdots&\vdots \\
 -a_{n-1,1} & 0 & -a_{n-1,3}& -a_{n-2,3}& -a_{n-3,3}&\cdots &-a_{n-i+3,3} &-a_{n-i+2,3}&\cdots&-a_{4,3} &-b& 0\\
 -a_{n,1} & 0 & -a_{n,3}& -a_{n-1,3}& -a_{n-2,3}&\cdots &-a_{n-i+4,3}  &-a_{n-i+3,3}&\cdots&-a_{5,3} &-a_{4,3}&-b
\end{smallmatrix}\right].$$
\begin{itemize}
\item The transformation $e^{\prime}_k=e_k,(1\leq k\leq n),e^{\prime}_{n+1}=e_{n+1}-a_{4,3}e_1$
removes $a_{4,3}$ in $\r_{e_{n+1}}$ and $-a_{4,3}$ in $\L_{e_{n+1}}$ from the $(i,i-1)^{st}$ positions, where $(4\leq i\leq n),$ 
 but it affects other entries as well,
such as
the entry in the $(2,1)^{st}$ position in $\r_{e_{n+1}}$ and $\L_{e_{n+1}},$
which we change to $a_{2,1}-a_{4,3}$ and $a_{2,3}+a_{4,1},$ respectively.
It also changes the entry in the $(2,3)^{rd}$ position in $\L_{e_{n+1}}$ to 
$a_{2,3}.$
At the same time, it affects the coefficient in front of $e_2$ in the bracket $[e_{n+1},e_{n+1}],$ which we change back to $a_{2,n+1}$.

\item Applying the transformation $e^{\prime}_i=e_i,(1\leq i\leq n),e^{\prime}_{n+1}=e_{n+1}+\sum_{k=3}^{n-1}a_{k+1,1}e_{k},$
we remove $a_{k+1,1}$ in $\r_{e_{n+1}}$ and $-a_{k+1,1}$ in $\L_{e_{n+1}}$ from the entries in the $(k+1,1)^{st}$
positions, where $(3\leq k\leq n-1).$ It changes the entry in the $(2,1)^{st}$ position in
$\r_{e_{n+1}}$ to $a_{2,1}+2a_{4,1}-a_{4,3},$ the entries in the $(2,3)^{rd}$ positions in $\r_{e_{n+1}}$
and $\L_{e_{n+1}}$ to $a_{2,3}+a_{4,1}.$ 
It also affects the coefficient in front of $e_2$ in $[e_{n+1},e_{n+1}],$ which
we rename back by $a_{2,n+1}.$ Then we assign $a_{2,1}+2a_{4,1}-a_{4,3}:=a_{2,1}$ and $a_{2,3}+a_{4,1}:=a_{2,3}.$
\item The transformation $e^{\prime}_j=e_j,(1\leq j\leq n),e^{\prime}_{n+1}=e_{n+1}-\frac{a_{2,n+1}}{2b}e_2$ 
removes the coefficient $a_{2,n+1}$ in front of $e_2$ in $[e_{n+1},e_{n+1}]$ and we prove the result.
\end{itemize}
\item[(4)] The right (a derivation) and left (not a derivation) multiplication operators 
restricted to the nilradical are given below:
$$\r_{e_{n+1}}=\left[\begin{smallmatrix}
 a & 0 & 0 & 0&0&&\cdots &0& \cdots&0 & 0&0 \\
  a_{2,3}-b_{2,1}+b_{2,3} & 0 & a_{2,3}& -a &0& & \cdots &0&\cdots  & 0& 0&0\\
  -a & 0 & 0 & 0 & 0& &\cdots &0&\cdots &0& 0&0 \\
  a_{4,1} & 0 &  a_{4,3} & a &0 & &\cdots&0&\cdots &0 & 0&0\\
  a_{5,1} & 0 & a_{5,3} & a_{4,3} & 2a  &&\cdots &0 &\cdots&0 & 0&0 \\
 \boldsymbol{\cdot} & \boldsymbol{\cdot} & \boldsymbol{\cdot} & a_{5,3} & a_{4,3}  &\ddots& &\vdots &&\vdots & \vdots&\vdots \\
  \vdots & \vdots & \vdots &\vdots &\vdots  &\ddots& \ddots&\vdots &&\vdots & \vdots&\vdots \\
  a_{i,1} & 0 & a_{i,3} & a_{i-1,3} & a_{i-2,3}  &\cdots&a_{4,3}& (i-3)a&\cdots&0 & 0&0\\
   \vdots  & \vdots  & \vdots &\vdots &\vdots&&\vdots &\vdots &&\vdots  & \vdots&\vdots \\
 a_{n-1,1} & 0 & a_{n-1,3}& a_{n-2,3}& a_{n-3,3}&\cdots &a_{n-i+3,3} &a_{n-i+2,3}&\cdots&a_{4,3} &(n-4)a& 0\\
 a_{n,1} & 0 & a_{n,3}& a_{n-1,3}& a_{n-2,3}&\cdots &a_{n-i+4,3}  &a_{n-i+3,3}&\cdots&a_{5,3} &a_{4,3}&(n-3)a
\end{smallmatrix}\right],$$
$$\L_{e_{n+1}}=\left[\begin{smallmatrix}
 -a & 0 & 0 & 0&0&&\cdots &0& \cdots&0 & 0&0 \\
  b_{2,1} & 0 & b_{2,3}& a &0& & \cdots &0&\cdots  & 0& 0&0\\
  a & 0 & 0 & 0 & 0& &\cdots &0&\cdots &0& 0&0 \\
  -a_{4,1} & 0 &  -a_{4,3} & -a &0 & &\cdots&0&\cdots &0 & 0&0\\
  -a_{5,1} & 0 & -a_{5,3} & -a_{4,3} & -2a  &&\cdots &0 &\cdots&0 & 0&0 \\
 \boldsymbol{\cdot} & \boldsymbol{\cdot} & \boldsymbol{\cdot} & -a_{5,3} & -a_{4,3}  &\ddots& &\vdots &&\vdots & \vdots&\vdots \\
  \vdots & \vdots & \vdots &\vdots &\vdots  &\ddots& \ddots&\vdots &&\vdots & \vdots&\vdots \\
  -a_{i,1} & 0 & -a_{i,3} & -a_{i-1,3} & -a_{i-2,3}  &\cdots&-a_{4,3}& (3-i)a&\cdots&0 & 0&0\\
   \vdots  & \vdots  & \vdots &\vdots &\vdots&&\vdots &\vdots &&\vdots  & \vdots&\vdots \\
 -a_{n-1,1} & 0 & -a_{n-1,3}& -a_{n-2,3}& -a_{n-3,3}&\cdots &-a_{n-i+3,3} &-a_{n-i+2,3}&\cdots&-a_{4,3} &(4-n)a& 0\\
 -a_{n,1} & 0 & -a_{n,3}& -a_{n-1,3}& -a_{n-2,3}&\cdots &-a_{n-i+4,3}  &-a_{n-i+3,3}&\cdots&-a_{5,3} &-a_{4,3}&(3-n)a
\end{smallmatrix}\right].$$
\begin{itemize}
\item The transformation $e^{\prime}_k=e_k,(1\leq k\leq n),e^{\prime}_{n+1}=e_{n+1}-a_{4,3}e_1$
removes $a_{4,3}$ in $\r_{e_{n+1}}$ and $-a_{4,3}$ in $\L_{e_{n+1}}$ from the entries in the $(i,i-1)^{st}$ positions, where $(4\leq i\leq n),$ 
 but it affects other entries as well,
such as
the entry in the $(2,1)^{st}$ position in $\r_{e_{n+1}}$ and $\L_{e_{n+1}},$
which we change to $a_{2,3}-b_{2,1}-a_{4,3}+b_{2,3}$ and $b_{2,1}-a_{4,3},$ respectively.
It also changes the entry in the $(2,3)^{rd}$ position in $\L_{e_{n+1}}$ to 
$b_{2,3}-2a_{4,3}.$
At the same time, it affects the coefficient in front of $e_2$ in the bracket $[e_{n+1},e_{n+1}],$ which we change back to $a_{2,n+1}$.

\item Applying the transformation $e^{\prime}_i=e_i,(1\leq i\leq n),e^{\prime}_{n+1}=e_{n+1}+\sum_{k=3}^{n-1}a_{k+1,1}e_{k},$
we remove $a_{k+1,1}$ in $\r_{e_{n+1}}$ and $-a_{k+1,1}$ in $\L_{e_{n+1}}$ from the entries in the $(k+1,1)^{st}$
positions, where $(3\leq k\leq n-1).$ It changes the entry in the $(2,1)^{st}$ position in
$\r_{e_{n+1}}$ to $2a_{4,1}+a_{2,3}-b_{2,1}-a_{4,3}+b_{2,3},$ the entries in the $(2,3)^{rd}$ positions in $\r_{e_{n+1}}$
and $\L_{e_{n+1}}$ to $a_{2,3}+a_{4,1}$ and $b_{2,3}+a_{4,1}-2a_{4,3},$ respectively. 
It also affects the coefficient in front of $e_2$ in $[e_{n+1},e_{n+1}],$ which
we rename back by $a_{2,n+1}.$ We assign $a_{2,3}+a_{4,1}:=a_{2,3},b_{2,1}-a_{4,3}:=b_{2,1}$
and $b_{2,3}+a_{4,1}-2a_{4,3}:=b_{2,3}.$ Then $2a_{4,1}+a_{2,3}-b_{2,1}-a_{4,3}+b_{2,3}:=a_{2,3}-b_{2,1}+b_{2,3}.$
\end{itemize}
\item[(5)] We apply the transformation
$e^{\prime}_i=e_i,(1\leq i\leq 4),e^{\prime}_5=e_5-A_{4,3}e_1+\frac{1}{2(b+c)}(A_{2,1}A_{4,3}+2a_{4,1}A_{4,3}+b_{2,1}A_{4,3}-A^2_{4,3}-a_{2,3}a_{4,1}-a^2_{4,1}-a_{4,1}b_{2,3}-a_{2,5})e_2+a_{4,1}e_3$
and then we assign $a_{2,3}+a_{4,1}:=a_{2,3}$ and $b_{2,3}-2b_{2,1}+a_{4,1}:=b_{2,3}.$
\item[(6)] 
The transformation we apply is
$e^{\prime}_i=e_i,(1\leq i\leq 4),e^{\prime}_5=e_5-A_{4,3}e_1+\frac{1}{2(c-a)}(A_{2,1}A_{4,3}+2a_{4,1}A_{4,3}+b_{2,1}A_{4,3}-A^2_{4,3}-a_{2,3}a_{4,1}-a^2_{4,1}-a_{4,1}b_{2,3}-a_{2,5})e_2+a_{4,1}e_3.$
We assign $a_{2,3}+a_{4,1}:=a_{2,3}$ and $b_{2,3}-2b_{2,1}+a_{4,1}:=b_{2,3}.$
\item[(7)] We apply the transformation
$e^{\prime}_i=e_i,(1\leq i\leq 4),e^{\prime}_5=e_5-a_{4,3}e_1+a_{4,1}e_3$ and rename the coefficient in front of $e_2$
in $[e_5,e_5]$ back by $a_{2,5}.$ Then we assign $a_{2,3}+a_{4,1}:=a_{2,3},b_{2,1}-a_{4,3}:=b_{2,1}$ and $b_{2,3}+a_{4,1}-2a_{4,3}:=b_{2,3}.$
\item[(8)] We apply
$e^{\prime}_i=e_i,(1\leq i\leq 4),e^{\prime}_5=e_5-a_{4,3}e_1-\frac{1}{2c}(2a_{2,3}a_{4,3}-a^2_{4,1}+2a_{4,1}a_{4,3}+2a_{4,1}b_{2,1}+3a_{4,3}^2-4a_{4,3}b_{2,1}+a_{2,5})e_2+a_{4,1}e_3.$
We assign $a_{2,3}+a_{4,1}:=a_{2,3}$ and $b_{2,1}-a_{4,3}:=b_{2,1}.$
 \item[(9)] The transformation is
 $e^{\prime}_i=e_i,(1\leq i\leq 4),$
 $e^{\prime}_5=e_5-A_{4,3}e_1+\frac{1}{4(b+c)}(2a_{2,1}A_{4,3}-2a_{2,3}a_{4,1}+a_{2,3}A_{4,3}-2a^2_{4,1}+6a_{4,1}A_{4,3}-2a_{4,1}b_{2,3}-2A^2_{4,3}+b_{2,3}A_{4,3}-2a_{2,5})e_2+a_{4,1}e_3.$
Then we assign $a_{2,3}+a_{4,1}:=a_{2,3}$ and $b_{2,3}-2a_{2,1}-3a_{4,1}:=b_{2,3}.$
\item[(10)] We apply the transformation
$e^{\prime}_i=e_i,(1\leq i\leq 4),e^{\prime}_5=e_5-a_{4,3}e_1+a_{4,1}e_3$ and rename the coefficient in front of $e_2$
in $[e_5,e_5]$ back by $a_{2,5}.$ We assign $a_{2,3}+a_{4,1}:=a_{2,3},a_{2,1}+2a_{4,1}-a_{4,3}:=a_{2,1}$
and $b_{2,3}+a_{4,1}-2a_{4,3}:=b_{2,3}.$
\end{enumerate}
\end{proof}
\allowdisplaybreaks
 \begin{theorem}\label{RL4(Change of Basis)} There are eight solvable
indecomposable right Leibniz algebras up to isomorphism with a codimension one nilradical
$\mathcal{L}^4,(n\geq4),$ which are given below:
\begin{equation}
\begin{array}{l}
\displaystyle \nonumber (i)\,\,\, \g_{n+1,1}: [e_1,e_{n+1}]=e_1+(a-1)e_3,[e_2,e_{n+1}]=2ae_2,[e_3,e_{n+1}]=ae_3,\\
\displaystyle [e_4,e_{n+1}]=(a-1)e_2+(a+1)e_4,[e_{i},e_{n+1}]=\left(a+i-3\right)e_{i},[e_{n+1},e_1]=-e_1-(a-1)e_3,\\
\displaystyle[e_{n+1},e_3]=-ae_3,[e_{n+1},e_4]=(a+1)\left(e_2-e_4\right),[e_{n+1},e_i]=\left(3-i-a\right)e_i,(5\leq i\leq n),\\
\displaystyle  \nonumber(ii)\,\,\, \g_{n+1,2}:[e_1,e_{n+1}]=e_1+(2-n)e_3,[e_2,e_{n+1}]=2(3-n)e_2, [e_3,e_{n+1}]=(3-n)e_3,\\
\displaystyle [e_4,e_{n+1}]=(2-n)e_2+(4-n)e_4,[e_{i},e_{n+1}]=(i-n)e_{i},[e_{n+1},e_{n+1}]=e_n,\\
\displaystyle [e_{n+1},e_1]=-e_1+(n-2)e_3,[e_{n+1},e_3]=(n-3)e_3,[e_{n+1},e_4]=(4-n)\left(e_2-e_4\right),\\
\displaystyle [e_{n+1},e_i]=(n-i)e_i,(5\leq i\leq n),\\
\displaystyle (iii)\,\,\,\g_{n+1,3}:
[e_1,e_{n+1}]=e_3,[e_2,e_{n+1}]=2e_2,[e_4,e_{n+1}]=e_2,[e_{i},e_{n+1}]=e_{i}+\epsilon e_{i+2}+\\
\displaystyle \sum_{k=i+3}^n{b_{k-i-2}e_k},[e_{n+1},e_1]=-e_3,[e_{n+1},e_4]=e_2,[e_{n+1},e_i]=-e_i-\epsilon e_{i+2}-\sum_{k=i+3}^n{b_{k-i-2}e_k},\\
\displaystyle (\epsilon=0,1,3\leq i\leq n),\\
\displaystyle \nonumber(iv)\,\,\,\g_{n+1,4}: [e_1,e_{n+1}]=e_1-e_3,
[e_3,e_{n+1}]=de_2,[e_4,e_{n+1}]=e_4-e_2,[e_{i},e_{n+1}]=(i-3)e_{i},\\
\displaystyle 
[e_{n+1},e_{n+1}]=\epsilon e_2,[e_{n+1},e_1]=-e_1+\left(d+f\right)e_2+e_3,[e_{n+1},e_3]=fe_2, [e_{n+1},e_4]=e_2-e_4,\\
\displaystyle [e_{n+1},e_i]=(3-i)e_i,(5\leq i\leq n;\epsilon=0,1; if\,\,\epsilon=0,\,\,then\,\,d^2+f^2\neq0),\\  
\displaystyle \nonumber(v)\,\,\,\g_{5,5}: [e_1,e_{5}]=ae_1+(b-a+1)e_3,[e_2,e_{5}]=2(b+1)e_2,
[e_3,e_{5}]=e_1+be_3,\\
\displaystyle [e_4,e_{5}]=(b-a+2)e_2+(a+b)e_4,[e_{5},e_1]=-ae_1+(a-b-1)e_3, [e_{5},e_3]=-e_1-be_3,\\
\displaystyle [e_{5},e_4]=(a+b)(e_2-e_4),(if\,\,b=-1,then\,\,a\neq1),\\
\displaystyle \nonumber(vi)\,\,\g_{5,6}: [e_1,e_{5}]=ae_1+(1-2a)e_3,[e_2,e_{5}]=2(1-a)e_2,
[e_3,e_{5}]=e_1-ae_3,\\
\displaystyle [e_4,e_{5}]=2(1-a)e_2,[e_{5},e_{5}]=e_4,[e_{5},e_1]=-ae_1+(2a-1)e_3,[e_{5},e_3]=ae_3-e_1, (a\neq1),\\
\displaystyle \nonumber(vii)\,\,\g_{5,7}:[e_1,e_{5}]=a(e_1-e_3),
[e_3,e_{5}]=e_1+fe_2-e_3,[e_4,e_{5}]=(1-a)\left(e_2-e_4\right),\\
\displaystyle [e_{5},e_{5}]=\epsilon e_2,[e_{5},e_1]=-ae_1+\left(d+f\right)e_2+ae_3,[e_{5},e_3]=-e_1+de_2+e_3,\\
\displaystyle [e_{5},e_4]=(a-1)\left(e_2-e_4\right),(\epsilon=0,1;if\,\,\epsilon=0,then\,\,d^2+f^2\neq0;a\neq1),\\
\displaystyle  \nonumber(viii)\,\g_{5,8}: [e_1,e_{5}]=ce_2,
[e_3,e_{5}]=e_1-e_3,[e_4,e_{5}]=e_2-e_4,[e_{5},e_{5}]=\epsilon e_2,\\
\displaystyle [e_{5},e_1]=\left(c+d\right)e_2,[e_{5},e_3]=-e_1+\left(d+2c\right)e_2+e_3,[e_{5},e_4]=e_4-e_2,(c\neq0,\epsilon=0,1).
\end{array} 
\end{equation} 
\end{theorem}
\vskip 20pt    
\begin{proof}
One applies the change of basis transformations keeping the nilradical $\mathcal{L}^4$ given in $(\ref{L4})$ unchanged.
\begin{enumerate}[noitemsep, topsep=1pt]
\allowdisplaybreaks
\item[(1)] We have the right (a derivation) and the left (not a derivation) multiplication operators restricted to the nilradical given below:
$$\r_{e_{n+1}}=\left[\begin{smallmatrix}
 a & 0 & 0 & 0&0&0&\cdots && 0&0 & 0\\
  \mathcal{A}_{2,1} & 2b & a_{2,3}& b-a &0&0 & \cdots &&0  & 0& 0\\
  b-a & 0 & b & 0 & 0&0 &\cdots &&0 &0& 0\\
  0 & 0 &  0 & a+b &0 &0 &\cdots&&0 &0 & 0\\
 0 & 0 & a_{5,3} & 0 & 2a+b  &0&\cdots & &0&0 & 0\\
  0 & 0 &\boldsymbol{\cdot} & a_{5,3} & 0 &3a+b&\cdots & &0&0 & 0\\
    0 & 0 &\boldsymbol{\cdot} & \boldsymbol{\cdot} & \ddots &0&\ddots &&\vdots&\vdots &\vdots\\
  \vdots & \vdots & \vdots &\vdots &  &\ddots&\ddots &\ddots &\vdots&\vdots & \vdots\\
 0 & 0 & a_{n-2,3}& a_{n-3,3}& \cdots&\cdots &a_{5,3}&0&(n-5)a+b &0& 0\\
0 & 0 & a_{n-1,3}& a_{n-2,3}& \cdots&\cdots &\boldsymbol{\cdot}&a_{5,3}&0 &(n-4)a+b& 0\\
 0 & 0 & a_{n,3}& a_{n-1,3}& \cdots&\cdots &\boldsymbol{\cdot}&\boldsymbol{\cdot}&a_{5,3} &0& (n-3)a+b
\end{smallmatrix}\right],$$
 $$\L_{e_{n+1}}=\left[\begin{smallmatrix}
 -a & 0 & 0 & 0&0&0&\cdots && 0&0 & 0\\
  b_{2,1} & 0 & a_{2,3}& a+b &0&0 & \cdots &&0  & 0& 0\\
  a-b & 0 & -b & 0 & 0&0 &\cdots &&0 &0& 0\\
  0 & 0 &  0 & -a-b &0 &0 &\cdots&&0 &0 & 0\\
 0 & 0 & -a_{5,3} & 0 & -2a-b  &0&\cdots & &0&0 & 0\\
  0 & 0 &\boldsymbol{\cdot} & -a_{5,3} & 0 &-3a-b&\cdots & &0&0 & 0\\
    0 & 0 &\boldsymbol{\cdot} & \boldsymbol{\cdot} & \ddots &0&\ddots &&\vdots&\vdots &\vdots\\
  \vdots & \vdots & \vdots &\vdots &  &\ddots&\ddots &\ddots &\vdots&\vdots & \vdots\\
 0 & 0 & -a_{n-2,3}& -a_{n-3,3}& \cdots&\cdots &-a_{5,3}&0&(5-n)a-b &0& 0\\
0 & 0 & -a_{n-1,3}& -a_{n-2,3}& \cdots&\cdots &\boldsymbol{\cdot}&-a_{5,3}&0 &(4-n)a-b& 0\\
 0 & 0 & -a_{n,3}& -a_{n-1,3}& \cdots&\cdots &\boldsymbol{\cdot}&\boldsymbol{\cdot}&-a_{5,3} &0& (3-n)a-b
\end{smallmatrix}\right].$$
\begin{itemize}[noitemsep, topsep=0pt]
\allowdisplaybreaks 
\item We apply the transformation $e^{\prime}_1=e_1,e^{\prime}_2=e_2,e^{\prime}_i=e_i-\frac{a_{k-i+3,3}}{(k-i)a}e_k,(3\leq i\leq n-2,
i+2\leq k\leq n,n\geq5),e^{\prime}_{j}=e_{j},(n-1\leq j\leq n+1),$ where $k$ is fixed, renaming all the affected entries back.
This transformation removes $a_{5,3},a_{6,3},...,a_{n,3}$ in $\r_{e_{n+1}}$ and $-a_{5,3},-a_{6,3},...,-a_{n,3}$ in $\L_{e_{n+1}}.$
Besides it introduces the entries in the $(5,1)^{st},(6,1)^{st},...,(n,1)^{st}$ positions in $\r_{e_{n+1}}$ and $\L_{e_{n+1}},$ 
which we set to be $a_{5,1},a_{6,1},...,a_{n,1}$ and $-a_{5,1},-a_{6,1},...,-a_{n,1},$ respectively.

\noindent $(I)$ Suppose $b\neq\frac{a}{2}.$
\item The transformation $e^{\prime}_1=e_1+\frac{1}{a-2b}\left(\mathcal{A}_{2,1}+\frac{(b-a)a_{2,3}}{b}\right)e_2,e^{\prime}_2=e_2,
e^{\prime}_3=e_3-\frac{a_{2,3}}{b}e_2,e^{\prime}_{i}=e_{i},
(4\leq i\leq n,n\geq4),e^{\prime}_{n+1}=e_{n+1}-a_{5,1}e_2+\sum_{k=4}^{n-1}a_{k+1,1}e_{k}$
removes $\mathcal{A}_{2,1}$ and $b_{2,1}$ from the $(2,1)^{st}$ positions in $\r_{e_{n+1}}$ and $\L_{e_{n+1}},$
respectively. It also removes $a_{2,3}$
from the $(2,3)^{rd}$ positions in $\r_{e_{n+1}}$ and $\L_{e_{n+1}}$ as well as 
$a_{k+1,1}$ in $\r_{e_{n+1}}$ and $-a_{k+1,1}$ in $\L_{e_{n+1}}$ from the entries in the $(k+1,1)^{st}$
positions, where $(4\leq k\leq n-1)$.
\begin{remark}\label{Remark{a_{5,1}}}
If $n=4,$ then $a_{5,1}=0$ and the same is in case $(II).$
\end{remark}
\item Then we scale $a$ to unity applying the transformation $e^{\prime}_i=e_i,(1\leq i\leq n,n\geq4),e^{\prime}_{n+1}=\frac{e_{n+1}}{a}.$ 
Renaming $\frac{b}{a}$ by $b,$ we obtain a continuous family of Leibniz algebras:
\begin{equation}
\left\{
\begin{array}{l}
\displaystyle   [e_1,e_{n+1}]=e_1+(b-1)e_3,[e_2,e_{n+1}]=2be_2,[e_3,e_{n+1}]=be_3,\\
\displaystyle [e_4,e_{n+1}]=(b-1)e_2+(b+1)e_4,[e_{i},e_{n+1}]=\left(b+i-3\right)e_{i},\\
\displaystyle  [e_{n+1},e_1]=-e_1-(b-1)e_3,[e_{n+1},e_3]=-be_3,[e_{n+1},e_4]=(b+1)e_2-\\
\displaystyle (b+1)e_4,[e_{n+1},e_i]=\left(3-i-b\right)e_i,(5\leq i\leq n),\\
\displaystyle(b\neq0,b\neq\frac{1}{2},b\neq3-n).
\end{array} 
\right.
\label{g_{n+1,1}}
\end{equation} 
\noindent $(II)$ Suppose $b:=\frac{a}{2}.$ We have that $\mathcal{A}_{2,1}=a_{2,3}.$
\item The transformation $e^{\prime}_1=e_1-\frac{b_{2,1}+a_{2,3}}{a}e_2,e^{\prime}_2=e_2,
e^{\prime}_3=e_3-\frac{2a_{2,3}}{a}e_2,e^{\prime}_{i}=e_{i},
(4\leq i\leq n,n\geq4),e^{\prime}_{n+1}=e_{n+1}-a_{5,1}e_2+\sum_{k=4}^{n-1}a_{k+1,1}e_{k}$
removes $a_{2,3}$ from the $(2,1)^{st},(2,3)^{rd}$ positions in $\r_{e_{n+1}}$
and from the $(2,3)^{rd}$ position in $\L_{e_{n+1}}.$
It also removes $b_{2,1}$ from the $(2,1)^{st}$ position in $\L_{e_{n+1}}$ as well as
$a_{k+1,1}$ in $\r_{e_{n+1}}$ and $-a_{k+1,1}$ in $\L_{e_{n+1}}$ from the entries in the $(k+1,1)^{st}$
positions, where $(4\leq k\leq n-1)$.
\item To scale $a$ to unity, we apply the transformation $e^{\prime}_i=e_i,(1\leq i\leq n,n\geq4),e^{\prime}_{n+1}=\frac{e_{n+1}}{a}$ 
and obtain a limiting case of $(\ref{g_{n+1,1}})$ with $b=\frac{1}{2}$ given below:
\begin{equation}
\left\{
\begin{array}{l}
\displaystyle  \nonumber [e_1,e_{n+1}]=e_1-\frac{e_3}{2},[e_2,e_{n+1}]=e_2,[e_3,e_{n+1}]=\frac{e_3}{2},[e_4,e_{n+1}]=-\frac{e_2}{2}+\frac{3e_4}{2},\\
\displaystyle [e_{i},e_{n+1}]=\left(i-\frac{5}{2}\right)e_{i},[e_{n+1},e_1]=-e_1+\frac{e_3}{2},[e_{n+1},e_3]=-\frac{e_3}{2},\\
\displaystyle [e_{n+1},e_4]=\frac{3}{2}\left(e_2-e_4\right),[e_{n+1},e_i]=\left(\frac{5}{2}-i\right)e_i,(5\leq i\leq n).
\end{array} 
\right.
\end{equation} 
\end{itemize}
\item[(2)] We have the right (a derivation) and the left (not a derivation) multiplication operators restricted to the nilradical are as follows:
$$\r_{e_{n+1}}\left[\begin{smallmatrix}
 a & 0 & 0 & 0&0&0&\cdots && 0&0 & 0\\
 \mathcal{A}_{2,1} & 2(3-n)a & a_{2,3}& (2-n)a &0&0 & \cdots &&0  & 0& 0\\
  (2-n)a & 0 & (3-n)a & 0 & 0&0 &\cdots &&0 &0& 0\\
  0 & 0 &  0 & (4-n)a &0 &0 &\cdots&&0 &0 & 0\\
 0 & 0 & a_{5,3} & 0 & (5-n)a  &0&\cdots & &0&0 & 0\\
  0 & 0 &\boldsymbol{\cdot} & a_{5,3} & 0 &(6-n)a&\cdots & &0&0 & 0\\
    0 & 0 &\boldsymbol{\cdot} & \boldsymbol{\cdot} & \ddots &0&\ddots &&\vdots&\vdots &\vdots\\
  \vdots & \vdots & \vdots &\vdots &  &\ddots&\ddots &\ddots &\vdots&\vdots & \vdots\\
 0 & 0 & a_{n-2,3}& a_{n-3,3}& \cdots&\cdots &a_{5,3}&0&-2a &0& 0\\
0 & 0 & a_{n-1,3}& a_{n-2,3}& \cdots&\cdots &\boldsymbol{\cdot}&a_{5,3}&0 &-a& 0\\
 0 & 0 & a_{n,3}& a_{n-1,3}& \cdots&\cdots &\boldsymbol{\cdot}&\boldsymbol{\cdot}&a_{5,3} &0&0
\end{smallmatrix}\right],$$
 $$\L_{e_{n+1}}=\left[\begin{smallmatrix}
 -a & 0 & 0 & 0&0&0&\cdots && 0&0 & 0\\
  b_{2,1} & 0 & a_{2,3}& (4-n)a &0&0 & \cdots &&0  & 0& 0\\
  (n-2)a & 0 & (n-3)a & 0 & 0&0 &\cdots &&0 &0& 0\\
  0 & 0 &  0 & (n-4)a &0 &0 &\cdots&&0 &0 & 0\\
 0 & 0 & -a_{5,3} & 0 & (n-5)a  &0&\cdots & &0&0 & 0\\
  0 & 0 &\boldsymbol{\cdot} & -a_{5,3} & 0 &(n-6)a&\cdots & &0&0 & 0\\
    0 & 0 &\boldsymbol{\cdot} & \boldsymbol{\cdot} & \ddots &0&\ddots &&\vdots&\vdots &\vdots\\
  \vdots & \vdots & \vdots &\vdots &  &\ddots&\ddots &\ddots &\vdots&\vdots & \vdots\\
 0 & 0 & -a_{n-2,3}& -a_{n-3,3}& \cdots&\cdots &-a_{5,3}&0&2a &0& 0\\
0 & 0 & -a_{n-1,3}& -a_{n-2,3}& \cdots&\cdots &\boldsymbol{\cdot}&-a_{5,3}&0 &a& 0\\
 0 & 0 & -a_{n,3}& -a_{n-1,3}& \cdots&\cdots &\boldsymbol{\cdot}&\boldsymbol{\cdot}&-a_{5,3} &0&0
\end{smallmatrix}\right].$$

\begin{itemize}[noitemsep, topsep=0pt]
\allowdisplaybreaks 
\item We apply the transformation $e^{\prime}_1=e_1,e^{\prime}_2=e_2,e^{\prime}_i=e_i-\frac{a_{k-i+3,3}}{(k-i)a}e_k,(3\leq i\leq n-2,
i+2\leq k\leq n,n\geq5),e^{\prime}_{j}=e_{j},(n-1\leq j\leq n+1),$ where $k$ is fixed, renaming all the affected entries back.
This transformation removes $a_{5,3},a_{6,3},...,a_{n,3}$ in $\r_{e_{n+1}}$ and $-a_{5,3},-a_{6,3},...,-a_{n,3}$ in $\L_{e_{n+1}}.$
Besides it introduces the entries in the $(5,1)^{st},(6,1)^{st},...,(n,1)^{st}$ positions in $\r_{e_{n+1}}$ and $\L_{e_{n+1}},$ 
which we set to be $a_{5,1},a_{6,1},...,a_{n,1}$ and $-a_{5,1},-a_{6,1},...,-a_{n,1},$ respectively.

\item The transformation $e^{\prime}_1=e_1+\frac{1}{(2n-5)a}\left(\mathcal{A}_{2,1}+\frac{(n-2)a_{2,3}}{n-3}\right)e_2,e^{\prime}_2=e_2,
e^{\prime}_3=e_3+\frac{a_{2,3}}{(n-3)a}e_2,e^{\prime}_{i}=e_{i},
(4\leq i\leq n,n\geq4),e^{\prime}_{n+1}=e_{n+1}-a_{5,1}e_2+\sum_{k=4}^{n-1}a_{k+1,1}e_{k}$
removes $\mathcal{A}_{2,1}$ and $b_{2,1}$ from the $(2,1)^{st}$ positions in $\r_{e_{n+1}}$ and $\L_{e_{n+1}},$
respectively. It also removes $a_{2,3}$
from the $(2,3)^{rd}$ positions in $\r_{e_{n+1}}$ and $\L_{e_{n+1}};$ 
$a_{k+1,1}$ and $-a_{k+1,1}$ from the entries in the $(k+1,1)^{st}$
positions in $\r_{e_{n+1}}$ and $\L_{e_{n+1}}$, respectively, where $(4\leq k\leq n-1)$. (See Remark \ref{Remark{a_{5,1}}}).
\item To scale $a$ to unity, we apply the transformation $e^{\prime}_i=e_i,(1\leq i\leq n),e^{\prime}_{n+1}=\frac{e_{n+1}}{a}$ 
renaming the coefficient $\frac{a_{n,n+1}}{a^2}$ in front of $e_n$ in $[e_{n+1},e_{n+1}]$ back by $a_{n,n+1}.$ We obtain a Leibniz algebra:
\begin{equation}
\left\{
\begin{array}{l}
\displaystyle  \nonumber [e_1,e_{n+1}]=e_1+(2-n)e_3,[e_2,e_{n+1}]=2(3-n)e_2, [e_3,e_{n+1}]=(3-n)e_3,\\
\displaystyle [e_4,e_{n+1}]=(2-n)e_2+(4-n)e_4,[e_{i},e_{n+1}]=(i-n)e_{i},[e_{n+1},e_{n+1}]=a_{n,n+1}e_n,\\
\displaystyle [e_{n+1},e_1]=-e_1+(n-2)e_3,[e_{n+1},e_3]=(n-3)e_3,[e_{n+1},e_4]=(4-n)\left(e_2-e_4\right),\\
\displaystyle [e_{n+1},e_i]=(n-i)e_i,(5\leq i\leq n).
\end{array} 
\right.
\end{equation} 
 \end{itemize}
If $a_{n,n+1}=0,$ then we have a limiting case of (\ref{g_{n+1,1}}) with $b=3-n$. If $a_{n,n+1}\neq0,$
then $a_{n,n+1}=re^{i\phi}$ and we apply the transformation
$e^{\prime}_j=\left(re^{i\phi}\right)^{\frac{j}{n-2}}e_j,(1\leq j\leq 2),e^{\prime}_k=\left(re^{i\phi}\right)^{\frac{k-2}{n-2}}e_k,(3\leq k\leq n),
e^{\prime}_{n+1}=e_{n+1}$
to scale $a_{n,n+1}$ to $1$. We have the algebra $\g_{n+1,2}$ given below:
\begin{equation}
\left\{
\begin{array}{l}
\displaystyle  \nonumber [e_1,e_{n+1}]=e_1+(2-n)e_3,[e_2,e_{n+1}]=2(3-n)e_2, [e_3,e_{n+1}]=(3-n)e_3,\\
\displaystyle [e_4,e_{n+1}]=(2-n)e_2+(4-n)e_4,[e_{i},e_{n+1}]=(i-n)e_{i},[e_{n+1},e_{n+1}]=e_n,\\
\displaystyle [e_{n+1},e_1]=-e_1+(n-2)e_3,[e_{n+1},e_3]=(n-3)e_3,[e_{n+1},e_4]=(4-n)\left(e_2-e_4\right),\\
\displaystyle [e_{n+1},e_i]=(n-i)e_i,(5\leq i\leq n).
\end{array} 
\label{g_{n+1,2}}
\right.
\end{equation} 
\item[(3)] We have the right (a derivation) and the left (not a derivation) multiplication operators restricted to the nilradical are as follows:
$$\r_{e_{n+1}}=\left[\begin{smallmatrix}
 0 & 0 & 0 & 0&0&0&\cdots && 0&0 & 0\\
  a_{2,1} & 2b & a_{2,3}& b &0&0 & \cdots &&0  & 0& 0\\
  b & 0 & b & 0 & 0&0 &\cdots &&0 &0& 0\\
  0 & 0 &  0 & b &0 &0 &\cdots&&0 &0 & 0\\
 0 & 0 & a_{5,3} & 0 & b  &0&\cdots & &0&0 & 0\\
  0 & 0 &\boldsymbol{\cdot} & a_{5,3} & 0 &b&\cdots & &0&0 & 0\\
    0 & 0 &\boldsymbol{\cdot} & \boldsymbol{\cdot} & \ddots &0&\ddots &&\vdots&\vdots &\vdots\\
  \vdots & \vdots & \vdots &\vdots &  &\ddots&\ddots &\ddots &\vdots&\vdots & \vdots\\
 0 & 0 & a_{n-2,3}& a_{n-3,3}& \cdots&\cdots &a_{5,3}&0&b &0& 0\\
0 & 0 & a_{n-1,3}& a_{n-2,3}& \cdots&\cdots &\boldsymbol{\cdot}&a_{5,3}&0 &b& 0\\
 0 & 0 & a_{n,3}& a_{n-1,3}& \cdots&\cdots &\boldsymbol{\cdot}&\boldsymbol{\cdot}&a_{5,3} &0& b
\end{smallmatrix}\right],$$
 $$\L_{e_{n+1}}=\left[\begin{smallmatrix}
 0 & 0 & 0 & 0&0&0&\cdots && 0&0 & 0\\
  a_{2,3} & 0 & a_{2,3}& b &0&0 & \cdots &&0  & 0& 0\\
  -b & 0 & -b & 0 & 0&0 &\cdots &&0 &0& 0\\
  0 & 0 &  0 & -b &0 &0 &\cdots&&0 &0 & 0\\
 0 & 0 & -a_{5,3} & 0 & -b  &0&\cdots & &0&0 & 0\\
  0 & 0 &\boldsymbol{\cdot} & -a_{5,3} & 0 &-b&\cdots & &0&0 & 0\\
    0 & 0 &\boldsymbol{\cdot} & \boldsymbol{\cdot} & \ddots &0&\ddots &&\vdots&\vdots &\vdots\\
  \vdots & \vdots & \vdots &\vdots &  &\ddots&\ddots &\ddots &\vdots&\vdots & \vdots\\
 0 & 0 & -a_{n-2,3}& -a_{n-3,3}& \cdots&\cdots &-a_{5,3}&0&-b &0& 0\\
0 & 0 & -a_{n-1,3}& -a_{n-2,3}& \cdots&\cdots &\boldsymbol{\cdot}&-a_{5,3}&0 &-b& 0\\
 0 & 0 & -a_{n,3}& -a_{n-1,3}& \cdots&\cdots &\boldsymbol{\cdot}&\boldsymbol{\cdot}&-a_{5,3} &0&-b
\end{smallmatrix}\right].$$
\begin{itemize}[noitemsep, topsep=0pt]
\allowdisplaybreaks 
\item Applying the transformation $e^{\prime}_1=e_1-\frac{a_{2,1}+a_{2,3}}{2b}e_2,e^{\prime}_2=e_2,
e^{\prime}_3=e_3-\frac{a_{2,3}}{b}e_2,e^{\prime}_{i}=e_{i},
(4\leq i\leq n+1),$
we
remove $a_{2,1}$ from the $(2,1)^{st}$ position in $\r_{e_{n+1}}$ and 
$a_{2,3}$ from the $(2,1)^{st},$ the $(2,3)^{rd}$ positions in
$\L_{e_{n+1}}$ and from the $(2,3)^{rd}$ position in $\r_{e_{n+1}}$ keeping other entries unchanged.

\item To scale $b$ to unity, we apply the transformation $e^{\prime}_i=e_i,(1\leq i\leq n),e^{\prime}_{n+1}=\frac{e_{n+1}}{b}.$ 
Then we rename $\frac{a_{5,3}}{b},\frac{a_{6,3}}{b},...,\frac{a_{n,3}}{b}$ by $a_{5,3},a_{6,3},...,a_{n,3},$ respectively.
We obtain a Leibniz algebra
\begin{equation}
\left\{
\begin{array}{l}
\displaystyle  \nonumber [e_1,e_{n+1}]=e_3,[e_2,e_{n+1}]=2e_2,[e_4,e_{n+1}]=e_2,[e_{i},e_{n+1}]=e_{i}+\sum_{k=i+2}^n{a_{k-i+3,3}e_k},\\
\displaystyle [e_{n+1},e_1]=-e_3,[e_{n+1},e_4]=e_2,[e_{n+1},e_i]=-e_i-\sum_{k=i+2}^n{a_{k-i+3,3}e_k},(3\leq i\leq n).
\end{array} 
\right.
\end{equation}
 If $a_{5,3}\neq0,(n\geq5),$ then $a_{5,3}=re^{i\phi}$ and applying
the transformation $e^{\prime}_j=\left(re^{i\phi}\right)^{\frac{j}{2}}e_j,(1\leq j\leq 2),e^{\prime}_k=\left(re^{i\phi}\right)^{\frac{k-2}{2}}e_k,
(3\leq k\leq n),e^{\prime}_{n+1}=e_{n+1},$ we scale $a_{5,3}$ to $1.$ We also rename all the affected entries back
and then we rename $a_{6,3},...,a_{n,3}$ by $b_1,...,b_{n-5},$ respectively.
We combine with the case when $a_{5,3}=0$ and obtain a Leibniz algebra $\g_{n+1,3}$ given below:
\begin{equation}
\left\{
\begin{array}{l}
\displaystyle  \nonumber [e_1,e_{n+1}]=e_3,[e_2,e_{n+1}]=2e_2,[e_4,e_{n+1}]=e_2,[e_{i},e_{n+1}]=e_{i}+\epsilon e_{i+2}+\sum_{k=i+3}^n{b_{k-i-2}e_k},\\
\displaystyle [e_{n+1},e_1]=-e_3,[e_{n+1},e_4]=e_2,[e_{n+1},e_i]=-e_i-\epsilon e_{i+2}-\sum_{k=i+3}^n{b_{k-i-2}e_k},\\
\displaystyle (\epsilon=0,1,3\leq i\leq n).
\end{array} 
\right.
\end{equation}
\begin{remark}
If $n=4,$ then $\epsilon=0.$
\end{remark}
\end{itemize}
\item[(4)] We have the right (a derivation) and the left (not a derivation) multiplication operators restricted to the nilradical are as follows:
$$\r_{e_{n+1}}=\left[\begin{smallmatrix}
 a & 0 & 0 & 0&0&0&\cdots && 0&0 & 0\\
 a_{2,3}-b_{2,1}+b_{2,3} & 0 & a_{2,3}& -a &0&0 & \cdots &&0  & 0& 0\\
  -a & 0 & 0 & 0 & 0&0 &\cdots &&0 &0& 0\\
  0 & 0 &  0 & a &0 &0 &\cdots&&0 &0 & 0\\
 0 & 0 & a_{5,3} & 0 & 2a  &0&\cdots & &0&0 & 0\\
  0 & 0 &\boldsymbol{\cdot} & a_{5,3} & 0 &3a&\cdots & &0&0 & 0\\
    0 & 0 &\boldsymbol{\cdot} & \boldsymbol{\cdot} & \ddots &0&\ddots &&\vdots&\vdots &\vdots\\
  \vdots & \vdots & \vdots &\vdots &  &\ddots&\ddots &\ddots &\vdots&\vdots & \vdots\\
 0 & 0 & a_{n-2,3}& a_{n-3,3}& \cdots&\cdots &a_{5,3}&0&(n-5)a &0& 0\\
0 & 0 & a_{n-1,3}& a_{n-2,3}& \cdots&\cdots &\boldsymbol{\cdot}&a_{5,3}&0 &(n-4)a& 0\\
 0 & 0 & a_{n,3}& a_{n-1,3}& \cdots&\cdots &\boldsymbol{\cdot}&\boldsymbol{\cdot}&a_{5,3} &0& (n-3)a
\end{smallmatrix}\right],$$
 $$\L_{e_{n+1}}=\left[\begin{smallmatrix}
 -a & 0 & 0 & 0&0&0&\cdots && 0&0 & 0\\
  b_{2,1}& 0 & b_{2,3}& a &0&0 & \cdots &&0  & 0& 0\\
  a & 0 & 0 & 0 & 0&0 &\cdots &&0 &0& 0\\
  0 & 0 &  0 & -a &0 &0 &\cdots&&0 &0 & 0\\
 0 & 0 & -a_{5,3} & 0 & -2a  &0&\cdots & &0&0 & 0\\
  0 & 0 &\boldsymbol{\cdot} & -a_{5,3} & 0 &-3a&\cdots & &0&0 & 0\\
    0 & 0 &\boldsymbol{\cdot} & \boldsymbol{\cdot} & \ddots &0&\ddots &&\vdots&\vdots &\vdots\\
  \vdots & \vdots & \vdots &\vdots &  &\ddots&\ddots &\ddots &\vdots&\vdots & \vdots\\
 0 & 0 & -a_{n-2,3}& -a_{n-3,3}& \cdots&\cdots &-a_{5,3}&0&(5-n)a &0& 0\\
0 & 0 & -a_{n-1,3}& -a_{n-2,3}& \cdots&\cdots &\boldsymbol{\cdot}&-a_{5,3}&0 &(4-n)a& 0\\
 0 & 0 & -a_{n,3}& -a_{n-1,3}& \cdots&\cdots &\boldsymbol{\cdot}&\boldsymbol{\cdot}&-a_{5,3} &0& (3-n)a
\end{smallmatrix}\right].$$
\begin{itemize}[noitemsep, topsep=0pt]
\allowdisplaybreaks 
\item We apply the transformation $e^{\prime}_1=e_1,e^{\prime}_2=e_2,e^{\prime}_i=e_i-\frac{a_{k-i+3,3}}{(k-i)a}e_k,(3\leq i\leq n-2,
i+2\leq k\leq n,n\geq5),e^{\prime}_{j}=e_{j},(n-1\leq j\leq n+1),$ where $k$ is fixed, renaming all the affected entries back.
This transformation removes $a_{5,3},a_{6,3},...,a_{n,3}$ in $\r_{e_{n+1}}$ and $-a_{5,3},-a_{6,3},...,-a_{n,3}$ in $\L_{e_{n+1}}.$
Besides it introduces the entries in the $(5,1)^{st},(6,1)^{st},...,(n,1)^{st}$ positions in $\r_{e_{n+1}}$ and $\L_{e_{n+1}},$ 
which we set to be $a_{5,1},a_{6,1},...,a_{n,1}$ and $-a_{5,1},-a_{6,1},...,-a_{n,1},$ respectively.

\item The transformation $e^{\prime}_1=e_1+\frac{a_{2,3}-b_{2,1}+b_{2,3}}{a}e_2,e^{\prime}_{i}=e_{i},
(2\leq i\leq n,n\geq4),e^{\prime}_{n+1}=e_{n+1}+\sum_{k=4}^{n-1}a_{k+1,1}e_{k}$
removes $a_{2,3}-b_{2,1}+b_{2,3}$ from the $(2,1)^{st}$ position in $\r_{e_{n+1}}$.
It changes the entry in the $(2,1)^{st}$ position in $\L_{e_{n+1}}$ to $a_{2,3}+b_{2,3}.$
It also removes
$a_{k+1,1}$ and $-a_{k+1,1}$ from the entries in the $(k+1,1)^{st}$
positions, where $(4\leq k\leq n-1)$ in $\r_{e_{n+1}}$ and $\L_{e_{n+1}},$ respectively.

\item We assign $a_{2,3}:=d$ and $b_{2,3}:=f$
and
then we scale $a$ to unity applying the transformation $e^{\prime}_i=e_i,(1\leq i\leq n,n\geq4),e^{\prime}_{n+1}=\frac{e_{n+1}}{a}.$ 
Renaming $\frac{d}{a},\frac{f}{a}$ and $\frac{a_{2,n+1}}{a^2}$ by $d,f$ and $a_{2,n+1},$ respectively, we obtain a right and left Leibniz algebra:
\begin{equation}
\left\{
\begin{array}{l}
\displaystyle  \nonumber [e_1,e_{n+1}]=e_1-e_3,
[e_3,e_{n+1}]=de_2,[e_4,e_{n+1}]=e_4-e_2,[e_{i},e_{n+1}]=(i-3)e_{i},\\
\displaystyle 
[e_{n+1},e_{n+1}]=a_{2,n+1}e_2,[e_{n+1},e_1]=-e_1+\left(d+f\right)e_2+e_3,[e_{n+1},e_3]=fe_2,\\
\displaystyle [e_{n+1},e_4]=e_2-e_4,[e_{n+1},e_i]=(3-i)e_i,(5\leq i\leq n),
\end{array} 
\right.
\end{equation} 
which is a limiting case of (\ref{g_{n+1,1}}) with $b=0,$ when $d=f=a_{2,n+1}=0$. Altogether (\ref{g_{n+1,1}}) and all its limiting
cases after replacing $b$ with $a$ give us a Leibniz algebra $\g_{n+1,1}$ given below:
\begin{equation}
\left\{
\begin{array}{l}
\displaystyle \nonumber  [e_1,e_{n+1}]=e_1+(a-1)e_3,[e_2,e_{n+1}]=2ae_2,[e_3,e_{n+1}]=ae_3,\\
\displaystyle [e_4,e_{n+1}]=(a-1)e_2+(a+1)e_4,[e_{i},e_{n+1}]=\left(a+i-3\right)e_{i},[e_{n+1},e_1]=-e_1-\\
\displaystyle (a-1)e_3,[e_{n+1},e_3]=-ae_3,[e_{n+1},e_4]=(a+1)\left(e_2-e_4\right),[e_{n+1},e_i]=\left(3-i-a\right)e_i,\\
\displaystyle(5\leq i\leq n).
\end{array} 
\right.
\end{equation} 
It remains to consider a continuous family of Leibniz algebras given below
 and scale any nonzero entries as much as possible.
\begin{equation}
\left\{
\begin{array}{l}
\displaystyle  \nonumber [e_1,e_{n+1}]=e_1-e_3,
[e_3,e_{n+1}]=de_2,[e_4,e_{n+1}]=e_4-e_2,[e_{i},e_{n+1}]=(i-3)e_{i},\\
\displaystyle 
[e_{n+1},e_{n+1}]=a_{2,n+1}e_2,[e_{n+1},e_1]=-e_1+\left(d+f\right)e_2+e_3,[e_{n+1},e_3]=fe_2,\\
\displaystyle [e_{n+1},e_4]=e_2-e_4,[e_{n+1},e_i]=(3-i)e_i,(a_{2,n+1}^2+d^2+f^2\neq0,5\leq i\leq n),
\end{array} 
\right.
\end{equation} 
 If $a_{2,n+1}\neq0,$ then $a_{2,n+1}=re^{i\phi}$ and applying
the transformation $e^{\prime}_j=\left(re^{i\phi}\right)^{\frac{j}{2}}e_j,(1\leq j\leq 2),e^{\prime}_k=\left(re^{i\phi}\right)^{\frac{k-2}{2}}e_k,$
$(3\leq k\leq n,n\geq4),e^{\prime}_{n+1}=e_{n+1},$ we scale it to $1.$ We also rename all the affected entries back.
Then we combine with the case when $a_{2,n+1}=0$ and obtain a right and left Leibniz algebra $\g_{n+1,4}$ given below:
\begin{equation}
\left\{
\begin{array}{l}
\displaystyle  \nonumber [e_1,e_{n+1}]=e_1-e_3,
[e_3,e_{n+1}]=de_2,[e_4,e_{n+1}]=e_4-e_2,[e_{i},e_{n+1}]=(i-3)e_{i},\\
\displaystyle 
[e_{n+1},e_{n+1}]=\epsilon e_2,[e_{n+1},e_1]=-e_1+\left(d+f\right)e_2+e_3,[e_{n+1},e_3]=fe_2,\\
\displaystyle  [e_{n+1},e_4]=e_2-e_4,[e_{n+1},e_i]=(3-i)e_i,\\
\displaystyle(5\leq i\leq n;\epsilon=0,1; if\,\,\epsilon=0,\,\,then\,\,d^2+f^2\neq0).
\end{array} 
\right.
\end{equation} 
\end{itemize}
\item[(5)] Applying the transformation
$e^{\prime}_1=e_1+\frac{(b+2c-2a)a_{2,3}+(b+2c)b_{2,3}}{2a(b+c)}e_2,e^{\prime}_2=e_2,$
$e^{\prime}_3=e_3+\frac{(c-2a)a_{2,3}+c\cdot b_{2,3}}{2a(b+c)}e_2,
e^{\prime}_4=e_4,e^{\prime}_5=\frac{e_5}{c}$ and renaming
$\frac{a}{c}$ and $\frac{b}{c}$ by $a$ and $b,$ respectively, we obtain a continuous family of Leibniz algebras given below:
\begin{equation}
\left\{
\begin{array}{l}
\displaystyle  [e_1,e_{5}]=ae_1+(b-a+1)e_3,[e_2,e_{5}]=2(b+1)e_2,
[e_3,e_{5}]=e_1+be_3,\\
\displaystyle [e_4,e_{5}]=(b-a+2)e_2+(a+b)e_4,[e_{5},e_1]=-ae_1+(a-b-1)e_3,\\
\displaystyle [e_{5},e_3]=-e_1-be_3,[e_{5},e_4]=(a+b)(e_2-e_4),(b\neq-a,a\neq0,b\neq-1).
\end{array} 
\right.
\label{g_{5,5}}
\end{equation} 
\item[(6)] We
apply the transformation
$e^{\prime}_1=e_1+\frac{(3a-2c)a_{2,3}+(a-2c)b_{2,3}}{2a(a-c)}e_2,
e^{\prime}_2=e_2,$
$e^{\prime}_3=e_3+\frac{(2a-c)a_{2,3}-c\cdot b_{2,3}}{2a(a-c)}e_2,e^{\prime}_4=e_4,e^{\prime}_5=\frac{e_5}{c}$
and rename
$\frac{a}{c}$ and $\frac{a_{4,5}}{c^2}$ by $a$ and $a_{4,5},$ respectively,
to obtain a Leibniz algebra given below:
\begin{equation}
\left\{
\begin{array}{l}
\displaystyle  [e_1,e_{5}]=ae_1+(1-2a)e_3,[e_2,e_{5}]=2(1-a)e_2,
[e_3,e_{5}]=e_1-ae_3,\\
\displaystyle [e_4,e_{5}]=2(1-a)e_2,[e_{5},e_{5}]=a_{4,5}e_4,[e_{5},e_1]=-ae_1+(2a-1)e_3,\\
\displaystyle [e_{5},e_3]=ae_3-e_1,(a\neq0,a\neq1),
\end{array} 
\label{general1}
\right.
\end{equation} 
which is a limiting case of $(\ref{g_{5,5}})$ with $b:=-a$ if $a_{4,5}=0.$
If $a_{4,5}\neq0,$ then $a_{4,5}=re^{i\phi}$. To scale $a_{4,5}$ to $1,$
we apply the transformation
$e^{\prime}_1=\sqrt{r}e^{i\frac{\phi}{2}}e_1,e^{\prime}_2=re^{i\phi}e_2,e^{\prime}_3=\sqrt{r}e^{i\frac{\phi}{2}}e_3,e^{\prime}_4=re^{i\phi}e_4,e^{\prime}_{5}=e_{5}$
and obtain a Leibniz algebra given below:
\begin{equation}
\left\{
\begin{array}{l}
\displaystyle  [e_1,e_{5}]=ae_1+(1-2a)e_3,[e_2,e_{5}]=2(1-a)e_2,
[e_3,e_{5}]=e_1-ae_3,\\
\displaystyle [e_4,e_{5}]=2(1-a)e_2,[e_{5},e_{5}]=e_4,[e_{5},e_1]=-ae_1+(2a-1)e_3,\\
\displaystyle [e_{5},e_3]=-e_1+ae_3,(a\neq0,a\neq1).
\end{array} 
\right.
\label{g_{5,6}}
\end{equation} 
\item[(7)] We apply the transformation
$e^{\prime}_1=e_1+\frac{a_{2,3}-b_{2,1}+b_{2,3}}{a}e_2,e^{\prime}_i=e_i,(2\leq i\leq 5)$
and assign $d:=\frac{a\cdot b_{2,3}+c\left(a_{2,3}-b_{2,1}+b_{2,3}\right)}{a},$
$f:=\frac{a\cdot a_{2,3}-c\left(a_{2,3}-b_{2,1}+b_{2,3}\right)}{a}$ to obtain a continuous family
of Leibniz algebras:
\begin{equation}
\left\{
\begin{array}{l}
\displaystyle  \nonumber [e_1,e_{5}]=a(e_1-e_3),
[e_3,e_{5}]=ce_1+fe_2-ce_3,[e_4,e_{5}]=(c-a)(e_2-e_4),[e_{5},e_{5}]=a_{2,5}e_2,\\
\displaystyle [e_{5},e_1]=-ae_1+\left(d+f\right)e_2+ae_3,[e_{5},e_3]=-ce_1+de_2+ce_3,[e_{5},e_4]=(a-c)(e_2-e_4).
\end{array} 
\right.
\end{equation} 
Then we continue with the transformation $e^{\prime}_i=e_i,(1\leq i\leq 4),e^{\prime}_5=\frac{e_5}{c}$ renaming
$\frac{d}{c},\frac{f}{c},\frac{a}{c}$ and $\frac{a_{2,5}}{c^2}$
by $d,f,a$ and $a_{2,5},$ respectively,
to obtain a Leibniz algebra:
\begin{equation}
\left\{
\begin{array}{l}
\displaystyle   [e_1,e_{5}]=a(e_1-e_3),
[e_3,e_{5}]=e_1+fe_2-e_3,[e_4,e_{5}]=(1-a)\left(e_2-e_4\right),\\
\displaystyle [e_{5},e_{5}]=a_{2,5}e_2,[e_{5},e_1]=-ae_1+\left(d+f\right)e_2+ae_3,[e_{5},e_3]=-e_1+de_2+e_3,\\
\displaystyle [e_{5},e_4]=(a-1)\left(e_2-e_4\right),(a\neq0,a\neq 1).
\end{array} 
\right.
\label{general}
\end{equation} 
If $d=f=a_{2,5}=0,$ then we have a limiting case of $(\ref{g_{5,5}})$ with $b=-1.$
If $a_{2,5}\neq0,$ then we apply
the same transformation as we did to scale $a_{4,5}$ to $1.$ We also rename all the affected entries back.
Then we combine with the case when $a_{2,5}=0$ and obtain a Leibniz algebra given below:
\begin{equation}
\left\{
\begin{array}{l}
\displaystyle  [e_1,e_{5}]=a(e_1-e_3),
[e_3,e_{5}]=e_1+fe_2-e_3,[e_4,e_{5}]=(1-a)\left(e_2-e_4\right),\\
\displaystyle [e_{5},e_{5}]=\epsilon e_2,[e_{5},e_1]=-ae_1+\left(d+f\right)e_2+ae_3,[e_{5},e_3]=-e_1+de_2+e_3,\\
\displaystyle [e_{5},e_4]=(a-1)\left(e_2-e_4\right),(\epsilon=0,1;if\,\,\epsilon=0,then\,\,d^2+f^2\neq0),\\
\displaystyle (a\neq0,a\neq 1).
\end{array} 
\label{g_{5,7}}
\right.
\end{equation}
\begin{remark}\label{Remark{g_{5,7}}}
We notice by applying the transformation
$e^{\prime}_1=e_1+fe_2,e^{\prime}_i=e_i,(2\leq i\leq 5)$ that this algebra $(\ref{g_{5,7}})$
is isomorphic to
\begin{equation}
\left\{
\begin{array}{l}
\displaystyle \nonumber [e_1,e_{5}]=ae_1-afe_2-ae_3,
[e_3,e_{5}]=e_1-e_3,[e_4,e_{5}]=(1-a)\left(e_2-e_4\right),[e_{5},e_{5}]=\epsilon e_2,\\
\displaystyle [e_{5},e_1]=-ae_1+\left(af+d+f\right)e_2+ae_3,[e_{5},e_3]=-e_1+(d+f)e_2+e_3,\\
\displaystyle [e_{5},e_4]=(a-1)\left(e_2-e_4\right),(\epsilon=0,1;if\,\,\epsilon=0,then\,\,d^2+f^2\neq0),(a\neq0,a\neq 1).
\end{array} 
\right.
\end{equation}
\end{remark} 
\item[(8)] Applying the transformation
$e^{\prime}_1=e_1+\frac{a_{2,3}-2b_{2,1}}{c}e_2,e^{\prime}_2=e_2,
e^{\prime}_3=e_3-\frac{b_{2,1}}{c}e_2,e^{\prime}_4=e_4,e^{\prime}_5=\frac{e_5}{c}$
and renaming $\frac{a_{4,5}}{c^2}$ back by $a_{4,5},$ we obtain a Leibniz algebra:
\begin{equation}
\left\{
\begin{array}{l}
\displaystyle  \nonumber [e_1,e_{5}]=e_3,[e_2,e_{5}]=2e_2,
[e_3,e_{5}]=e_1,[e_4,e_{5}]=2e_2,[e_{5},e_{5}]=a_{4,5}e_4,[e_{5},e_1]=-e_3,\\
\displaystyle [e_{5},e_3]=-e_1,
\end{array} 
\right.
\end{equation} 
which is a limiting case of $(\ref{general1})$ with $a=0$ and at the same time a limiting case of $(\ref{g_{5,5}})$
with $b:=-a$ if $a_{4,5}=0.$ Therefore we only consider the case when $a_{4,5}\neq0.$
One applies the transformation below $(\ref{general1})$ to scale $a_{4,5}$ to $1$ and we have a limiting case of $(\ref{g_{5,6}})$
with $a=0.$ Altogether $(\ref{g_{5,6}})$ and all its limiting cases give us the algebra $g_{5,6}$ in full generality:
\begin{equation}
\left\{
\begin{array}{l}
\displaystyle \nonumber [e_1,e_{5}]=ae_1+(1-2a)e_3,[e_2,e_{5}]=2(1-a)e_2,
[e_3,e_{5}]=e_1-ae_3,[e_4,e_{5}]=2(1-a)e_2,\\
\displaystyle [e_{5},e_{5}]=e_4,[e_{5},e_1]=-ae_1+(2a-1)e_3,[e_{5},e_3]=-e_1+ae_3,
\end{array} 
\right.
\end{equation} 
where $a\neq1,$ otherwise nilpotent.
\item[(9)] If $a=0,b\neq0,b\neq-c,c\neq0,(n=4),$ then we apply the transformation 
$e^{\prime}_1=e_1+\frac{(b+2c)b_{2,3}-(3b+2c)a_{2,3}}{(b+c)^2}e_2,
e^{\prime}_2=e_2,e^{\prime}_3=e_3-\frac{\mathcal{B}_{2,1}}{b+c}e_2,e^{\prime}_4=e_4,
e^{\prime}_5=\frac{e_5}{c}.$ We rename $\frac{b}{c}$ by $b$ and obtain a Leibniz algebra:
\begin{equation}
\left\{
\begin{array}{l}
\displaystyle  \nonumber [e_1,e_{5}]=(b+1)e_3,[e_2,e_{5}]=2(b+1)e_2,
[e_3,e_{5}]=e_1+be_3,[e_4,e_{5}]=(b+2)e_2+be_4,\\
\displaystyle [e_{5},e_1]=(-b-1)e_3,[e_{5},e_3]=-e_1-be_3,[e_{5},e_4]=b(e_2-e_4),(b\neq0,b\neq-1),
\end{array} 
\right.
\end{equation}
which is a limiting case of $(\ref{g_{5,5}})$ with $a=0.$ Altogether $(\ref{g_{5,5}})$ and all its
limiting cases give us the algebra $g_{5,5}$ in full generality:
\begin{equation}
\left\{
\begin{array}{l}
\displaystyle \nonumber  [e_1,e_{5}]=ae_1+(b-a+1)e_3,[e_2,e_{5}]=2(b+1)e_2,
[e_3,e_{5}]=e_1+be_3,\\
\displaystyle [e_4,e_{5}]=(b-a+2)e_2+(a+b)e_4,[e_{5},e_1]=-ae_1+(a-b-1)e_3,[e_{5},e_3]=-e_1-be_3,\\
\displaystyle [e_{5},e_4]=(a+b)(e_2-e_4),
\end{array} 
\right.
\end{equation}
where if $b=-1,$ then $a\neq1,$ otherwise nilpotent.
\item[(10)] We apply the transformation
$e^{\prime}_1=e_1+\frac{a_{2,3}}{c}e_2,e^{\prime}_i=e_i,(2\leq i\leq 4),e^{\prime}_5=\frac{e_5}{c}$
and rename $\frac{a_{2,1}}{c},\frac{a_{2,3}}{c},\frac{b_{2,3}}{c}$ and $\frac{a_{2,5}}{c^2}$ by
$a_{2,1},a_{2,3},b_{2,3}$ and $a_{2,5},$ respectively, to
obtain a Leibniz algebra:
\begin{equation}
\left\{
\begin{array}{l}
\displaystyle  \nonumber [e_1,e_{5}]=a_{2,1}e_2,
[e_3,e_{5}]=e_1-e_3,[e_4,e_{5}]=e_2-e_4,[e_{5},e_{5}]=a_{2,5}e_2,\\
\displaystyle [e_{5},e_1]=\left(a_{2,3}+b_{2,3}-a_{2,1}\right)e_2,[e_{5},e_3]=-e_1+\left(a_{2,3}+b_{2,3}\right)e_2+e_3,[e_{5},e_4]=e_4-e_2.
\end{array} 
\right.
\end{equation} 
We scale $a_{2,5}$ to $1$ if nonzero.
Combining with the case when $a_{2,5}=0$ and renaming all the affected entries back, we have the algebra:
\begin{equation}
\left\{
\begin{array}{l}
\displaystyle  \nonumber [e_1,e_{5}]=a_{2,1}e_2,
[e_3,e_{5}]=e_1-e_3,[e_4,e_{5}]=e_2-e_4,[e_{5},e_{5}]=\epsilon e_2,\\
\displaystyle [e_{5},e_1]=\left(a_{2,3}+b_{2,3}-a_{2,1}\right)e_2,[e_{5},e_3]=-e_1+\left(a_{2,3}+b_{2,3}\right)e_2+e_3,[e_{5},e_4]=e_4-e_2,\\
\displaystyle (\epsilon=0,1).
\end{array} 
\right.
\end{equation}
If $a_{2,1}=0,$ then we obtain a limiting case of $(\ref{general})$ and $(\ref{g_{5,7}})$ with $a=0$ according to Remark \ref{Remark{g_{5,7}}}.
Altogether $(\ref{g_{5,7}})$ and all its limiting cases, give us the algebra $g_{5,7}$ in full generality:
\begin{equation}
\left\{
\begin{array}{l}
\displaystyle \nonumber [e_1,e_{5}]=a(e_1-e_3),
[e_3,e_{5}]=e_1+fe_2-e_3,[e_4,e_{5}]=(1-a)\left(e_2-e_4\right),[e_{5},e_{5}]=\epsilon e_2,\\
\displaystyle [e_{5},e_1]=-ae_1+\left(d+f\right)e_2+ae_3,[e_{5},e_3]=-e_1+de_2+e_3,[e_{5},e_4]=(a-1)\left(e_2-e_4\right),\\
\displaystyle (\epsilon=0,1;if\,\,\epsilon=0,then\,\,d^2+f^2\neq0),
\end{array} 
\right.
\end{equation}
where $a\neq1,$ otherwise nilpotent.

If $c:=a_{2,1}\neq0$ and we assign $d:=a_{2,3}+b_{2,3}-2a_{2,1},$ then
we obtain a Leibniz algebra $g_{5,8}$ given below:
\begin{equation}
\left\{
\begin{array}{l}
\displaystyle  \nonumber [e_1,e_{5}]=ce_2,
[e_3,e_{5}]=e_1-e_3,[e_4,e_{5}]=e_2-e_4,[e_{5},e_{5}]=\epsilon e_2,[e_{5},e_1]=\left(c+d\right)e_2,\\
\displaystyle [e_{5},e_3]=-e_1+\left(d+2c\right)e_2+e_3,[e_{5},e_4]=e_4-e_2,(c\neq0,\epsilon=0,1).
\end{array} 
\right.
\end{equation}
 \end{enumerate}
\end{proof}

\subsubsection{Codimension two and three solvable extensions of $\mathcal{L}^4$}\label{Two&Three}
The non-zero inner derivations of $\mathcal{L}^4,(n\geq4)$ are
given by
 \[
\r_{e_1}=\left[\begin{smallmatrix}
0&0 & 0 & 0 & \cdots & 0 & 0  & 0 \\
1&0 & 0 & 0 & \cdots & 0 & 0  & 0 \\
0&0 & 0 & 0 & \cdots & 0 & 0 & 0 \\
 0&0 & 1 & 0 & \cdots & 0 & 0 & 0 \\
0& 0 & 0 & 1 & \cdots & 0 & 0 & 0 \\
\vdots& \vdots  & \vdots  & \vdots  & \ddots & \vdots & \vdots & \vdots\\
 0& 0 & 0 & 0&\cdots & 1 & 0 &0\\
  0& 0 & 0 & 0&\cdots & 0 & 1 &0
\end{smallmatrix}\right],\r_{e_3}=\left[\begin{smallmatrix}
0&0 & 0 & 0 & \cdots & 0 \\
2&0 & 1 & 0 & \cdots & 0  \\
0&0 & 0 & 0 & \cdots & 0 \\
 -1&0 & 0 & 0 & \cdots & 0\\
0& 0 & 0 & 0 & \cdots & 0\\
\vdots& \vdots  & \vdots  & \vdots  &  & \vdots\\
  0& 0 & 0 & 0&\cdots & 0
\end{smallmatrix}\right],
\r_{e_i}=-E_{i+1,1}=\left[\begin{smallmatrix} 0 & 0&0&\cdots &  0 \\
 0& 0&0&\cdots &  0 \\
  0 & 0&0&\cdots &  0 \\
    0 & 0&0&\cdots &  0 \\
 \vdots &\vdots &\vdots& & \vdots\\
 -1 &0&0& \cdots & 0\\
  \vdots &\vdots &\vdots& & \vdots\\
  \boldsymbol{\cdot} & 0&0&\cdots &  0
 \end{smallmatrix}\right]\,(4\leq i\leq n-1),\] where $E_{i+1,1}$ is the $n\times n$ matrix that has $1$ in the 
$(i+1,1)^{st}$
position and all other entries are zero.
 \begin{remark}\label{NumberOuterDerivations}
If we have four or more outer derivations, then they are
``nil-dependent''(see Section \ref{Pr}.).
Therefore the solvable algebras we are constructing might be of codimension at most three. (As Section \ref{Codim3} shows, we have at most two outer
derivations.)
\end{remark}
\begin{center}\textit{General approach to find right solvable Leibniz
algebras with a codimension two nilradical $\mathcal{L}^4$.}\footnote{When we work with left Leibniz algebras, we first change the right multiplication operator to the left  everywhere and the right Leibniz identity to the left Leibniz identity in step $(iii).$ We also interchange
$s$ and  $r$ in the very left in steps $(i)$ and $(ii)$ as well. }
\end{center}
\begin{enumerate}[noitemsep, topsep=0pt]
\item[(i)] We consider $\r_{[e_r,e_s]}=[\r_{e_s},\r_{e_r}]=\r_{e_s}\r_{e_r}-\r_{e_r}\r_{e_s},\,(n+1\leq r\leq n+2,\,1\leq s\leq n+2)$
and compare with $c_1\r_{e_1}+\sum_{k=3}^{n-1}c_k\r_{e_k},$ because $e_{2}$
and $e_n$ are in the center of $\mathcal{L}^4,\,(n\geq4)$ defined in $(\ref{L4})$ to find all the unknown commutators.
\item[(ii)] We write down
$[e_{r},e_{s}],\,(n+1\leq r\leq n+2,\,1\leq s\leq n+2)$ including a linear combination of
$e_{2}$ and $e_n$ as well. We add the brackets of the nilradical $\mathcal{L}^4$ and outer derivations $\r_{e_{n+1}}$ and $\r_{e_{n+2}}.$
\item[(iii)] One satisfies the right Leibniz identity: $[[e_r,e_s],e_t]=[[e_r,e_t],e_s]+[e_r,[e_s,e_t]]$ 
or, equivalently, $\r_{e_t}\left([e_r,e_s]\right)=[\r_{e_t}(e_r),e_s]+[e_r,\r_{e_t}(e_s)],\,(1\leq r,s,t\leq n+2)$ for all the brackets obtained in step $(ii).$
\item[(iv)] Then we carry out the technique of ``absorption'' (see Section \ref{Solvable left Leibniz algebras}) to remove some parameters to simplify the algebra. 
\item[(v)]  We apply the change of basis transformations without affecting the nilradical to remove as many parameters as possible.
\end{enumerate}
\paragraph{Codimension two solvable extensions of $\mathcal{L}^4,(n=4)$}\label{Twodim(n=4)}
There are the following cases based on the conditions involving parameters $a,b$ and $c$:
\begin{enumerate}
\item[(1)] $b\neq-a,a\neq0,b\neq-c,$ 
\item[(2)] $b:=-a,a\neq0,a\neq c,$
\item[(3)] $b:=-c,a\neq c,a\neq0,$
\item[(4)] $a=0,b\neq0,b\neq-c.$
\end{enumerate}
\noindent (1) (a) One could set $\left(
\begin{array}{c}
  a^1 \\
 b^1\\
 c^1
\end{array}\right)=\left(
\begin{array}{c}
  1\\
 a\\
 0
\end{array}\right)$ and $\left(
\begin{array}{c}
  a^2 \\
 b^2\\
 c^2
\end{array}\right)=\left(
\begin{array}{c}
  1\\
 b\\
 1
\end{array}\right),(a\neq-1,a\neq0,b\neq-1).$ Therefore the vector space of outer derivations as $4\times 4$
matrices is as follows:
$$\r_{e_{5}}=\begin{array}{llll} \left[\begin{matrix}
 1 & 0 & 0 & 0\\
 \frac{3-2a}{2}a_{2,3}+\frac{1-2a}{2}b_{2,3} & 2a & a_{2,3}& a-1\\
  a-1 & 0 & a & 0 \\
  0 & 0 &  0 & a+1
 \end{matrix}\right]
\end{array},$$
$$\r_{e_{6}}=\begin{array}{llll} \left[\begin{matrix}
 1 & 0 & 1 & 0\\
-b\cdot \alpha_{2,3}-(1+b)\beta_{2,3} & 2(b+1) & \alpha_{2,3}& b+1\\
  b & 0 & b & 0 \\
  0 & 0 &  0 & b+1
    \end{matrix}\right]
\end{array}.$$
\noindent $(i)$ Considering $\r_{[e_5,e_6]},$ we obtain that $a:=1$ and
$\alpha_{2,3}:=\left(b+\frac{3}{2}\right)a_{2,3}+\frac{b_{2,3}}{2}.$ Since $b\neq-1$, it follows that $\beta_{2,3}:=\frac{b_{2,3}}{2}-\left(b+\frac{1}{2}\right)a_{2,3}$ 
and $\r_{[e_5,e_6]}=0.$ As a result,
$$\r_{e_{5}}=\begin{array}{llll} \left[\begin{matrix}
 1 & 0 & 0 & 0\\
 \frac{a_{2,3}-b_{2,3}}{2} & 2 & a_{2,3}& 0\\
  0 & 0 & 1 & 0 \\
  0 & 0 &  0 & 2
 \end{matrix}\right]
\end{array},$$
$$\r_{e_{6}}=\begin{array}{llll} \left[\begin{matrix}
 1 & 0 & 1 & 0\\
 \frac{a_{2,3}}{2}-\left(b+\frac{1}{2}\right)b_{2,3} & 2(b+1) &  \left(b+\frac{3}{2}\right)a_{2,3}+\frac{b_{2,3}}{2}& b+1\\
  b & 0 & b & 0 \\
  0 & 0 &  0 & b+1
\end{matrix}\right]
\end{array},(b\neq-1).$$
Further, we find the following commutators:
\allowdisplaybreaks
\begin{equation}
\left\{
\begin{array}{l}
\displaystyle  \nonumber \r_{[e_{5},e_{1}]}=-\r_{e_1},\r_{[e_{5},e_2]}=0,\r_{[e_{5},e_3]}=-\r_{e_3},\r_{[e_{5},e_i]}=0,\r_{[e_{6},e_1]}=-\r_{e_1}-
b\r_{e_3},\\
\displaystyle \r_{[e_{6},e_2]}=0,\r_{[e_{6},e_{3}]}=-\r_{e_1}-b\r_{e_3}, \r_{[e_{6},e_i]}=0,(b\neq-1,4\leq i\leq 6).
\end{array} 
\right.
\end{equation} 
\noindent $(ii)$ We include a linear combination of $e_2$ and $e_4$:
\begin{equation}
\left\{
\begin{array}{l}
\displaystyle  \nonumber [e_{5},e_{1}]=-e_1+c_{2,1}e_2+c_{4,1}e_4,[e_{5},e_2]=c_{2,2}e_2+c_{4,2}e_4,
[e_{5},e_3]=c_{2,3}e_2-e_3+c_{4,3}e_4,\\
\displaystyle [e_{5},e_i]=c_{2,i}e_2+c_{4,i}e_4,[e_{6},e_1]=-e_1+d_{2,1}e_2-
be_3+d_{4,1}e_4,[e_{6},e_2]=d_{2,2}e_2+d_{4,2}e_4,\\
\displaystyle [e_{6},e_{3}]=-e_1+d_{2,3}e_2-be_3+d_{4,3}e_4, [e_{6},e_i]=d_{2,i}e_2+d_{4,i}e_4,(b\neq-1,4\leq i\leq 6).
\end{array} 
\right.
\end{equation} 
Besides we have the brackets from $\mathcal{L}^4$ and from outer derivations $\r_{e_{5}}$ and $\r_{e_{6}}$ as well.

\noindent $(iii)$ We satisfy the right Leibniz identity, which is shown in Table \ref{RightCodimTwo(L4,(n=4))}.

\begin{table}[h!]
\caption{Right Leibniz identities in case (1) (a) with a nilradical $\mathcal{L}^4,(n=4)$.}
\label{RightCodimTwo(L4,(n=4))}
\begin{tabular}{lp{2.4cm}p{12cm}}
\hline
\scriptsize Steps &\scriptsize Ordered triple &\scriptsize
Result\\ \hline
\scriptsize $1.$ &\scriptsize $\r_{e_1}\left([e_{5},e_{1}]\right)$ &\scriptsize
$[e_{5},e_2]=0$
$\implies$ $c_{2,2}=c_{4,2}=0.$\\ \hline
\scriptsize $2.$ &\scriptsize $\r_{e_1}\left([e_{6},e_{1}]\right)$ &\scriptsize
$[e_{6},e_2]=0$
$\implies$ $d_{2,2}=d_{4,2}=0.$\\ \hline
\scriptsize $3.$ &\scriptsize $\r_{e_3}\left([e_{5},e_{1}]\right)$ &\scriptsize
$c_{2,4}:=2,c_{4,4}:=-2$
$\implies$  $[e_{5},e_4]=2\left(e_2-e_4\right).$ \\ \hline
\scriptsize $4.$ &\scriptsize $\r_{e_3}\left([e_{6},e_{1}]\right)$ &\scriptsize
$d_{2,4}:=1+b,d_{4,4}:=-1-b$ 
$\implies$  $[e_{6},e_4]=(b+1)\left(e_2-e_4\right).$ \\ \hline
\scriptsize $5.$ &\scriptsize $\r_{e_{5}}\left([e_{5},e_{5}]\right)$ &\scriptsize
$c_{4,5}=0$
$\implies$  $[e_{5},e_{5}]=c_{2,5}e_{2}.$\\ \hline
\scriptsize $6.$ &\scriptsize $\r_{e_6}\left([e_{5},e_{6}]\right)$ &\scriptsize
$d_{4,6}=0$ 
$\implies$  $[e_{6},e_6]=d_{2,6}e_2.$   \\ \hline  
\scriptsize $7.$ &\scriptsize $\r_{e_5}\left([e_{5},e_{3}]\right)$ &\scriptsize
$c_{2,3}:=a_{2,3},c_{4,3}=0$  
$\implies$  $[e_{5},e_3]=a_{2,3}e_2-e_3.$   \\ \hline
\scriptsize $8.$ &\scriptsize $\r_{e_{6}}\left([e_{5},e_{3}]\right)$ &\scriptsize
 $c_{2,1}:=\frac{a_{2,3}-b_{2,3}}{2},$ $c_{4,1}=0$ 
$\implies$  $[e_{5},e_{1}]=-e_{1}+\frac{a_{2,3}-b_{2,3}}{2}e_2.$\\ \hline
\scriptsize $9.$ &\scriptsize $\r_{e_{1}}\left([e_{5},e_{6}]\right)$ &\scriptsize
$d_{4,1}=0$ 
$\implies$ $[e_{6},e_{1}]=-e_1+d_{2,1}e_2-be_3.$   \\ \hline
\scriptsize $10.$ &\scriptsize $\r_{e_{3}}\left([e_{5},e_{6}]\right)$ &\scriptsize
$d_{4,3}=0$
$\implies$ $[e_{6},e_{3}]=-e_1+d_{2,3}e_2-be_3.$ \\ \hline
\scriptsize $11.$ &\scriptsize $\r_{e_5}\left([e_{5},e_{6}]\right)$ &\scriptsize
 $d_{4,5}:=-c_{4,6}$ $\implies$ $c_{4,6}:=(b+1)c_{2,5}-c_{2,6}$
$\implies$  $[e_{5},e_6]=c_{2,6}e_2+\left((b+1)c_{2,5}-c_{2,6}\right)e_4,
[e_6,e_5]=d_{2,5}e_2+\left(c_{2,6}-(b+1)c_{2,5}\right)e_4$\\ \hline
\scriptsize $12.$ &\scriptsize $\r_{e_{5}}\left([e_{6},e_{3}]\right)$ &\scriptsize
$d_{2,3}:=\left(b+\frac{1}{2}\right)a_{2,3}-\frac{b_{2,3}}{2}$ 
$\implies$  $[e_{6},e_{3}]=-e_1+\left(\left(b+\frac{1}{2}\right)a_{2,3}-\frac{b_{2,3}}{2}\right)e_2-be_3.$   \\ \hline
\scriptsize $13.$ &\scriptsize $\r_{e_{6}}\left([e_{6},e_{3}]\right)$ &\scriptsize
$d_{2,1}:=\left(b+\frac{1}{2}\right)a_{2,3}-\frac{b_{2,3}}{2}$
$\implies$  $[e_{6},e_{1}]=-e_1+\left(\left(b+\frac{1}{2}\right)a_{2,3}-\frac{b_{2,3}}{2}\right)e_2-be_3.$ \\ \hline
\scriptsize $14.$ &\scriptsize $\r_{e_{6}}\left([e_{6},e_{5}]\right)$ &\scriptsize
$d_{2,6}:=(b+1)\left(c_{2,6}+d_{2,5}\right)-(b+1)^2c_{2,5}$
$\implies$ $[e_{6},e_6]=\left((b+1)\left(c_{2,6}+d_{2,5}\right)-(b+1)^2c_{2,5}\right)e_2.$\\ \hline
\end{tabular}
\end{table}
We obtain that $\r_{e_5}$ and $\r_{e_6}$ restricted to the nilradical do not change, but the remaining brackets are as follows:
\begin{equation}
\left\{
\begin{array}{l}
\displaystyle  \nonumber [e_{5},e_{1}]=-e_1+\frac{a_{2,3}-b_{2,3}}{2}e_2,
[e_{5},e_3]=a_{2,3}e_2-e_3,[e_{5},e_4]=2\left(e_2-e_4\right),[e_5,e_5]=c_{2,5}e_2,\\
\displaystyle [e_5,e_6]=c_{2,6}e_2+\left((b+1)c_{2,5}-c_{2,6}\right)e_4,
[e_{6},e_1]=-e_1+\left(\left(b+\frac{1}{2}\right)a_{2,3}-\frac{b_{2,3}}{2}\right)e_2-
be_3,\\
\displaystyle [e_{6},e_{3}]=-e_1+\left(\left(b+\frac{1}{2}\right)a_{2,3}-\frac{b_{2,3}}{2}\right)e_2-
be_3, [e_{6},e_4]=(b+1)\left(e_2-e_4\right),\\
\displaystyle [e_6,e_5]=d_{2,5}e_2+\left(c_{2,6}-(b+1)c_{2,5}\right)e_4,[e_{6},e_6]=\left((b+1)\left(c_{2,6}+d_{2,5}\right)-(b+1)^2c_{2,5}\right)e_2,\\
\displaystyle (b\neq-1).
\end{array} 
\right.
\end{equation} 
Altogether the nilradical $\mathcal{L}^4$ $(\ref{L4}),$ the outer derivations $\r_{e_{5}}$ and $\r_{e_{6}}$
 written in the bracket notation and the remaining brackets given above define a continuous family of Leibniz algebras
depending on the parameters.

\noindent $(iv)\&(v)$ We apply the following transformation: $e^{\prime}_1=e_1+\frac{b_{2,3}-a_{2,3}}{2}e_2,
e^{\prime}_2=e_2,e^{\prime}_3=e_3-a_{2,3}e_2,e^{\prime}_4=e_4,e^{\prime}_5=e_5-\frac{c_{2,5}}{2}e_2,
e^{\prime}_6=e_6-\frac{d_{2,5}}{2}e_2+\left(\frac{b+1}{2}c_{2,5}-\frac{c_{2,6}}{2}\right)e_4$
and obtain a Leibniz algebra $\g_{6,2}$ given below:
\begin{equation}
\left\{
\begin{array}{l}
\displaystyle  \nonumber [e_1,e_5]=e_1,[e_2,e_5]=2e_2,[e_3,e_5]=e_3,[e_4,e_5]=2e_4,[e_{5},e_{1}]=-e_1,[e_{5},e_3]=-e_3,\\
\displaystyle 
[e_{5},e_4]=2\left(e_2-e_4\right), [e_1,e_6]=e_1+be_3,[e_2,e_6]=2(b+1)e_2,[e_3,e_6]=e_1+be_3,\\
\displaystyle[e_4,e_6]=(b+1)\left(e_2+e_4\right),
[e_{6},e_1]=-e_1-be_3,[e_{6},e_{3}]=-e_1-be_3,\\
\displaystyle  [e_{6},e_4]=(b+1)\left(e_2-e_4\right),(b\neq-1).
\end{array} 
\right.
\end{equation}
\begin{remark}
We notice that if $b=-1$, then the outer derivation $\r_{e_6}$ is nilpotent.
\end{remark}
\noindent (1) (b) We set $\left(
\begin{array}{c}
  a^1 \\
 b^1\\
 c^1
\end{array}\right)=\left(
\begin{array}{c}
  1\\
 2\\
 c
\end{array}\right)$ and $\left(
\begin{array}{c}
  a^2 \\
 b^2\\
 c^2
\end{array}\right)=\left(
\begin{array}{c}
  1\\
 1\\
 d
\end{array}\right),(c\neq-2,d\neq-1).$ Therefore the vector space of outer derivations as $4\times 4$
matrices is as follows:
$$\r_{e_{5}}=\begin{array}{llll} \left[\begin{matrix}
 1 & 0 & c & 0\\
-\frac{1+3c}{2}a_{2,3}-\frac{3+3c}{2}b_{2,3} & 2(c+2) & a_{2,3}& 2c+1\\
 c+1 & 0 & 2 & 0 \\
  0 & 0 &  0 & 3
\end{matrix}\right]
\end{array},$$
$$\r_{e_{6}}=\begin{array}{llll} \left[\begin{matrix}
 1 & 0 & d & 0\\
\frac{1-3d}{2}\alpha_{2,3}-\frac{1+3d}{2}\beta_{2,3} & 2(d+1) & \alpha_{2,3}& 2d\\
  d & 0 & 1 & 0 \\
  0 & 0 &  0 & 2
  \end{matrix}\right]
\end{array}.$$
\noindent $(i)$ Considering $\r_{[e_5,e_6]},$ we obtain that $d=0$ and we have
the system of equations: 
$$\left\{ \begin{array}{ll}
a_{2,3}:=\left(\frac{3}{2}c+2\right)\alpha_{2,3}+\frac{c}{2}\beta_{2,3} {,}  \\
\left(\frac{3}{2}+c\right)\beta_{2,3}-\frac{1}{2}\alpha_{2,3}-\left(\frac{3}{2}c+\frac{1}{2}\right)a_{2,3}-\left(\frac{3}{2}c+\frac{3}{2}\right)b_{2,3}=0{.}
\end{array}
\right. $$ 
There are the following two cases:
\begin{enumerate}[noitemsep, topsep=0pt]
\item[(I)] If $c\neq-1,$ then $b_{2,3}:=-\frac{3c+2}{2}\alpha_{2,3}-\frac{c-2}{2}\beta_{2,3}$
and $\r_{[e_5,e_6]}=0.$ As a result,
$$\r_{e_{5}}=\begin{array}{llll} \left[\begin{matrix}
 1 & 0 & c & 0\\
\frac{\alpha_{2,3}}{2}-\frac{2c+3}{2}\beta_{2,3} & 2(c+2) & \left(\frac{3c}{2}+2\right)\alpha_{2,3}+\frac{c}{2}\beta_{2,3}& 2c+1\\
 c+1 & 0 & 2 & 0 \\
  0 & 0 &  0 & 3
 \end{matrix}\right]
\end{array},$$
$$\r_{e_{6}}=\begin{array}{llll} \left[\begin{matrix}
 1 & 0 & 0 & 0\\
\frac{\alpha_{2,3}-\beta_{2,3}}{2} & 2 & \alpha_{2,3}& 0\\
  0 & 0 & 1 & 0 \\
  0 & 0 &  0 & 2
   \end{matrix}\right]
\end{array},(c\neq-2).$$
\noindent Further, we find the following:
\allowdisplaybreaks
\begin{equation}
\left\{
\begin{array}{l}
\displaystyle  \nonumber \r_{[e_{5},e_{1}]}=-\r_{e_1}-(c+1)\r_{e_3},\r_{[e_{5},e_2]}=0,\r_{[e_{5},e_3]}=-c\r_{e_1}-2\r_{e_3},\r_{[e_{5},e_i]}=0,\\
\displaystyle \r_{[e_{6},e_1]}=-\r_{e_1},\r_{[e_{6},e_2]}=0,\r_{[e_{6},e_{3}]}=-\r_{e_3}, \r_{[e_{6},e_i]}=0,(c\neq-2,4\leq i\leq 6).
\end{array} 
\right.
\end{equation} 
\item[(II)] If $c=-1,$ then $a_{2,3}:=\frac{\alpha_{2,3}-\beta_{2,3}}{2}$
and $\r_{[e_5,e_6]}=0.$ It follows that
$$\r_{e_{5}}=\begin{array}{llll} \left[\begin{matrix}
 1 & 0 & -1 & 0\\
\frac{\alpha_{2,3}-\beta_{2,3}}{2} & 2& \frac{\alpha_{2,3}-\beta_{2,3}}{2}& -1\\
0 & 0 & 2 & 0 \\
  0 & 0 &  0 & 3
 \end{matrix}\right]
\end{array},
\r_{e_{6}}=\begin{array}{llll} \left[\begin{matrix}
 1 & 0 & 0 & 0\\
\frac{\alpha_{2,3}-\beta_{2,3}}{2} & 2 & \alpha_{2,3}& 0\\
  0 & 0 & 1 & 0 \\
  0 & 0 &  0 & 2
   \end{matrix}\right]
\end{array}$$
 and we have the following commutators:
\allowdisplaybreaks
\begin{equation}
\left\{
\begin{array}{l}
\displaystyle  \nonumber \r_{[e_{5},e_{1}]}=-\r_{e_1},\r_{[e_{5},e_2]}=0,\r_{[e_{5},e_3]}=\r_{e_1}-2\r_{e_3},\r_{[e_{5},e_i]}=0,\\
\displaystyle \r_{[e_{6},e_1]}=-\r_{e_1},\r_{[e_{6},e_2]}=0,\r_{[e_{6},e_{3}]}=-\r_{e_3}, \r_{[e_{6},e_i]}=0,(4\leq i\leq 6).
\end{array} 
\right.
\end{equation} 
\end{enumerate}
\noindent $(ii)$ We combine cases (I) and (II) together, include a linear combination of $e_2$ and $e_4$,
and have the following brackets:
\allowdisplaybreaks
\begin{equation}
\left\{
\begin{array}{l}
\displaystyle  \nonumber [e_{5},e_{1}]=-e_1+a_{2,1}e_2-(c+1)e_3+a_{4,1}e_4,[e_{5},e_2]=a_{2,2}e_2+a_{4,2}e_4,[e_{5},e_3]=-ce_1+\\
\displaystyle a_{2,3}e_2-2e_3+a_{4,3}e_4,[e_{5},e_i]=a_{2,i}e_2+a_{4,i}e_4,[e_{6},e_1]=-e_1+b_{2,1}e_2+b_{4,1}e_4,\\
\displaystyle [e_{6},e_2]=b_{2,2}e_2+b_{4,2}e_4,[e_{6},e_{3}]=b_{2,3}e_2-e_3+b_{4,3}e_4,[e_{6},e_i]=b_{2,i}e_2+b_{4,i}e_4,\\
\displaystyle(c\neq-2,4\leq i\leq 6).
\end{array} 
\right.
\end{equation}
Besides we have the brackets from $\mathcal{L}^4$ and from outer derivations $\r_{e_{5}}$ and $\r_{e_{6}}$ as well.

\noindent $(iii)$ To satisfy the right Leibniz identity, we refer to the identities given in Table \ref{RightCodimTwo(L4,(n=4))}. as much as possible.
As a result, the identities we apply are the following: $1.-6.,8.,$
$\r_{e_6}{\left([e_5,e_1]\right)}=[\r_{e_6}(e_5),e_1]+[e_5,\r_{e_6}(e_1)],$
$9.-11.,$ $\r_{e_6}\left([e_6,e_1]\right)=[\r_{e_6}(e_6),e_1]+[e_6,\r_{e_6}(e_1)],$ $13.$ and $14.$
We obtain that $\r_{e_5}$ and $\r_{e_6}$ restricted to the nilradical do not change, but the remaining brackets are the following:
\begin{equation}
\left\{
\begin{array}{l}
\displaystyle  \nonumber [e_{5},e_{1}]=-e_1+\left(\left(c+\frac{3}{2}\right)\alpha_{2,3}-\frac{\beta_{2,3}}{2}\right)e_2-(c+1)e_3,\\
\displaystyle [e_{5},e_3]=-ce_1+\left(\left(\frac{c}{2}+2\right)\alpha_{2,3}-\frac{c}{2}\beta_{2,3}\right)e_2-2e_3,[e_{5},e_4]=3\left(e_2-e_4\right),\\
\displaystyle [e_5,e_5]=(c+2)\left(a_{2,6}+a_{4,6}\right)e_2,[e_5,e_6]=a_{2,6}e_2+a_{4,6}e_4,[e_{6},e_1]=-e_1+\frac{\alpha_{2,3}-\beta_{2,3}}{2}e_2,\\
\displaystyle [e_{6},e_{3}]=\alpha_{2,3}e_2-e_3, [e_{6},e_4]=2\left(e_2-e_4\right),[e_6,e_5]=\left((c+2)b_{2,6}+a_{4,6}\right)e_2-a_{4,6}e_4,\\
\displaystyle [e_{6},e_6]=b_{2,6}e_2,(c\neq-2).
\end{array} 
\right.
\end{equation} 
Altogether the nilradical $\mathcal{L}^4$ $(\ref{L4}),$ the outer derivations $\r_{e_{5}}$ and $\r_{e_{6}}$
 written in the bracket notation and the remaining brackets given above define a continuous family of Leibniz algebras
depending on the parameters.

\noindent $(iv)\&(v)$ We apply the transformation: $e^{\prime}_1=e_1-\frac{\alpha_{2,3}-\beta_{2,3}}{2}e_2,
e^{\prime}_2=e_2,e^{\prime}_3=e_3-\alpha_{2,3}e_2,e^{\prime}_4=e_4,e^{\prime}_5=e_5-\frac{a_{2,6}}{2}e_2-\frac{a_{4,6}}{2}e_4,
e^{\prime}_6=e_6-\frac{b_{2,6}}{2}e_2$ and obtain a Leibniz algebra $\g_{6,3}$:
\begin{equation}
\left\{
\begin{array}{l}
\displaystyle  \nonumber [e_1,e_5]=e_1+(c+1)e_3,[e_2,e_5]=2(c+2)e_2,[e_3,e_5]=ce_1+2e_3,[e_4,e_5]=(2c+1)e_2+3e_4,\\
\displaystyle [e_{5},e_{1}]=-e_1-(c+1)e_3,[e_{5},e_3]=-ce_1-2e_3,[e_{5},e_4]=3\left(e_2-e_4\right),[e_1,e_6]=e_1,\\
\displaystyle [e_2,e_6]=2e_2,[e_3,e_6]=e_3,[e_4,e_6]=2e_4, [e_{6},e_1]=-e_1,[e_{6},e_{3}]=-e_3,[e_{6},e_4]=2\left(e_2-e_4\right),\\
\displaystyle  (c\neq-2).
\end{array} 
\right.
\end{equation} 

\noindent (1) (c) We set $\left(
\begin{array}{c}
  a^1 \\
 b^1\\
 c^1
\end{array}\right)=\left(
\begin{array}{c}
  a\\
 1\\
 0
\end{array}\right)$ and $\left(
\begin{array}{c}
  a^2 \\
 b^2\\
 c^2
\end{array}\right)=\left(
\begin{array}{c}
  b\\
 0\\
 1
\end{array}\right),(a,b\neq0,a\neq-1).$ Therefore the vector space of outer derivations as $4\times 4$
matrices is as follows:
$$\r_{e_{5}}=\begin{array}{llll} \left[\begin{matrix}
 a & 0 & 0 & 0\\
\frac{(3a-2)a_{2,3}+(a-2)b_{2,3}}{2a} & 2 & a_{2,3}& 1-a\\
 1-a & 0 & 1 & 0 \\
  0 & 0 &  0 & a+1
\end{matrix}\right]
\end{array},\r_{e_{6}}=\begin{array}{llll} \left[\begin{matrix}
 b & 0 & 1 & 0\\
\frac{(3b-3)\alpha_{2,3}+(b-3)\beta_{2,3}}{2b} & 2 & \alpha_{2,3}& 2-b\\
  1-b & 0 & 0 & 0 \\
  0 & 0 &  0 & b
  \end{matrix}\right]
\end{array}.$$
\noindent $(i)$ Considering $\r_{[e_5,e_6]},$ we obtain that $a:=1$ and we have
the system of equations: 
$$\left\{ \begin{array}{ll}
\alpha_{2,3}:=\frac{3a_{2,3}+b_{2,3}}{2} {,}  \\
b^2a_{2,3}+(b^2-2b)b_{2,3}+(3-3b)\alpha_{2,3}+(3-b)\beta_{2,3}=0{.}
\end{array}
\right. $$ 
There are the following two cases:
\begin{enumerate}[noitemsep, topsep=0pt]
\item[(I)] If $b\neq3,$ then $\beta_{2,3}:=\frac{(2b-3)a_{2,3}+(2b-1)b_{2,3}}{2}$
and $\r_{[e_5,e_6]}=0.$ Consequently,
$$\r_{e_{5}}=\begin{array}{llll} \left[\begin{matrix}
 1 & 0 & 0 & 0\\
\frac{a_{2,3}-b_{2,3}}{2} & 2 & a_{2,3}& 0\\
 0 & 0 & 1 & 0 \\
  0 & 0 &  0 & 2
 \end{matrix}\right]
\end{array},\r_{e_{6}}=\begin{array}{llll} \left[\begin{matrix}
 b & 0 & 1 & 0\\
\frac{b\cdot a_{2,3}+(b-2)b_{2,3}}{2} & 2 &\frac{3a_{2,3}+b_{2,3}}{2}& 2-b\\
 1-b & 0 & 0 & 0 \\
  0 & 0 &  0 & b
   \end{matrix}\right]
\end{array},(b\neq0).$$
\noindent Further, we find the commutators:
\allowdisplaybreaks
\begin{equation}
\left\{
\begin{array}{l}
\displaystyle  \nonumber \r_{[e_{5},e_{1}]}=-\r_{e_1},\r_{[e_{5},e_2]}=0,\r_{[e_{5},e_3]}=-\r_{e_3},\r_{[e_{5},e_i]}=0,\r_{[e_{6},e_1]}=-b\r_{e_1}+(b-1)\r_{e_3},\\
\displaystyle \r_{[e_{6},e_2]}=0,\r_{[e_{6},e_{3}]}=-\r_{e_1}, \r_{[e_{6},e_i]}=0,(b\neq0,4\leq i\leq 6).
\end{array} 
\right.
\end{equation} 
\item[(II)] If $b:=3,$ then $\r_{[e_5,e_6]}=0$ and $\r_{e_5},\r_{e_6}$ are as follows: 
$$\r_{e_{5}}=\begin{array}{llll} \left[\begin{matrix}
 1 & 0 & 0 & 0\\
\frac{a_{2,3}-b_{2,3}}{2} & 2& a_{2,3}& 0\\
0 & 0 & 1 & 0 \\
  0 & 0 &  0 & 2
 \end{matrix}\right]
\end{array},
\r_{e_{6}}=\begin{array}{llll} \left[\begin{matrix}
 3 & 0 & 1 & 0\\
\frac{3a_{2,3}+b_{2,3}}{2} & 2 & \frac{3a_{2,3}+b_{2,3}}{2}& -1\\
  -2 & 0 & 0 & 0 \\
  0 & 0 &  0 & 3
   \end{matrix}\right]
\end{array}.$$
We have the following commutators:
\allowdisplaybreaks
\begin{equation}
\left\{
\begin{array}{l}
\displaystyle  \nonumber \r_{[e_{5},e_{1}]}=-\r_{e_1},\r_{[e_{5},e_2]}=0,\r_{[e_{5},e_3]}=-\r_{e_3},\r_{[e_{5},e_i]}=0,\r_{[e_{6},e_1]}=-3\r_{e_1}+2\r_{e_3},\\
\displaystyle \r_{[e_{6},e_2]}=0,\r_{[e_{6},e_{3}]}=-\r_{e_1}, \r_{[e_{6},e_i]}=0,(4\leq i\leq 6).
\end{array} 
\right.
\end{equation} 
\end{enumerate}
\noindent $(ii)$ We combine cases (I) and (II) together and include a linear combination of $e_2$ and $e_4:$
\allowdisplaybreaks
\begin{equation}
\left\{
\begin{array}{l}
\displaystyle  \nonumber [e_{5},e_{1}]=-e_1+c_{2,1}e_2+c_{4,1}e_4,[e_{5},e_2]=c_{2,2}e_2+c_{4,2}e_4,[e_{5},e_3]=c_{2,3}e_2-e_3+c_{4,3}e_4,\\
\displaystyle [e_{5},e_i]=c_{2,i}e_2+c_{4,i}e_4,[e_{6},e_1]=-be_1+d_{2,1}e_2+(b-1)e_3+d_{4,1}e_4,[e_{6},e_2]=d_{2,2}e_2+d_{4,2}e_4,\\
\displaystyle [e_{6},e_{3}]=-e_1+d_{2,3}e_2+d_{4,3}e_4,[e_{6},e_i]=d_{2,i}e_2+d_{4,i}e_4,(b\neq0,4\leq i\leq 6).
\end{array} 
\right.
\end{equation}
We also have the brackets from $\mathcal{L}^4$ and from outer derivations $\r_{e_{5}}$ and $\r_{e_{6}}$ as well.

\noindent $(iii)$ To satisfy the right Leibniz identity, we apply the same identities as in Table \ref{RightCodimTwo(L4,(n=4))}. 
We have that $\r_{e_5}$ and $\r_{e_6}$ restricted to the nilradical do not change, but the remaining brackets are the following:
\begin{equation}
\left\{
\begin{array}{l}
\displaystyle  \nonumber [e_{5},e_{1}]=-e_1+\frac{a_{2,3}-b_{2,3}}{2}e_2, [e_{5},e_3]=a_{2,3}e_2-e_3,[e_{5},e_4]=2\left(e_2-e_4\right),[e_5,e_5]=c_{2,5}e_2,\\
\displaystyle [e_5,e_6]=c_{2,6}e_2+\left(c_{2,5}-c_{2,6}\right)e_4,[e_{6},e_1]=-be_1+\left(\frac{2-b}{2}a_{2,3}-\frac{b}{2}b_{2,3}\right)e_2+
(b-1)e_3,\\
\displaystyle [e_{6},e_{3}]=-e_1+\frac{a_{2,3}-b_{2,3}}{2}e_2, [e_{6},e_4]=b\left(e_2-e_4\right),
[e_6,e_5]=d_{2,5}e_2+\left(c_{2,6}-c_{2,5}\right)e_4,\\
\displaystyle [e_{6},e_6]=\left(d_{2,5}-c_{2,5}+c_{2,6}\right)e_2,(b\neq0).
\end{array} 
\right.
\end{equation} 
Altogether the nilradical $\mathcal{L}^4$ $(\ref{L4}),$ the outer derivations $\r_{e_{5}}$ and $\r_{e_{6}}$
 written in the bracket notation and the remaining brackets given above define a continuous family of Leibniz algebras
depending on the parameters.

\noindent $(iv)\&(v)$ We apply the transformation: $e^{\prime}_1=e_1+\frac{b_{2,3}-a_{2,3}}{2}e_2,
e^{\prime}_2=e_2,e^{\prime}_3=e_3-a_{2,3}e_2,e^{\prime}_4=e_4,e^{\prime}_5=e_5-\frac{c_{2,5}}{2}e_2,
e^{\prime}_6=e_6-\frac{d_{2,5}}{2}e_2+\frac{c_{2,5}-c_{2,6}}{2}e_4$ and obtain a Leibniz algebra $\g_{6,4}$:
\begin{equation}
\left\{
\begin{array}{l}
\displaystyle  \nonumber [e_1,e_5]=e_1,[e_2,e_5]=2e_2,[e_3,e_5]=e_3,[e_4,e_5]=2e_4,[e_{5},e_{1}]=-e_1,[e_{5},e_3]=-e_3,\\
\displaystyle [e_{5},e_4]=2\left(e_2-e_4\right),[e_1,e_6]=be_1+(1-b)e_3,[e_2,e_6]=2e_2,[e_3,e_6]=e_1,\\
\displaystyle [e_4,e_6]=(2-b)e_2+be_4,[e_{6},e_1]=-be_1+(b-1)e_3,[e_{6},e_{3}]=-e_1,[e_{6},e_4]=b\left(e_2-e_4\right),\\
\displaystyle  (b\neq0).
\end{array} 
\right.
\end{equation} 
\noindent $(2)$ In this case we consider $\r_{[e_5,e_6]}$ and compare with $c_1\r_{e_1}+c_3\r_{e_3},$
where 
$$\r_{e_{5}}=\begin{array}{llll} \left[\begin{matrix}
 a & 0 & c & 0\\
 \frac{(5a-3c)a_{2,3}+(3a-3c)b_{2,3}}{2a} & 2(c-a) &a_{2,3}& 2(c-a) \\
 c-2a & 0 &-a & 0\\
  0 & 0 &  0 &0
\end{matrix}\right]
\end{array},$$
$$\r_{e_{6}}=\begin{array}{llll} \left[\begin{matrix} \alpha & 0 & \gamma & 0\\
 \frac{(5\alpha-3\gamma)\alpha_{2,3}+(3\alpha-3\gamma)\beta_{2,3}}{2\alpha} & 2(\gamma-\alpha) &\alpha_{2,3}& 2(\gamma-\alpha) \\
 \gamma-2\alpha & 0 &-\alpha & 0\\
  0 & 0 &  0 &0
\end{matrix}\right]
\end{array}.$$
We obtain that $a\gamma-\alpha c=0,$ where $a,\alpha\neq0,$ which is impossible, because $\left(
\begin{array}{c}
  a\\
 c
\end{array}\right)$ and $\left(\begin{array}{c}
  \alpha\\
 \gamma
\end{array}\right)$ are supposed to be linearly independent.

\noindent $(3)$ Considering $\r_{[e_5,e_6]}$ and comparing with $c_1\r_{e_1}+c_3\r_{e_3},$
where 
$$\r_{e_{5}}=\begin{array}{llll} \left[\begin{matrix}
 a & 0 & c & 0\\
a_{2,3}-b_{2,1}+b_{2,3} & 0 &a_{2,3}& c-a\\
 -a & 0 &-c & 0\\
  0 & 0 &  0 &a-c
\end{matrix}\right]
\end{array},\r_{e_{6}}=\begin{array}{llll} \left[\begin{matrix} \alpha & 0 & \gamma & 0\\
\alpha_{2,3}-\beta_{2,1}+\beta_{2,3} & 0 &\alpha_{2,3}& \gamma-\alpha\\
-\alpha & 0 &-\gamma & 0\\
  0 & 0 &  0 &\alpha-\gamma
\end{matrix}\right]
\end{array},$$
we have to have that $a\gamma-\alpha c=0,$ where $a,\alpha\neq0,$ which is impossible.

\noindent $(4)$ Finally we consider $\r_{[e_5,e_6]}$ and compare with $c_1\r_{e_1}+c_3\r_{e_3}$ as well.
We obtain that $b\gamma-\beta c=0,$ where $b,\beta\neq0,$ which is impossible.

\paragraph{Codimension two solvable extensions of $\mathcal{L}^4,(n\geq5)$}
By taking a linear combination of $\r_{e_{n+1}}$ and $\r_{e_{n+2}}$ and keeping in mind that no nontrivial linear combination of the matrices $\r_{e_{n+1}}$ and $\r_{e_{n+2}}$ 
can be a nilpotent matrix, one could set $\left(
\begin{array}{c}
  a \\
 b
\end{array}\right)=\left(
\begin{array}{c}
  1\\
2
\end{array}\right)$ and $\left(
\begin{array}{c}
 \alpha \\
 \beta
\end{array}\right)=\left(
\begin{array}{c}
  1\\
 1
\end{array}\right).$
Therefore the vector space of outer derivations as $n \times n$ matrices is as follows:
{ $$\r_{e_{n+1}}=\left[\begin{smallmatrix}
 1 & 0 & 0 & 0&0&0&\cdots && 0&0 & 0\\
  3b_{2,1}-2a_{2,3} & 4 & a_{2,3}& 1 &0&0 & \cdots &&0  & 0& 0\\
  1 & 0 & 2 & 0 & 0&0 &\cdots &&0 &0& 0\\
  0 & 0 &  0 & 3 &0 &0 &\cdots&&0 &0 & 0\\
 0 & 0 & a_{5,3} & 0 & 4  &0&\cdots & &0&0 & 0\\
  0 & 0 &\boldsymbol{\cdot} & a_{5,3} & 0 &5&\cdots & &0&0 & 0\\
    0 & 0 &\boldsymbol{\cdot} & \boldsymbol{\cdot} & \ddots &0&\ddots &&\vdots&\vdots &\vdots\\
  \vdots & \vdots & \vdots &\vdots &  &\ddots&\ddots &\ddots &\vdots&\vdots & \vdots\\
 0 & 0 & a_{n-2,3}& a_{n-3,3}& \cdots&\cdots &a_{5,3}&0&n-3 &0& 0\\
0 & 0 & a_{n-1,3}& a_{n-2,3}& \cdots&\cdots &\boldsymbol{\cdot}&a_{5,3}&0 &n-2& 0\\
 0 & 0 & a_{n,3}& a_{n-1,3}& \cdots&\cdots &\boldsymbol{\cdot}&\boldsymbol{\cdot}&a_{5,3} &0& n-1
\end{smallmatrix}\right],$$}
{ $$ \r_{e_{n+2}}=\left[\begin{smallmatrix}
 1 & 0 & 0 & 0&0&0&\cdots && 0&0 & 0\\
  \beta_{2,1} & 2 & \alpha_{2,3}& 0 &0&0 & \cdots &&0  & 0& 0\\
  0 & 0 & 1 & 0 & 0&0 &\cdots &&0 &0& 0\\
  0 & 0 &  0 & 2 &0 &0 &\cdots&&0 &0 & 0\\
 0 & 0 & \alpha_{5,3} & 0 & 3  &0&\cdots & &0&0 & 0\\
  0 & 0 &\boldsymbol{\cdot} & \alpha_{5,3} & 0 &4&\cdots & &0&0 & 0\\
    0 & 0 &\boldsymbol{\cdot} & \boldsymbol{\cdot} & \ddots &0&\ddots &&\vdots&\vdots &\vdots\\
  \vdots & \vdots & \vdots &\vdots &  &\ddots&\ddots &\ddots &\vdots&\vdots & \vdots\\
 0 & 0 & \alpha_{n-2,3}& \alpha_{n-3,3}& \cdots&\cdots &\alpha_{5,3}&0&n-4 &0& 0\\
0 & 0 & \alpha_{n-1,3}& \alpha_{n-2,3}& \cdots&\cdots &\boldsymbol{\cdot}&\alpha_{5,3}&0 &n-3& 0\\
 0 & 0 & \alpha_{n,3}& \alpha_{n-1,3}& \cdots&\cdots &\boldsymbol{\cdot}&\boldsymbol{\cdot}&\alpha_{5,3} &0& n-2
\end{smallmatrix}\right].$$}

\noindent $(i)$ Considering $\r_{[e_{n+1},e_{n+2}]},$ which is the same as $[\r_{e_{n+2}},\r_{e_{n+1}}],$ we deduce that 
$\alpha_{i,3}:=a_{i,3},(5\leq i\leq n),$ $\alpha_{2,3}:=\frac{a_{2,3}}{2}$ and
$\beta_{2,1}:=b_{2,1}-\frac{a_{2,3}}{2}.$ As a result, the outer derivation $\r_{e_{n+2}}$
changes as follows:
$$\r_{e_{n+2}}=\left[\begin{smallmatrix}
 1 & 0 & 0 & 0&0&0&\cdots && 0&0 & 0\\
 b_{2,1}-\frac{a_{2,3}}{2} & 2 & \frac{a_{2,3}}{2}& 0 &0&0 & \cdots &&0  & 0& 0\\
  0 & 0 & 1 & 0 & 0&0 &\cdots &&0 &0& 0\\
  0 & 0 &  0 & 2 &0 &0 &\cdots&&0 &0 & 0\\
 0 & 0 & a_{5,3} & 0 & 3  &0&\cdots & &0&0 & 0\\
  0 & 0 &\boldsymbol{\cdot} & a_{5,3} & 0 &4&\cdots & &0&0 & 0\\
    0 & 0 &\boldsymbol{\cdot} & \boldsymbol{\cdot} & \ddots &0&\ddots &&\vdots&\vdots &\vdots\\
  \vdots & \vdots & \vdots &\vdots &  &\ddots&\ddots &\ddots &\vdots&\vdots & \vdots\\
 0 & 0 & a_{n-2,3}& a_{n-3,3}& \cdots&\cdots &a_{5,3}&0&n-4 &0& 0\\
0 & 0 & a_{n-1,3}& a_{n-2,3}& \cdots&\cdots &\boldsymbol{\cdot}&a_{5,3}&0 &n-3& 0\\
 0 & 0 & a_{n,3}& a_{n-1,3}& \cdots&\cdots &\boldsymbol{\cdot}&\boldsymbol{\cdot}&a_{5,3} &0& n-2
\end{smallmatrix}\right].$$
Altogether we find the following commutators:
\allowdisplaybreaks
\begin{equation}
\left\{
\begin{array}{l}
\displaystyle  \nonumber \r_{[e_{n+1},e_{1}]}=-\r_{e_1}-\r_{e_3},\r_{[e_{n+1},e_2]}=0,\r_{[e_{n+1},e_i]}=(1-i)\r_{e_i}-\sum_{k=i+2}^{n-1}{a_{k-i+3,3}\r_{e_k}},\\
\displaystyle \r_{[e_{n+1},e_j]}=0,(n\leq j\leq n+1),\r_{[e_{n+1},e_{n+2}]}=-\sum_{k=4}^{n-1}{a_{k+1,3}\r_{e_{k}}},\r_{[e_{n+2},e_1]}=-\r_{e_1},\\
\displaystyle \r_{[e_{n+2},e_2]}=0,\r_{[e_{n+2},e_{i}]}=(2-i)\r_{e_i}-\sum_{k=i+2}^{n-1}{a_{k-i+3,3}\r_{e_k}},(3\leq i\leq n-1),\\
\displaystyle \r_{[e_{n+2},e_n]}=0,\r_{[e_{n+2},e_{n+1}]}=\sum_{k=4}^{n-1}{a_{k+1,3}\r_{e_{k}}},\r_{[e_{n+2},e_{n+2}]}=0.
\end{array} 
\right.
\end{equation} 
\noindent $(ii)$ We include a linear combination of $e_2$ and $e_n:$
\begin{equation}
\left\{
\begin{array}{l}
\displaystyle  \nonumber [e_{n+1},e_{1}]=-e_1+c_{2,1}e_2-e_3+c_{n,1}e_n,[e_{n+1},e_2]=c_{2,2}e_2+c_{n,2}e_n,[e_{n+1},e_i]=c_{2,i}e_2+\\
\displaystyle(1-i)e_i-\sum_{k=i+2}^{n-1}{a_{k-i+3,3}e_k}+c_{n,i}e_n,[e_{n+1},e_j]=c_{2,j}e_2+c_{n,j}e_n,(n\leq j\leq n+1),\\
\displaystyle [e_{n+1},e_{n+2}]=c_{2,n+2}e_2-\sum_{k=4}^{n-1}{a_{k+1,3}e_{k}}+
c_{n,n+2}e_n,[e_{n+2},e_1]=-e_1+d_{2,1}e_2+d_{n,1}e_n,\\
\displaystyle [e_{n+2},e_2]=d_{2,2}e_2+d_{n,2}e_n, [e_{n+2},e_{i}]=d_{2,i}e_2+(2-i)e_i-\sum_{k=i+2}^{n-1}{a_{k-i+3,3}e_k}+d_{n,i}e_n,\\
\displaystyle (3\leq i\leq n-1),[e_{n+2},e_n]=d_{2,n}e_2+d_{n,n}e_n,[e_{n+2},e_{n+1}]=d_{2,n+1}e_2+\sum_{k=4}^{n-1}{a_{k+1,3}e_{k}}+\\
\displaystyle d_{n,n+1}e_n, [e_{n+2},e_{n+2}]=d_{2,n+2}e_2+d_{n,n+2}e_n.
\end{array} 
\right.
\end{equation} 
Besides we have the brackets from $\mathcal{L}^4$ and from outer derivations $\r_{e_{n+1}}$ and $\r_{e_{n+2}}$ as well.

\noindent $(iii)$ We satisfy the right Leibniz identity shown in Table \ref{RightCodimTwo(L4)}. We notice that $\r_{e_{n+1}}$ and $\r_{e_{n+2}}$ restricted to the nilradical do not change, but the remaining brackets are as follows:
\begin{equation}
\left\{
\begin{array}{l}
\displaystyle  \nonumber [e_{n+1},e_{1}]=-e_1+b_{2,1}e_2-e_3,[e_{n+1},e_3]=a_{2,3}e_2-2e_3-\sum_{k=5}^n{a_{k,3}e_k},
[e_{n+1},e_4]=3\left(e_2-e_4\right)-\\
\displaystyle \sum_{k=6}^n{a_{k-1,3}e_k},
 [e_{n+1},e_i]=(1-i)e_i-\sum_{k=i+2}^{n}{a_{k-i+3,3}e_k},[e_{n+1},e_{n+1}]=\left(2c_{2,n+2}-2a_{5,3}\right)e_2,\\
\displaystyle [e_{n+1},e_{n+2}]=c_{2,n+2}e_2-\sum_{k=4}^{n-1}{a_{k+1,3}e_{k}}+
c_{n,n+2}e_n,[e_{n+2},e_1]=-e_1+\left(b_{2,1}-\frac{a_{2,3}}{2}\right)e_2,\\
\displaystyle [e_{n+2},e_3]=\frac{a_{2,3}}{2}e_2-e_3-\sum_{k=5}^n{a_{k,3}e_k},
[e_{n+2},e_4]=2\left(e_2-e_4\right)-\sum_{k=6}^n{a_{k-1,3}e_k},\\
\displaystyle [e_{n+2},e_{i}]=(2-i)e_i-\sum_{k=i+2}^{n}{a_{k-i+3,3}e_k},
 (5\leq i\leq n),[e_{n+2},e_{n+1}]=d_{2,n+1}e_2+\sum_{k=4}^{n-1}{a_{k+1,3}e_{k}}-\\
\displaystyle c_{n,n+2}e_n, [e_{n+2},e_{n+2}]=\frac{d_{2,n+1}+a_{5,3}}{2}e_2.
\end{array} 
\right.
\end{equation} 

\begin{table}[h!]
\caption{Right Leibniz identities in the codimension two nilradical $\mathcal{L}^4,(n\geq5)$.}
\label{RightCodimTwo(L4)}
\begin{tabular}{lp{2.4cm}p{12cm}}
\hline
\scriptsize Steps &\scriptsize Ordered triple &\scriptsize
Result\\ \hline
\scriptsize $1.$ &\scriptsize $\r_{e_1}\left([e_{n+1},e_{1}]\right)$ &\scriptsize
$[e_{n+1},e_2]=0$
$\implies$ $c_{2,2}=c_{n,2}=0.$\\ \hline
\scriptsize $2.$ &\scriptsize $\r_{e_1}\left([e_{n+2},e_{1}]\right)$ &\scriptsize
$[e_{n+2},e_2]=0$
$\implies$ $d_{2,2}=d_{n,2}=0.$\\ \hline
\scriptsize $3.$ &\scriptsize $\r_{e_3}\left([e_{n+1},e_{1}]\right)$ &\scriptsize
$c_{2,4}:=3,c_{n,4}:=-a_{n-1,3},$ where $a_{4,3}=0$ 
$\implies$  $[e_{n+1},e_4]=3(e_2-e_4)-\sum_{k=6}^{n}{a_{k-1,3}e_k}.$ \\ \hline
\scriptsize $4.$ &\scriptsize $\r_{e_i}\left([e_{n+1},e_{1}]\right)$ &\scriptsize
$c_{2,i+1}=0,c_{n,i+1}:=-a_{n-i+2,3},(4\leq i\leq n-2),$ where $a_{4,3}=0$ 
$\implies$  $[e_{n+1},e_j]=\left(1-j\right)e_j-\sum_{k=j+2}^{n}{a_{k-j+3,3}e_k},(5\leq j\leq n-1).$ \\ \hline
\scriptsize $5.$ &\scriptsize $\r_{e_{n-1}}\left([e_{n+1},e_{1}]\right)$ &\scriptsize
$c_{2,n}=0,$ $c_{n,n}:=1-n$
$\implies$  $[e_{n+1},e_{n}]=\left(1-n\right)e_{n}.$ Altogether with $4.,$
  $[e_{n+1},e_i]=\left(1-i\right)e_i-\sum_{k=i+2}^{n}{a_{k-i+3,3}e_k},(5\leq i\leq n).$  \\ \hline
 \scriptsize $6.$ &\scriptsize $\r_{e_3}\left([e_{n+2},e_{1}]\right)$ &\scriptsize
$d_{2,4}:=2,d_{n,4}:=-a_{n-1,3},$ where $a_{4,3}=0$ 
$\implies$  $[e_{n+2},e_4]=2(e_2-e_4)-\sum_{k=6}^{n}{a_{k-1,3}e_k}.$   \\ \hline  
\scriptsize $7.$ &\scriptsize $\r_{e_i}\left([e_{n+2},e_{1}]\right)$ &\scriptsize
$d_{2,i+1}=0,d_{n,i+1}:=-a_{n-i+2,3},(4\leq i\leq n-2),$ where $a_{4,3}=0$ 
$\implies$  $[e_{n+2},e_j]=\left(2-j\right)e_j-\sum_{k=j+2}^{n}{a_{k-j+3,3}e_k},(5\leq j\leq n-1).$   \\ \hline
\scriptsize $8.$ &\scriptsize $\r_{e_{n-1}}\left([e_{n+2},e_{1}]\right)$ &\scriptsize
 $d_{2,n}=0,$ $d_{n,n}:=2-n$ 
$\implies$  $[e_{n+2},e_{n}]=\left(2-n\right)e_{n}.$ Combining with $7.,$
  $[e_{n+2},e_i]=\left(2-i\right)e_i-\sum_{k=i+2}^{n}{a_{k-i+3,3}e_k},(5\leq i\leq n).$   \\ \hline
\scriptsize $9.$ &\scriptsize $\r_{e_{n+1}}\left([e_{n+1},e_{n+1}]\right)$ &\scriptsize
$c_{n,n+1}=0$ 
$\implies$  $[e_{n+1},e_{n+1}]=c_{2,n+1}e_2.$   \\ \hline
\scriptsize $10.$ &\scriptsize $\r_{e_{n+2}}\left([e_{n+1},e_{n+2}]\right)$ &\scriptsize
$d_{n,n+2}=0$
$\implies$ $[e_{n+2},e_{n+2}]=d_{2,n+2}e_2.$ \\ \hline
\scriptsize $11.$ &\scriptsize $\r_{e_{n+1}}\left([e_{n+1},e_{3}]\right)$ &\scriptsize
 $c_{2,3}:=a_{2,3}, c_{n,3}:=-a_{n,3}$
$\implies$  $[e_{n+1},e_3]=a_{2,3}e_2-2e_3-\sum_{k=5}^n{a_{k,3}e_k}.$\\ \hline
\scriptsize $12.$ &\scriptsize $\r_{e_{n+2}}\left([e_{n+2},e_{3}]\right)$ &\scriptsize
$d_{2,3}:=\frac{a_{2,3}}{2},$ $d_{n,3}:=-a_{n,3}$ 
$\implies$  $[e_{n+2},e_{3}]=\frac{a_{2,3}}{2}e_2-e_3-\sum_{k=5}^n{a_{k,3}e_k}.$   \\ \hline
\scriptsize $13.$ &\scriptsize $\r_{e_{n+2}}\left([e_{n+1},e_{1}]\right)$ &\scriptsize
$c_{2,1}:=b_{2,1},c_{n,1}=0$
$\implies$  $[e_{n+1},e_{1}]=-e_1+b_{2,1}e_2-e_3$. \\ \hline
\scriptsize $14.$ &\scriptsize $\r_{e_{n+2}}\left([e_{n+2},e_{1}]\right)$ &\scriptsize
$d_{2,1}:=b_{2,1}-\frac{a_{2,3}}{2},d_{n,1}=0$
$\implies$ $[e_{n+2},e_1]=-e_1+\left(b_{2,1}-\frac{a_{2,3}}{2}\right)e_2.$ \\ \hline
\scriptsize $15.$ &\scriptsize $\r_{e_{n+1}}\left([e_{n+1},e_{n+2}]\right)$ &\scriptsize
$c_{2,n+1}:=2c_{2,n+2}-2a_{5,3},d_{n,n+1}:=-c_{n,n+2}$
$\implies$ $[e_{n+1},e_{n+1}]=\left(2c_{2,n+2}-2a_{5,3}\right)e_2,
[e_{n+2},e_{n+1}]=d_{2,n+1}e_2+\sum_{k=4}^{n-1}a_{k+1,3}e_k-c_{n,n+2}e_n.$ \\ \hline
\scriptsize $16.$ &\scriptsize $\r_{e_{n+2}}\left([e_{n+2},e_{n+1}]\right)$ &\scriptsize
$d_{2,n+2}:=\frac{d_{2,n+1}+a_{5,3}}{2}$
$\implies$ $[e_{n+2},e_{n+2}]=\frac{d_{2,n+1}+a_{5,3}}{2}e_2.$\\ \hline
\end{tabular}
\end{table}
Altogether the nilradical $\mathcal{L}^4$ $(\ref{L4}),$ the outer derivations $\r_{e_{n+1}}$ and $\r_{e_{n+2}}$
and the remaining brackets given above define a continuous family of solvable right Leibniz algebras
depending on the parameters. Then we apply the technique of ``absorption'' according to step $(iv)$.
\begin{itemize}[noitemsep, topsep=0pt]
\item First we apply the transformation $e^{\prime}_i=e_i,(1\leq i\leq n,n\geq5),e^{\prime}_{n+1}=e_{n+1}-\frac{c_{2,n+2}}{2}e_2,
e^{\prime}_{n+2}=e_{n+2}-\frac{d_{2,n+1}+a_{5,3}}{4}e_2$
to remove the coefficients $c_{2,n+2}$ and $\frac{d_{2,n+1}+a_{5,3}}{2}$ in front of $e_2$ in
$[e_{n+1},e_{n+2}]$ and $[e_{n+2},e_{n+2}],$ respectively.
This transformation changes the coefficient in front of $e_2$ in $[e_{n+1},e_{n+1}]$
and $[e_{n+2},e_{n+1}]$ to $-2a_{5,3}$ and $-a_{5,3},$ respectively.
\item Then we apply
$e^{\prime}_i=e_i,(1\leq i\leq n,n\geq5),e^{\prime}_{n+1}=e_{n+1}+\frac{a_{5,3}}{2}e_4,
e^{\prime}_{n+2}=e_{n+2}$ to remove the coefficients $-a_{5,3}$ and $-2a_{5,3}$
in front of $e_2$ in $[e_{n+2},e_{n+1}]$ and $[e_{n+1},e_{n+1}]$, respectively. Besides this transformation removes
$a_{5,3}$ and $-a_{5,3}$ in front of $e_4$ in $[e_{n+2},e_{n+1}]$ and  $[e_{n+1},e_{n+2}],$ respectively, and
changes the coefficients in front of $e_{k},(6\leq k\leq n-1)$
in $[e_{n+2},e_{n+1}]$ and $[e_{n+1},e_{n+2}]$ to $a_{k+1,3}-\frac{a_{5,3}}{2}a_{k-1,3}$ and 
$\frac{a_{5,3}}{2}a_{k-1,3}-a_{k+1,3},$ respectively. It also affects the coefficients in front $e_n,(n\geq6)$ in
$[e_{n+1},e_{n+2}]$ and $[e_{n+2},e_{n+1}],$ which we rename back by $c_{n,n+2}$ and $-c_{n,n+2},$
respectively. The following entries are introduced by the transformation: $-\frac{a_{5,3}}{2}$ and $\frac{a_{5,3}}{2}$ in the $(5,1)^{st},(n\geq5)$ position in $\r_{e_{n+1}}$ and $\L_{e_{n+1}},$
respectively.
\item
Applying the transformation $e^{\prime}_j=e_j,(1\leq j\leq n+1,n\geq6),
e^{\prime}_{n+2}=e_{n+2}-\sum_{k=5}^{n-1}{\frac{A_{k+1,3}}{k-1}e_k},$
where $A_{6,3}:=a_{6,3}$ and $A_{k+1,3}:=a_{k+1,3}-\frac{1}{2}a_{5,3}a_{k-1,3}-\sum_{i=7}^k{\frac{A_{i-1,3}a_{k-i+5,3}}{i-3}},$
$(6\leq k\leq n-1,n\geq7),$
we remove the coefficients $a_{6,3}$ and $-a_{6,3}$ in front of $e_5$ in $[e_{n+2},e_{n+1}]$
and $[e_{n+1},e_{n+2}],$ respectively. We also remove $a_{k+1,3}-\frac{a_{5,3}}{2}a_{k-1,3}$ and 
$\frac{a_{5,3}}{2}a_{k-1,3}-a_{k+1,3}$
in front of $e_k,(6\leq k\leq n-1)$ in $[e_{n+2},e_{n+1}]$ and $[e_{n+1},e_{n+2}],$ respectively.
This transformation introduces $\frac{a_{6,3}}{4}$ and $-\frac{a_{6,3}}{4}$ in the $(6,1)^{st},(n\geq6)$ position as well as
$\frac{A_{k+1,3}}{k-1}$ and $\frac{A_{k+1,3}}{1-k}$ in the $(k+1,1)^{st},(6\leq k\leq n-1)$ in $\r_{e_{n+2}}$
and $\L_{e_{n+2}},$ respectively. It also affects the coefficients in front $e_n,(n\geq7)$ in
$[e_{n+1},e_{n+2}]$ and $[e_{n+2},e_{n+1}],$ which we rename back by $c_{n,n+2}$ and $-c_{n,n+2},$
respectively.

\item Finally applying the transformation $e^{\prime}_i=e_i,(1\leq i\leq n+1,n\geq5),
e^{\prime}_{n+2}=e_{n+2}+\frac{c_{n,n+2}}{n-1}e_n,$
we remove $c_{n,n+2}$ and $-c_{n,n+2}$ in front of $e_n$
in $[e_{n+1},e_{n+2}]$ and $[e_{n+2},e_{n+1}],$ respectively,
without affecting other entries. We obtain that $\r_{e_{n+1}}$ and $\r_{e_{n+2}}$ are as follows:
 \end{itemize}
{$$\r_{e_{n+1}}=\left[\begin{smallmatrix}
 1 & 0 & 0 & 0&0&0&\cdots && 0&0 & 0\\
  3b_{2,1}-2a_{2,3} & 4 & a_{2,3}& 1 &0&0 & \cdots &&0  & 0& 0\\
  1 & 0 & 2 & 0 & 0&0 &\cdots &&0 &0& 0\\
  0 & 0 &  0 & 3 &0 &0 &\cdots&&0 &0 & 0\\
 -\frac{a_{5,3}}{2} & 0 & a_{5,3} & 0 & 4  &0&\cdots & &0&0 & 0\\
  0 & 0 &\boldsymbol{\cdot} & a_{5,3} & 0 &5&\cdots & &0&0 & 0\\
    0 & 0 &\boldsymbol{\cdot} & \boldsymbol{\cdot} & \ddots &0&\ddots &&\vdots&\vdots &\vdots\\
  \vdots & \vdots & \vdots &\vdots &  &\ddots&\ddots &\ddots &\vdots&\vdots & \vdots\\
 0 & 0 & a_{n-2,3}& a_{n-3,3}& \cdots&\cdots &a_{5,3}&0&n-3 &0& 0\\
0 & 0 & a_{n-1,3}& a_{n-2,3}& \cdots&\cdots &\boldsymbol{\cdot}&a_{5,3}&0 &n-2& 0\\
 0 & 0 & a_{n,3}& a_{n-1,3}& \cdots&\cdots &\boldsymbol{\cdot}&\boldsymbol{\cdot}&a_{5,3} &0& n-1
\end{smallmatrix}\right],$$}
$$\r_{e_{n+2}}=\left[\begin{smallmatrix}
 1 & 0 & 0 & 0&0&0&0&\cdots && 0&0 & 0\\
b_{2,1}-\frac{a_{2,3}}{2} & 2 & \frac{a_{2,3}}{2}& 0 &0&0 & 0&\cdots &&0  & 0& 0\\
  0 & 0 & 1 & 0 & 0&0 &0&\cdots &&0 &0& 0\\
  0 & 0 &  0 & 2 &0 &0 &0&\cdots&&0 &0 & 0\\
  0 & 0 & a_{5,3} & 0 & 3  &0&0&\cdots & &0&0 & 0\\
\frac{a_{6,3}}{4} & 0 &\boldsymbol{\cdot} & a_{5,3} & 0 &4&0&\cdots & &0&0 & 0\\
  \frac{A_{7,3}}{5} & 0 &\boldsymbol{\cdot} & \boldsymbol{\cdot} & a_{5,3}&0 &5& &&\vdots&\vdots &\vdots\\
      \frac{A_{8,3}}{6} & 0 &\boldsymbol{\cdot} & \boldsymbol{\cdot} &\boldsymbol{\cdot} &a_{5,3}&0 &\ddots&&\vdots&\vdots &\vdots\\
  \vdots & \vdots & \vdots &\vdots &  &&\ddots&\ddots &\ddots &\vdots&\vdots & \vdots\\
  \frac{A_{n-2,3}}{n-4} & 0 & a_{n-2,3}& a_{n-3,3}& \cdots&\cdots&\cdots &a_{5,3}&0&n-4 &0& 0\\
    \frac{A_{n-1,3}}{n-3} & 0 & a_{n-1,3}& a_{n-2,3}& \cdots&\cdots&\cdots &\boldsymbol{\cdot}&a_{5,3}&0 &n-3& 0\\
 \frac{A_{n,3}}{n-2} & 0 & a_{n,3}&a_{n-1,3}& \cdots&\cdots &\cdots&\boldsymbol{\cdot}&\boldsymbol{\cdot}&a_{5,3} &0& n-2
\end{smallmatrix}\right],(n\geq5).$$
The remaining brackets are given below:
\begin{equation}
\left\{
\begin{array}{l}
\displaystyle  \nonumber [e_{n+1},e_{1}]=-e_1+b_{2,1}e_2-e_3+\frac{a_{5,3}}{2}e_5,[e_{n+1},e_3]=a_{2,3}e_2-2e_3-\sum_{k=5}^n{a_{k,3}e_k},\\
\displaystyle [e_{n+1},e_4]=3\left(e_2-e_4\right)-\sum_{k=6}^n{a_{k-1,3}e_k},
 [e_{n+1},e_i]=(1-i)e_i-\sum_{k=i+2}^{n}{a_{k-i+3,3}e_k},\\
\displaystyle [e_{n+2},e_1]=-e_1+\left(b_{2,1}-\frac{a_{2,3}}{2}\right)e_2-\sum_{k=6}^n{\frac{A_{k,3}}{k-2}e_k},[e_{n+2},e_3]=\frac{a_{2,3}}{2}e_2-e_3-\sum_{k=5}^n{a_{k,3}e_k},\\
\displaystyle 
[e_{n+2},e_4]=2\left(e_2-e_4\right)-\sum_{k=6}^n{a_{k-1,3}e_k},[e_{n+2},e_{i}]=(2-i)e_i-\sum_{k=i+2}^{n}{a_{k-i+3,3}e_k},(5\leq i\leq n).
\end{array} 
\right.
\end{equation} 
\noindent $(v)$ Finally we apply the following two change of basis transformations:
\begin{itemize}[noitemsep, topsep=0pt]
\item $e^{\prime}_1=e_1-\left(b_{2,1}-\frac{a_{2,3}}{2}\right)e_2,e^{\prime}_2=e_2,
e^{\prime}_3=e_3-\frac{a_{2,3}}{2}e_2,
e^{\prime}_i=e_i-\sum_{k=i+2}^n{\frac{B_{k-i+3,3}}{k-i}e_k},$
$(3\leq i\leq n-2),e^{\prime}_{n-1}=e_{n-1},e^{\prime}_{n}=e_{n},e^{\prime}_{n+1}=e_{n+1},e^{\prime}_{n+2}=e_{n+2},$
where $B_{j,3}:=a_{j,3}-\sum_{k=7}^j{\frac{B_{k-2,3}a_{j-k+5,3}}{k-5}},(5\leq j\leq n)$.\\
This transformation removes $a_{5,3},a_{6,3},...,a_{n,3}$ in $\r_{e_{n+1}},$ $\r_{e_{n+2}}$
and $-a_{5,3},-a_{6,3},...,-a_{n,3}$ in $\L_{e_{n+1}},\L_{e_{n+2}}.$ It also removes $-\frac{a_{5,3}}{2}$
and $\frac{a_{5,3}}{2}$ from the $(5,1)^{st}$ positions in $\r_{e_{n+1}}$ and $\L_{e_{n+1}},$ respectively.
Besides it removes $3b_{2,1}-2a_{2,3}$ and $b_{2,1}$ from the $(2,1)^{st}$ positions in $\r_{e_{n+1}}$
and $\L_{e_{n+1}},$ respectively, as well as $b_{2,1}-\frac{a_{2,3}}{2}$ from the $(2,1)^{st}$ positions in $\r_{e_{n+2}}$
and $\L_{e_{n+2}}.$
The transformation also removes $a_{2,3}$ from the $(2,3)^{rd}$ positions in $\r_{e_{n+1}},$ $\L_{e_{n+1}}$ and
$\frac{a_{2,3}}{2}$ from the same positions in $\r_{e_{n+2}},$ $\L_{e_{n+2}}.$
This transformation introduces the entries in the $(i,1)^{st}$ positions in $\r_{e_{n+1}}$ and $\L_{e_{n+1}},$
which we set to be $a_{i,1}$ and $-a_{i,1},(6\leq i\leq n),$ respectively. The transformation affects the entries in the $(j,1)^{st},(8\leq j\leq n)$
positions in $\r_{e_{n+2}}$ and $\L_{e_{n+2}},$ but we rename all the entries in the $(i,1)^{st}$ positions by $\frac{i-3}{i-2}a_{i,1}$ in $\r_{e_{n+2}}$ and by $\frac{3-i}{i-2}a_{i,1},(6\leq i\leq n)$ in $\L_{e_{n+2}}.$
\item Finally applying the transformation $e_k^{\prime}=e_k,(1\leq k\leq n),e^{\prime}_{n+1}=e_{n+1}+\sum_{k=5}^{n-1}{a_{k+1,1}e_k},$

$e^{\prime}_{n+2}=e_{n+2}+\sum_{k=5}^{n-1}{\frac{k-2}{k-1}a_{k+1,1}e_k},$
we remove $a_{i,1}$ in $\r_{e_{n+1}}$ and $-a_{i,1}$ in $\L_{e_{n+1}}$ as well as
$\frac{i-3}{i-2}a_{i,1}$ and $\frac{3-i}{i-2}a_{i,1},(6\leq i\leq n)$ in $\r_{e_{n+2}}$ and $\L_{e_{n+2}},$ respectively. 
\end{itemize}
We obtain a Leibniz algebra $\g_{n+2,1}$ given below:
\begin{equation}
\left\{
\begin{array}{l}
\displaystyle  \nonumber [e_1,e_{n+1}]=e_1+e_3,[e_2,e_{n+1}]=4e_2,[e_3,e_{n+1}]=2e_3,
[e_4,e_{n+1}]=e_2+3e_4,\\
\displaystyle  [e_{i},e_{n+1}]=(i-1)e_i,[e_{n+1},e_{1}]=-e_1-e_3,[e_{n+1},e_3]=-2e_3, [e_{n+1},e_4]=3(e_2-e_4),\\
\displaystyle [e_{n+1},e_{i}]=(1-i)e_i,[e_{1},e_{n+2}]=e_1,[e_2,e_{n+2}]=2e_2,[e_j,e_{n+2}]=(j-2)e_j,(3\leq j\leq n),\\
\displaystyle [e_{n+2},e_1]=-e_1,[e_{n+2},e_3]=-e_3,[e_{n+2},e_4]=2(e_2-e_4),[e_{n+2},e_i]=(2-i)e_i,(5\leq i\leq n).\end{array} 
\right.
\end{equation} 
 We summarize a result 
in the following theorem: 
 \begin{theorem}\label{RCodim2L4} There are four solvable
indecomposable right Leibniz algebras up to isomorphism with a codimension two nilradical
$\mathcal{L}^4,(n\geq4),$ which are given below:
\begin{equation}
\begin{array}{l}
\displaystyle  \nonumber (i)\,\g_{n+2,1}: [e_1,e_{n+1}]=e_1+e_3,[e_2,e_{n+1}]=4e_2,[e_3,e_{n+1}]=2e_3,
[e_4,e_{n+1}]=e_2+3e_4,\\
\displaystyle  [e_{i},e_{n+1}]=(i-1)e_i,[e_{n+1},e_{1}]=-e_1-e_3,[e_{n+1},e_3]=-2e_3, [e_{n+1},e_4]=3(e_2-e_4),\\
\displaystyle [e_{n+1},e_{i}]=(1-i)e_i,[e_{1},e_{n+2}]=e_1,[e_2,e_{n+2}]=2e_2,[e_j,e_{n+2}]=(j-2)e_j,(3\leq j\leq n),\\
\displaystyle [e_{n+2},e_1]=-e_1,[e_{n+2},e_3]=-e_3,[e_{n+2},e_4]=2(e_2-e_4),[e_{n+2},e_i]=(2-i)e_i,(5\leq i\leq n),\\
\displaystyle (n\geq5),\\
\displaystyle  (ii)\,\g_{6,2}: [e_1,e_5]=e_1,[e_2,e_5]=2e_2,[e_3,e_5]=e_3,[e_4,e_5]=2e_4,[e_{5},e_{1}]=-e_1,[e_{5},e_3]=-e_3,\\
\displaystyle 
[e_{5},e_4]=2\left(e_2-e_4\right), [e_1,e_6]=e_1+be_3,[e_2,e_6]=2(b+1)e_2,[e_3,e_6]=e_1+be_3,\\
\displaystyle[e_4,e_6]=(b+1)\left(e_2+e_4\right),
[e_{6},e_1]=-e_1-be_3,[e_{6},e_{3}]=-e_1-be_3, [e_{6},e_4]=(b+1)\left(e_2-e_4\right),\\
\displaystyle (b\neq-1),\\
\displaystyle (iii)\, \g_{6,3}: [e_1,e_5]=e_1+(c+1)e_3,[e_2,e_5]=2(c+2)e_2,[e_3,e_5]=ce_1+2e_3,\\
\displaystyle [e_4,e_5]=(2c+1)e_2+3e_4,[e_{5},e_{1}]=-e_1-(c+1)e_3,[e_{5},e_3]=-ce_1-2e_3,[e_{5},e_4]=3\left(e_2-e_4\right),\\
\displaystyle [e_1,e_6]=e_1,[e_2,e_6]=2e_2,[e_3,e_6]=e_3,[e_4,e_6]=2e_4, [e_{6},e_1]=-e_1,[e_{6},e_{3}]=-e_3,\\
\displaystyle  [e_{6},e_4]=2\left(e_2-e_4\right),(c\neq-2),\\
\displaystyle  (iv)\,\g_{6,4}: [e_1,e_5]=e_1,[e_2,e_5]=2e_2,[e_3,e_5]=e_3,[e_4,e_5]=2e_4,[e_{5},e_{1}]=-e_1,[e_{5},e_3]=-e_3,\\
\displaystyle [e_{5},e_4]=2\left(e_2-e_4\right),[e_1,e_6]=be_1+(1-b)e_3,[e_2,e_6]=2e_2,[e_3,e_6]=e_1,[e_4,e_6]=(2-b)e_2+\\
\displaystyle be_4,[e_{6},e_1]=-be_1+(b-1)e_3,[e_{6},e_{3}]=-e_1,[e_{6},e_4]=b\left(e_2-e_4\right),(b\neq0).
\end{array} 
\end{equation} 
\end{theorem}
\paragraph{Codimension three solvable extensions of $\mathcal{L}^4,(n=4)$}\label{Codim3}
Considering $\r_{[e_5,e_6]},\r_{[e_5,e_7]}$ and $\r_{[e_6,e_7]}$ and comparing with $c_1\r_{e_1}+c_3\r_{e_3},$
where 
$$\r_{e_{i}}=\begin{array}{llll} \left[\begin{matrix}
 a_i & 0 & c_i & 0\\
\frac{(3a_i-2b_i-3c_i){a_i}_{2,3}+(a_i-2b_i-3c_i){b_i}_{2,3}}{2a_i} & 2(b_i+c_i) &{a_i}_{2,3}& 2c_i+b_i-a_i \\
 b_i+c_i-a_i & 0 &b_i & 0\\
  0 & 0 &  0 &a_i+b_i
\end{matrix}\right]
\end{array},(5\leq i\leq7),$$
we have to have, respectively, that 
\begin{equation}\nonumber a_1c_2-a_2c_1-b_1c_2+b_2c_1=0,
a_1c_3-a_3c_1-b_1c_3+b_3c_1=0,
a_2c_3-a_3c_2-b_2c_3+b_3c_2=0.
\end{equation} 
However, it means that 
 $\left(
\begin{array}{c}
  a_1\\
 b_1\\
 c_1
\end{array}\right),\left(
\begin{array}{c}
  a_2\\
 b_2\\
 c_2
\end{array}\right)$ and $\left(
\begin{array}{c}
  a_3\\
 b_3\\
 c_3
\end{array}\right)$ are linearly dependent: one could see that
considering 
$\det\begin{pmatrix} a_1 & a_2 &a_3 \\ 
b_1 & b_2 &b_3\\
c_1 &c_2 &c_3
 \end{pmatrix}=a_1\left(b_2c_3-b_3c_2\right)-a_2\left(b_1c_3-b_3c_1\right)+a_3\left(b_1c_2-b_2c_1\right)=
 a_1\left(a_2c_3-a_3c_2\right)-a_2\left(a_1c_3-a_3c_1\right)+a_3\left(a_1c_2-a_2c_1\right)=0.$

\subsection{Solvable indecomposable left Leibniz algebras with a nilradical $\mathcal{L}^4$}
\subsubsection{Codimension one solvable extensions of $\mathcal{L}^4$}
Classification follows the same steps in theorems with the same cases per step.
\begin{theorem}\label{TheoremLL4} We set $a_{1,1}:=a$ and $a_{3,3}:=b$ in $(\ref{BRLeibniz})$. To satisfy the left Leibniz identity, there are the following cases based on the conditions involving parameters,
each gives a continuous family of solvable Leibniz algebras:
\begin{enumerate}[noitemsep, topsep=0pt]
\item[(1)] If $a_{1,3}=0, b\neq-a,a\neq0,b\neq0,(n=4)$ or $b\neq(3-n)a,a\neq0,b\neq0,(n\geq5),$ then we have the following brackets for the algebra:
\begin{equation}
\left\{
\begin{array}{l}
\displaystyle  \nonumber [e_1,e_{n+1}]=ae_1+a_{2,1}e_2+(b-a)e_3+\sum_{k=4}^n{a_{k,1}e_k},
[e_3,e_{n+1}]=a_{2,3}e_2+be_3+ \sum_{k=4}^n{a_{k,3}e_k},\\
\displaystyle[e_4,e_{n+1}]=(a+b)(e_4-e_2)+\sum_{k=5}^n{a_{k-1,3}e_k},
[e_{i},e_{n+1}]=\left((i-3)a+b\right)e_{i}+\sum_{k=i+1}^n{a_{k-i+3,3}e_k},\\
\displaystyle 
[e_{n+1},e_{n+1}]=a_{2,n+1}e_2,[e_{n+1},e_1]=-ae_1+B_{2,1}e_2+(a-b)e_3-\sum_{k=4}^n{a_{k,1}e_k},\\
\displaystyle [e_{n+1},e_2]=-2be_2,[e_{n+1},e_3]=(a_{2,3}+2a_{4,3})e_2-be_3-\sum_{k=4}^n{a_{k,3}e_k},[e_{n+1},e_4]=(a-b)e_2-\\
\displaystyle (a+b)e_4-\sum_{k=5}^n{a_{k-1,3}e_k},[e_{n+1},e_i]=\left((3-i)a-b\right)e_i-\sum_{k=i+1}^n{a_{k-i+3,3}e_k},(5\leq i\leq n),\\
\displaystyle where\,\,B_{2,1}:=\frac{(2b-a)a_{2,1}+2b\cdot a_{4,1}+2(a-b)(a_{2,3}+a_{4,3})}{a},
\end{array} 
\right.
\end{equation} 
$$\L_{e_{n+1}}=\left[\begin{smallmatrix}
 -a & 0 & 0 & 0&0&&\cdots &0& \cdots&0 & 0&0 \\
  B_{2,1} & -2b & a_{2,3}+2a_{4,3}& a-b &0& & \cdots &0&\cdots  & 0& 0&0\\
  a-b & 0 & -b & 0 & 0& &\cdots &0&\cdots &0& 0&0 \\
  -a_{4,1} & 0 &  -a_{4,3} & -a-b &0 & &\cdots&0&\cdots &0 & 0&0\\
  -a_{5,1} & 0 & -a_{5,3} & -a_{4,3} & -2a-b  &&\cdots &0 &\cdots&0 & 0&0 \\
 \boldsymbol{\cdot} & \boldsymbol{\cdot} & \boldsymbol{\cdot} & -a_{5,3} & -a_{4,3}  &\ddots& &\vdots &&\vdots & \vdots&\vdots \\
  \vdots & \vdots & \vdots &\vdots &\vdots  &\ddots& \ddots&\vdots &&\vdots & \vdots&\vdots \\
  -a_{i,1} & 0 & -a_{i,3} & -a_{i-1,3} & -a_{i-2,3}  &\cdots&-a_{4,3}& (3-i)a-b&\cdots&0 & 0&0\\
   \vdots  & \vdots  & \vdots &\vdots &\vdots&&\vdots &\vdots &&\vdots  & \vdots&\vdots \\
 -a_{n-1,1} & 0 & -a_{n-1,3}& -a_{n-2,3}& -a_{n-3,3}&\cdots &-a_{n-i+3,3} &-a_{n-i+2,3}&\cdots&-a_{4,3} &(4-n)a-b& 0\\
 -a_{n,1} & 0 & -a_{n,3}& -a_{n-1,3}& -a_{n-2,3}&\cdots &-a_{n-i+4,3}  &-a_{n-i+3,3}&\cdots&-a_{5,3} &-a_{4,3}& (3-n)a-b
\end{smallmatrix}\right].$$
\item[(2)] If $a_{1,3}=0,b:=-a,a\neq0,(n=4)$ or $b:=(3-n)a,a\neq0,(n\geq5),$
then the brackets for the algebra are  
\begin{equation}
\left\{
\begin{array}{l}
\displaystyle  \nonumber [e_1,e_{n+1}]=ae_1+a_{2,1}e_2-(n-2)ae_3+\sum_{k=4}^n{a_{k,1}e_k},
[e_3,e_{n+1}]=a_{2,3}e_2+(3-n)ae_3+\\
\displaystyle  \sum_{k=4}^n{a_{k,3}e_k},[e_4,e_{n+1}]=(n-4)a(e_2-e_4)+\sum_{k=5}^n{a_{k-1,3}e_k},
[e_{i},e_{n+1}]=\left(i-n\right)ae_{i}+\\
\displaystyle \sum_{k=i+1}^n{a_{k-i+3,3}e_k},
[e_{n+1},e_{n+1}]=a_{2,n+1}e_2+a_{n,n+1}e_n,[e_{n+1},e_1]=-ae_1+B_{2,1}e_2+\\
\displaystyle (n-2)ae_3-\sum_{k=4}^n{a_{k,1}e_k},[e_{n+1},e_2]=(2n-6)ae_2,[e_{n+1},e_3]=(a_{2,3}+2a_{4,3})e_2+\\
\displaystyle (n-3)ae_3-\sum_{k=4}^n{a_{k,3}e_k},[e_{n+1},e_4]=(n-2)ae_2+(n-4)ae_4-\sum_{k=5}^n{a_{k-1,3}e_k},\\
\displaystyle [e_{n+1},e_i]=\left(n-i\right)ae_i-\sum_{k=i+1}^n{a_{k-i+3,3}e_k},(5\leq i\leq n),\\
\displaystyle where\,\,B_{2,1}:=(5-2n)a_{2,1}+(6-2n)a_{4,1}+2(n-2)(a_{2,3}+a_{4,3}),
\end{array} 
\right.
\end{equation}
$$\L_{e_{n+1}}=\left[\begin{smallmatrix}
 -a & 0 & 0 & 0&0&&\cdots &0& \cdots&0 & 0&0 \\
  B_{2,1} & (2n-6)a & a_{2,3}+2a_{4,3}& (n-2)a &0& & \cdots &0&\cdots  & 0& 0&0\\
  (n-2)a & 0 & (n-3)a & 0 & 0& &\cdots &0&\cdots &0& 0&0 \\
  -a_{4,1} & 0 &  -a_{4,3} & (n-4)a &0 & &\cdots&0&\cdots &0 & 0&0\\
  -a_{5,1} & 0 & -a_{5,3} & -a_{4,3} & (n-5)a  &&\cdots &0 &\cdots&0 & 0&0 \\
 \boldsymbol{\cdot} & \boldsymbol{\cdot} & \boldsymbol{\cdot} & -a_{5,3} & -a_{4,3}  &\ddots& &\vdots &&\vdots & \vdots&\vdots \\
  \vdots & \vdots & \vdots &\vdots &\vdots  &\ddots& \ddots&\vdots &&\vdots & \vdots&\vdots \\
  -a_{i,1} & 0 & -a_{i,3} & -a_{i-1,3} & -a_{i-2,3}  &\cdots&-a_{4,3}& (n-i)a&\cdots&0 & 0&0\\
   \vdots  & \vdots  & \vdots &\vdots &\vdots&&\vdots &\vdots &&\vdots  & \vdots&\vdots \\
 -a_{n-1,1} & 0 & -a_{n-1,3}& -a_{n-2,3}& -a_{n-3,3}&\cdots &-a_{n-i+3,3} &-a_{n-i+2,3}&\cdots&-a_{4,3} &a& 0\\
 -a_{n,1} & 0 & -a_{n,3}& -a_{n-1,3}& -a_{n-2,3}&\cdots &-a_{n-i+4,3}  &-a_{n-i+3,3}&\cdots&-a_{5,3} &-a_{4,3}&0
\end{smallmatrix}\right].$$
 \item[(3)] If $a_{1,3}=0,a=0,b\neq0,(n=4)$ or $a=0$ and $b\neq0,(n\geq5),$ then
\begin{equation}
\left\{
\begin{array}{l}
\displaystyle  \nonumber [e_1,e_{n+1}]=\left(a_{2,3}+a_{4,3}-a_{4,1}\right)e_2+be_3+\sum_{k=4}^n{a_{k,1}e_k},
[e_3,e_{n+1}]=a_{2,3}e_2+be_3+\sum_{k=4}^n{a_{k,3}e_k},\\
\displaystyle [e_4,e_{n+1}]=b\left(e_4-e_2\right)+\sum_{k=5}^n{a_{k-1,3}e_k},
[e_{i},e_{n+1}]=be_{i}+\sum_{k=i+1}^n{a_{k-i+3,3}e_k},\\
\displaystyle [e_{n+1},e_{n+1}]=a_{2,n+1}e_2, 
[e_{n+1},e_1]=b_{2,1}e_2-be_3-\sum_{k=4}^n{a_{k,1}e_k},[e_{n+1},e_2]=-2be_2,\\
\displaystyle [e_{n+1},e_3]=(a_{2,3}+2a_{4,3})e_2-be_3-\sum_{k=4}^n{a_{k,3}e_k},[e_{n+1},e_4]=-b\left(e_2+e_4\right)-\sum_{k=5}^n{a_{k-1,3}e_k},\\
\displaystyle [e_{n+1},e_i]=-be_i-\sum_{k=i+1}^n{a_{k-i+3,3}e_k},(5\leq i\leq n),
\end{array} 
\right.
\end{equation} 
$$\L_{e_{n+1}}=\left[\begin{smallmatrix}
 0 & 0 & 0 & 0&0&&\cdots &0& \cdots&0 & 0&0 \\
  b_{2,1} & -2b & a_{2,3}+2a_{4,3}&-b &0& & \cdots &0&\cdots  & 0& 0&0\\
  -b & 0 & -b & 0 & 0& &\cdots &0&\cdots &0& 0&0 \\
  -a_{4,1} & 0 &  -a_{4,3} & -b &0 & &\cdots&0&\cdots &0 & 0&0\\
  -a_{5,1} & 0 & -a_{5,3} & -a_{4,3} & -b  &&\cdots &0 &\cdots&0 & 0&0 \\
 \boldsymbol{\cdot} & \boldsymbol{\cdot} & \boldsymbol{\cdot} & -a_{5,3} & -a_{4,3}  &\ddots& &\vdots &&\vdots & \vdots&\vdots \\
  \vdots & \vdots & \vdots &\vdots &\vdots  &\ddots& \ddots&\vdots &&\vdots & \vdots&\vdots \\
  -a_{i,1} & 0 & -a_{i,3} & -a_{i-1,3} & -a_{i-2,3}  &\cdots&-a_{4,3}& -b&\cdots&0 & 0&0\\
   \vdots  & \vdots  & \vdots &\vdots &\vdots&&\vdots &\vdots &&\vdots  & \vdots&\vdots \\
 -a_{n-1,1} & 0 & -a_{n-1,3}& -a_{n-2,3}& -a_{n-3,3}&\cdots &-a_{n-i+3,3} &-a_{n-i+2,3}&\cdots&-a_{4,3} &-b& 0\\
 -a_{n,1} & 0 & -a_{n,3}& -a_{n-1,3}& -a_{n-2,3}&\cdots &-a_{n-i+4,3}  &-a_{n-i+3,3}&\cdots&-a_{5,3} &-a_{4,3}&-b
\end{smallmatrix}\right].$$
   \allowdisplaybreaks
\item[(4)] If $a_{1,3}=0,b=0,a\neq0,(n=4)$ or $b=0,a\neq0,(n\geq5),$ then
\begin{equation}
\left\{
\begin{array}{l}
\displaystyle  \nonumber [e_1,e_{n+1}]=ae_1+a_{2,1}e_2-ae_3+\sum_{k=4}^n{a_{k,1}e_k},
[e_3,e_{n+1}]=a_{2,3}e_2+\sum_{k=4}^n{a_{k,3}e_k},\\
\displaystyle[e_4,e_{n+1}]=a(e_4-e_2)+\sum_{k=5}^n{a_{k-1,3}e_k},
[e_{i},e_{n+1}]=(i-3)ae_{i}+\sum_{k=i+1}^n{a_{k-i+3,3}e_k},\\
\displaystyle 
[e_{n+1},e_{n+1}]=a_{2,n+1}e_2,[e_{n+1},e_1]=-ae_1+\left(a_{2,3}-a_{2,1}+b_{2,3}\right)e_2+ae_3-\sum_{k=4}^n{a_{k,1}e_k},\\
\displaystyle [e_{n+1},e_3]=b_{2,3}e_2-\sum_{k=4}^n{a_{k,3}e_k},[e_{n+1},e_4]=a(e_2-e_4)-\sum_{k=5}^n{a_{k-1,3}e_k},\\
\displaystyle [e_{n+1},e_i]=(3-i)ae_i-\sum_{k=i+1}^n{a_{k-i+3,3}e_k},(5\leq i\leq n),
\end{array} 
\right.
\end{equation} 
$$\L_{e_{n+1}}=\left[\begin{smallmatrix}
 -a & 0 & 0 & 0&0&&\cdots &0& \cdots&0 & 0&0 \\
  a_{2,3}-a_{2,1}+b_{2,3} & 0 & b_{2,3}& a &0& & \cdots &0&\cdots  & 0& 0&0\\
  a & 0 & 0 & 0 & 0& &\cdots &0&\cdots &0& 0&0 \\
  -a_{4,1} & 0 &  -a_{4,3} & -a &0 & &\cdots&0&\cdots &0 & 0&0\\
  -a_{5,1} & 0 & -a_{5,3} & -a_{4,3} & -2a  &&\cdots &0 &\cdots&0 & 0&0 \\
 \boldsymbol{\cdot} & \boldsymbol{\cdot} & \boldsymbol{\cdot} & -a_{5,3} & -a_{4,3}  &\ddots& &\vdots &&\vdots & \vdots&\vdots \\
  \vdots & \vdots & \vdots &\vdots &\vdots  &\ddots& \ddots&\vdots &&\vdots & \vdots&\vdots \\
  -a_{i,1} & 0 & -a_{i,3} & -a_{i-1,3} & -a_{i-2,3}  &\cdots&-a_{4,3}& (3-i)a&\cdots&0 & 0&0\\
   \vdots  & \vdots  & \vdots &\vdots &\vdots&&\vdots &\vdots &&\vdots  & \vdots&\vdots \\
 -a_{n-1,1} & 0 & -a_{n-1,3}& -a_{n-2,3}& -a_{n-3,3}&\cdots &-a_{n-i+3,3} &-a_{n-i+2,3}&\cdots&-a_{4,3} &(4-n)a& 0\\
 -a_{n,1} & 0 & -a_{n,3}& -a_{n-1,3}& -a_{n-2,3}&\cdots &-a_{n-i+4,3}  &-a_{n-i+3,3}&\cdots&-a_{5,3} &-a_{4,3}& (3-n)a
\end{smallmatrix}\right].$$ 
In the remaining cases $a_{1,3}:=c.$
\item[(5)] If $b\neq-a,a\neq0,b\neq-c,c\neq0,$ then
\begin{equation}
\left\{
\begin{array}{l}
\displaystyle  \nonumber [e_1,e_{5}]=ae_1+a_{2,1}e_2+(b+c-a)e_3+a_{4,1}e_4,
[e_3,e_{5}]=ce_1+a_{2,3}e_2+be_3+A_{4,3}e_4,\\
\displaystyle [e_4,e_{5}]=(a+b)(e_4-e_2),[e_{5},e_{5}]=a_{2,5}e_2,[e_{5},e_1]=-ae_1+B_{2,1}e_2+(a-b-c)e_3-a_{4,1}e_4,\\
\displaystyle [e_5,e_{2}]=-2(b+c)e_2,[e_{5},e_3]=-ce_1+b_{2,3}e_2-be_3-A_{4,3}e_4,[e_{5},e_4]=(a-2c-b)e_2-\\
\displaystyle (a+b)e_4; B_{2,1}:=\frac{(2b-a+2c)a_{2,1}+2(b+c)a_{4,1}+(a-b-c)(a_{2,3}+b_{2,3})}{a}\,\,and \\
\displaystyle 
A_{4,3}:=\frac{2c(a_{2,1}+a_{4,1})-(a+c)a_{2,3}+(a-c)b_{2,3}}{2a},
\end{array} 
\right.
\end{equation} 
$\L_{e_{5}}=\left[\begin{smallmatrix}
 -a & 0 & -c & 0\\
  B_{2,1} & -2(b+c) & b_{2,3}& a-2c-b \\
  a-b-c & 0 & -b & 0\\
  -a_{4,1} & 0 &  -A_{4,3} & -a-b
\end{smallmatrix}\right].$
\item[(6)] If $b:=-a,a\neq0,a\neq c,c\neq0,$ then
\begin{equation}
\left\{
\begin{array}{l}
\displaystyle  \nonumber [e_1,e_{5}]=ae_1+a_{2,1}e_2+(c-2a)e_3+a_{4,1}e_4,
[e_3,e_{5}]=ce_1+a_{2,3}e_2-ae_3+A_{4,3}e_4,\\
\displaystyle[e_{5},e_{5}]=a_{2,5}e_2+a_{4,5}e_4,[e_{5},e_1]=-ae_1+B_{2,1}e_2+(2a-c)e_3-a_{4,1}e_4,\\
\displaystyle [e_5,e_{2}]=2(a-c)e_2,[e_{5},e_3]=-ce_1+b_{2,3}e_2+ae_3-A_{4,3}e_4,[e_{5},e_4]=2(a-c)e_2,\\
\displaystyle where\,\, B_{2,1}:=\frac{(2c-3a)a_{2,1}+2(c-a)a_{4,1}+(2a-c)(a_{2,3}+b_{2,3})}{a}\,\,and \\
\displaystyle 
A_{4,3}:=\frac{2c(a_{2,1}+a_{4,1})-(a+c)a_{2,3}+(a-c)b_{2,3}}{2a},
\end{array} 
\right.
\end{equation} 
$\L_{e_{5}}=\left[\begin{smallmatrix}
 -a & 0 & -c & 0\\
  B_{2,1} & 2(a-c) & b_{2,3}& 2(a-c) \\
  2a-c & 0 & a & 0\\
  -a_{4,1} & 0 &  -A_{4,3} & 0
\end{smallmatrix}\right].$
\item[(7)] If $b:=-c,c\neq0,a\neq c,a\neq0,$ then
\begin{equation}
\left\{
\begin{array}{l}
\displaystyle  \nonumber [e_1,e_{5}]=ae_1+a_{2,1}e_2-ae_3+a_{4,1}e_4,
[e_3,e_{5}]=ce_1+a_{2,3}e_2-ce_3+a_{4,3}e_4,\\
\displaystyle [e_4,e_{5}]=(c-a)(e_2-e_4),[e_{5},e_{5}]=a_{2,5}e_2,[e_{5},e_1]=-ae_1+\left(a_{2,3}-a_{2,1}+b_{2,3}\right)e_2+ae_3-\\
\displaystyle a_{4,1}e_4,[e_{5},e_3]=-ce_1+b_{2,3}e_2+ce_3-a_{4,3}e_4,[e_{5},e_4]=(c-a)(e_4-e_2),
\end{array} 
\right.
\end{equation} 
$\L_{e_{5}}=\left[\begin{smallmatrix}
 -a & 0 & -c & 0\\
  a_{2,3}-a_{2,1}+b_{2,3} & 0 & b_{2,3}& a-c\\
  a & 0 & c & 0\\
  -a_{4,1} & 0 &  -a_{4,3} & c-a
\end{smallmatrix}\right].$
\item[(8)] If $c:=a,b:=-a,a\neq0,$ then\footnote{The outer derivation $\L_{e_5}$ is nilpotent, so we remove this case from further consideration.}
\begin{equation}
\left\{
\begin{array}{l}
\displaystyle  \nonumber [e_1,e_{5}]=ae_1+a_{2,1}e_2-ae_3+\left(a_{4,3}+b_{4,3}-b_{4,1}\right)e_4,
[e_3,e_{5}]=ae_1+a_{2,3}e_2-ae_3+a_{4,3}e_4,\\
\displaystyle [e_{5},e_{5}]=a_{2,5}e_2+a_{4,5}e_4,[e_{5},e_1]=-ae_1+\left(a_{2,3}-a_{2,1}+b_{2,3}\right)e_2+ae_3+b_{4,1}e_4,\\
\displaystyle [e_{5},e_3]=-ae_1+b_{2,3}e_2+ae_3+b_{4,3}e_4,
\end{array} 
\right.
\end{equation} 
$\L_{e_{5}}=\left[\begin{smallmatrix}
 -a & 0 & -a & 0\\
  a_{2,3}-a_{2,1}+b_{2,3}& 0& b_{2,3}& 0 \\
  a & 0 & a & 0\\
  b_{4,1} & 0 &  b_{4,3} & 0
\end{smallmatrix}\right].$
\item[(9)] If $a=0,b=0,c\neq0,$ then
\begin{equation}
\left\{
\begin{array}{l}
\displaystyle  \nonumber [e_1,e_{5}]=a_{2,1}e_2+ce_3+a_{4,1}e_4,
[e_3,e_{5}]=ce_1+\left(2a_{2,1}+2a_{4,1}-b_{2,3}\right)e_2+a_{4,3}e_4,\\
\displaystyle [e_{5},e_{5}]=a_{2,5}e_2+a_{4,5}e_4,[e_{5},e_1]=\left(3a_{2,1}+4a_{4,1}-2b_{2,3}+2a_{4,3}\right)e_2-ce_3-a_{4,1}e_4,\\
\displaystyle [e_5,e_{2}]=-2ce_2,[e_{5},e_3]=-ce_1+b_{2,3}e_2-a_{4,3}e_4,[e_5,e_{4}]=-2ce_2,
\end{array} 
\right.
\end{equation} 
$\L_{e_{5}}=\left[\begin{smallmatrix}
 0 & 0 & -c & 0\\
  3a_{2,1}+4a_{4,1}-2b_{2,3}+2a_{4,3} & -2c & b_{2,3}& -2c\\
  -c & 0 & 0 & 0\\
  -a_{4,1} & 0 &  -a_{4,3} & 0
\end{smallmatrix}\right].$
\item[(10)] If $a=0,b\neq0,b\neq-c,c\neq0,$ then
\begin{equation}
\left\{
\begin{array}{l}
\displaystyle  \nonumber [e_1,e_{5}]=\left(\frac{a_{2,3}+b_{2,3}}{2}-a_{4,1}\right)e_2+(b+c)e_3+a_{4,1}e_4,
[e_3,e_{5}]=ce_1+a_{2,3}e_2+be_3+A_{4,3}e_4,\\
\displaystyle [e_4,e_{5}]=b(e_4-e_2),[e_{5},e_{5}]=a_{2,5}e_2,[e_{5},e_1]=b_{2,1}e_2-(b+c)e_3-a_{4,1}e_4,\\
\displaystyle [e_5,e_{2}]=-2(b+c)e_2,[e_{5},e_3]=-ce_1+b_{2,3}e_2-be_3-A_{4,3}e_4,[e_{5},e_4]=-(2c+b)e_2-be_4,\\
\displaystyle where\,\,A_{4,3}:=\frac{(2b+c)b_{2,3}-(2b+3c)a_{2,3}+2c(b_{2,1}-a_{4,1})}{4(b+c)},
\end{array} 
\right.
\end{equation} 
$\L_{e_{5}}=\left[\begin{smallmatrix}
 0 & 0 & -c & 0\\
  b_{2,1} & -2(b+c) & b_{2,3}& -2c-b \\
  -b-c & 0 & -b & 0\\
  -a_{4,1} & 0 &  -A_{4,3} & -b
\end{smallmatrix}\right].$
\item[(11)] If $a=0,b:=-c,c\neq0,$ then
\begin{equation}
\left\{
\begin{array}{l}
\displaystyle  \nonumber [e_1,e_{5}]=\left(a_{2,3}+b_{2,3}-b_{2,1}\right)e_2+a_{4,1}e_4,
[e_3,e_{5}]=ce_1+a_{2,3}e_2-ce_3+a_{4,3}e_4,\\
\displaystyle [e_4,e_{5}]=c(e_2-e_4),[e_{5},e_{5}]=a_{2,5}e_2,[e_{5},e_1]=b_{2,1}e_2-a_{4,1}e_4,[e_{5},e_3]=-ce_1+b_{2,3}e_2+\\
\displaystyle ce_3-a_{4,3}e_4,[e_{5},e_4]=c(e_4-e_2),
\end{array} 
\right.
\end{equation} 
$\L_{e_{5}}=\left[\begin{smallmatrix}
 0 & 0 & -c & 0\\
  b_{2,1} & 0 & b_{2,3}& -c \\
  0 & 0 &c & 0\\
  -a_{4,1} & 0 &  -a_{4,3} & c
\end{smallmatrix}\right].$
\end{enumerate}
\end{theorem}
\begin{proof} 
\begin{enumerate}[noitemsep, topsep=0pt]
\item[(1)] The proof is off-loaded to Table \ref{Left(L4)}, when $(n\geq5)$. For $(n=4),$
we recalculate the applicable identities, which are $1.,2.,4.,6.-8.,10.-15.$
\item[(2)]  We repeat case $(1),$ except the identity $13.$ 
\item[(3)] If $(n\geq5),$ then we apply the identities given in Table \ref{Left(L4)},
except $14.$ and $15.$ applying instead: $\L_{e_3}[e_{n+1},e_{n+1}]=[\L_{e_3}(e_{n+1}),e_{n+1}]+[e_{n+1},\L_{e_3}(e_{n+1})]$
and $\L[e_1,e_{n+1}]=[\L(e_1),e_{n+1}]+[e_1,\L(e_{n+1})].$
For $(n=4),$ we apply the same identities as in case $(1)$ except $14.$ and $15.$
applying two identities given above.
\item[(4)] We apply the same identities as in case $(1).$
\item[(5)] We apply the following identities: $1.,2.,4.,6.-8.,10.-14.,\L_{e_3}[e_5,e_5]=[\L_{e_3}(e_5),e_5]+[e_5,\L_{e_3}(e_5)],15.$
\item[(6)] We apply the same identities as in case $(2)$ for $(n=4).$
\item[(7)] We apply the same identities as in case $(1)$ for $(n=4).$
\item[(8)] We apply the same identities as in case $(1)$ for $(n=4),$ except $13.$ and $15.$
\item[(9)] We apply the same identities as in case $(2)$ for $(n=4).$
\item[(10)] The same identities as in case $(5).$
\item[(11)] The same identities as in case $(5),$ except the identity $15.$
\end{enumerate}
\end{proof} 

\newpage
\begin{table}[h!]
\caption{Left Leibniz identities in case $(1)$ in Theorem \ref{TheoremLL4}, ($n\geq5$).}
\label{Left(L4)}
\begin{tabular}{lp{2.4cm}p{12cm}}
\hline
\scriptsize Steps &\scriptsize Ordered triple &\scriptsize
Result\\ \hline
\scriptsize $1.$ &\scriptsize $\L_{e_1}\left([e_1,e_{n+1}]\right)$ &\scriptsize
$[e_{2},e_{n+1}]=0$
$\implies$ $a_{k,2}=0,(1\leq k\leq n).$\\ \hline
\scriptsize $2.$ &\scriptsize $\L_{e_1}\left([e_{3},e_{n+1}]\right)$ &\scriptsize
 $a_{1,4}=0,a_{k,4}:=a_{k-1,3},(5\leq k\leq n),a_{3,1}:=a_{2,4}+a_{1,3}+2b,a_{3,4}=0,a_{4,4}:=a+b$ $\implies$ 
 $[e_1,e_{n+1}]=ae_1+a_{2,1}e_2+(a_{2,4}+a_{1,3}+2b)e_3+\sum_{k=4}^n{a_{k,1}e_k},[e_{4},e_{n+1}]=a_{2,4}e_2+(a+b)e_4+\sum_{k=5}^n{a_{k-1,3}e_k}.$ \\ \hline
 \scriptsize $3.$ &\scriptsize $\L_{e_1}\left([e_{i},e_{n+1}]\right)$ &\scriptsize
 $a_{1,i+1}=a_{3,i+1}=0\,\,and\,\,we\,\,had\,\,that\,\,a_{1,4}=a_{3,4}=0\implies a_{2,i+1}=a_{4,i+1}=0;a_{i+1,i+1}:=(i-2)a+b,a_{k,i+1}:=a_{k-1,i},(5\leq k\leq n,k\neq i+1,4\leq i\leq n-1)$ $\implies$ 
 $[e_{j},e_{n+1}]=\left((j-3)a+b\right)e_j+\sum_{k=j+1}^n{a_{k-j+3,3}e_k},(5\leq j\leq n).$ \\ \hline
 \scriptsize $4.$ &\scriptsize $\L_{e_3}\left([e_{n+1},e_3]\right)$ &\scriptsize
 $b_{1,2}=b_{3,2}=0,b_{k,2}=0,(5\leq k\leq n),b_{3,3}:=2a_{1,3}+b+b_{2,2},b_{4,2}:=a_{1,3}+b_{1,3}$ $\implies$ 
 $[e_{n+1},e_{2}]=b_{2,2}e_2+(a_{1,3}+b_{1,3})e_4,[e_{n+1},e_3]=b_{1,3}e_1+b_{2,3}e_2+(2a_{1,3}+b+b_{2,2})e_3+\sum_{k=4}^n{b_{k,3}e_k}.$ \\ \hline
 \scriptsize $5.$ &\scriptsize $\L_{e_3}\left([e_{n+1},e_{i}]\right)$ &\scriptsize
 $b_{1,i}=b_{3,i}=0,a_{1,3}=0$ $\implies$ 
 $[e_1,e_{n+1}]=ae_1+a_{2,1}e_2+(a_{2,4}+2b)e_3+\sum_{k=4}^n{a_{k,1}e_k},
 [e_3,e_{n+1}]=a_{2,3}e_2+be_3+\sum_{k=4}^n{a_{k,3}e_k},
 [e_{n+1},e_2]=b_{2,2}e_2+b_{1,3}e_4,
 [e_{n+1},e_3]=b_{1,3}e_1+b_{2,3}e_2+(b+b_{2,2})e_3+\sum_{k=4}^n{b_{k,3}e_k},
 [e_{n+1},e_{i}]=b_{2,i}e_2+\sum_{k=4}^n{b_{k,i}e_k},(4\leq i\leq n-1).$ \\ \hline
 \scriptsize $6.$ &\scriptsize $\L_{e_3}\left([e_{n+1},e_{n}]\right)$ &\scriptsize
 $b_{1,n}=b_{3,n}=0$ $\implies$ 
 $[e_{n+1},e_{n}]=b_{2,n}e_2+\sum_{k=4}^n{b_{k,n}e_k}.$\\ \hline
 \scriptsize $7.$ &\scriptsize $\L[e_{3},e_{3}]$ &\scriptsize
 $b_{1,3}=0\implies b_{2,2}:=-2b$ $\implies$ 
 $[e_{n+1},e_{2}]=-2be_2,[e_{n+1},e_3]=b_{2,3}e_2-be_3+\sum_{k=4}^n{b_{k,3}e_k}.$\\ \hline 
 \scriptsize $8.$ &\scriptsize $\L_{e_1}\left([e_{n+1},e_{3}]\right)$ &\scriptsize
 $b_{2,4}:=a_{2,4}+2a,b_{4,4}:=-a-b,b_{k,4}:=b_{k-1,3},(5\leq k\leq n)$ $\implies$ 
 $[e_{n+1},e_{4}]=\left(a_{2,4}+2a\right)e_2-\left(a+b\right)e_4+\sum_{k=5}^n{b_{k-1,3}e_k}.$ \\ \hline 
 \scriptsize $9.$ &\scriptsize $\L_{e_1}\left([e_{n+1},e_{i}]\right)$ &\scriptsize
 $b_{2,i+1}=b_{4,i+1}=0,b_{i+1,i+1}:=(2-i)a-b,b_{k,i+1}:=b_{k-1,i},(5\leq k\leq n,k\neq i+1,4\leq i\leq n-1)$ $\implies$ 
$[e_{n+1},e_j]=\left((3-j)a-b\right)e_j+\sum_{k=j+1}^n{b_{k-j+3,3}e_k},(5\leq j\leq n).$ \\ \hline
\scriptsize $10.$ &\scriptsize $\L_{e_3}\left([e_{n+1},e_{1}]\right)$ &\scriptsize
 $b_{1,1}:=-a,b_{3,1}:=a_{2,4}+2a,b_{k-1,3}:=-a_{k-1,3},(5\leq k\leq n)$ $\implies$ 
$[e_{n+1},e_{1}]=-ae_1+b_{2,1}e_2+(a_{2,4}+2a)e_3+\sum_{k=4}^n{b_{k,1}e_k},
[e_{n+1},e_{3}]=b_{2,3}e_2-be_3-\sum_{k=4}^{n-1}{a_{k,3}e_k}+b_{n,3}e_n,
[e_{n+1},e_{4}]=(a_{2,4}+2a)e_2-(a+b)e_4-\sum_{k=5}^n{a_{k-1,3}e_k},
[e_{n+1},e_i]=\left((3-i)a-b\right)e_i-\sum_{k=i+1}^n{a_{k-i+3,3}e_k},(5\leq i\leq n).$\\ \hline
\scriptsize $11.$ &\scriptsize $\L_{e_1}\left([e_{n+1},e_{1}]\right)$ &\scriptsize
 $a_{2,4}:=-a-b,b_{k-1,1}:=-a_{k-1,1},(5\leq k\leq n)$ $\implies$ 
 $[e_{1},e_{n+1}]=ae_1+a_{2,1}e_2+(b-a)e_3+\sum_{k=4}^n{a_{k,1}e_k},
 [e_{4},e_{n+1}]=(a+b)(e_4-e_2)+\sum_{k=5}^n{a_{k-1,3}e_k},
[e_{n+1},e_{1}]=-ae_1+b_{2,1}e_2+(a-b)e_3-\sum_{k=4}^{n-1}{a_{k,1}e_k}+b_{n,1}e_n,
[e_{n+1},e_{4}]=(a-b)e_2-(a+b)e_4-\sum_{k=5}^n{a_{k-1,3}e_k}.$ \\ \hline
\scriptsize $12.$ &\scriptsize $\L[e_{n+1},e_{1}]$ &\scriptsize
 $a_{1,n+1}=0,a_{k,n+1}=0,(3\leq k\leq n-1)$ $\implies$ 
 $[e_{n+1},e_{n+1}]=a_{2,n+1}e_2+a_{n,n+1}e_n.$ \\ \hline
 \scriptsize $13.$ &\scriptsize $\L[e_{n+1},e_{n+1}]$ &\scriptsize
 $a_{n,n+1}=0,(because\,\,b\neq(3-n)a)$ $\implies$ 
 $[e_{n+1},e_{n+1}]=a_{2,n+1}e_2.$ \\ \hline
 \scriptsize $14.$ &\scriptsize $\L[e_{3},e_{n+1}]$ &\scriptsize
$b_{n,3}:=-a_{n,3},(because\,\,a\neq0),b_{2,3}:=a_{2,3}+2a_{4,3},(because\,\,b\neq0)$ $\implies$ 
 $[e_{n+1},e_{3}]=(a_{2,3}+2a_{4,3})e_2-be_3-\sum_{k=4}^{n}{a_{k,3}e_k}.$ \\ \hline 
 \scriptsize $15.$ &\scriptsize $\L_{e_1}\left([e_{n+1},e_{n+1}]\right)$ &\scriptsize
 $b_{n,1}:=-a_{n,1},B_{2,1}:=\frac{(2b-a)a_{2,1}+2b\cdot a_{4,1}+2(a-b)(a_{2,3}+a_{4,3})}{a},(because\,\,a\neq0)$ $\implies$ 
 $[e_{n+1},e_{1}]=-ae_1+B_{2,1}e_2+(a-b)e_3-\sum_{k=4}^{n}{a_{k,1}e_k}.$ \\ \hline 
 \end{tabular}
\end{table}

\begin{theorem}\label{TheoremLL4Absorption} Applying the technique of ``absorption'' (see Section \ref{Solvable left Leibniz algebras}), we can further simplify the algebras 
in each of the cases in Theorem \ref{TheoremLL4} as follows:
\begin{enumerate}[noitemsep, topsep=0pt]
\allowdisplaybreaks
\item[(1)] If $a_{1,3}=0, b\neq-a,a\neq0,b\neq0,(n=4)$ or $b\neq(3-n)a,a\neq0,b\neq0,(n\geq5),$ then we have the following brackets for the algebra:
\begin{equation}
\left\{
\begin{array}{l}
\displaystyle  \nonumber [e_1,e_{n+1}]=ae_1+a_{2,1}e_2+(b-a)e_3,
[e_3,e_{n+1}]=a_{2,3}e_2+be_3+\sum_{k=5}^n{a_{k,3}e_k},\\
\displaystyle[e_4,e_{n+1}]=(a+b)(e_4-e_2)+\sum_{k=6}^n{a_{k-1,3}e_k},
[e_{i},e_{n+1}]=\left((i-3)a+b\right)e_{i}+\sum_{k=i+2}^n{a_{k-i+3,3}e_k},\\
\displaystyle [e_{n+1},e_1]=-ae_1+\mathcal{B}_{2,1}e_2+(a-b)e_3,[e_{n+1},e_2]=-2be_2,[e_{n+1},e_3]=a_{2,3}e_2-be_3-\\
\displaystyle \sum_{k=5}^n{a_{k,3}e_k},[e_{n+1},e_4]=(a-b)e_2-(a+b)e_4-\sum_{k=6}^n{a_{k-1,3}e_k}, [e_{n+1},e_i]=\left((3-i)a-b\right)e_i-\\
\displaystyle \sum_{k=i+2}^n{a_{k-i+3,3}e_k},(5\leq i\leq n);where\,\,\mathcal{B}_{2,1}:=\frac{(2b-a)a_{2,1}+2(a-b)a_{2,3}}{a},
\end{array} 
\right.
\end{equation} 
$$\L_{e_{n+1}}=\left[\begin{smallmatrix}
 -a & 0 & 0 & 0&0&0&\cdots && 0&0 & 0\\
  \mathcal{B}_{2,1} & -2b & a_{2,3}& a-b &0&0 & \cdots &&0  & 0& 0\\
  a-b & 0 & -b & 0 & 0&0 &\cdots &&0 &0& 0\\
  0 & 0 &  0 & -a-b &0 &0 &\cdots&&0 &0 & 0\\
 0 & 0 & -a_{5,3} & 0 & -2a-b  &0&\cdots & &0&0 & 0\\
  0 & 0 &\boldsymbol{\cdot} & -a_{5,3} & 0 &-3a-b&\cdots & &0&0 & 0\\
    0 & 0 &\boldsymbol{\cdot} & \boldsymbol{\cdot} & \ddots &0&\ddots &&\vdots&\vdots &\vdots\\
  \vdots & \vdots & \vdots &\vdots &  &\ddots&\ddots &\ddots &\vdots&\vdots & \vdots\\
 0 & 0 & -a_{n-2,3}& -a_{n-3,3}& \cdots&\cdots &-a_{5,3}&0&(5-n)a-b &0& 0\\
0 & 0 & -a_{n-1,3}& -a_{n-2,3}& \cdots&\cdots &\boldsymbol{\cdot}&-a_{5,3}&0 &(4-n)a-b& 0\\
 0 & 0 & -a_{n,3}& -a_{n-1,3}& \cdots&\cdots &\boldsymbol{\cdot}&\boldsymbol{\cdot}&-a_{5,3} &0& (3-n)a-b
\end{smallmatrix}\right].$$
\item[(2)] If $a_{1,3}=0,b:=-a,a\neq0,(n=4)$ or $b:=(3-n)a,a\neq0,(n\geq5),$
then the brackets for the algebra are as follows:  
\begin{equation}
\left\{
\begin{array}{l}
\displaystyle  \nonumber [e_1,e_{n+1}]=ae_1+a_{2,1}e_2-(n-2)ae_3,
[e_3,e_{n+1}]=a_{2,3}e_2+(3-n)ae_3+ \sum_{k=5}^n{a_{k,3}e_k},\\
\displaystyle[e_4,e_{n+1}]=(n-4)a(e_2-e_4)+\sum_{k=6}^n{a_{k-1,3}e_k},
[e_{i},e_{n+1}]=\left(i-n\right)ae_{i}+\sum_{k=i+2}^n{a_{k-i+3,3}e_k},\\
\displaystyle 
[e_{n+1},e_{n+1}]=a_{n,n+1}e_n,[e_{n+1},e_1]=-ae_1+\mathcal{B}_{2,1}e_2+(n-2)ae_3,[e_{n+1},e_2]=(2n-6)ae_2,\\
\displaystyle[e_{n+1},e_3]=a_{2,3}e_2+(n-3)ae_3-\sum_{k=5}^n{a_{k,3}e_k},[e_{n+1},e_4]=(n-2)ae_2+(n-4)ae_4-\\
\displaystyle \sum_{k=6}^n{a_{k-1,3}e_k},[e_{n+1},e_i]=\left(n-i\right)ae_i-\sum_{k=i+2}^n{a_{k-i+3,3}e_k},(5\leq i\leq n),\\
\displaystyle where\,\,\mathcal{B}_{2,1}:=(5-2n)a_{2,1}+2(n-2)a_{2,3},
\end{array} 
\right.
\end{equation}
$$\L_{e_{n+1}}=\left[\begin{smallmatrix}
 -a & 0 & 0 & 0&0&0&\cdots && 0&0 & 0\\
  \mathcal{B}_{2,1} & (2n-6)a & a_{2,3}& (n-2)a &0&0 & \cdots &&0  & 0& 0\\
  (n-2)a & 0 & (n-3)a & 0 & 0&0 &\cdots &&0 &0& 0\\
  0 & 0 &  0 & (n-4)a &0 &0 &\cdots&&0 &0 & 0\\
 0 & 0 & -a_{5,3} & 0 & (n-5)a  &0&\cdots & &0&0 & 0\\
  0 & 0 &\boldsymbol{\cdot} & -a_{5,3} & 0 &(n-6)a&\cdots & &0&0 & 0\\
    0 & 0 &\boldsymbol{\cdot} & \boldsymbol{\cdot} & \ddots &0&\ddots &&\vdots&\vdots &\vdots\\
  \vdots & \vdots & \vdots &\vdots &  &\ddots&\ddots &\ddots &\vdots&\vdots & \vdots\\
 0 & 0 & -a_{n-2,3}& -a_{n-3,3}& \cdots&\cdots &-a_{5,3}&0&2a &0& 0\\
0 & 0 & -a_{n-1,3}& -a_{n-2,3}& \cdots&\cdots &\boldsymbol{\cdot}&-a_{5,3}&0 &a& 0\\
 0 & 0 & -a_{n,3}& -a_{n-1,3}& \cdots&\cdots &\boldsymbol{\cdot}&\boldsymbol{\cdot}&-a_{5,3} &0& 0
\end{smallmatrix}\right].$$
\item[(3)] If $a_{1,3}=0,a=0,b\neq0,(n=4)$ or $a=0$ and $b\neq0,(n\geq5),$ then
\begin{equation}
\left\{
\begin{array}{l}
\displaystyle  \nonumber [e_1,e_{n+1}]=a_{2,3}e_2+be_3,
[e_3,e_{n+1}]=a_{2,3}e_2+be_3+\sum_{k=5}^n{a_{k,3}e_k},[e_4,e_{n+1}]=b\left(e_4-e_2\right)+\\
\displaystyle \sum_{k=6}^n{a_{k-1,3}e_k},
[e_{i},e_{n+1}]=be_{i}+\sum_{k=i+2}^n{a_{k-i+3,3}e_k},[e_{n+1},e_1]=b_{2,1}e_2-be_3,[e_{n+1},e_2]=-2be_2,\\
\displaystyle [e_{n+1},e_3]=a_{2,3}e_2-be_3-\sum_{k=5}^n{a_{k,3}e_k},[e_{n+1},e_4]=-b\left(e_2+e_4\right)-\sum_{k=6}^n{a_{k-1,3}e_k},\\
\displaystyle [e_{n+1},e_i]=-be_i-\sum_{k=i+2}^n{a_{k-i+3,3}e_k},(5\leq i\leq n),
\end{array} 
\right.
\end{equation} 
$$\L_{e_{n+1}}=\left[\begin{smallmatrix}
0 & 0 & 0 & 0&0&0&\cdots && 0&0 & 0\\
  b_{2,1} & -2b & a_{2,3}& -b &0&0 & \cdots &&0  & 0& 0\\
  -b & 0 & -b & 0 & 0&0 &\cdots &&0 &0& 0\\
  0 & 0 &  0 & -b &0 &0 &\cdots&&0 &0 & 0\\
 0 & 0 & -a_{5,3} & 0 & -b  &0&\cdots & &0&0 & 0\\
  0 & 0 &\boldsymbol{\cdot} & -a_{5,3} & 0 &-b&\cdots & &0&0 & 0\\
    0 & 0 &\boldsymbol{\cdot} & \boldsymbol{\cdot} & \ddots &0&\ddots &&\vdots&\vdots &\vdots\\
  \vdots & \vdots & \vdots &\vdots &  &\ddots&\ddots &\ddots &\vdots&\vdots & \vdots\\
 0 & 0 & -a_{n-2,3}& -a_{n-3,3}& \cdots&\cdots &-a_{5,3}&0&-b &0& 0\\
0 & 0 & -a_{n-1,3}& -a_{n-2,3}& \cdots&\cdots &\boldsymbol{\cdot}&-a_{5,3}&0 &-b& 0\\
 0 & 0 & -a_{n,3}& -a_{n-1,3}& \cdots&\cdots &\boldsymbol{\cdot}&\boldsymbol{\cdot}&-a_{5,3} &0&-b
\end{smallmatrix}\right].$$
\allowdisplaybreaks
\item[(4)] If $a_{1,3}=0,b=0,a\neq0,(n=4)$ or $b=0,a\neq0,(n\geq5),$ then
\begin{equation}
\left\{
\begin{array}{l}
\displaystyle  \nonumber [e_1,e_{n+1}]=ae_1+a_{2,1}e_2-ae_3,
[e_3,e_{n+1}]=a_{2,3}e_2+\sum_{k=5}^n{a_{k,3}e_k},[e_4,e_{n+1}]=a(e_4-e_2)+\\
\displaystyle\sum_{k=6}^n{a_{k-1,3}e_k},
[e_{i},e_{n+1}]=(i-3)ae_{i}+\sum_{k=i+2}^n{a_{k-i+3,3}e_k},[e_{n+1},e_{n+1}]=a_{2,n+1}e_2,\\
\displaystyle [e_{n+1},e_1]=-ae_1+\left(a_{2,3}-a_{2,1}+b_{2,3}\right)e_2+ae_3,[e_{n+1},e_3]=b_{2,3}e_2-\sum_{k=5}^n{a_{k,3}e_k},\\
\displaystyle [e_{n+1},e_4]=a(e_2-e_4)-\sum_{k=6}^n{a_{k-1,3}e_k},[e_{n+1},e_i]=(3-i)ae_i-\sum_{k=i+2}^n{a_{k-i+3,3}e_k},(5\leq i\leq n),
\end{array} 
\right.
\end{equation} 
$$\L_{e_{n+1}}=\left[\begin{smallmatrix}
 -a & 0 & 0 & 0&0&0&\cdots && 0&0 & 0\\
  a_{2,3}-a_{2,1}+b_{2,3} & 0 & b_{2,3}& a &0&0 & \cdots &&0  & 0& 0\\
  a & 0 &0 & 0 & 0&0 &\cdots &&0 &0& 0\\
  0 & 0 &  0 & -a &0 &0 &\cdots&&0 &0 & 0\\
 0 & 0 & -a_{5,3} & 0 & -2a  &0&\cdots & &0&0 & 0\\
  0 & 0 &\boldsymbol{\cdot} & -a_{5,3} & 0 &-3a&\cdots & &0&0 & 0\\
    0 & 0 &\boldsymbol{\cdot} & \boldsymbol{\cdot} & \ddots &0&\ddots &&\vdots&\vdots &\vdots\\
  \vdots & \vdots & \vdots &\vdots &  &\ddots&\ddots &\ddots &\vdots&\vdots & \vdots\\
 0 & 0 & -a_{n-2,3}& -a_{n-3,3}& \cdots&\cdots &-a_{5,3}&0&(5-n)a &0& 0\\
0 & 0 & -a_{n-1,3}& -a_{n-2,3}& \cdots&\cdots &\boldsymbol{\cdot}&-a_{5,3}&0 &(4-n)a& 0\\
 0 & 0 & -a_{n,3}& -a_{n-1,3}& \cdots&\cdots &\boldsymbol{\cdot}&\boldsymbol{\cdot}&-a_{5,3} &0& (3-n)a
\end{smallmatrix}\right].$$
In the remaining cases $a_{1,3}:=c.$
\item[(5)] If $b\neq-a,a\neq0,b\neq-c,c\neq0,$ then
\begin{equation}
\left\{
\begin{array}{l}
\displaystyle  \nonumber [e_1,e_{5}]=ae_1+\mathcal{A}_{2,1}e_2+(b+c-a)e_3,
[e_3,e_{5}]=ce_1+a_{2,3}e_2+be_3,[e_{4},e_5]=(a+b)e_4-\\
\displaystyle (a+b)e_2,[e_{5},e_1]=-ae_1+\mathcal{B}_{2,1}e_2+(a-b-c)e_3,[e_5,e_{2}]=-2(b+c)e_2,[e_{5},e_3]=-ce_1+\\
\displaystyle \mathcal{B}_{2,3}e_2-be_3,[e_5,e_{4}]=(a-2c-b)e_2-(a+b)e_4;where\,\,\mathcal{A}_{2,1}:=\frac{(a+c)a_{2,3}+(c-a)b_{2,3}}{2a},\\
\displaystyle \mathcal{B}_{2,1}:=\frac{(3a-2b-c)a_{2,3}+(a-2b-c)b_{2,3}}{2a}\,\,and\,\, 
\mathcal{B}_{2,3}:=\frac{(a+c)a_{2,3}+c\cdot b_{2,3}}{a},
\end{array} 
\right.
\end{equation} 
$\L_{e_{5}}=\left[\begin{smallmatrix}
 -a & 0 & -c & 0\\
 \mathcal{B}_{2,1} & -2(b+c) &  \mathcal{B}_{2,3}& a-2c-b \\
 a-b-c & 0 & -b & 0\\
 0 & 0 &  0 & -a-b
\end{smallmatrix}\right].$
\item[(6)] If $b:=-a,a\neq0,a\neq c,c\neq0,$ then
\begin{equation}
\left\{
\begin{array}{l}
\displaystyle  \nonumber [e_1,e_{5}]=ae_1+\mathcal{A}_{2,1}e_2+(c-2a)e_3,
[e_3,e_{5}]=ce_1+a_{2,3}e_2-ae_3,[e_{5},e_{5}]=a_{4,5}e_4,\\
\displaystyle [e_{5},e_1]=-ae_1+\mathcal{B}_{2,1}e_2+(2a-c)e_3,[e_5,e_2]=2(a-c)e_2,[e_{5},e_3]=-ce_1+\mathcal{B}_{2,3}e_2+ae_3,\\
\displaystyle [e_5,e_{4}]=2(a-c)e_2;where\,\,\mathcal{A}_{2,1}:=\frac{(a+c)a_{2,3}+(c-a)b_{2,3}}{2a},\mathcal{B}_{2,3}:=\frac{(a+c)a_{2,3}+c\cdot b_{2,3}}{a},\\
\displaystyle \mathcal{B}_{2,1}:=\frac{(5a-c)a_{2,3}+(3a-c)b_{2,3}}{2a},
\end{array} 
\right.
\end{equation} 
$\L_{e_{5}}=\left[\begin{smallmatrix}
 -a & 0 & -c & 0\\
  \mathcal{B}_{2,1} & 2(a-c) &\mathcal{B}_{2,3}& 2(a-c) \\
 2a-c & 0 &a & 0\\
  0 & 0 &  0 &0
\end{smallmatrix}\right].$
\item[(7)] If $b:=-c,c\neq0,a\neq c,a\neq0,$ then
\begin{equation}
\left\{
\begin{array}{l}
\displaystyle  \nonumber [e_1,e_{5}]=ae_1+a_{2,1}e_2-ae_3,
[e_3,e_{5}]=ce_1+a_{2,3}e_2-ce_3,[e_4,e_{5}]=(c-a)(e_2-e_4),\\
\displaystyle [e_{5},e_{5}]=a_{2,5}e_2,[e_{5},e_1]=-ae_1+\left(a_{2,3}-a_{2,1}+b_{2,3}\right)e_2+ae_3,[e_{5},e_3]=-ce_1+b_{2,3}e_2+ce_3,\\
\displaystyle [e_{5},e_4]=(c-a)(e_4-e_2),
\end{array} 
\right.
\end{equation} 
$\L_{e_{5}}=\left[\begin{smallmatrix}
 -a & 0 & -c & 0\\
  a_{2,3}-a_{2,1}+b_{2,3} & 0 & b_{2,3}& a-c\\
  a & 0 & c & 0\\
 0& 0 &  0 & c-a
\end{smallmatrix}\right].$
\item[(8)] If $a=0,b=0,c\neq0,$ then
\begin{equation}
\left\{
\begin{array}{l}
\displaystyle  \nonumber [e_1,e_{5}]=a_{2,1}e_2+ce_3,
[e_3,e_{5}]=ce_1+\left(2a_{2,1}-b_{2,3}\right)e_2,[e_{5},e_{5}]=a_{4,5}e_4,\\
\displaystyle [e_{5},e_1]=\left(3a_{2,1}-2b_{2,3}\right)e_2-ce_3,[e_5,e_{2}]=-2ce_2,[e_{5},e_3]=-ce_1+b_{2,3}e_2,[e_5,e_{4}]=-2ce_2,
\end{array} 
\right.
\end{equation} 
$\L_{e_{5}}=\left[\begin{smallmatrix}
 0 & 0 & -c & 0\\
  3a_{2,1}-2b_{2,3} & -2c & b_{2,3}& -2c\\
  -c & 0 & 0 & 0\\
  0& 0 & 0 & 0
\end{smallmatrix}\right].$
\item[(9)] If $a=0,b\neq0,b\neq-c,c\neq0,$ then
\begin{equation}
\left\{
\begin{array}{l}
\displaystyle  \nonumber [e_1,e_{5}]=\mathcal{A}_{2,1}e_2+(b+c)e_3,
[e_3,e_{5}]=ce_1+a_{2,3}e_2+be_3,[e_4,e_{5}]=b(e_4-e_2),\\
\displaystyle [e_{5},e_1]=\mathcal{B}_{2,1}e_2-(b+c)e_3,[e_5,e_{2}]=-2(b+c)e_2,[e_{5},e_3]=-ce_1+\mathcal{B}_{2,3}e_2-be_3,\\
\displaystyle [e_{5},e_4]=-(2c+b)e_2-be_4;where\,\,\mathcal{A}_{2,1}:=\frac{(4b+5c)a_{2,3}+c\cdot b_{2,3}}{4(b+c)},\\
\displaystyle \mathcal{B}_{2,1}:=\frac{(2b+3c)a_{2,3}-(2b+c)b_{2,3}}{4(b+c)}\,\,and\,\,\mathcal{B}_{2,3}:=\frac{(2b+3c)a_{2,3}+c\cdot b_{2,3}}{2(b+c)},
\end{array} 
\right.
\end{equation} 
$\L_{e_{5}}=\left[\begin{smallmatrix}
 0 & 0 & -c & 0\\
  \mathcal{B}_{2,1} & -2(b+c) & \mathcal{B}_{2,3}& -2c-b \\
  -b-c & 0 & -b & 0\\
 0 & 0 &0 & -b
\end{smallmatrix}\right].$

\item[(10)] If $a=0,b:=-c,c\neq0,$ then
\begin{equation}
\left\{
\begin{array}{l}
\displaystyle  \nonumber [e_1,e_{5}]=\left(a_{2,3}+b_{2,3}-b_{2,1}\right)e_2,
[e_3,e_{5}]=ce_1+a_{2,3}e_2-ce_3,[e_4,e_{5}]=c(e_2-e_4),\\
\displaystyle [e_{5},e_{5}]=a_{2,5}e_2,[e_{5},e_1]=b_{2,1}e_2,[e_{5},e_3]=-ce_1+b_{2,3}e_2+ce_3,[e_{5},e_4]=c(e_4-e_2),
\end{array} 
\right.
\end{equation} 
$\L_{e_{5}}=\left[\begin{smallmatrix}
 0 & 0 & -c & 0\\
  b_{2,1} & 0 & b_{2,3}& -c \\
  0 & 0 &c & 0\\
  0 & 0 &  0 & c
\end{smallmatrix}\right].$

\end{enumerate}
\end{theorem}
\begin{proof}
\begin{enumerate}[noitemsep, topsep=0pt]
\item[(1)] The right (not a derivation)
and the left (a derivation) multiplication operators  restricted to the nilradical are given below:
$$\r_{e_{n+1}}=\left[\begin{smallmatrix}
 a & 0 & 0 & 0&0&&\cdots &0& \cdots&0 & 0&0 \\
  a_{2,1} & 0 & a_{2,3}& -a-b &0& & \cdots &0&\cdots  & 0& 0&0\\
  b-a & 0 & b & 0 & 0& &\cdots &0&\cdots &0& 0&0 \\
  a_{4,1} & 0 &  a_{4,3} & a+b &0 & &\cdots&0&\cdots &0 & 0&0\\
  a_{5,1} & 0 & a_{5,3} & a_{4,3} & 2a+b  &&\cdots &0 &\cdots&0 & 0&0 \\
 \boldsymbol{\cdot} & \boldsymbol{\cdot} & \boldsymbol{\cdot} & a_{5,3} & a_{4,3}  &\ddots& &\vdots &&\vdots & \vdots&\vdots \\
  \vdots & \vdots & \vdots &\vdots &\vdots  &\ddots& \ddots&\vdots &&\vdots & \vdots&\vdots \\
  a_{i,1} & 0 & a_{i,3} & a_{i-1,3} & a_{i-2,3}  &\cdots&a_{4,3}& (i-3)a+b&\cdots&0 & 0&0\\
   \vdots  & \vdots  & \vdots &\vdots &\vdots&&\vdots &\vdots &&\vdots  & \vdots&\vdots \\
 a_{n-1,1} & 0 & a_{n-1,3}& a_{n-2,3}& a_{n-3,3}&\cdots &a_{n-i+3,3} &a_{n-i+2,3}&\cdots&a_{4,3} &(n-4)a+b& 0\\
 a_{n,1} & 0 & a_{n,3}& a_{n-1,3}& a_{n-2,3}&\cdots &a_{n-i+4,3}  &a_{n-i+3,3}&\cdots&a_{5,3} &a_{4,3}& (n-3)a+b
\end{smallmatrix}\right],$$
$$\L_{e_{n+1}}=\left[\begin{smallmatrix}
 -a & 0 & 0 & 0&0&&\cdots &0& \cdots&0 & 0&0 \\
  B_{2,1} & -2b & a_{2,3}+2a_{4,3}& a-b &0& & \cdots &0&\cdots  & 0& 0&0\\
  a-b & 0 & -b & 0 & 0& &\cdots &0&\cdots &0& 0&0 \\
  -a_{4,1} & 0 &  -a_{4,3} & -a-b &0 & &\cdots&0&\cdots &0 & 0&0\\
  -a_{5,1} & 0 & -a_{5,3} & -a_{4,3} & -2a-b  &&\cdots &0 &\cdots&0 & 0&0 \\
 \boldsymbol{\cdot} & \boldsymbol{\cdot} & \boldsymbol{\cdot} & -a_{5,3} & -a_{4,3}  &\ddots& &\vdots &&\vdots & \vdots&\vdots \\
  \vdots & \vdots & \vdots &\vdots &\vdots  &\ddots& \ddots&\vdots &&\vdots & \vdots&\vdots \\
  -a_{i,1} & 0 & -a_{i,3} & -a_{i-1,3} & -a_{i-2,3}  &\cdots&-a_{4,3}& (3-i)a-b&\cdots&0 & 0&0\\
   \vdots  & \vdots  & \vdots &\vdots &\vdots&&\vdots &\vdots &&\vdots  & \vdots&\vdots \\
 -a_{n-1,1} & 0 & -a_{n-1,3}& -a_{n-2,3}& -a_{n-3,3}&\cdots &-a_{n-i+3,3} &-a_{n-i+2,3}&\cdots&-a_{4,3} &(4-n)a-b& 0\\
 -a_{n,1} & 0 & -a_{n,3}& -a_{n-1,3}& -a_{n-2,3}&\cdots &-a_{n-i+4,3}  &-a_{n-i+3,3}&\cdots&-a_{5,3} &-a_{4,3}& (3-n)a-b
\end{smallmatrix}\right].$$
\allowdisplaybreaks
\begin{itemize}
\item We start with applying the transformation $e^{\prime}_k=e_k,(1\leq k\leq n),e^{\prime}_{n+1}=e_{n+1}-a_{4,3}e_1$
to
remove $a_{4,3}$ in $\r_{e_{n+1}}$ and $-a_{4,3}$ in $\L_{e_{n+1}}$ from the $(i,i-1)^{st}$ positions, where $(4\leq i\leq n),$ 
 but it affects other entries as well,
such as
the entry in the $(2,1)^{st}$ position in $\r_{e_{n+1}}$ and $\L_{e_{n+1}}$ that
we change to $a_{2,1}-a_{4,3}$ and $B_{2,1}-a_{4,3},$ respectively.
It also changes the entry in the $(2,3)^{rd}$ position in $\L_{e_{n+1}}$ to 
$a_{2,3}.$
At the same time, it affects the coefficient in front of $e_2$ in the bracket $[e_{n+1},e_{n+1}],$ which we change back to $a_{2,n+1}$.
\item Then we apply the transformation $e^{\prime}_i=e_i,(1\leq i\leq n),e^{\prime}_{n+1}=e_{n+1}+\sum_{k=3}^{n-1}a_{k+1,1}e_{k}$
to remove $a_{k+1,1}$ in $\r_{e_{n+1}}$ and $-a_{k+1,1}$ in $\L_{e_{n+1}}$ from the entries in the $(k+1,1)^{st}$
positions, where $(3\leq k\leq n-1).$ The transformation changes the entry in the $(2,1)^{st}$ position in
$\r_{e_{n+1}}$ to $2a_{4,1}+a_{2,1}-a_{4,3},$ the entries in the $(2,3)^{rd}$ positions in $\r_{e_{n+1}}$
and $\L_{e_{n+1}}$ to $a_{2,3}+a_{4,1}.$ 
It also affects the coefficient in front of $e_2$ in $[e_{n+1},e_{n+1}],$ which
we rename back by $a_{2,n+1}.$ We assign $2a_{4,1}+a_{2,1}-a_{4,3}:=a_{2,1}$ and $a_{2,3}+a_{4,1}:=a_{2,3}.$ Then
$B_{2,1}-a_{4,3}:=\frac{(2b-a)a_{2,1}+2(a-b)a_{2,3}}{a}.$
\item Applying the transformation $e^{\prime}_j=e_j,(1\leq j\leq n),e^{\prime}_{n+1}=e_{n+1}+\frac{a_{2,n+1}}{2b}e_2,$
we 
remove the coefficient $a_{2,n+1}$ in front of $e_2$ in $[e_{n+1},e_{n+1}].$
\end{itemize}
\item[(2)] The right (not a derivation) and left (a derivation) multiplication operators restricted to the nilradical are as follows:
$$\r_{e_{n+1}}=\left[\begin{smallmatrix}
 a & 0 & 0 & 0&0&&\cdots &0& \cdots&0 & 0&0 \\
  a_{2,1} & 0 & a_{2,3}& (n-4)a &0& & \cdots &0&\cdots  & 0& 0&0\\
  (2-n)a & 0 & (3-n)a & 0 & 0& &\cdots &0&\cdots &0& 0&0 \\
  a_{4,1} & 0 &  a_{4,3} & (4-n)a &0 & &\cdots&0&\cdots &0 & 0&0\\
  a_{5,1} & 0 & a_{5,3} & a_{4,3} & (5-n)a  &&\cdots &0 &\cdots&0 & 0&0 \\
 \boldsymbol{\cdot} & \boldsymbol{\cdot} & \boldsymbol{\cdot} & a_{5,3} & a_{4,3}  &\ddots& &\vdots &&\vdots & \vdots&\vdots \\
  \vdots & \vdots & \vdots &\vdots &\vdots  &\ddots& \ddots&\vdots &&\vdots & \vdots&\vdots \\
  a_{i,1} & 0 & a_{i,3} & a_{i-1,3} & a_{i-2,3}  &\cdots&a_{4,3}& (i-n)a&\cdots&0 & 0&0\\
   \vdots  & \vdots  & \vdots &\vdots &\vdots&&\vdots &\vdots &&\vdots  & \vdots&\vdots \\
 a_{n-1,1} & 0 & a_{n-1,3}& a_{n-2,3}& a_{n-3,3}&\cdots &a_{n-i+3,3} &a_{n-i+2,3}&\cdots&a_{4,3} &-a& 0\\
 a_{n,1} & 0 & a_{n,3}& a_{n-1,3}& a_{n-2,3}&\cdots &a_{n-i+4,3}  &a_{n-i+3,3}&\cdots&a_{5,3} &a_{4,3}& 0
\end{smallmatrix}\right],$$
$$\L_{e_{n+1}}=\left[\begin{smallmatrix}
 -a & 0 & 0 & 0&0&&\cdots &0& \cdots&0 & 0&0 \\
  B_{2,1} & (2n-6)a & a_{2,3}+2a_{4,3}& (n-2)a &0& & \cdots &0&\cdots  & 0& 0&0\\
  (n-2)a & 0 & (n-3)a & 0 & 0& &\cdots &0&\cdots &0& 0&0 \\
  -a_{4,1} & 0 &  -a_{4,3} & (n-4)a &0 & &\cdots&0&\cdots &0 & 0&0\\
  -a_{5,1} & 0 & -a_{5,3} & -a_{4,3} & (n-5)a  &&\cdots &0 &\cdots&0 & 0&0 \\
 \boldsymbol{\cdot} & \boldsymbol{\cdot} & \boldsymbol{\cdot} & -a_{5,3} & -a_{4,3}  &\ddots& &\vdots &&\vdots & \vdots&\vdots \\
  \vdots & \vdots & \vdots &\vdots &\vdots  &\ddots& \ddots&\vdots &&\vdots & \vdots&\vdots \\
  -a_{i,1} & 0 & -a_{i,3} & -a_{i-1,3} & -a_{i-2,3}  &\cdots&-a_{4,3}& (n-i)a&\cdots&0 & 0&0\\
   \vdots  & \vdots  & \vdots &\vdots &\vdots&&\vdots &\vdots &&\vdots  & \vdots&\vdots \\
 -a_{n-1,1} & 0 & -a_{n-1,3}& -a_{n-2,3}& -a_{n-3,3}&\cdots &-a_{n-i+3,3} &-a_{n-i+2,3}&\cdots&-a_{4,3} &a& 0\\
 -a_{n,1} & 0 & -a_{n,3}& -a_{n-1,3}& -a_{n-2,3}&\cdots &-a_{n-i+4,3}  &-a_{n-i+3,3}&\cdots&-a_{5,3} &-a_{4,3}&0
\end{smallmatrix}\right].$$
\begin{itemize}
\item We apply the transformation $e^{\prime}_k=e_k,(1\leq k\leq n),e^{\prime}_{n+1}=e_{n+1}-a_{4,3}e_1$ to
remove $a_{4,3}$ in $\r_{e_{n+1}}$ and $-a_{4,3}$ in $\L_{e_{n+1}}$ from the $(i,i-1)^{st}$ positions, where $(4\leq i\leq n),$ 
 but it affects other entries as well,
such as
the entry in the $(2,1)^{st}$ position in $\r_{e_{n+1}}$ and $\L_{e_{n+1}}$ that
we change to $a_{2,1}-a_{4,3}$ and $B_{2,1}-a_{4,3},$ respectively.
It also changes the entry in the $(2,3)^{rd}$ position in $\L_{e_{n+1}}$ to 
$a_{2,3}.$
At the same time, it affects the coefficient in front of $e_2$ in the bracket $[e_{n+1},e_{n+1}],$ which we change back to $a_{2,n+1}$.

\item Then we apply the transformation $e^{\prime}_i=e_i,(1\leq i\leq n),e^{\prime}_{n+1}=e_{n+1}+\sum_{k=3}^{n-1}a_{k+1,1}e_{k}$
to remove $a_{k+1,1}$ in $\r_{e_{n+1}}$ and $-a_{k+1,1}$ in $\L_{e_{n+1}}$ from the entries in the $(k+1,1)^{st}$
positions, where $(3\leq k\leq n-1).$ It changes the entry in the $(2,1)^{st}$ position in
$\r_{e_{n+1}}$ to $2a_{4,1}+a_{2,1}-a_{4,3},$ the entries in the $(2,3)^{rd}$ positions in $\r_{e_{n+1}}$
and $\L_{e_{n+1}}$ to $a_{2,3}+a_{4,1}.$ 
It also affects the coefficient in front of $e_2$ in $[e_{n+1},e_{n+1}],$ which
we rename back by $a_{2,n+1}.$ We assign $2a_{4,1}+a_{2,1}-a_{4,3}:=a_{2,1}$ and $a_{2,3}+a_{4,1}:=a_{2,3}.$ Then
$B_{2,1}-a_{4,3}:=(5-2n)a_{2,1}+2(n-2)a_{2,3}.$
\item Finally the transformation $e^{\prime}_j=e_j,(1\leq j\leq n),e^{\prime}_{n+1}=e_{n+1}-\frac{a_{2,n+1}}{2(n-3)a}e_2$ 
removes the coefficient $a_{2,n+1}$ in front of $e_2$ in $[e_{n+1},e_{n+1}].$
\end{itemize}
\item[(3)] The right (not a derivation) and left (a derivation) multiplication operators 
restricted to the nilradical are given below:
$$\r_{e_{n+1}}=\left[\begin{smallmatrix}
 0 & 0 & 0 & 0&0&&\cdots &0& \cdots&0 & 0&0 \\
  a_{2,3}+a_{4,3}-a_{4,1} & 0 & a_{2,3}& -b &0& & \cdots &0&\cdots  & 0& 0&0\\
  b & 0 & b & 0 & 0& &\cdots &0&\cdots &0& 0&0 \\
  a_{4,1} & 0 &  a_{4,3} & b &0 & &\cdots&0&\cdots &0 & 0&0\\
  a_{5,1} & 0 & a_{5,3} & a_{4,3} & b  &&\cdots &0 &\cdots&0 & 0&0 \\
 \boldsymbol{\cdot} & \boldsymbol{\cdot} & \boldsymbol{\cdot} & a_{5,3} & a_{4,3}  &\ddots& &\vdots &&\vdots & \vdots&\vdots \\
  \vdots & \vdots & \vdots &\vdots &\vdots  &\ddots& \ddots&\vdots &&\vdots & \vdots&\vdots \\
  a_{i,1} & 0 & a_{i,3} & a_{i-1,3} & a_{i-2,3}  &\cdots&a_{4,3}& b&\cdots&0 & 0&0\\
   \vdots  & \vdots  & \vdots &\vdots &\vdots&&\vdots &\vdots &&\vdots  & \vdots&\vdots \\
 a_{n-1,1} & 0 & a_{n-1,3}& a_{n-2,3}& a_{n-3,3}&\cdots &a_{n-i+3,3} &a_{n-i+2,3}&\cdots&a_{4,3} &b& 0\\
 a_{n,1} & 0 & a_{n,3}& a_{n-1,3}& a_{n-2,3}&\cdots &a_{n-i+4,3}  &a_{n-i+3,3}&\cdots&a_{5,3} &a_{4,3}& b
\end{smallmatrix}\right],$$
$$\L_{e_{n+1}}=\left[\begin{smallmatrix}
 0 & 0 & 0 & 0&0&&\cdots &0& \cdots&0 & 0&0 \\
  b_{2,1} & -2b & a_{2,3}+2a_{4,3}&-b &0& & \cdots &0&\cdots  & 0& 0&0\\
  -b & 0 & -b & 0 & 0& &\cdots &0&\cdots &0& 0&0 \\
  -a_{4,1} & 0 &  -a_{4,3} & -b &0 & &\cdots&0&\cdots &0 & 0&0\\
  -a_{5,1} & 0 & -a_{5,3} & -a_{4,3} & -b  &&\cdots &0 &\cdots&0 & 0&0 \\
 \boldsymbol{\cdot} & \boldsymbol{\cdot} & \boldsymbol{\cdot} & -a_{5,3} & -a_{4,3}  &\ddots& &\vdots &&\vdots & \vdots&\vdots \\
  \vdots & \vdots & \vdots &\vdots &\vdots  &\ddots& \ddots&\vdots &&\vdots & \vdots&\vdots \\
  -a_{i,1} & 0 & -a_{i,3} & -a_{i-1,3} & -a_{i-2,3}  &\cdots&-a_{4,3}& -b&\cdots&0 & 0&0\\
   \vdots  & \vdots  & \vdots &\vdots &\vdots&&\vdots &\vdots &&\vdots  & \vdots&\vdots \\
 -a_{n-1,1} & 0 & -a_{n-1,3}& -a_{n-2,3}& -a_{n-3,3}&\cdots &-a_{n-i+3,3} &-a_{n-i+2,3}&\cdots&-a_{4,3} &-b& 0\\
 -a_{n,1} & 0 & -a_{n,3}& -a_{n-1,3}& -a_{n-2,3}&\cdots &-a_{n-i+4,3}  &-a_{n-i+3,3}&\cdots&-a_{5,3} &-a_{4,3}&-b
\end{smallmatrix}\right].$$
\begin{itemize}
\item Applying the transformation $e^{\prime}_k=e_k,(1\leq k\leq n),e^{\prime}_{n+1}=e_{n+1}-a_{4,3}e_1,$ we
remove $a_{4,3}$ in $\r_{e_{n+1}}$ and $-a_{4,3}$ in $\L_{e_{n+1}}$ from the $(i,i-1)^{st}$ positions, where $(4\leq i\leq n),$ 
 but the transformation affects other entries,
such as
the entry in the $(2,1)^{st}$ position in $\r_{e_{n+1}}$ and $\L_{e_{n+1}},$
that we change to $a_{2,3}-a_{4,1}$ and $b_{2,1}-a_{4,3},$ respectively.
It also changes the entry in the $(2,3)^{rd}$ position in $\L_{e_{n+1}}$ to 
$a_{2,3}.$
At the same time, it affects the coefficient in front of $e_2$ in the bracket $[e_{n+1},e_{n+1}],$ which we change back to $a_{2,n+1}$.
\item Then we apply the transformation $e^{\prime}_i=e_i,(1\leq i\leq n),e^{\prime}_{n+1}=e_{n+1}+\sum_{k=3}^{n-1}a_{k+1,1}e_{k}$
to remove $a_{k+1,1}$ in $\r_{e_{n+1}}$ and $-a_{k+1,1}$ in $\L_{e_{n+1}}$ from the entries in the $(k+1,1)^{st}$
positions, where $(3\leq k\leq n-1).$ This transformation changes the entry in the $(2,1)^{st}$ position in
$\r_{e_{n+1}}$ as well as the entries in the $(2,3)^{rd}$ positions in $\r_{e_{n+1}}$
and $\L_{e_{n+1}}$ to $a_{2,3}+a_{4,1}.$
It also affects the coefficient in front of $e_2$ in $[e_{n+1},e_{n+1}],$ which
we rename back by $a_{2,n+1}.$ We assign $a_{2,3}+a_{4,1}:=a_{2,3}$ and $b_{2,1}-a_{4,3}:=b_{2,1}.$
\item Finally applying the transformation $e^{\prime}_j=e_j,(1\leq j\leq n),e^{\prime}_{n+1}=e_{n+1}+\frac{a_{2,n+1}}{2b}e_2,$
we 
remove the coefficient $a_{2,n+1}$ in front of $e_2$ in $[e_{n+1},e_{n+1}].$
\end{itemize}
\item[(4)] The right (not a derivation) and  left (a derivation) multiplication operators 
restricted to the nilradical are as follows:
$$\r_{e_{n+1}}=\left[\begin{smallmatrix}
 a & 0 & 0 & 0&0&&\cdots &0& \cdots&0 & 0&0 \\
  a_{2,1} & 0 & a_{2,3}& -a &0& & \cdots &0&\cdots  & 0& 0&0\\
  -a & 0 & 0 & 0 & 0& &\cdots &0&\cdots &0& 0&0 \\
  a_{4,1} & 0 &  a_{4,3} & a &0 & &\cdots&0&\cdots &0 & 0&0\\
  a_{5,1} & 0 & a_{5,3} & a_{4,3} & 2a  &&\cdots &0 &\cdots&0 & 0&0 \\
 \boldsymbol{\cdot} & \boldsymbol{\cdot} & \boldsymbol{\cdot} & a_{5,3} & a_{4,3}  &\ddots& &\vdots &&\vdots & \vdots&\vdots \\
  \vdots & \vdots & \vdots &\vdots &\vdots  &\ddots& \ddots&\vdots &&\vdots & \vdots&\vdots \\
  a_{i,1} & 0 & a_{i,3} & a_{i-1,3} & a_{i-2,3}  &\cdots&a_{4,3}& (i-3)a&\cdots&0 & 0&0\\
   \vdots  & \vdots  & \vdots &\vdots &\vdots&&\vdots &\vdots &&\vdots  & \vdots&\vdots \\
 a_{n-1,1} & 0 & a_{n-1,3}& a_{n-2,3}& a_{n-3,3}&\cdots &a_{n-i+3,3} &a_{n-i+2,3}&\cdots&a_{4,3} &(n-4)a& 0\\
 a_{n,1} & 0 & a_{n,3}& a_{n-1,3}& a_{n-2,3}&\cdots &a_{n-i+4,3}  &a_{n-i+3,3}&\cdots&a_{5,3} &a_{4,3}&(n-3)a
\end{smallmatrix}\right],$$
$$\L_{e_{n+1}}=\left[\begin{smallmatrix}
 -a & 0 & 0 & 0&0&&\cdots &0& \cdots&0 & 0&0 \\
  a_{2,3}-a_{2,1}+b_{2,3} & 0 & b_{2,3}& a &0& & \cdots &0&\cdots  & 0& 0&0\\
  a & 0 & 0 & 0 & 0& &\cdots &0&\cdots &0& 0&0 \\
  -a_{4,1} & 0 &  -a_{4,3} & -a &0 & &\cdots&0&\cdots &0 & 0&0\\
  -a_{5,1} & 0 & -a_{5,3} & -a_{4,3} & -2a  &&\cdots &0 &\cdots&0 & 0&0 \\
 \boldsymbol{\cdot} & \boldsymbol{\cdot} & \boldsymbol{\cdot} & -a_{5,3} & -a_{4,3}  &\ddots& &\vdots &&\vdots & \vdots&\vdots \\
  \vdots & \vdots & \vdots &\vdots &\vdots  &\ddots& \ddots&\vdots &&\vdots & \vdots&\vdots \\
  -a_{i,1} & 0 & -a_{i,3} & -a_{i-1,3} & -a_{i-2,3}  &\cdots&-a_{4,3}& (3-i)a&\cdots&0 & 0&0\\
   \vdots  & \vdots  & \vdots &\vdots &\vdots&&\vdots &\vdots &&\vdots  & \vdots&\vdots \\
 -a_{n-1,1} & 0 & -a_{n-1,3}& -a_{n-2,3}& -a_{n-3,3}&\cdots &-a_{n-i+3,3} &-a_{n-i+2,3}&\cdots&-a_{4,3} &(4-n)a& 0\\
 -a_{n,1} & 0 & -a_{n,3}& -a_{n-1,3}& -a_{n-2,3}&\cdots &-a_{n-i+4,3}  &-a_{n-i+3,3}&\cdots&-a_{5,3} &-a_{4,3}& (3-n)a
\end{smallmatrix}\right].$$ 
\begin{itemize}
\item We continue with the transformation $e^{\prime}_k=e_k,(1\leq k\leq n),e^{\prime}_{n+1}=e_{n+1}-a_{4,3}e_1$ to
remove $a_{4,3}$ in $\r_{e_{n+1}}$ and $-a_{4,3}$ in $\L_{e_{n+1}}$ from the $(i,i-1)^{st}$ positions, where $(4\leq i\leq n),$ 
 but other entries are affected as well,
such as
the entry in the $(2,1)^{st}$ position in $\r_{e_{n+1}}$ and $\L_{e_{n+1}},$
that we change to $a_{2,1}-a_{4,3}$ and $a_{2,3}-a_{2,1}-a_{4,3}+b_{2,3},$ respectively.
The transformation also changes the entry in the $(2,3)^{rd}$ position in $\L_{e_{n+1}}$ to 
$b_{2,3}-2a_{4,3}$
and affects the coefficient in front of $e_2$ in the bracket $[e_{n+1},e_{n+1}],$ which we change back to $a_{2,n+1}$.
\item Applying the transformation $e^{\prime}_i=e_i,(1\leq i\leq n),e^{\prime}_{n+1}=e_{n+1}+\sum_{k=3}^{n-1}a_{k+1,1}e_{k},$
we remove $a_{k+1,1}$ in $\r_{e_{n+1}}$ and $-a_{k+1,1}$ in $\L_{e_{n+1}}$ from the entries in the $(k+1,1)^{st}$
positions, where $(3\leq k\leq n-1).$ This transformation changes the entry in the $(2,1)^{st}$ position in
$\r_{e_{n+1}}$ to $2a_{4,1}+a_{2,1}-a_{4,3},$ the entries in the $(2,3)^{rd}$ positions in $\r_{e_{n+1}}$
and $\L_{e_{n+1}}$ to $a_{2,3}+a_{4,1}$ and $a_{4,1}+b_{2,3}-2a_{4,3},$ respectively. 
It also affects the coefficient in front of $e_2$ in $[e_{n+1},e_{n+1}],$ that
we rename back by $a_{2,n+1}.$ We assign $2a_{4,1}+a_{2,1}-a_{4,3}:=a_{2,1},a_{2,3}+a_{4,1}:=a_{2,3}$
and $a_{4,1}+b_{2,3}-2a_{4,3}:=b_{2,3}.$ 
\end{itemize}
\item[(5)] We apply the transformation
$e^{\prime}_i=e_i,(1\leq i\leq 4),e^{\prime}_5=e_5-A_{4,3}e_1-
\frac{1}{2(b+c)}(B_{2,1}A_{4,3}+a_{2,1}A_{4,3}+2a_{4,1}A_{4,3}-A^2_{4,3}-a_{2,3}a_{4,1}-a^2_{4,1}-a_{4,1}b_{2,3}-a_{2,5})e_2+a_{4,1}e_3.$
Then we assign $a_{2,3}+a_{4,1}:=a_{2,3}$ and $b_{2,3}-2a_{2,1}-3a_{4,1}:=b_{2,3}.$
\item[(6)] One applies the transformation
$e^{\prime}_i=e_i,(1\leq i\leq 4),e^{\prime}_5=e_5-A_{4,3}e_1+
\frac{1}{2(a-c)}(B_{2,1}A_{4,3}+a_{2,1}A_{4,3}+2a_{4,1}A_{4,3}-A^2_{4,3}-a_{2,3}a_{4,1}-a^2_{4,1}-a_{4,1}b_{2,3}-a_{2,5})e_2+a_{4,1}e_3.$
Then we assign $a_{2,3}+a_{4,1}:=a_{2,3}$ and $b_{2,3}-2a_{2,1}-3a_{4,1}:=b_{2,3}.$
\item[(7)] We apply the transformation
$e^{\prime}_i=e_i,(1\leq i\leq 4),e^{\prime}_5=e_5-a_{4,3}e_1+a_{4,1}e_3$ and rename the coefficient in front of $e_2$
in $[e_5,e_5]$ back by $a_{2,5}.$ Then we assign $a_{2,1}+2a_{4,1}-a_{4,3}:=a_{2,1},a_{2,3}+a_{4,1}:=a_{2,3}$ and $b_{2,3}+a_{4,1}-2a_{4,3}:=b_{2,3}.$
\item[(8)] The transformation is as follows:
$e^{\prime}_i=e_i,(1\leq i\leq 4),e^{\prime}_5=e_5-a_{4,3}e_1+
\frac{1}{2c}(2a_{2,1}a_{4,1}-4a_{2,1}a_{4,3}+3a^2_{4,1}-6a_{4,1}a_{4,3}-a^2_{4,3}+2a_{4,3}b_{2,3}+a_{2,5})e_2+a_{4,1}e_3.$
We assign $a_{2,1}+2a_{4,1}-a_{4,3}:=a_{2,1}$ and $b_{2,3}+a_{4,1}-2a_{4,3}:=b_{2,3}.$
 \item[(9)] The transformation is
 $e^{\prime}_i=e_i,(1\leq i\leq 4),$
 $e^{\prime}_5=e_5-A_{4,3}e_1+
\frac{1}{4(b+c)}(2A^2_{4,3}-a_{2,3}A_{4,3}-2a_{4,1}A_{4,3}-2b_{2,1}A_{4,3}-b_{2,3}A_{4,3}+2a_{2,3}a_{4,1}+2a^2_{4,1}+2a_{4,1}b_{2,3}+2a_{2,5})e_2+a_{4,1}e_3.$
We assign $a_{2,3}+a_{4,1}:=a_{2,3}$ and $b_{2,3}+a_{4,1}-2b_{2,1}:=b_{2,3}.$
\item[(10)] We apply the transformation
$e^{\prime}_i=e_i,(1\leq i\leq 4),e^{\prime}_5=e_5-a_{4,3}e_1+a_{4,1}e_3$ and rename the coefficient in front of $e_2$
in $[e_5,e_5]$ back by $a_{2,5}.$ We assign $a_{2,3}+a_{4,1}:=a_{2,3},b_{2,1}-a_{4,3}:=b_{2,1}$
and $b_{2,3}+a_{4,1}-2a_{4,3}:=b_{2,3}.$
\end{enumerate}
\end{proof}
\allowdisplaybreaks
 \begin{theorem}\label{LL4(Change of Basis)} There are eight solvable
indecomposable left Leibniz algebras up to isomorphism with a codimension one nilradical
$\mathcal{L}^4,(n\geq4),$ which are given below:
\begin{equation}
\begin{array}{l}
\displaystyle \nonumber (i)\,\,\, l_{n+1,1}: [e_1,e_{n+1}]=e_1+(a-1)e_3,
[e_3,e_{n+1}]=ae_3,[e_4,e_{n+1}]=(a+1)(e_4-e_2),\\
\displaystyle[e_{i},e_{n+1}]=\left(a+i-3\right)e_{i},[e_{n+1},e_1]=-e_1-(a-1)e_3,[e_{n+1},e_2]=-2ae_2, [e_{n+1},e_3]=-ae_3,\\
\displaystyle [e_{n+1},e_4]=(1-a)e_2-(a+1)e_4, [e_{n+1},e_i]=\left(3-i-a\right)e_i,(5\leq i\leq n),\\
\displaystyle  \nonumber(ii)\,\,\, l_{n+1,2}:[e_1,e_{n+1}]=e_1+(2-n)e_3,
[e_3,e_{n+1}]=(3-n)e_3,[e_4,e_{n+1}]=(n-4)(e_2-e_4),\\
\displaystyle[e_{i},e_{n+1}]=\left(i-n\right)e_{i},[e_{n+1},e_{n+1}]=e_n,[e_{n+1},e_1]=-e_1+(n-2)e_3,[e_{n+1},e_2]=(2n-6)e_2,\\
\displaystyle [e_{n+1},e_3]=(n-3)e_3,[e_{n+1},e_4]=(n-2)e_2+(n-4)e_4,[e_{n+1},e_i]=\left(n-i\right)e_i,\\
\displaystyle (5\leq i\leq n),\\
\displaystyle (iii)\,\,\,l_{n+1,3}:
 [e_1,e_{n+1}]=e_3,[e_4,e_{n+1}]=-e_2,[e_{i},e_{n+1}]=e_{i}+\epsilon e_{i+2}+\sum_{k=i+3}^n{b_{k-i-2}e_k},\\
\displaystyle [e_{n+1},e_1]=-e_3,[e_{n+1},e_2]=-2e_2,[e_{n+1},e_4]=-e_2,[e_{n+1},e_i]=-e_i-\epsilon e_{i+2}-\sum_{k=i+3}^n{b_{k-i-2}e_k},\\
\displaystyle (\epsilon=0,1,3\leq i\leq n),\\
\displaystyle \nonumber(iv)\,\,\,\g_{n+1,4}: [e_1,e_{n+1}]=e_1-e_3,
[e_3,e_{n+1}]=fe_2,[e_4,e_{n+1}]=-e_2+e_4,[e_{i},e_{n+1}]=(i-3)e_{i},\\
\displaystyle 
[e_{n+1},e_{n+1}]=\epsilon e_2,[e_{n+1},e_1]=-e_1+\left(d+f\right)e_2+e_3,[e_{n+1},e_3]=de_2,\\
\displaystyle [e_{n+1},e_4]=e_2-e_4,[e_{n+1},e_i]=(3-i)e_i,(5\leq i\leq n;\epsilon=0,1; if\,\,\epsilon=0,\,\,then\,\,d^2+f^2\neq0),\\
\displaystyle \nonumber(v)\,\,\,l_{5,5}:[e_1,e_{5}]=ae_1+(b-a+1)e_3,
[e_3,e_{5}]=e_1+be_3,[e_{4},e_5]=(a+b)(e_4-e_2),\\
\displaystyle [e_{5},e_1]=-ae_1+(a-b-1)e_3,[e_5,e_{2}]=-2(b+1)e_2,[e_{5},e_3]=-e_1-be_3,\\
\displaystyle [e_5,e_4]=(a-b-2)e_2-(a+b)e_4,(if\,\,b=-1,then\,\,a\neq1),\\
\displaystyle \nonumber(vi)\,\,l_{5,6}:  [e_1,e_{5}]=ae_1+(1-2a)e_3,
[e_3,e_{5}]=e_1-ae_3,[e_{5},e_{5}]=e_4,[e_{5},e_1]=-ae_1+\\
\displaystyle (2a-1)e_3,[e_5,e_{2}]=2(a-1)e_2,[e_{5},e_3]=ae_3-e_1,[e_5,e_4]=2(a-1)e_2,(a\neq1),\\
\displaystyle \nonumber(vii)\,\,\g_{5,7}:[e_1,e_{5}]=a(e_1-e_3),
[e_3,e_{5}]=e_1+fe_2-e_3,[e_4,e_{5}]=(1-a)\left(e_2-e_4\right),\\
\displaystyle [e_{5},e_{5}]=\epsilon e_2,[e_{5},e_1]=-ae_1+\left(d+f\right)e_2+ae_3,[e_{5},e_3]=-e_1+de_2+e_3,\\
\displaystyle [e_{5},e_4]=(a-1)\left(e_2-e_4\right),(\epsilon=0,1;if\,\,\epsilon=0,then\,\,d^2+f^2\neq0;a\neq1),\\
\displaystyle  \nonumber(viii)\,\g_{5,8}: [e_1,e_{5}]=ce_2,
[e_3,e_{5}]=e_1-e_3,[e_4,e_{5}]=e_2-e_4,[e_{5},e_{5}]=\epsilon e_2,[e_{5},e_1]=\left(c+d\right)e_2,\\
\displaystyle [e_{5},e_3]=-e_1+\left(d+2c\right)e_2+e_3,[e_{5},e_4]=e_4-e_2,(c\neq0,\epsilon=0,1).
\end{array} 
\end{equation} 
\end{theorem}
\vskip 20pt    
\begin{proof}
One applies the change of basis transformations keeping the nilradical $\mathcal{L}^4$ given in $(\ref{L4})$ unchanged.
\begin{enumerate}[noitemsep, topsep=1pt]
\allowdisplaybreaks
\item[(1)] We have the right (not a derivation) and the left (a derivation) multiplication operators restricted to the nilradical are as follows: 

$\r_{e_{n+1}}=\left[\begin{smallmatrix}
 a & 0 & 0 & 0&0&0&\cdots && 0&0 & 0\\
  a_{2,1} & 0 & a_{2,3}& -a-b &0&0 & \cdots &&0  & 0& 0\\
  b-a & 0 & b & 0 & 0&0 &\cdots &&0 &0& 0\\
  0 & 0 &  0 & a+b &0 &0 &\cdots&&0 &0 & 0\\
 0 & 0 & a_{5,3} & 0 & 2a+b  &0&\cdots & &0&0 & 0\\
  0 & 0 &\boldsymbol{\cdot} & a_{5,3} & 0 &3a+b&\cdots & &0&0 & 0\\
    0 & 0 &\boldsymbol{\cdot} & \boldsymbol{\cdot} & \ddots &0&\ddots &&\vdots&\vdots &\vdots\\
  \vdots & \vdots & \vdots &\vdots &  &\ddots&\ddots &\ddots &\vdots&\vdots & \vdots\\
 0 & 0 & a_{n-2,3}& a_{n-3,3}& \cdots&\cdots &a_{5,3}&0&(n-5)a+b &0& 0\\
0 & 0 & a_{n-1,3}& a_{n-2,3}& \cdots&\cdots &\boldsymbol{\cdot}&a_{5,3}&0 &(n-4)a+b& 0\\
 0 & 0 & a_{n,3}& a_{n-1,3}& \cdots&\cdots &\boldsymbol{\cdot}&\boldsymbol{\cdot}&a_{5,3} &0& (n-3)a+b
\end{smallmatrix}\right],$

$\L_{e_{n+1}}=\left[\begin{smallmatrix}
 -a & 0 & 0 & 0&0&0&\cdots && 0&0 & 0\\
  \mathcal{B}_{2,1} & -2b & a_{2,3}& a-b &0&0 & \cdots &&0  & 0& 0\\
  a-b & 0 & -b & 0 & 0&0 &\cdots &&0 &0& 0\\
  0 & 0 &  0 & -a-b &0 &0 &\cdots&&0 &0 & 0\\
 0 & 0 & -a_{5,3} & 0 & -2a-b  &0&\cdots & &0&0 & 0\\
  0 & 0 &\boldsymbol{\cdot} & -a_{5,3} & 0 &-3a-b&\cdots & &0&0 & 0\\
    0 & 0 &\boldsymbol{\cdot} & \boldsymbol{\cdot} & \ddots &0&\ddots &&\vdots&\vdots &\vdots\\
  \vdots & \vdots & \vdots &\vdots &  &\ddots&\ddots &\ddots &\vdots&\vdots & \vdots\\
 0 & 0 & -a_{n-2,3}& -a_{n-3,3}& \cdots&\cdots &-a_{5,3}&0&(5-n)a-b &0& 0\\
0 & 0 & -a_{n-1,3}& -a_{n-2,3}& \cdots&\cdots &\boldsymbol{\cdot}&-a_{5,3}&0 &(4-n)a-b& 0\\
 0 & 0 & -a_{n,3}& -a_{n-1,3}& \cdots&\cdots &\boldsymbol{\cdot}&\boldsymbol{\cdot}&-a_{5,3} &0& (3-n)a-b
\end{smallmatrix}\right].$
\begin{itemize}[noitemsep, topsep=0pt]
\allowdisplaybreaks 
\item We apply the transformation $e^{\prime}_1=e_1,e^{\prime}_2=e_2,e^{\prime}_i=e_i-\frac{a_{k-i+3,3}}{(k-i)a}e_k,(3\leq i\leq n-2,
i+2\leq k\leq n,n\geq5),e^{\prime}_{j}=e_{j},(n-1\leq j\leq n+1),$ where $k$ is fixed, renaming all the affected entries back.
This transformation removes $a_{5,3},a_{6,3},...,a_{n,3}$ in $\r_{e_{n+1}}$ and $-a_{5,3},-a_{6,3},...,-a_{n,3}$ in $\L_{e_{n+1}}.$
Besides it introduces the entries in the $(5,1)^{st},(6,1)^{st},...,(n,1)^{st}$ positions in $\r_{e_{n+1}}$ and $\L_{e_{n+1}},$ 
which we set to be $a_{5,1},a_{6,1},...,a_{n,1}$ and $-a_{5,1},-a_{6,1},...,-a_{n,1},$ respectively.

\noindent $(I)$ Suppose $b\neq\frac{a}{2}.$
\item Applying the transformation $e^{\prime}_1=e_1+\frac{1}{2b-a}\left(\mathcal{B}_{2,1}+\frac{(b-a)a_{2,3}}{b}\right)e_2,e^{\prime}_2=e_2,
e^{\prime}_3=e_3+\frac{a_{2,3}}{b}e_2,e^{\prime}_{i}=e_{i},
(4\leq i\leq n,n\geq4),e^{\prime}_{n+1}=e_{n+1}-a_{5,1}e_2+\sum_{k=4}^{n-1}a_{k+1,1}e_{k},$ we
remove $a_{2,1}$ and $\mathcal{B}_{2,1}$ from the $(2,1)^{st}$ positions in $\r_{e_{n+1}}$ and $\L_{e_{n+1}},$
respectively. This transformation also removes $a_{2,3}$
from the $(2,3)^{rd}$ positions in $\r_{e_{n+1}}$ and $\L_{e_{n+1}}$ as well as 
$a_{k+1,1}$ in $\r_{e_{n+1}}$ and $-a_{k+1,1}$ in $\L_{e_{n+1}}$ from the entries in the $(k+1,1)^{st}$
positions, where $(4\leq k\leq n-1)$.
(For $n=4,$ we direct to the Remark \ref{Remark{a_{5,1}}}.)

\item Then we scale $a$ to unity applying the transformation $e^{\prime}_i=e_i,(1\leq i\leq n,n\geq4),e^{\prime}_{n+1}=\frac{e_{n+1}}{a}.$ 
Renaming $\frac{b}{a}$ by $b,$ we obtain a continuous family of Leibniz algebras:
\begin{equation}
\left\{
\begin{array}{l}
\displaystyle  [e_1,e_{n+1}]=e_1+(b-1)e_3,
[e_3,e_{n+1}]=be_3,[e_4,e_{n+1}]=(b+1)(e_4-e_2),\\
\displaystyle[e_{i},e_{n+1}]=\left(b+i-3\right)e_{i},[e_{n+1},e_1]=-e_1-(b-1)e_3,\\
\displaystyle [e_{n+1},e_2]=-2be_2,[e_{n+1},e_3]=-be_3,[e_{n+1},e_4]=(1-b)e_2-(b+1)e_4, \\
\displaystyle [e_{n+1},e_i]=\left(3-i-b\right)e_i,(5\leq i\leq n),(b\neq0,b\neq\frac{1}{2},b\neq3-n).
\end{array} 
\right.
\label{l_{n+1,1}}
\end{equation} 
\noindent $(II)$ Suppose $b:=\frac{a}{2}.$ We have that $\mathcal{B}_{2,1}=a_{2,3}.$
\item We apply the transformation $e^{\prime}_1=e_1+\frac{a_{2,1}+a_{2,3}}{a}e_2,e^{\prime}_2=e_2,
e^{\prime}_3=e_3+\frac{2a_{2,3}}{a}e_2,e^{\prime}_{i}=e_{i},
(4\leq i\leq n,n\geq4),e^{\prime}_{n+1}=e_{n+1}-a_{5,1}e_2+\sum_{k=4}^{n-1}a_{k+1,1}e_{k}$
to remove $a_{2,3}$ from the $(2,1)^{st},(2,3)^{rd}$ positions in $\L_{e_{n+1}}$
and from the $(2,3)^{rd}$ position in $\r_{e_{n+1}}.$
This transformation also removes $a_{2,1}$ from the $(2,1)^{st}$ position in $\r_{e_{n+1}}$ as well as
$a_{k+1,1}$ and $-a_{k+1,1}$ from the entries in the $(k+1,1)^{st}$
positions in $\r_{e_{n+1}}$ and $\L_{e_{n+1}},$ respectively, where $(4\leq k\leq n-1)$. (For $n=4,$ we refer to Remark \ref{Remark{a_{5,1}}}.)
\item To scale $a$ to unity, we apply the transformation $e^{\prime}_i=e_i,(1\leq i\leq n,n\geq4),e^{\prime}_{n+1}=\frac{e_{n+1}}{a}$ 
and obtain a limiting case of $(\ref{l_{n+1,1}})$ with $b=\frac{1}{2}$ given below:
\begin{equation}
\left\{
\begin{array}{l}
\displaystyle \nonumber [e_1,e_{n+1}]=e_1-\frac{e_3}{2},
[e_3,e_{n+1}]=\frac{e_3}{2},[e_4,e_{n+1}]=\frac{3}{2}(e_4-e_2),[e_{i},e_{n+1}]=\left(i-\frac{5}{2}\right)e_{i},\\
\displaystyle[e_{n+1},e_1]=-e_1+\frac{e_3}{2},[e_{n+1},e_2]=-e_2,[e_{n+1},e_3]=-\frac{e_3}{2},[e_{n+1},e_4]=\frac{e_2}{2}-\frac{3e_4}{2},\\
\displaystyle [e_{n+1},e_i]=\left(\frac{5}{2}-i\right)e_i,(5\leq i\leq n).
\end{array} 
\right.
\end{equation} 
\end{itemize}
\item[(2)] We have the right (not a derivation) and the left (a derivation) multiplication operators restricted to the nilradical are as follows:
$$\r_{e_{n+1}}\left[\begin{smallmatrix}
 a & 0 & 0 & 0&0&0&\cdots && 0&0 & 0\\
 a_{2,1} & 0 & a_{2,3}& (n-4)a &0&0 & \cdots &&0  & 0& 0\\
  (2-n)a & 0 & (3-n)a & 0 & 0&0 &\cdots &&0 &0& 0\\
  0 & 0 &  0 & (4-n)a &0 &0 &\cdots&&0 &0 & 0\\
 0 & 0 & a_{5,3} & 0 & (5-n)a  &0&\cdots & &0&0 & 0\\
  0 & 0 &\boldsymbol{\cdot} & a_{5,3} & 0 &(6-n)a&\cdots & &0&0 & 0\\
    0 & 0 &\boldsymbol{\cdot} & \boldsymbol{\cdot} & \ddots &0&\ddots &&\vdots&\vdots &\vdots\\
  \vdots & \vdots & \vdots &\vdots &  &\ddots&\ddots &\ddots &\vdots&\vdots & \vdots\\
 0 & 0 & a_{n-2,3}& a_{n-3,3}& \cdots&\cdots &a_{5,3}&0&-2a &0& 0\\
0 & 0 & a_{n-1,3}& a_{n-2,3}& \cdots&\cdots &\boldsymbol{\cdot}&a_{5,3}&0 &-a& 0\\
 0 & 0 & a_{n,3}& a_{n-1,3}& \cdots&\cdots &\boldsymbol{\cdot}&\boldsymbol{\cdot}&a_{5,3} &0&0
\end{smallmatrix}\right],$$
 $$\L_{e_{n+1}}=\left[\begin{smallmatrix}
 -a & 0 & 0 & 0&0&0&\cdots && 0&0 & 0\\
  \mathcal{B}_{2,1} & (2n-6)a & a_{2,3}& (n-2)a &0&0 & \cdots &&0  & 0& 0\\
  (n-2)a & 0 & (n-3)a & 0 & 0&0 &\cdots &&0 &0& 0\\
  0 & 0 &  0 & (n-4)a &0 &0 &\cdots&&0 &0 & 0\\
 0 & 0 & -a_{5,3} & 0 & (n-5)a  &0&\cdots & &0&0 & 0\\
  0 & 0 &\boldsymbol{\cdot} & -a_{5,3} & 0 &(n-6)a&\cdots & &0&0 & 0\\
    0 & 0 &\boldsymbol{\cdot} & \boldsymbol{\cdot} & \ddots &0&\ddots &&\vdots&\vdots &\vdots\\
  \vdots & \vdots & \vdots &\vdots &  &\ddots&\ddots &\ddots &\vdots&\vdots & \vdots\\
 0 & 0 & -a_{n-2,3}& -a_{n-3,3}& \cdots&\cdots &-a_{5,3}&0&2a &0& 0\\
0 & 0 & -a_{n-1,3}& -a_{n-2,3}& \cdots&\cdots &\boldsymbol{\cdot}&-a_{5,3}&0 &a& 0\\
 0 & 0 & -a_{n,3}& -a_{n-1,3}& \cdots&\cdots &\boldsymbol{\cdot}&\boldsymbol{\cdot}&-a_{5,3} &0& 0
\end{smallmatrix}\right].$$

\begin{itemize}[noitemsep, topsep=0pt]
\allowdisplaybreaks 
\item We apply the transformation $e^{\prime}_1=e_1,e^{\prime}_2=e_2,e^{\prime}_i=e_i-\frac{a_{k-i+3,3}}{(k-i)a}e_k,(3\leq i\leq n-2,
i+2\leq k\leq n,n\geq5),e^{\prime}_{j}=e_{j},(n-1\leq j\leq n+1),$ where $k$ is fixed, renaming all the affected entries back.
This transformation removes $a_{5,3},a_{6,3},...,a_{n,3}$ in $\r_{e_{n+1}}$ and $-a_{5,3},-a_{6,3},...,-a_{n,3}$ in $\L_{e_{n+1}}.$
Moreover it introduces the entries in the $(5,1)^{st},(6,1)^{st},...,(n,1)^{st}$ positions in $\r_{e_{n+1}}$ and $\L_{e_{n+1}},$ 
which we set to be $a_{5,1},a_{6,1},...,a_{n,1}$ and $-a_{5,1},-a_{6,1},...,-a_{n,1},$ respectively.

\item Applying the transformation $e^{\prime}_1=e_1+\frac{1}{(5-2n)a}\left(\mathcal{B}_{2,1}+\frac{(2-n)a_{2,3}}{3-n}\right)e_2,e^{\prime}_2=e_2,
e^{\prime}_3=e_3+\frac{a_{2,3}}{(3-n)a}e_2,e^{\prime}_{i}=e_{i},
(4\leq i\leq n,n\geq4),e^{\prime}_{n+1}=e_{n+1}-a_{5,1}e_2+\sum_{k=4}^{n-1}a_{k+1,1}e_{k},$ we
remove $a_{2,1}$ and $\mathcal{B}_{2,1}$ from the $(2,1)^{st}$ positions in $\r_{e_{n+1}}$ and $\L_{e_{n+1}},$
respectively. We also remove $a_{2,3}$
from the $(2,3)^{rd}$ positions in $\r_{e_{n+1}}$ and $\L_{e_{n+1}}$ as well as 
$a_{k+1,1}$ in $\r_{e_{n+1}}$ and $-a_{k+1,1}$ in $\L_{e_{n+1}}$ from the entries in the $(k+1,1)^{st}$
positions, where $(4\leq k\leq n-1)$. (See Remark \ref{Remark{a_{5,1}}}.)

\item To scale $a$ to unity, we apply the transformation $e^{\prime}_i=e_i,(1\leq i\leq n),e^{\prime}_{n+1}=\frac{e_{n+1}}{a}$ 
renaming the coefficient $\frac{a_{n,n+1}}{a^2}$ in front of $e_n$ in $[e_{n+1},e_{n+1}]$ back by $a_{n,n+1}.$ We obtain a Leibniz algebra given below:
\begin{equation}
\left\{
\begin{array}{l}
\displaystyle  \nonumber [e_1,e_{n+1}]=e_1+(2-n)e_3,
[e_3,e_{n+1}]=(3-n)e_3,[e_4,e_{n+1}]=(n-4)(e_2-e_4),\\
\displaystyle
[e_{i},e_{n+1}]=\left(i-n\right)e_{i},[e_{n+1},e_{n+1}]=a_{n,n+1}e_n,[e_{n+1},e_1]=-e_1+(n-2)e_3,\\
\displaystyle 
[e_{n+1},e_2]=(2n-6)e_2,[e_{n+1},e_3]=(n-3)e_3,[e_{n+1},e_4]=(n-2)e_2+(n-4)e_4,\\
\displaystyle [e_{n+1},e_i]=\left(n-i\right)e_i,(5\leq i\leq n),
\end{array} 
\right.
\end{equation}
 \end{itemize}
If $a_{n,n+1}=0,$ then we have a limiting case of (\ref{l_{n+1,1}}) with $b=3-n$. If $a_{n,n+1}\neq0,$
 then we scale it to $1.$ It gives us the algebra $l_{n+1,2}.$
 \item[(3)] We have the right (not a derivation) and the left (a derivation) multiplication operators restricted to the nilradical are as follows:
$$\r_{e_{n+1}}=\left[\begin{smallmatrix}
 0 & 0 & 0 & 0&0&0&\cdots && 0&0 & 0\\
  a_{2,3} & 0 & a_{2,3}& -b &0&0 & \cdots &&0  & 0& 0\\
  b & 0 & b & 0 & 0&0 &\cdots &&0 &0& 0\\
  0 & 0 &  0 & b &0 &0 &\cdots&&0 &0 & 0\\
 0 & 0 & a_{5,3} & 0 & b  &0&\cdots & &0&0 & 0\\
  0 & 0 &\boldsymbol{\cdot} & a_{5,3} & 0 &b&\cdots & &0&0 & 0\\
    0 & 0 &\boldsymbol{\cdot} & \boldsymbol{\cdot} & \ddots &0&\ddots &&\vdots&\vdots &\vdots\\
  \vdots & \vdots & \vdots &\vdots &  &\ddots&\ddots &\ddots &\vdots&\vdots & \vdots\\
 0 & 0 & a_{n-2,3}& a_{n-3,3}& \cdots&\cdots &a_{5,3}&0&b &0& 0\\
0 & 0 & a_{n-1,3}& a_{n-2,3}& \cdots&\cdots &\boldsymbol{\cdot}&a_{5,3}&0 &b& 0\\
 0 & 0 & a_{n,3}& a_{n-1,3}& \cdots&\cdots &\boldsymbol{\cdot}&\boldsymbol{\cdot}&a_{5,3} &0& b
\end{smallmatrix}\right],$$
 $$\L_{e_{n+1}}=\left[\begin{smallmatrix}
0 & 0 & 0 & 0&0&0&\cdots && 0&0 & 0\\
  b_{2,1} & -2b & a_{2,3}& -b &0&0 & \cdots &&0  & 0& 0\\
  -b & 0 & -b & 0 & 0&0 &\cdots &&0 &0& 0\\
  0 & 0 &  0 & -b &0 &0 &\cdots&&0 &0 & 0\\
 0 & 0 & -a_{5,3} & 0 & -b  &0&\cdots & &0&0 & 0\\
  0 & 0 &\boldsymbol{\cdot} & -a_{5,3} & 0 &-b&\cdots & &0&0 & 0\\
    0 & 0 &\boldsymbol{\cdot} & \boldsymbol{\cdot} & \ddots &0&\ddots &&\vdots&\vdots &\vdots\\
  \vdots & \vdots & \vdots &\vdots &  &\ddots&\ddots &\ddots &\vdots&\vdots & \vdots\\
 0 & 0 & -a_{n-2,3}& -a_{n-3,3}& \cdots&\cdots &-a_{5,3}&0&-b &0& 0\\
0 & 0 & -a_{n-1,3}& -a_{n-2,3}& \cdots&\cdots &\boldsymbol{\cdot}&-a_{5,3}&0 &-b& 0\\
 0 & 0 & -a_{n,3}& -a_{n-1,3}& \cdots&\cdots &\boldsymbol{\cdot}&\boldsymbol{\cdot}&-a_{5,3} &0&-b
\end{smallmatrix}\right].$$
\begin{itemize}[noitemsep, topsep=0pt]
\allowdisplaybreaks 
\item Applying the transformation $e^{\prime}_1=e_1+\frac{b_{2,1}+a_{2,3}}{2b}e_2,e^{\prime}_2=e_2,
e^{\prime}_3=e_3+\frac{a_{2,3}}{b}e_2,e^{\prime}_{i}=e_{i},
(4\leq i\leq n+1),$
we
remove $b_{2,1}$ from the $(2,1)^{st}$ position in $\L_{e_{n+1}}$ and 
$a_{2,3}$ from the $(2,1)^{st}$ position in $\r_{e_{n+1}}$ and from the $(2,3)^{rd}$ positions in
$\r_{e_{n+1}}$ and $\L_{e_{n+1}}$ keeping other entries unchanged.

\item To scale $b$ to unity, we apply the transformation $e^{\prime}_i=e_i,(1\leq i\leq n),e^{\prime}_{n+1}=\frac{e_{n+1}}{b}.$ 
Then we rename $\frac{a_{5,3}}{b},\frac{a_{6,3}}{b},...,\frac{a_{n,3}}{b}$ by $a_{5,3},a_{6,3},...,a_{n,3},$ respectively.
We obtain a Leibniz algebra
\begin{equation}
\left\{
\begin{array}{l}
\displaystyle  \nonumber [e_1,e_{n+1}]=e_3,[e_4,e_{n+1}]=-e_2,[e_{i},e_{n+1}]=e_{i}+\sum_{k=i+2}^n{a_{k-i+3,3}e_k},[e_{n+1},e_1]=-e_3,\\
\displaystyle [e_{n+1},e_2]=-2e_2,[e_{n+1},e_4]=-e_2, [e_{n+1},e_i]=-e_i-\sum_{k=i+2}^n{a_{k-i+3,3}e_k},(3\leq i\leq n),
\end{array} 
\right.
\end{equation} 
 If $a_{5,3}\neq0,(n\geq5),$ then we scale it  to $1.$ We also rename all the affected entries back
and then we rename $a_{6,3},...,a_{n,3}$ by $b_1,...,b_{n-5},$ respectively.
We combine with the case when $a_{5,3}=0$ and obtain a Leibniz algebra $l_{n+1,3}.$
\begin{remark}
If $n=4,$ then $\epsilon=0.$
\end{remark}
\end{itemize}
\item[(4)] We have the right (not a derivation) and the left (a derivation) multiplication operators restricted to the nilradical are as follows:
$$\r_{e_{n+1}}=\left[\begin{smallmatrix}
 a & 0 & 0 & 0&0&0&\cdots && 0&0 & 0\\
 a_{2,1} & 0 & a_{2,3}& -a &0&0 & \cdots &&0  & 0& 0\\
  -a & 0 & 0 & 0 & 0&0 &\cdots &&0 &0& 0\\
  0 & 0 &  0 & a &0 &0 &\cdots&&0 &0 & 0\\
 0 & 0 & a_{5,3} & 0 & 2a  &0&\cdots & &0&0 & 0\\
  0 & 0 &\boldsymbol{\cdot} & a_{5,3} & 0 &3a&\cdots & &0&0 & 0\\
    0 & 0 &\boldsymbol{\cdot} & \boldsymbol{\cdot} & \ddots &0&\ddots &&\vdots&\vdots &\vdots\\
  \vdots & \vdots & \vdots &\vdots &  &\ddots&\ddots &\ddots &\vdots&\vdots & \vdots\\
 0 & 0 & a_{n-2,3}& a_{n-3,3}& \cdots&\cdots &a_{5,3}&0&(n-5)a &0& 0\\
0 & 0 & a_{n-1,3}& a_{n-2,3}& \cdots&\cdots &\boldsymbol{\cdot}&a_{5,3}&0 &(n-4)a& 0\\
 0 & 0 & a_{n,3}& a_{n-1,3}& \cdots&\cdots &\boldsymbol{\cdot}&\boldsymbol{\cdot}&a_{5,3} &0& (n-3)a
\end{smallmatrix}\right],$$
$$\L_{e_{n+1}}=\left[\begin{smallmatrix}
 -a & 0 & 0 & 0&0&0&\cdots && 0&0 & 0\\
  a_{2,3}-a_{2,1}+b_{2,3} & 0 & b_{2,3}& a &0&0 & \cdots &&0  & 0& 0\\
  a & 0 &0 & 0 & 0&0 &\cdots &&0 &0& 0\\
  0 & 0 &  0 & -a &0 &0 &\cdots&&0 &0 & 0\\
 0 & 0 & -a_{5,3} & 0 & -2a  &0&\cdots & &0&0 & 0\\
  0 & 0 &\boldsymbol{\cdot} & -a_{5,3} & 0 &-3a&\cdots & &0&0 & 0\\
    0 & 0 &\boldsymbol{\cdot} & \boldsymbol{\cdot} & \ddots &0&\ddots &&\vdots&\vdots &\vdots\\
  \vdots & \vdots & \vdots &\vdots &  &\ddots&\ddots &\ddots &\vdots&\vdots & \vdots\\
 0 & 0 & -a_{n-2,3}& -a_{n-3,3}& \cdots&\cdots &-a_{5,3}&0&(5-n)a &0& 0\\
0 & 0 & -a_{n-1,3}& -a_{n-2,3}& \cdots&\cdots &\boldsymbol{\cdot}&-a_{5,3}&0 &(4-n)a& 0\\
 0 & 0 & -a_{n,3}& -a_{n-1,3}& \cdots&\cdots &\boldsymbol{\cdot}&\boldsymbol{\cdot}&-a_{5,3} &0& (3-n)a
\end{smallmatrix}\right].$$
\begin{itemize}[noitemsep, topsep=0pt]
\allowdisplaybreaks 
\item We apply the transformation $e^{\prime}_1=e_1,e^{\prime}_2=e_2,e^{\prime}_i=e_i-\frac{a_{k-i+3,3}}{(k-i)a}e_k,(3\leq i\leq n-2,
i+2\leq k\leq n,n\geq5),e^{\prime}_{j}=e_{j},(n-1\leq j\leq n+1),$ where $k$ is fixed renaming all the affected entries back.
This transformation removes $a_{5,3},a_{6,3},...,a_{n,3}$ in $\r_{e_{n+1}}$ and $-a_{5,3},-a_{6,3},...,-a_{n,3}$ in $\L_{e_{n+1}}.$
Besides it introduces the entries in the $(5,1)^{st},(6,1)^{st},...,(n,1)^{st}$ positions in $\r_{e_{n+1}}$ and $\L_{e_{n+1}},$ 
which we set to be $a_{5,1},a_{6,1},...,a_{n,1}$ and $-a_{5,1},-a_{6,1},...,-a_{n,1},$ respectively.

\item Applying the transformation $e^{\prime}_1=e_1+\frac{a_{2,1}}{a}e_2,e^{\prime}_{i}=e_{i},
(2\leq i\leq n,n\geq4),e^{\prime}_{n+1}=e_{n+1}+\sum_{k=4}^{n-1}a_{k+1,1}e_{k},$ we
remove $a_{2,1}$ from the $(2,1)^{st}$ position in $\r_{e_{n+1}}$.
This transformation changes the entry in the $(2,1)^{st}$ position in $\L_{e_{n+1}}$ to $a_{2,3}+b_{2,3}.$
It also removes
$a_{k+1,1}$ and $-a_{k+1,1}$ from the entries in the $(k+1,1)^{st}$
positions, where $(4\leq k\leq n-1)$ in $\r_{e_{n+1}}$ and $\L_{e_{n+1}},$ respectively.\footnote{Except this transformation it is the same as case $(4)$
for the right Leibniz algebras.}

\item We assign $a_{2,3}:=d$ and $b_{2,3}:=f$
and
then we scale $a$ to unity applying the transformation $e^{\prime}_i=e_i,(1\leq i\leq n,n\geq4),e^{\prime}_{n+1}=\frac{e_{n+1}}{a}.$ 
Renaming $\frac{d}{a},\frac{f}{a}$ and $\frac{a_{2,n+1}}{a^2}$ by $d,f$ and $a_{2,n+1},$ respectively, we obtain a Leibniz algebra, which is right and left at the same time
and a limiting case of (\ref{l_{n+1,1}}) with $b=0,$ when $d=f=a_{2,n+1}=0$:
\begin{equation}
\left\{
\begin{array}{l}
\displaystyle  \nonumber [e_1,e_{n+1}]=e_1-e_3,
[e_3,e_{n+1}]=de_2,[e_4,e_{n+1}]=e_4-e_2,[e_{i},e_{n+1}]=(i-3)e_{i},\\
\displaystyle [e_{n+1},e_{n+1}]=a_{2,n+1}e_2,[e_{n+1},e_1]=-e_1+\left(d+f\right)e_2+e_3,[e_{n+1},e_3]=fe_2,\\
\displaystyle [e_{n+1},e_4]=e_2-e_4,[e_{n+1},e_i]=(3-i)e_i,(5\leq i\leq n).
\end{array} 
\right.
\end{equation} 
Altogether (\ref{l_{n+1,1}}) and all its limiting
cases after replacing $b$ with $a$ give us a Leibniz algebra $l_{n+1,1}.$
It remains to consider a continuous family of Leibniz algebras given below
 and scale any nonzero entries as much as possible.
\begin{equation}
\left\{
\begin{array}{l}
\displaystyle  \nonumber [e_1,e_{n+1}]=e_1-e_3,
[e_3,e_{n+1}]=de_2,[e_4,e_{n+1}]=e_4-e_2,[e_{i},e_{n+1}]=(i-3)e_{i},\\
\displaystyle 
[e_{n+1},e_{n+1}]=a_{2,n+1}e_2,[e_{n+1},e_1]=-e_1+\left(d+f\right)e_2+e_3,[e_{n+1},e_3]=fe_2,\\
\displaystyle [e_{n+1},e_4]=e_2-e_4,[e_{n+1},e_i]=(3-i)e_i,(a_{2,n+1}^2+d^2+f^2\neq0,5\leq i\leq n),
\end{array} 
\right.
\end{equation} 
 If $a_{2,n+1}\neq0,$ then we scale it to $1.$ We also rename all the affected entries back.
Then we combine with the case when $a_{2,n+1}=0$ and obtain a right and left Leibniz algebra $\g_{n+1,4}.$
\end{itemize}
\item[(5)] Applying the transformation
$e^{\prime}_1=e_1+\frac{(2a-b)a_{2,3}-b\cdot b_{2,3}}{2a(b+c)}e_2,e^{\prime}_2=e_2,$
$e^{\prime}_3=e_3+\frac{(2a+c)a_{2,3}+c\cdot b_{2,3}}{2a(b+c)}e_2,$\\
$e^{\prime}_4=e_4,e^{\prime}_5=\frac{e_5}{c}$ and renaming
$\frac{a}{c}$ and $\frac{b}{c}$ by $a$ and $b,$ respectively, we obtain a continuous family of Leibniz algebras given below:
\begin{equation}
\left\{
\begin{array}{l}
\displaystyle  [e_1,e_{5}]=ae_1+(b-a+1)e_3,
[e_3,e_{5}]=e_1+be_3,[e_{4},e_5]=(a+b)(e_4-e_2),\\
\displaystyle [e_{5},e_1]=-ae_1+(a-b-1)e_3,[e_5,e_{2}]=-2(b+1)e_2,[e_{5},e_3]=-e_1-be_3,\\
\displaystyle [e_5,e_4]=(a-b-2)e_2-(a+b)e_4,(b\neq-a,a\neq0,b\neq-1).
\end{array} 
\right.
\label{l_{5,5}}
\end{equation} 
\item[(6)] We
apply the transformation
$e^{\prime}_1=e_1+\frac{3a_{2,3}+b_{2,3}}{2(c-a)}e_2,
e^{\prime}_2=e_2,$
$e^{\prime}_3=e_3+\frac{(2a+c)a_{2,3}+c\cdot b_{2,3}}{2a(c-a)}e_2,$\\
$e^{\prime}_4=e_4,e^{\prime}_5=\frac{e_5}{c}$
and rename
$\frac{a}{c}$ and $\frac{a_{4,5}}{c^2}$ by $a$ and $a_{4,5},$ respectively,
to obtain a Leibniz algebra given below:
\begin{equation}
\left\{
\begin{array}{l}
\displaystyle  [e_1,e_{5}]=ae_1+(1-2a)e_3,
[e_3,e_{5}]=e_1-ae_3,[e_{5},e_{5}]=a_{4,5}e_4,\\
\displaystyle [e_{5},e_1]=-ae_1+(2a-1)e_3,[e_5,e_{2}]=2(a-1)e_2,[e_{5},e_3]=ae_3-e_1,\\
\displaystyle [e_5,e_4]=2(a-1)e_2,(a\neq0,a\neq1),
\end{array} 
\label{g1}
\right.
\end{equation} 
which is a limiting case of $(\ref{l_{5,5}})$ with $b:=-a$ when $a_{4,5}=0.$
If $a_{4,5}\neq0,$ then
we scale it  to $1$
and obtain a continuous family of Leibniz algebras:
\begin{equation}
\left\{
\begin{array}{l}
\displaystyle  [e_1,e_{5}]=ae_1+(1-2a)e_3,
[e_3,e_{5}]=e_1-ae_3,[e_{5},e_{5}]=e_4,\\
\displaystyle [e_{5},e_1]=-ae_1+(2a-1)e_3,[e_5,e_{2}]=2(a-1)e_2,[e_{5},e_3]=ae_3-e_1,\\
\displaystyle [e_5,e_4]=2(a-1)e_2,(a\neq0,a\neq1).
\end{array} 
\right.
\label{l_{5,6}}
\end{equation} 
\item[(7)] We apply the transformation
$e^{\prime}_1=e_1+\frac{a_{2,1}}{a}e_2,e^{\prime}_i=e_i,(2\leq i\leq 5)$
and assign $d:=\frac{a\cdot b_{2,3}+c\cdot a_{2,1}}{a},$
$f:=\frac{a\cdot a_{2,3}-c\cdot a_{2,1}}{a}.$ Then this case becomes the same as case $(7)$ of Theorem \ref{RL4(Change of Basis)}.
\item[(8)] Applying the transformation
$e^{\prime}_1=e_1+\frac{2a_{2,1}-b_{2,3}}{c}e_2,e^{\prime}_2=e_2,
e^{\prime}_3=e_3+\frac{a_{2,1}}{c}e_2,e^{\prime}_4=e_4,e^{\prime}_5=\frac{e_5}{c}$
and renaming $\frac{a_{4,5}}{c^2}$ back by $a_{4,5},$ we obtain a Leibniz algebra:
\begin{equation}
\left\{
\begin{array}{l}
\displaystyle  \nonumber [e_1,e_{5}]=e_3,
[e_3,e_{5}]=e_1,[e_{5},e_{5}]=a_{4,5}e_4,[e_{5},e_1]=-e_3,[e_5,e_{2}]=-2e_2,[e_{5},e_3]=-e_1,\\
\displaystyle [e_5,e_{4}]=-2e_2,
\end{array} 
\right.
\end{equation} 
which is a limiting case of $(\ref{g1})$ with $a=0.$ If $a_{4,5}\neq0,$ then
we scale $a_{4,5}$ to $1$ and obtain a limiting case of $(\ref{l_{5,6}})$
with $a=0.$ Altogether $(\ref{l_{5,6}})$ and all its limiting cases give us the algebra $l_{5,6}.$
\item[(9)] One applies the transformation 
$e^{\prime}_1=e_1+\frac{(3b+4c)a_{2,3}-b\cdot b_{2,3}}{4(b+c)^2}e_2,
e^{\prime}_2=e_2,e^{\prime}_3=e_3+\frac{(4b+5c)a_{2,3}+c\cdot b_{2,3}}{4(b+c)^2}e_2,e^{\prime}_4=e_4,
e^{\prime}_5=\frac{e_5}{c}.$ Renaming $\frac{b}{c}$ by $b,$ we obtain a Leibniz algebra given below:
\begin{equation}
\left\{
\begin{array}{l}
\displaystyle  \nonumber [e_1,e_{5}]=(b+1)e_3,
[e_3,e_{5}]=e_1+be_3,[e_{4},e_5]=b(e_4-e_2),[e_{5},e_1]=(-b-1)e_3,\\
\displaystyle [e_5,e_{2}]=-2(b+1)e_2,[e_{5},e_3]=-e_1-be_3,[e_5,e_{4}]=-(b+2)e_2-be_4,(b\neq0,b\neq-1),
\end{array} 
\right.
\end{equation}
which is a limiting case of $(\ref{l_{5,5}})$ with $a=0.$ Altogether $(\ref{l_{5,5}})$ and all its
limiting cases give us the algebra $l_{5,5}.$ 
\item[(10)] We apply the transformation
$e^{\prime}_1=e_1+\frac{a_{2,3}}{c}e_2,e^{\prime}_i=e_i,(2\leq i\leq 4),e^{\prime}_5=\frac{e_5}{c}$
and rename $\frac{b_{2,1}}{c},\frac{b_{2,3}}{c},\frac{a_{2,3}}{c}$ and $\frac{a_{2,5}}{c^2}$ by
$b_{2,1},b_{2,3},a_{2,3}$ and $a_{2,5},$ respectively. Then we assign $a_{2,1}:=a_{2,3}+b_{2,3}-b_{2,1}$
and this case becomes the same as case $(10)$ of Theorem \ref{RL4(Change of Basis)}.
 \end{enumerate}
\end{proof}
\subsubsection{Codimension two and three solvable extensions of $\mathcal{L}^4$}
The non-zero inner derivations of $\mathcal{L}^4,(n\geq4)$ are
given by
 \[
\L_{e_1}=\left[\begin{smallmatrix}
0&0 & 0 & 0 & \cdots & 0 & 0  & 0 \\
1&0 & 2 & 0 & \cdots & 0 & 0  & 0 \\
0&0 & 0 & 0 & \cdots & 0 & 0 & 0 \\
 0&0 & -1 & 0 & \cdots & 0 & 0 & 0 \\
0& 0 & 0 & -1 & \cdots & 0 & 0 & 0 \\
\vdots& \vdots  & \vdots  & \vdots  & \ddots & \vdots & \vdots & \vdots\\
 0& 0 & 0 & 0&\cdots & -1 & 0 &0\\
  0& 0 & 0 & 0&\cdots & 0 & -1 &0
\end{smallmatrix}\right],\L_{e_3}=\left[\begin{smallmatrix}
0&0 & 0 & 0 & \cdots & 0 \\
0&0 & 1 & 0 & \cdots & 0  \\
0&0 & 0 & 0 & \cdots & 0 \\
 1&0 & 0 & 0 & \cdots & 0\\
0& 0 & 0 & 0 & \cdots & 0\\
\vdots& \vdots  & \vdots  & \vdots  &  & \vdots\\
  0& 0 & 0 & 0&\cdots & 0
\end{smallmatrix}\right],
\L_{e_i}=E_{i+1,1}=\left[\begin{smallmatrix} 0 & 0&0&\cdots &  0 \\
 0& 0&0&\cdots &  0 \\
  0 & 0&0&\cdots &  0 \\
    0 & 0&0&\cdots &  0 \\
 \vdots &\vdots &\vdots& & \vdots\\
 1 &0&0& \cdots & 0\\
  \vdots &\vdots &\vdots& & \vdots\\
  \boldsymbol{\cdot} & 0&0&\cdots &  0
 \end{smallmatrix}\right]\,(4\leq i\leq n-1),\] where $E_{i+1,1}$ is the $n\times n$ matrix that has $1$ in the 
$(i+1,1)^{st}$
position and all other entries are zero. According to Remark \ref{NumberOuterDerivations}, we have at most two outer derivations.

\paragraph{Codimension two solvable extensions of $\mathcal{L}^4,(n=4)$}
We consider the same cases as in Section \ref{Twodim(n=4)} and follow the General approach given in Section \ref{Two&Three}.

\noindent (1) (a) One could set $\left(
\begin{array}{c}
  a^1 \\
 b^1\\
 c^1
\end{array}\right)=\left(
\begin{array}{c}
  1\\
 a\\
 0
\end{array}\right)$ and $\left(
\begin{array}{c}
  a^2 \\
 b^2\\
 c^2
\end{array}\right)=\left(
\begin{array}{c}
  1\\
 b\\
 1
\end{array}\right),(a\neq-1,a\neq0,b\neq-1).$ Therefore the vector space of outer derivations as $4\times 4$
matrices is as follows:
$$\L_{e_{5}}=\begin{array}{llll} \left[\begin{matrix}
 -1 & 0 & 0 & 0\\
 \frac{3-2a}{2}a_{2,3}+\frac{1-2a}{2}b_{2,3} & -2a & a_{2,3}& 1-a\\
  1-a & 0 & -a & 0 \\
  0 & 0 &  0 & -1-a
 \end{matrix}\right]
\end{array},$$
$$\L_{e_{6}}=\begin{array}{llll} \left[\begin{matrix}
 -1 & 0 & -1 & 0\\
(1-b)\alpha_{2,3}-b\cdot\beta_{2,3} & -2(b+1) & 2\alpha_{2,3}+\beta_{2,3}& -1-b\\
  -b & 0 & -b & 0 \\
  0 & 0 &  0 & -1-b
    \end{matrix}\right]
\end{array}.$$
\noindent $(i)$ Considering $\L_{[e_5,e_6]},$ we obtain that $a:=1.$ Since $b\neq-1,$ we have that 
$\alpha_{2,3}:=\left(b+\frac{1}{2}\right)a_{2,3}-\frac{b_{2,3}}{2}$ and it follows that $\beta_{2,3}:=\left(\frac{1}{2}-b\right)a_{2,3}+\frac{3}{2}b_{2,3}.$ 
It implies that $\L_{[e_5,e_6]}=0.$ As a result,
$$\L_{e_{5}}=\begin{array}{llll} \left[\begin{matrix}
 -1 & 0 & 0 & 0\\
 \frac{a_{2,3}-b_{2,3}}{2} & -2 & a_{2,3}& 0\\
  0 & 0 & -1 & 0 \\
  0 & 0 &  0 & -2
 \end{matrix}\right]
\end{array},$$
$$\L_{e_{6}}=\begin{array}{llll} \left[\begin{matrix}
 -1 & 0 & -1 & 0\\
\frac{a_{2,3}}{2}-\left(b+\frac{1}{2}\right)b_{2,3} & -2(b+1) & \left(b+\frac{3}{2}\right)a_{2,3}+\frac{b_{2,3}}{2}& -1-b\\
  -b & 0 & -b & 0 \\
  0 & 0 &  0 & -1-b
    \end{matrix}\right]
\end{array},(b\neq-1)\footnote{We notice that $\L_{e_{6}}$ is nilpotent if $b=-1.$}.$$
Further, we find the following commutators:
\allowdisplaybreaks
\begin{equation}
\left\{
\begin{array}{l}
\displaystyle  \nonumber \L_{[e_{1},e_{5}]}=\L_{e_1},\L_{[e_{2},e_5]}=0,\L_{[e_{3},e_5]}=\L_{e_3},\L_{[e_{i},e_5]}=0,\L_{[e_{1},e_6]}=\L_{e_1}+
b\L_{e_3},\\
\displaystyle \L_{[e_{2},e_6]}=0,\L_{[e_{3},e_{6}]}=\L_{e_1}+b\L_{e_3}, \L_{[e_{i},e_6]}=0,(b\neq-1,4\leq i\leq 6).
\end{array} 
\right.
\end{equation} 
\noindent $(ii)$ We include a linear combination of $e_2$ and $e_4$:
\begin{equation}
\left\{
\begin{array}{l}
\displaystyle  \nonumber [e_{1},e_{5}]=e_1+c_{2,1}e_2+c_{4,1}e_4,[e_{2},e_5]=c_{2,2}e_2+c_{4,2}e_4,
[e_{3},e_5]=c_{2,3}e_2+e_3+c_{4,3}e_4,\\
\displaystyle [e_{i},e_5]=c_{2,i}e_2+c_{4,i}e_4,[e_{1},e_6]=e_1+d_{2,1}e_2+
be_3+d_{4,1}e_4,[e_{2},e_6]=d_{2,2}e_2+d_{4,2}e_4,\\
\displaystyle [e_{3},e_{6}]=e_1+d_{2,3}e_2+be_3+d_{4,3}e_4, [e_{i},e_6]=d_{2,i}e_2+d_{4,i}e_4,(b\neq-1,4\leq i\leq 6).
\end{array} 
\right.
\end{equation} 
Besides we have the brackets from $\mathcal{L}^4$ and from outer derivations $\L_{e_{5}}$ and $\L_{e_{6}}$ as well.

\noindent $(iii)$ We satisfy the right Leibniz identity shown in Table \ref{LeftCodimTwo(L4,(n=4))}.

\begin{table}[h!]
\caption{Left Leibniz identities in case (1) (a) with a nilradical $\mathcal{L}^4,(n=4)$.}
\label{LeftCodimTwo(L4,(n=4))}
\begin{tabular}{lp{2.4cm}p{12cm}}
\hline
\scriptsize Steps &\scriptsize Ordered triple &\scriptsize
Result\\ \hline
\scriptsize $1.$ &\scriptsize $\L_{e_1}\left([e_{1},e_{5}]\right)$ &\scriptsize
$[e_{2},e_5]=0$
$\implies$ $c_{2,2}=c_{4,2}=0.$\\ \hline
\scriptsize $2.$ &\scriptsize $\L_{e_1}\left([e_{1},e_{6}]\right)$ &\scriptsize
$[e_{2},e_6]=0$
$\implies$ $d_{2,2}=d_{4,2}=0.$\\ \hline
\scriptsize $3.$ &\scriptsize $\L_{e_1}\left([e_{3},e_{5}]\right)$ &\scriptsize
$c_{2,4}:=-2,c_{4,4}:=2$
$\implies$  $[e_{4},e_5]=2\left(e_4-e_2\right).$ \\ \hline
\scriptsize $4.$ &\scriptsize $\L_{e_1}\left([e_{3},e_{6}]\right)$ &\scriptsize
$d_{2,4}:=-b-1,d_{4,4}:=b+1$ 
$\implies$  $[e_{4},e_6]=(b+1)\left(e_4-e_2\right).$ \\ \hline
\scriptsize $5.$ &\scriptsize $\L_{e_{5}}\left([e_{5},e_{5}]\right)$ &\scriptsize
$c_{4,5}=0$
$\implies$  $[e_{5},e_{5}]=c_{2,5}e_{2}.$\\ \hline
\scriptsize $6.$ &\scriptsize $\L_{e_6}\left([e_{6},e_{6}]\right)$ &\scriptsize
$b\neq-1\implies$ $d_{4,6}=0$ 
$\implies$  $[e_{6},e_6]=d_{2,6}e_2.$   \\ \hline  
\scriptsize $7.$ &\scriptsize $\L_{e_3}\left([e_{5},e_{5}]\right)$ &\scriptsize
$c_{2,3}:=a_{2,3},c_{4,3}=0$  
$\implies$  $[e_{3},e_5]=a_{2,3}e_2+e_3.$   \\ \hline
\scriptsize $8.$ &\scriptsize $\L_{e_{1}}\left([e_{5},e_{5}]\right)$ &\scriptsize
 $c_{2,1}:=\frac{a_{2,3}-b_{2,3}}{2},$ $c_{4,1}=0$ 
$\implies$  $[e_{1},e_{5}]=e_{1}+\frac{a_{2,3}-b_{2,3}}{2}e_2.$\\ \hline
\scriptsize $9.$ &\scriptsize $\L_{e_{1}}\left([e_{5},e_{6}]\right)$ &\scriptsize
$d_{2,1}:=\left(b+\frac{1}{2}\right)a_{2,3}-\frac{b_{2,3}}{2},d_{4,1}=0$ 
$\implies$ $[e_{1},e_{6}]=e_1+\left(\left(b+\frac{1}{2}\right)a_{2,3}-\frac{b_{2,3}}{2}\right)e_2+be_3.$   \\ \hline
\scriptsize $10.$ &\scriptsize $\L_{e_{5}}\left([e_{3},e_{6}]\right)$ &\scriptsize
$d_{2,3}:=\left(b+\frac{1}{2}\right)a_{2,3}-\frac{b_{2,3}}{2},d_{4,3}=0$
$\implies$ $[e_{3},e_{6}]=e_1+\left(\left(b+\frac{1}{2}\right)a_{2,3}-\frac{b_{2,3}}{2}\right)e_2+be_3.$ \\ \hline
\scriptsize $11.$ &\scriptsize $\L_{e_5}\left([e_{6},e_{5}]\right)$ &\scriptsize
 $d_{4,5}:=-c_{4,6}$ $\implies$ $c_{4,6}:=(b+1)c_{2,5}-c_{2,6}$
$\implies$  $[e_5,e_6]=d_{2,5}e_2+\left(c_{2,6}-(b+1)c_{2,5}\right)e_4,[e_{6},e_5]=c_{2,6}e_2+\left((b+1)c_{2,5}-c_{2,6}\right)e_4$\\ \hline
\scriptsize $12.$ &\scriptsize $\L_{e_{5}}\left([e_{6},e_{6}]\right)$ &\scriptsize
$d_{2,6}:=(b+1)\left(c_{2,6}+d_{2,5}\right)-(b+1)^2c_{2,5}$
$\implies$ $[e_{6},e_6]=\left((b+1)\left(c_{2,6}+d_{2,5}\right)-(b+1)^2c_{2,5}\right)e_2.$\\ \hline
\end{tabular}
\end{table}
We obtain that $\L_{e_5}$ and $\L_{e_6}$ restricted to the nilradical do not change, but the remaining brackets are as follows:
\begin{equation}
\left\{
\begin{array}{l}
\displaystyle  \nonumber [e_{1},e_{5}]=e_1+\frac{a_{2,3}-b_{2,3}}{2}e_2,
[e_{3},e_5]=a_{2,3}e_2+e_3,[e_{4},e_5]=2\left(e_4-e_2\right),[e_5,e_5]=c_{2,5}e_2,\\
\displaystyle [e_6,e_5]=c_{2,6}e_2+\left((b+1)c_{2,5}-c_{2,6}\right)e_4,
[e_{1},e_6]=e_1+\left(\left(b+\frac{1}{2}\right)a_{2,3}-\frac{b_{2,3}}{2}\right)e_2+
be_3,\\
\displaystyle [e_{3},e_{6}]=e_1+\left(\left(b+\frac{1}{2}\right)a_{2,3}-\frac{b_{2,3}}{2}\right)e_2+
be_3, [e_{4},e_6]=(b+1)\left(e_4-e_2\right),\\
\displaystyle [e_5,e_6]=d_{2,5}e_2+\left(c_{2,6}-(b+1)c_{2,5}\right)e_4,[e_{6},e_6]=\left((b+1)\left(c_{2,6}+d_{2,5}\right)-(b+1)^2c_{2,5}\right)e_2,\\
\displaystyle (b\neq-1).
\end{array} 
\right.
\end{equation} 
Altogether the nilradical $\mathcal{L}^4$ $(\ref{L4}),$ the outer derivations $\L_{e_{5}}$ and $\L_{e_{6}}$
and the remaining brackets given above define a continuous family of Leibniz algebras
depending on the parameters.

\noindent $(iv)\&(v)$ We apply the following transformation: $e^{\prime}_1=e_1+\frac{a_{2,3}-b_{2,3}}{2}e_2,
e^{\prime}_2=e_2,e^{\prime}_3=e_3+a_{2,3}e_2,e^{\prime}_4=e_4,e^{\prime}_5=e_5+\frac{c_{2,5}}{2}e_2,
e^{\prime}_6=e_6+\frac{d_{2,5}}{2}e_2+\frac{c_{2,6}-(b+1)c_{2,5}}{2}e_4$
and obtain a Leibniz algebra $l_{6,2}$ given below:
\begin{equation}
\left\{
\begin{array}{l}
\displaystyle  \nonumber [e_1,e_5]=e_1,[e_3,e_5]=e_3,[e_4,e_5]=2(e_4-e_2),[e_{5},e_{1}]=-e_1,[e_5,e_2]=-2e_2,[e_{5},e_3]=-e_3,\\
\displaystyle 
[e_{5},e_4]=-2e_4, [e_1,e_6]=e_1+be_3,[e_3,e_6]=e_1+be_3,[e_4,e_6]=(b+1)\left(e_4-e_2\right),\\
\displaystyle
[e_{6},e_1]=-e_1-be_3,[e_6,e_2]=-2(b+1)e_2,[e_{6},e_{3}]=-e_1-be_3,[e_{6},e_4]=-(b+1)\left(e_2+e_4\right),\\
\displaystyle (b\neq-1).
\end{array} 
\right.
\end{equation}
\begin{remark}
We notice that if $b=-1$ in the algebra, then the outer derivation $\L_{e_6}$ is nilpotent.
\end{remark}
\noindent (1) (b) We set $\left(
\begin{array}{c}
  a^1 \\
 b^1\\
 c^1
\end{array}\right)=\left(
\begin{array}{c}
  1\\
 2\\
 c
\end{array}\right)$ and $\left(
\begin{array}{c}
  a^2 \\
 b^2\\
 c^2
\end{array}\right)=\left(
\begin{array}{c}
  1\\
 1\\
 d
\end{array}\right),(c\neq-2,d\neq-1).$ Therefore the vector space of outer derivations as $4\times 4$
matrices is as follows:
$$\L_{e_{5}}=\begin{array}{llll} \left[\begin{matrix}
 -1 & 0 & -c & 0\\
-\frac{c+1}{2}a_{2,3}-\frac{c+3}{2}b_{2,3} & -2(c+2) & (c+1)a_{2,3}+c\cdot b_{2,3}& -1-2c\\
 -c-1 & 0 & -2 & 0 \\
  0 & 0 &  0 & -3
\end{matrix}\right]
\end{array},$$
$$\L_{e_{6}}=\begin{array}{llll} \left[\begin{matrix}
 -1 & 0 & -d & 0\\
\frac{1-d}{2}\alpha_{2,3}-\frac{d+1}{2}\beta_{2,3} & -2(d+1) & (d+1)\alpha_{2,3}+d\cdot\beta_{2,3}& -2d\\
  -d & 0 & -1 & 0 \\
  0 & 0 &  0 & -2
  \end{matrix}\right]
\end{array}.$$
\noindent $(i)$ Considering $\L_{[e_5,e_6]},$ we obtain that $d=0$ and we have
the system of equations: 
$$\left\{ \begin{array}{ll}
b_{2,3}:=\left(\frac{c}{2}+1\right)\beta_{2,3}-\left(\frac{c}{2}+1\right)\alpha_{2,3} {,}  \\
\left(2c+3\right)\beta_{2,3}-\alpha_{2,3}-\left(c+1\right)a_{2,3}-\left(c+3\right)b_{2,3}=0{.}
\end{array}
\right. $$ 
There are the following two cases:
\begin{enumerate}[noitemsep, topsep=0pt]
\item[(I)] If $c\neq-1,$ then $a_{2,3}:=\frac{c+4}{2}\alpha_{2,3}-\frac{c}{2}\beta_{2,3}$
and $\L_{[e_5,e_6]}=0.$ Consequently,
$$\L_{e_{5}}=\begin{array}{llll} \left[\begin{matrix}
 -1 & 0 & -c & 0\\
\frac{\alpha_{2,3}}{2}-\frac{2c+3}{2}\beta_{2,3} & -2(c+2) & \frac{3c+4}{2}\alpha_{2,3}+\frac{c}{2}\beta_{2,3}& -2c-1\\
 -c-1 & 0 & -2 & 0 \\
  0 & 0 &  0 & -3
 \end{matrix}\right]
\end{array},$$
$$\L_{e_{6}}=\begin{array}{llll} \left[\begin{matrix}
 -1 & 0 & 0 & 0\\
\frac{\alpha_{2,3}-\beta_{2,3}}{2} & -2 & \alpha_{2,3}& 0\\
  0 & 0 & -1 & 0 \\
  0 & 0 &  0 & -2
   \end{matrix}\right]
\end{array},(c\neq-2)$$
\noindent and we find the commutators given below:
\allowdisplaybreaks
\begin{equation}
\left\{
\begin{array}{l}
\displaystyle  \nonumber \L_{[e_{1},e_{5}]}=\L_{e_1}+(c+1)\L_{e_3},\L_{[e_{2},e_5]}=0,\L_{[e_{3},e_5]}=c\L_{e_1}+2\L_{e_3},\L_{[e_{i},e_5]}=0,\\
\displaystyle \L_{[e_{1},e_6]}=\L_{e_1},\L_{[e_{2},e_6]}=0,\L_{[e_{3},e_{6}]}=\L_{e_3}, \L_{[e_{i},e_6]}=0,(c\neq-2,4\leq i\leq 6).
\end{array} 
\right.
\end{equation} 
\item[(II)] If $c=-1,$ then $b_{2,3}:=\frac{\beta_{2,3}-\alpha_{2,3}}{2}$
and $\L_{[e_5,e_6]}=0.$ As a result,
$$\L_{e_{5}}=\begin{array}{llll} \left[\begin{matrix}
 -1 & 0 & 1 & 0\\
\frac{\alpha_{2,3}-\beta_{2,3}}{2} & -2& \frac{\alpha_{2,3}-\beta_{2,3}}{2}& 1\\
0 & 0 & -2 & 0 \\
  0 & 0 &  0 & -3
 \end{matrix}\right]
\end{array},
\L_{e_{6}}=\begin{array}{llll} \left[\begin{matrix}
 -1 & 0 & 0 & 0\\
\frac{\alpha_{2,3}-\beta_{2,3}}{2} & -2 & \alpha_{2,3}& 0\\
  0 & 0 & -1 & 0 \\
  0 & 0 &  0 & -2
   \end{matrix}\right]
\end{array}$$
 and we have the following commutators:
\allowdisplaybreaks
\begin{equation}
\left\{
\begin{array}{l}
\displaystyle  \nonumber \L_{[e_{1},e_{5}]}=\L_{e_1},\L_{[e_{2},e_5]}=0,\L_{[e_{3},e_5]}=2\L_{e_3}-\L_{e_1},\L_{[e_{i},e_5]}=0,\\
\displaystyle \L_{[e_{1},e_6]}=\L_{e_1},\L_{[e_{2},e_6]}=0,\L_{[e_{3},e_{6}]}=\L_{e_3}, \L_{[e_{i},e_6]}=0,(4\leq i\leq 6).
\end{array} 
\right.
\end{equation} 
\end{enumerate}
\noindent $(ii)$ We combine cases (I) and (II), include a linear combination of $e_2$ and $e_4$,
and obtain the following:
\allowdisplaybreaks
\begin{equation}
\left\{
\begin{array}{l}
\displaystyle  \nonumber [e_{1},e_{5}]=e_1+a_{2,1}e_2+(c+1)e_3+a_{4,1}e_4,[e_{2},e_5]=a_{2,2}e_2+a_{4,2}e_4,[e_{3},e_5]=ce_1+a_{2,3}e_2+\\
\displaystyle 2e_3+a_{4,3}e_4,[e_{i},e_5]=a_{2,i}e_2+a_{4,i}e_4,[e_{1},e_6]=e_1+b_{2,1}e_2+b_{4,1}e_4,[e_{2},e_6]=b_{2,2}e_2+b_{4,2}e_4,\\
\displaystyle [e_{3},e_{6}]=b_{2,3}e_2+e_3+b_{4,3}e_4,[e_{i},e_6]=b_{2,i}e_2+b_{4,i}e_4,(c\neq-2,4\leq i\leq 6).
\end{array} 
\right.
\end{equation}
Besides we have the brackets from $\mathcal{L}^4$ and from outer derivations $\L_{e_{5}}$ and $\L_{e_{6}}$ as well.

\noindent $(iii)$ To satisfy the left Leibniz identity, we refer to the identities given in Table \ref{LeftCodimTwo(L4,(n=4))}. as much as possible.
The identities we apply are the following: $1.-6.,$
$\L_{e_3}{\left([e_6,e_6]\right)}=[\L_{e_3}(e_6),e_6]+[e_6,\L_{e_3}(e_6)],$
$\L_{e_1}\left([e_6,e_6]\right)=[\L_{e_1}(e_6),e_6]+[e_6,\L_{e_1}(e_6)],$ $\L_{e_1}\left([e_6,e_5]\right)=[\L_{e_1}(e_6),e_5]+[e_6,\L_{e_1}(e_5)],
\L_{e_3}\left([e_6,e_5]\right)=[\L_{e_3}(e_6),e_5]+[e_6,\L_{e_3}(e_5)],$ $11.$
 and $12.$
We have that $\L_{e_5}$ and $\L_{e_6}$ restricted to the nilradical do not change, but the remaining brackets are as follows:
\begin{equation}
\left\{
\begin{array}{l}
\displaystyle  \nonumber [e_{1},e_{5}]=e_1+\left(\left(c+\frac{3}{2}\right)\alpha_{2,3}-\frac{\beta_{2,3}}{2}\right)e_2+(c+1)e_3,\\
\displaystyle [e_{3},e_5]=ce_1+\left(\left(\frac{c}{2}+2\right)\alpha_{2,3}-\frac{c}{2}\beta_{2,3}\right)e_2+2e_3,[e_{4},e_5]=3\left(e_4-e_2\right),\\
\displaystyle [e_5,e_5]=(c+2)\left(a_{2,6}+a_{4,6}\right)e_2,[e_6,e_5]=a_{2,6}e_2+a_{4,6}e_4,[e_{1},e_6]=e_1+\frac{\alpha_{2,3}-\beta_{2,3}}{2}e_2,\\
\displaystyle [e_{3},e_{6}]=\alpha_{2,3}e_2+e_3, [e_{4},e_6]=2\left(e_4-e_2\right),[e_5,e_6]=\left((c+2)b_{2,6}+a_{4,6}\right)e_2-a_{4,6}e_4,\\
\displaystyle [e_{6},e_6]=b_{2,6}e_2,(c\neq-2).
\end{array} 
\right.
\end{equation} 
Altogether the nilradical $\mathcal{L}^4$ $(\ref{L4}),$ the outer derivations $\L_{e_{5}}$ and $\L_{e_{6}}$
and the remaining brackets given above define a continuous family of Leibniz algebras.

\noindent $(iv)\&(v)$ Applying the transformation: $e^{\prime}_1=e_1+\frac{\alpha_{2,3}-\beta_{2,3}}{2}e_2,
e^{\prime}_2=e_2,e^{\prime}_3=e_3+\alpha_{2,3}e_2,e^{\prime}_4=e_4,e^{\prime}_5=e_5+\frac{a_{2,6}}{2}e_2+\frac{a_{4,6}}{2}e_4,
e^{\prime}_6=e_6+\frac{b_{2,6}}{2}e_2,$ we obtain a Leibniz algebra $l_{6,3}$:
\begin{equation}
\left\{
\begin{array}{l}
\displaystyle  \nonumber [e_1,e_5]=e_1+(c+1)e_3,[e_3,e_5]=ce_1+2e_3,[e_{4},e_5]=3\left(e_4-e_2\right),[e_{5},e_{1}]=-e_1-(c+1)e_3,\\
\displaystyle [e_5,e_2]=-2(c+2)e_2,[e_{5},e_3]=-ce_1-2e_3,[e_5,e_4]=(-2c-1)e_2-3e_4,[e_1,e_6]=e_1,\\
\displaystyle [e_3,e_6]=e_3, [e_{4},e_6]=2\left(e_4-e_2\right),[e_{6},e_1]=-e_1,[e_6,e_2]=-2e_2,[e_{6},e_{3}]=-e_3,\\
\displaystyle  [e_6,e_4]=-2e_4,(c\neq-2).
\end{array} 
\right.
\end{equation} 

\noindent (1) (c) We set $\left(
\begin{array}{c}
  a^1 \\
 b^1\\
 c^1
\end{array}\right)=\left(
\begin{array}{c}
  a\\
 1\\
 0
\end{array}\right)$ and $\left(
\begin{array}{c}
  a^2 \\
 b^2\\
 c^2
\end{array}\right)=\left(
\begin{array}{c}
  b\\
 0\\
 1
\end{array}\right),(a,b\neq0,a\neq-1).$ Therefore the vector space of outer derivations as $4\times 4$
matrices is as follows:
$$\L_{e_{5}}=\begin{array}{llll} \left[\begin{matrix}
 -a & 0 & 0 & 0\\
\frac{(3a-2)a_{2,3}+(a-2)b_{2,3}}{2a} & -2 & a_{2,3}& a-1\\
 a-1 & 0 & -1 & 0 \\
  0 & 0 &  0 & -a-1
\end{matrix}\right]
\end{array},$$
$$\L_{e_{6}}=\begin{array}{llll} \left[\begin{matrix}
 -b & 0 & -1 & 0\\
\frac{(3b-1)\alpha_{2,3}+(b-1)\beta_{2,3}}{2b} & -2 & \frac{(b+1)\alpha_{2,3}+\beta_{2,3}}{b}& b-2\\
  b-1 & 0 & 0 & 0 \\
  0 & 0 &  0 & -b
  \end{matrix}\right]
\end{array}.$$
\noindent $(i)$ Considering $\L_{[e_5,e_6]},$ we obtain that $a:=1$ and we have
the system of equations: 
$$\left\{ \begin{array}{ll}
\beta_{2,3}:=\frac{3b\cdot a_{2,3}+b\cdot b_{2,3}}{2}-(b+1)\alpha_{2,3} {,}  \\
(b-3)\alpha_{2,3}+\frac{b-3}{2}\left(b_{2,3}-a_{2,3}\right)=0{.}
\end{array}
\right. $$ 
There are the following two cases:
\begin{enumerate}[noitemsep, topsep=0pt]
\item[(I)] If $b\neq3,$ then $\alpha_{2,3}:=\frac{a_{2,3}-b_{2,3}}{2}\implies$ $\beta_{2,3}:=\frac{(2b-1)a_{2,3}+(2b+1)b_{2,3}}{2}$
and $\L_{[e_5,e_6]}=0.$ As a result,
$$\hspace*{-0.6cm}\L_{e_{5}}=\begin{array}{llll} \left[\begin{matrix}
 -1 & 0 & 0 & 0\\
\frac{a_{2,3}-b_{2,3}}{2} & -2 & a_{2,3}& 0\\
 0 & 0 & -1 & 0 \\
  0 & 0 &  0 & -2
\end{matrix}\right]
\end{array},\L_{e_{6}}=\begin{array}{llll} \left[\begin{matrix}
 -b & 0 & -1 & 0\\
\frac{b\cdot a_{2,3}+(b-2)b_{2,3}}{2} & -2 & \frac{3a_{2,3}+b_{2,3}}{2}& b-2\\
  b-1 & 0 & 0 & 0 \\
  0 & 0 &  0 & -b
  \end{matrix}\right]
\end{array},(b\neq0).$$
\noindent Further, we find the following:
\allowdisplaybreaks
\begin{equation}
\left\{
\begin{array}{l}
\displaystyle  \nonumber \L_{[e_{1},e_{5}]}=\L_{e_1},\L_{[e_{2},e_5]}=0,\L_{[e_{3},e_5]}=\L_{e_3},\L_{[e_{i},e_5]}=0,\L_{[e_{1},e_6]}=b\L_{e_1}+(1-b)\L_{e_3},\\
\displaystyle \L_{[e_{2},e_6]}=0,\L_{[e_{3},e_{6}]}=\L_{e_1}, \L_{[e_{i},e_6]}=0,(b\neq0,4\leq i\leq 6).
\end{array} 
\right.
\end{equation} 
\item[(II)] If $b:=3,$ then $\beta_{2,3}:=\frac{9a_{2,3}+3b_{2,3}}{2}-4\alpha_{2,3}$ and $\L_{[e_5,e_6]}=0$ as well as $\L_{e_5},\L_{e_6}$ are as follows: 
$$\L_{e_{5}}=\begin{array}{llll} \left[\begin{matrix}
 -1 & 0 & 0 & 0\\
\frac{a_{2,3}-b_{2,3}}{2} & -2 & a_{2,3}& 0\\
 0 & 0 & -1 & 0 \\
  0 & 0 &  0 & -2
\end{matrix}\right]
\end{array},\L_{e_{6}}=\begin{array}{llll} \left[\begin{matrix}
 -3 & 0 & -1 & 0\\
\frac{3a_{2,3}+b_{2,3}}{2} & -2 & \frac{3a_{2,3}+b_{2,3}}{2}& 1\\
  2 & 0 & 0 & 0 \\
  0 & 0 &  0 & -3
  \end{matrix}\right]
\end{array}.$$
We have the following commutators:
\allowdisplaybreaks
\begin{equation}
\left\{
\begin{array}{l}
\displaystyle  \nonumber \L_{[e_{1},e_{5}]}=\L_{e_1},\L_{[e_{2},e_5]}=0,\L_{[e_{3},e_5]}=\L_{e_3},\L_{[e_{i},e_5]}=0,\L_{[e_{1},e_6]}=3\L_{e_1}-2\L_{e_3},\\
\displaystyle \L_{[e_{2},e_6]}=0,\L_{[e_{3},e_{6}]}=\L_{e_1}, \L_{[e_{i},e_6]}=0,(4\leq i\leq 6).
\end{array} 
\right.
\end{equation} 
\end{enumerate}
\noindent $(ii)$ We combine cases (I) and (II) and include a linear combination of $e_2$ and $e_4:$
\allowdisplaybreaks
\begin{equation}
\left\{
\begin{array}{l}
\displaystyle  \nonumber [e_{1},e_{5}]=e_1+c_{2,1}e_2+c_{4,1}e_4,[e_{2},e_5]=c_{2,2}e_2+c_{4,2}e_4,[e_{3},e_5]=c_{2,3}e_2+e_3+c_{4,3}e_4,\\
\displaystyle [e_{i},e_5]=c_{2,i}e_2+c_{4,i}e_4,[e_{1},e_6]=be_1+d_{2,1}e_2+(1-b)e_3+d_{4,1}e_4,[e_{2},e_6]=d_{2,2}e_2+\\
\displaystyle d_{4,2}e_4,[e_{3},e_{6}]=e_1+d_{2,3}e_2+d_{4,3}e_4,[e_{i},e_6]=d_{2,i}e_2+d_{4,i}e_4,(b\neq0,4\leq i\leq 6).
\end{array} 
\right.
\end{equation}
Besides we have the brackets from $\mathcal{L}^4$ and from outer derivations $\L_{e_{5}}$ and $\L_{e_{6}}$ as well.

\noindent $(iii)$ To satisfy the left Leibniz identity, we mostly apply the identities given in Table \ref{LeftCodimTwo(L4,(n=4))}:
$1.-5.,\L_{e_6}\left([e_6,e_5]\right)=[\L_{e_6}(e_6),e_5]+[e_6,\L_{e_6}(e_5)],7.-12.$ 
We have that $\L_{e_5}$ and $\L_{e_6}$ restricted to the nilradical do not change, but the remaining brackets are the following:
\begin{equation}
\left\{
\begin{array}{l}
\displaystyle  \nonumber [e_{1},e_{5}]=e_1+\frac{a_{2,3}-b_{2,3}}{2}e_2, [e_{3},e_5]=a_{2,3}e_2+e_3,[e_{4},e_5]=2\left(e_4-e_2\right),[e_5,e_5]=c_{2,5}e_2,\\
\displaystyle [e_6,e_5]=c_{2,6}e_2+\left(c_{2,5}-c_{2,6}\right)e_4,[e_{1},e_6]=be_1+\frac{(2-b)a_{2,3}-b\cdot b_{2,3}}{2}e_2+
(1-b)e_3,\\
\displaystyle [e_{3},e_{6}]=e_1+\frac{a_{2,3}-b_{2,3}}{2}e_2, [e_{4},e_6]=b\left(e_4-e_2\right),
[e_5,e_6]=d_{2,5}e_2+\left(c_{2,6}-c_{2,5}\right)e_4,\\
\displaystyle [e_{6},e_6]=\left(d_{2,5}-c_{2,5}+c_{2,6}\right)e_2,(b\neq0).
\end{array} 
\right.
\end{equation} 
Altogether the nilradical $\mathcal{L}^4$ $(\ref{L4}),$ the outer derivations $\L_{e_{5}}$ and $\L_{e_{6}}$
and the remaining brackets given above define a continuous family of Leibniz algebras
depending on the parameters.

\noindent $(iv)\&(v)$ We apply the transformation: $e^{\prime}_1=e_1+\frac{a_{2,3}-b_{2,3}}{2}e_2,
e^{\prime}_2=e_2,e^{\prime}_3=e_3+a_{2,3}e_2,e^{\prime}_4=e_4,e^{\prime}_5=e_5+\frac{c_{2,5}}{2}e_2,
e^{\prime}_6=e_6+\frac{d_{2,5}}{2}e_2+\frac{c_{2,6}-c_{2,5}}{2}e_4$ and obtain a Leibniz algebra $l_{6,4}$:
\begin{equation}
\left\{
\begin{array}{l}
\displaystyle  \nonumber [e_1,e_5]=e_1,[e_3,e_5]=e_3,[e_{4},e_5]=2\left(e_4-e_2\right),[e_{5},e_{1}]=-e_1,[e_5,e_2]=-2e_2,[e_{5},e_3]=-e_3,\\
\displaystyle [e_5,e_4]=-2e_4,[e_1,e_6]=be_1+(1-b)e_3,[e_3,e_6]=e_1,[e_{4},e_6]=b\left(e_4-e_2\right),\\
\displaystyle [e_{6},e_1]=-be_1+(b-1)e_3,[e_6,e_2]=-2e_2,[e_{6},e_{3}]=-e_1,[e_6,e_4]=(b-2)e_2-be_4,(b\neq0).
\end{array} 
\right.
\end{equation} 
\paragraph{Codimension two solvable extensions of $\mathcal{L}^4,(n\geq5)$}
One could set $\left(
\begin{array}{c}
  a \\
 b
\end{array}\right)=\left(
\begin{array}{c}
  1\\
2
\end{array}\right)$ and $\left(
\begin{array}{c}
 \alpha \\
 \beta
\end{array}\right)=\left(
\begin{array}{c}
  1\\
 1
\end{array}\right).$
Therefore the vector space of outer derivations as $n \times n$ matrices is as follows:
{ $$\L_{e_{n+1}}=\left[\begin{smallmatrix}
 -1 & 0 & 0 & 0&0&0&\cdots && 0&0 & 0\\
  3a_{2,1}-2a_{2,3} & -4 & a_{2,3}& -1 &0&0 & \cdots &&0  & 0& 0\\
  -1 & 0 & -2 & 0 & 0&0 &\cdots &&0 &0& 0\\
  0 & 0 &  0 & -3 &0 &0 &\cdots&&0 &0 & 0\\
 0 & 0 & -a_{5,3} & 0 & -4  &0&\cdots & &0&0 & 0\\
  0 & 0 &\boldsymbol{\cdot} & -a_{5,3} & 0 &-5&\cdots & &0&0 & 0\\
    0 & 0 &\boldsymbol{\cdot} & \boldsymbol{\cdot} & \ddots &0&\ddots &&\vdots&\vdots &\vdots\\
  \vdots & \vdots & \vdots &\vdots &  &\ddots&\ddots &\ddots &\vdots&\vdots & \vdots\\
 0 & 0 & -a_{n-2,3}& -a_{n-3,3}& \cdots&\cdots &-a_{5,3}&0&3-n &0& 0\\
0 & 0 & -a_{n-1,3}& -a_{n-2,3}& \cdots&\cdots &\boldsymbol{\cdot}&-a_{5,3}&0 &2-n& 0\\
 0 & 0 & -a_{n,3}& -a_{n-1,3}& \cdots&\cdots &\boldsymbol{\cdot}&\boldsymbol{\cdot}&-a_{5,3} &0& 1-n
\end{smallmatrix}\right],$$}
{ $$ \L_{e_{n+2}}=\left[\begin{smallmatrix}
 -1 & 0 & 0 & 0&0&0&\cdots && 0&0 & 0\\
  \alpha_{2,1} & -2 & \alpha_{2,3}& 0 &0&0 & \cdots &&0  & 0& 0\\
  0 & 0 & -1 & 0 & 0&0 &\cdots &&0 &0& 0\\
  0 & 0 &  0 & -2 &0 &0 &\cdots&&0 &0 & 0\\
 0 & 0 & -\alpha_{5,3} & 0 & -3  &0&\cdots & &0&0 & 0\\
  0 & 0 &\boldsymbol{\cdot} & -\alpha_{5,3} & 0 &-4&\cdots & &0&0 & 0\\
    0 & 0 &\boldsymbol{\cdot} & \boldsymbol{\cdot} & \ddots &0&\ddots &&\vdots&\vdots &\vdots\\
  \vdots & \vdots & \vdots &\vdots &  &\ddots&\ddots &\ddots &\vdots&\vdots & \vdots\\
 0 & 0 & -\alpha_{n-2,3}& -\alpha_{n-3,3}& \cdots&\cdots &-\alpha_{5,3}&0&4-n &0& 0\\
0 & 0 & -\alpha_{n-1,3}& -\alpha_{n-2,3}& \cdots&\cdots &\boldsymbol{\cdot}&-\alpha_{5,3}&0 &3-n& 0\\
 0 & 0 & -\alpha_{n,3}& -\alpha_{n-1,3}& \cdots&\cdots &\boldsymbol{\cdot}&\boldsymbol{\cdot}&-\alpha_{5,3} &0& 2-n
\end{smallmatrix}\right].$$}

\noindent $(i)$ Considering $\L_{[e_{n+1},e_{n+2}]}$ and comparing with
$\sum_{i=1}^n{c_i\L_{e_i}},$ we deduce that 
$\alpha_{i,3}:=a_{i,3},(5\leq i\leq n),$ $\alpha_{2,3}:=\frac{a_{2,3}}{2}$ and
$\alpha_{2,1}:=a_{2,1}-\frac{a_{2,3}}{2}.$ As a result, the outer derivation $\L_{e_{n+2}}$
changes as follows:
$$\L_{e_{n+2}}=\left[\begin{smallmatrix}
 -1 & 0 & 0 & 0&0&0&\cdots && 0&0 & 0\\
  a_{2,1}-\frac{a_{2,3}}{2} & -2 & \frac{a_{2,3}}{2}& 0 &0&0 & \cdots &&0  & 0& 0\\
  0 & 0 & -1 & 0 & 0&0 &\cdots &&0 &0& 0\\
  0 & 0 &  0 & -2 &0 &0 &\cdots&&0 &0 & 0\\
 0 & 0 & -a_{5,3} & 0 & -3  &0&\cdots & &0&0 & 0\\
  0 & 0 &\boldsymbol{\cdot} & -a_{5,3} & 0 &-4&\cdots & &0&0 & 0\\
    0 & 0 &\boldsymbol{\cdot} & \boldsymbol{\cdot} & \ddots &0&\ddots &&\vdots&\vdots &\vdots\\
  \vdots & \vdots & \vdots &\vdots &  &\ddots&\ddots &\ddots &\vdots&\vdots & \vdots\\
 0 & 0 & -a_{n-2,3}& -a_{n-3,3}& \cdots&\cdots &-a_{5,3}&0&4-n &0& 0\\
0 & 0 & -a_{n-1,3}& -a_{n-2,3}& \cdots&\cdots &\boldsymbol{\cdot}&-a_{5,3}&0 &3-n& 0\\
 0 & 0 & -a_{n,3}& -a_{n-1,3}& \cdots&\cdots &\boldsymbol{\cdot}&\boldsymbol{\cdot}&-a_{5,3} &0& 2-n
\end{smallmatrix}\right].$$
Altogether we find the following commutators:
\allowdisplaybreaks
\begin{equation}
\left\{
\begin{array}{l}
\displaystyle  \nonumber \L_{[e_{1},e_{n+1}]}=\L_{e_1}+\L_{e_3},\L_{[e_{2},e_{n+1}]}=0,\L_{[e_{i},e_{n+1}]}=(i-1)\L_{e_i}+\sum_{k=i+2}^{n-1}{a_{k-i+3,3}\L_{e_k}},\\
\displaystyle \L_{[e_{j},e_{n+1}]}=0,(n\leq j\leq n+1),\L_{[e_{n+2},e_{n+1}]}=\sum_{k=4}^{n-1}{a_{k+1,3}\L_{e_{k}}},\L_{[e_{1},e_{n+2}]}=\L_{e_1},\\
\displaystyle \L_{[e_{2},e_{n+2}]}=0,\L_{[e_{i},e_{n+2}]}=(i-2)\L_{e_i}+\sum_{k=i+2}^{n-1}{a_{k-i+3,3}\L_{e_k}},(3\leq i\leq n-1),\\
\displaystyle \L_{[e_{n},e_{n+2}]}=0,\L_{[e_{n+1},e_{n+2}]}=-\sum_{k=4}^{n-1}{a_{k+1,3}\L_{e_{k}}},\L_{[e_{n+2},e_{n+2}]}=0.
\end{array} 
\right.
\end{equation} 
\noindent $(ii)$ We include a linear combination of $e_2$ and $e_n:$
\begin{equation}
\left\{
\begin{array}{l}
\displaystyle  \nonumber [e_{1},e_{n+1}]=e_1+c_{2,1}e_2+e_3+c_{n,1}e_n,[e_{2},e_{n+1}]=c_{2,2}e_2+c_{n,2}e_n,[e_{i},e_{n+1}]=c_{2,i}e_2+\\
\displaystyle(i-1)e_i+\sum_{k=i+2}^{n-1}{a_{k-i+3,3}e_k}+c_{n,i}e_n,[e_{j},e_{n+1}]=c_{2,j}e_2+c_{n,j}e_n,(n\leq j\leq n+1),\\
\displaystyle [e_{n+2},e_{n+1}]=c_{2,n+2}e_2+\sum_{k=4}^{n-1}{a_{k+1,3}e_{k}}+
c_{n,n+2}e_n,[e_{1},e_{n+2}]=e_1+d_{2,1}e_2+d_{n,1}e_n,\\
\displaystyle [e_{2},e_{n+2}]=d_{2,2}e_2+d_{n,2}e_n, [e_{i},e_{n+2}]=d_{2,i}e_2+(i-2)e_i+\sum_{k=i+2}^{n-1}{a_{k-i+3,3}e_k}+d_{n,i}e_n,\\
\displaystyle (3\leq i\leq n-1),[e_{n},e_{n+2}]=d_{2,n}e_2+d_{n,n}e_n,[e_{n+1},e_{n+2}]=d_{2,n+1}e_2-\sum_{k=4}^{n-1}{a_{k+1,3}e_{k}}+\\
\displaystyle d_{n,n+1}e_n, [e_{n+2},e_{n+2}]=d_{2,n+2}e_2+d_{n,n+2}e_n.
\end{array} 
\right.
\end{equation} 
Besides we have the brackets from $\mathcal{L}^4$ and from outer derivations $\L_{e_{n+1}}$ and $\L_{e_{n+2}}$ as well.

\noindent $(iii)$ We satisfy the right Leibniz identity shown in Table \ref{LeftCodimTwo(L4)}. We notice that $\L_{e_{n+1}}$ and $\L_{e_{n+2}}$ restricted to the nilradical do not change, but the remaining brackets are as follows:
\begin{equation}
\left\{
\begin{array}{l}
\displaystyle  \nonumber [e_{1},e_{n+1}]=e_1+a_{2,1}e_2+e_3,[e_{3},e_{n+1}]=a_{2,3}e_2+2e_3+\sum_{k=5}^n{a_{k,3}e_k},
[e_{4},e_{n+1}]=3\left(e_4-e_2\right)+\\
\displaystyle \sum_{k=6}^n{a_{k-1,3}e_k},
 [e_{i},e_{n+1}]=(i-1)e_i+\sum_{k=i+2}^{n}{a_{k-i+3,3}e_k},[e_{n+1},e_{n+1}]=\left(2c_{2,n+2}+2a_{5,3}\right)e_2,\\
\displaystyle [e_{n+2},e_{n+1}]=c_{2,n+2}e_2+\sum_{k=4}^{n-1}{a_{k+1,3}e_{k}}+
c_{n,n+2}e_n,[e_{1},e_{n+2}]=e_1+\left(a_{2,1}-\frac{a_{2,3}}{2}\right)e_2,\\
\displaystyle [e_{3},e_{n+2}]=\frac{a_{2,3}}{2}e_2+e_3+\sum_{k=5}^n{a_{k,3}e_k},
[e_{4},e_{n+2}]=2\left(e_4-e_2\right)+\sum_{k=6}^n{a_{k-1,3}e_k},\\
\displaystyle [e_{i},e_{n+2}]=(i-2)e_i+\sum_{k=i+2}^{n}{a_{k-i+3,3}e_k},
 (5\leq i\leq n),[e_{n+1},e_{n+2}]=d_{2,n+1}e_2-\sum_{k=4}^{n-1}{a_{k+1,3}e_{k}}-\\
\displaystyle c_{n,n+2}e_n, [e_{n+2},e_{n+2}]=\frac{d_{2,n+1}-a_{5,3}}{2}e_2.
\end{array} 
\right.
\end{equation} 

\begin{table}[h!]
\caption{Left Leibniz identities in the codimension two nilradical $\mathcal{L}^4,(n\geq5)$.}
\label{LeftCodimTwo(L4)}
\begin{tabular}{lp{2.4cm}p{12cm}}
\hline
\scriptsize Steps &\scriptsize Ordered triple &\scriptsize
Result\\ \hline
\scriptsize $1.$ &\scriptsize $\L_{e_1}\left([e_{1},e_{n+1}]\right)$ &\scriptsize
$[e_{2},e_{n+1}]=0$
$\implies$ $c_{2,2}=c_{n,2}=0.$\\ \hline
\scriptsize $2.$ &\scriptsize $\L_{e_1}\left([e_{1},e_{n+2}]\right)$ &\scriptsize
$[e_{2},e_{n+2}]=0$
$\implies$ $d_{2,2}=d_{n,2}=0.$\\ \hline
\scriptsize $3.$ &\scriptsize $\L_{e_1}\left([e_{3},e_{n+1}]\right)$ &\scriptsize
$c_{2,4}:=-3,c_{n,4}:=a_{n-1,3},$ where $a_{4,3}=0$ 
$\implies$  $[e_{4},e_{n+1}]=3(e_4-e_2)+\sum_{k=6}^{n}{a_{k-1,3}e_k}.$ \\ \hline
\scriptsize $4.$ &\scriptsize $\L_{e_1}\left([e_{i},e_{n+1}]\right)$ &\scriptsize
$c_{2,i+1}=0,c_{n,i+1}:=a_{n-i+2,3},(4\leq i\leq n-2),$ where $a_{4,3}=0$ 
$\implies$  $[e_{j},e_{n+1}]=\left(j-1\right)e_j+\sum_{k=j+2}^{n}{a_{k-j+3,3}e_k},(5\leq j\leq n-1).$ \\ \hline
\scriptsize $5.$ &\scriptsize $\L_{e_{1}}\left([e_{n-1},e_{n+1}]\right)$ &\scriptsize
$c_{2,n}=0,$ $c_{n,n}:=n-1$
$\implies$  $[e_{n},e_{n+1}]=\left(n-1\right)e_{n}.$ Altogether with $4.,$
  $[e_{i},e_{n+1}]=\left(i-1\right)e_i+\sum_{k=i+2}^{n}{a_{k-i+3,3}e_k},(5\leq i\leq n).$  \\ \hline
 \scriptsize $6.$ &\scriptsize $\L_{e_1}\left([e_{3},e_{n+2}]\right)$ &\scriptsize
$d_{2,4}:=-2,d_{n,4}:=a_{n-1,3},$ where $a_{4,3}=0$ 
$\implies$  $[e_{4},e_{n+2}]=2(e_4-e_2)+\sum_{k=6}^{n}{a_{k-1,3}e_k}.$   \\ \hline  
\scriptsize $7.$ &\scriptsize $\L_{e_1}\left([e_{i},e_{n+2}]\right)$ &\scriptsize
$d_{2,i+1}=0,d_{n,i+1}:=a_{n-i+2,3},(4\leq i\leq n-2),$ where $a_{4,3}=0$ 
$\implies$  $[e_{j},e_{n+2}]=\left(j-2\right)e_j+\sum_{k=j+2}^{n}{a_{k-j+3,3}e_k},(5\leq j\leq n-1).$   \\ \hline
\scriptsize $8.$ &\scriptsize $\L_{e_{1}}\left([e_{n-1},e_{n+2}]\right)$ &\scriptsize
 $d_{2,n}=0,$ $d_{n,n}:=n-2$ 
$\implies$  $[e_{n},e_{n+2}]=\left(n-2\right)e_{n}.$ Combining with $7.,$
  $[e_{i},e_{n+2}]=\left(i-2\right)e_i+\sum_{k=i+2}^{n}{a_{k-i+3,3}e_k},(5\leq i\leq n).$   \\ \hline
\scriptsize $9.$ &\scriptsize $\L_{e_{n+1}}\left([e_{n+1},e_{n+1}]\right)$ &\scriptsize
$c_{n,n+1}=0$ 
$\implies$  $[e_{n+1},e_{n+1}]=c_{2,n+1}e_2.$   \\ \hline
\scriptsize $10.$ &\scriptsize $\L_{e_{n+2}}\left([e_{n+2},e_{n+2}]\right)$ &\scriptsize
$d_{n,n+2}=0$
$\implies$ $[e_{n+2},e_{n+2}]=d_{2,n+2}e_2.$ \\ \hline
\scriptsize $11.$ &\scriptsize $\L_{e_{3}}\left([e_{n+1},e_{n+1}]\right)$ &\scriptsize
 $c_{2,3}:=a_{2,3}, c_{n,3}:=a_{n,3}$
$\implies$  $[e_{3},e_{n+1}]=a_{2,3}e_2+2e_3+\sum_{k=5}^n{a_{k,3}e_k}.$\\ \hline
\scriptsize $12.$ &\scriptsize $\L_{e_{3}}\left([e_{n+2},e_{n+2}]\right)$ &\scriptsize
$d_{2,3}:=\frac{a_{2,3}}{2},$ $d_{n,3}:=a_{n,3}$ 
$\implies$  $[e_{3},e_{n+2}]=\frac{a_{2,3}}{2}e_2+e_3+\sum_{k=5}^n{a_{k,3}e_k}.$   \\ \hline
\scriptsize $13.$ &\scriptsize $\L_{e_{1}}\left([e_{n+1},e_{n+1}]\right)$ &\scriptsize
$c_{2,1}:=a_{2,1},c_{n,1}=0$
$\implies$  $[e_{1},e_{n+1}]=e_1+a_{2,1}e_2+e_3$. \\ \hline
\scriptsize $14.$ &\scriptsize $\L_{e_{1}}\left([e_{n+2},e_{n+2}]\right)$ &\scriptsize
$d_{2,1}:=a_{2,1}-\frac{a_{2,3}}{2},d_{n,1}=0$
$\implies$ $[e_{1},e_{n+2}]=e_1+\left(a_{2,1}-\frac{a_{2,3}}{2}\right)e_2.$ \\ \hline
\scriptsize $15.$ &\scriptsize $\L_{e_{n+1}}\left([e_{n+2},e_{n+1}]\right)$ &\scriptsize
$c_{2,n+1}:=2c_{2,n+2}+2a_{5,3},d_{n,n+1}:=-c_{n,n+2}$
$\implies$ $[e_{n+1},e_{n+1}]=\left(2c_{2,n+2}+2a_{5,3}\right)e_2,
[e_{n+1},e_{n+2}]=d_{2,n+1}e_2-\sum_{k=4}^{n-1}a_{k+1,3}e_k-c_{n,n+2}e_n.$ \\ \hline
\scriptsize $16.$ &\scriptsize $\L_{e_{n+1}}\left([e_{n+2},e_{n+2}]\right)$ &\scriptsize
$d_{2,n+2}:=\frac{d_{2,n+1}-a_{5,3}}{2}$
$\implies$ $[e_{n+2},e_{n+2}]=\frac{d_{2,n+1}-a_{5,3}}{2}e_2.$\\ \hline
\end{tabular}
\end{table}
Altogether the nilradical $\mathcal{L}^4$ $(\ref{L4}),$ the outer derivations $\L_{e_{n+1}}$ and $\L_{e_{n+2}}$
and the remaining brackets given above define a continuous family of the solvable left Leibniz algebras.
Then we apply the technique of ``absorption'' according to step $(iv)$.
\begin{itemize}[noitemsep, topsep=0pt]
\item We start with the transformation $e^{\prime}_i=e_i,(1\leq i\leq n,n\geq5),e^{\prime}_{n+1}=e_{n+1}+\frac{c_{2,n+2}}{2}e_2,
e^{\prime}_{n+2}=e_{n+2}+\frac{d_{2,n+1}-a_{5,3}}{4}e_2.$
This transformation removes the coefficients $c_{2,n+2}$ and $\frac{d_{2,n+1}-a_{5,3}}{2}$ in front of $e_2$ in
$[e_{n+2},e_{n+1}]$ and $[e_{n+2},e_{n+2}],$ respectively,
and changes the coefficient in front of $e_2$ in $[e_{n+1},e_{n+1}]$
and $[e_{n+1},e_{n+2}]$ to $2a_{5,3}$ and $a_{5,3},$ correspondingly.
\item Then we apply the transformation
$e^{\prime}_i=e_i,(1\leq i\leq n,n\geq5),e^{\prime}_{n+1}=e_{n+1}+\frac{a_{5,3}}{2}e_4,
e^{\prime}_{n+2}=e_{n+2}$ to remove the coefficients $a_{5,3}$ and $2a_{5,3}$
in front of $e_2$ in $[e_{n+1},e_{n+2}]$ and $[e_{n+1},e_{n+1}]$, respectively. At the same time this transformation removes
$a_{5,3}$ and $-a_{5,3}$ in front of $e_4$ in $[e_{n+2},e_{n+1}]$ and  $[e_{n+1},e_{n+2}],$ respectively, and changes the coefficients in front of $e_{k},(6\leq k\leq n-1)$
in $[e_{n+2},e_{n+1}]$ and $[e_{n+1},e_{n+2}]$ to $a_{k+1,3}-\frac{a_{5,3}}{2}a_{k-1,3}$ and 
$\frac{a_{5,3}}{2}a_{k-1,3}-a_{k+1,3},$ respectively. It also affects the coefficients in front $e_n,(n\geq6)$ in
$[e_{n+2},e_{n+1}]$ and $[e_{n+1},e_{n+2}],$ which we rename back by $c_{n,n+2}$ and $-c_{n,n+2},$
respectively. The following entries are introduced by the transformation: $-\frac{a_{5,3}}{2}$ and $\frac{a_{5,3}}{2}$ in the $(5,1)^{st},(n\geq5)$ position in $\r_{e_{n+1}}$ and $\L_{e_{n+1}},$
respectively.
\item
Applying the transformation $e^{\prime}_j=e_j,(1\leq j\leq n+1,n\geq6),
e^{\prime}_{n+2}=e_{n+2}-\sum_{k=5}^{n-1}{\frac{A_{k+1,3}}{k-1}e_k},$
where $A_{6,3}:=a_{6,3}$ and $A_{k+1,3}:=a_{k+1,3}-\frac{1}{2}a_{5,3}a_{k-1,3}-\sum_{i=7}^k{\frac{A_{i-1,3}a_{k-i+5,3}}{i-3}},$
$(6\leq k\leq n-1,n\geq7),$
we remove the coefficients $a_{6,3}$ and $-a_{6,3}$ in front of $e_5$ in $[e_{n+2},e_{n+1}]$
and $[e_{n+1},e_{n+2}],$ respectively. Besides the transformation removes $a_{k+1,3}-\frac{a_{5,3}}{2}a_{k-1,3}$ and 
$\frac{a_{5,3}}{2}a_{k-1,3}-a_{k+1,3}$
in front of $e_k,(6\leq k\leq n-1)$ in $[e_{n+2},e_{n+1}]$ and $[e_{n+1},e_{n+2}],$ respectively.
This transformation introduces $\frac{a_{6,3}}{4}$ and $-\frac{a_{6,3}}{4}$ in the $(6,1)^{st},(n\geq6)$ position
as well as
$\frac{A_{k+1,3}}{k-1}$ and $\frac{A_{k+1,3}}{1-k}$ in the $(k+1,1)^{st},(6\leq k\leq n-1)$ position in $\r_{e_{n+2}}$
and $\L_{e_{n+2}},$ respectively. It also affects the coefficients in front $e_n,(n\geq7)$ in
$[e_{n+2},e_{n+1}]$ and $[e_{n+1},e_{n+2}],$ which we rename back by $c_{n,n+2}$ and $-c_{n,n+2},$
respectively.
\item Finally applying the transformation $e^{\prime}_i=e_i,(1\leq i\leq n+1,n\geq5),
e^{\prime}_{n+2}=e_{n+2}-\frac{c_{n,n+2}}{n-1}e_n,$
we remove $c_{n,n+2}$ and $-c_{n,n+2}$ in front of $e_n$
in $[e_{n+2},e_{n+1}]$ and $[e_{n+1},e_{n+2}],$ respectively,
without affecting other entries. 

We obtain that $\L_{e_{n+1}}$ and $\L_{e_{n+2}}$ are as follows:
 \end{itemize}
{$$\L_{e_{n+1}}=\left[\begin{smallmatrix}
 -1 & 0 & 0 & 0&0&0&\cdots && 0&0 & 0\\
  3a_{2,1}-2a_{2,3} & -4 & a_{2,3}& -1 &0&0 & \cdots &&0  & 0& 0\\
  -1 & 0 & -2 & 0 & 0&0 &\cdots &&0 &0& 0\\
  0 & 0 &  0 & -3 &0 &0 &\cdots&&0 &0 & 0\\
 \frac{a_{5,3}}{2} & 0 & -a_{5,3} & 0 & -4  &0&\cdots & &0&0 & 0\\
  0 & 0 &\boldsymbol{\cdot} & -a_{5,3} & 0 &-5&\cdots & &0&0 & 0\\
    0 & 0 &\boldsymbol{\cdot} & \boldsymbol{\cdot} & \ddots &0&\ddots &&\vdots&\vdots &\vdots\\
  \vdots & \vdots & \vdots &\vdots &  &\ddots&\ddots &\ddots &\vdots&\vdots & \vdots\\
 0 & 0 & -a_{n-2,3}& -a_{n-3,3}& \cdots&\cdots &-a_{5,3}&0&3-n &0& 0\\
0 & 0 & -a_{n-1,3}& -a_{n-2,3}& \cdots&\cdots &\boldsymbol{\cdot}&-a_{5,3}&0 &2-n& 0\\
 0 & 0 & -a_{n,3}& -a_{n-1,3}& \cdots&\cdots &\boldsymbol{\cdot}&\boldsymbol{\cdot}&-a_{5,3} &0& 1-n
\end{smallmatrix}\right],$$}
$$\L_{e_{n+2}}=\left[\begin{smallmatrix}
 -1 & 0 & 0 & 0&0&0&0&\cdots && 0&0 & 0\\
a_{2,1}-\frac{a_{2,3}}{2} & -2 & \frac{a_{2,3}}{2}& 0 &0&0 & 0&\cdots &&0  & 0& 0\\
  0 & 0 & -1 & 0 & 0&0 &0&\cdots &&0 &0& 0\\
  0 & 0 &  0 & -2 &0 &0 &0&\cdots&&0 &0 & 0\\
  0 & 0 & -a_{5,3} & 0 & -3  &0&0&\cdots & &0&0 & 0\\
-\frac{a_{6,3}}{4} & 0 &\boldsymbol{\cdot} & -a_{5,3} & 0 &-4&0&\cdots & &0&0 & 0\\
  -\frac{A_{7,3}}{5} & 0 &\boldsymbol{\cdot} & \boldsymbol{\cdot} & -a_{5,3}&0 &-5& &&\vdots&\vdots &\vdots\\
      -\frac{A_{8,3}}{6} & 0 &\boldsymbol{\cdot} & \boldsymbol{\cdot} &\boldsymbol{\cdot} &-a_{5,3}&0 &\ddots&&\vdots&\vdots &\vdots\\
  \vdots & \vdots & \vdots &\vdots &  &&\ddots&\ddots &\ddots &\vdots&\vdots & \vdots\\
  \frac{A_{n-2,3}}{4-n} & 0 & -a_{n-2,3}& -a_{n-3,3}& \cdots&\cdots&\cdots &-a_{5,3}&0&4-n &0& 0\\
    \frac{A_{n-1,3}}{3-n} & 0 & -a_{n-1,3}& -a_{n-2,3}& \cdots&\cdots&\cdots &\boldsymbol{\cdot}&-a_{5,3}&0 &3-n& 0\\
 \frac{A_{n,3}}{2-n} & 0 & -a_{n,3}&-a_{n-1,3}& \cdots&\cdots &\cdots&\boldsymbol{\cdot}&\boldsymbol{\cdot}&-a_{5,3} &0&2-n
\end{smallmatrix}\right],(n\geq5).$$
The remaining brackets are given below:
\begin{equation}
\left\{
\begin{array}{l}
\displaystyle  \nonumber [e_{1},e_{n+1}]=e_1+a_{2,1}e_2+e_3-\frac{a_{5,3}}{2}e_5,[e_{3},e_{n+1}]=a_{2,3}e_2+2e_3+\sum_{k=5}^n{a_{k,3}e_k},\\
\displaystyle [e_{4},e_{n+1}]=3\left(e_4-e_2\right)+\sum_{k=6}^n{a_{k-1,3}e_k},
 [e_{i},e_{n+1}]=(i-1)e_i+\sum_{k=i+2}^{n}{a_{k-i+3,3}e_k},\\
\displaystyle [e_{1},e_{n+2}]=e_1+\left(a_{2,1}-\frac{a_{2,3}}{2}\right)e_2+\sum_{k=6}^n{\frac{A_{k,3}}{k-2}e_k},[e_{3},e_{n+2}]=\frac{a_{2,3}}{2}e_2+e_3+\sum_{k=5}^n{a_{k,3}e_k},\\
\displaystyle 
[e_{4},e_{n+2}]=2\left(e_4-e_2\right)+\sum_{k=6}^n{a_{k-1,3}e_k},[e_{i},e_{n+2}]=(i-2)e_i+\sum_{k=i+2}^{n}{a_{k-i+3,3}e_k},(5\leq i\leq n).
\end{array} 
\right.
\end{equation} 
\noindent $(v)$ Finally we apply the following two change of basis transformations:
\begin{itemize}[noitemsep, topsep=0pt]
\item $e^{\prime}_1=e_1+\left(a_{2,1}-\frac{a_{2,3}}{2}\right)e_2,e^{\prime}_2=e_2,
e^{\prime}_3=e_3+\frac{a_{2,3}}{2}e_2,
e^{\prime}_i=e_i-\sum_{k=i+2}^n{\frac{B_{k-i+3,3}}{k-i}e_k},$
$(3\leq i\leq n-2),e^{\prime}_{n-1}=e_{n-1},e^{\prime}_{n}=e_{n},e^{\prime}_{n+1}=e_{n+1},e^{\prime}_{n+2}=e_{n+2},$
where $B_{j,3}:=a_{j,3}-\sum_{k=7}^j{\frac{B_{k-2,3}a_{j-k+5,3}}{k-5}},(5\leq j\leq n)$.\\
This transformation removes $a_{5,3},a_{6,3},...,a_{n,3}$ in $\r_{e_{n+1}},$ $\r_{e_{n+2}}$
and $-a_{5,3},-a_{6,3},...,-a_{n,3}$ in $\L_{e_{n+1}},\L_{e_{n+2}}.$ It also removes $-\frac{a_{5,3}}{2}$
and $\frac{a_{5,3}}{2}$ from the $(5,1)^{st}$ positions in $\r_{e_{n+1}}$ and $\L_{e_{n+1}},$ respectively.
Besides it removes $a_{2,1}$ and $3a_{2,1}-2a_{2,3}$ from the $(2,1)^{st}$ positions in $\r_{e_{n+1}}$
and $\L_{e_{n+1}},$ respectively, as well as $a_{2,1}-\frac{a_{2,3}}{2}$ from the $(2,1)^{st}$ positions in $\r_{e_{n+2}}$
and $\L_{e_{n+2}}.$
The transformation also removes $a_{2,3}$ from the $(2,3)^{rd}$ positions in $\r_{e_{n+1}},$ $\L_{e_{n+1}}$ and
$\frac{a_{2,3}}{2}$ from the same positions in $\r_{e_{n+2}},$ $\L_{e_{n+2}}.$
It introduces the entries in the $(i,1)^{st}$ positions in $\r_{e_{n+1}}$ and $\L_{e_{n+1}},$
that we set to be $a_{i,1}$ and $-a_{i,1},(6\leq i\leq n),$ respectively. The transformation affects the entries in the $(j,1)^{st},(8\leq j\leq n)$
positions in $\r_{e_{n+2}}$ and $\L_{e_{n+2}},$ but we rename all the entries in the $(i,1)^{st}$ positions
by $\frac{i-3}{i-2}a_{i,1}$ in $\r_{e_{n+2}}$ and by $\frac{3-i}{i-2}a_{i,1},(6\leq i\leq n)$ in $\L_{e_{n+2}}.$
\item Applying the transformation $e_k^{\prime}=e_k,(1\leq k\leq n),e^{\prime}_{n+1}=e_{n+1}+\sum_{k=5}^{n-1}{a_{k+1,1}e_k},
e^{\prime}_{n+2}=e_{n+2}+\sum_{k=5}^{n-1}{\frac{k-2}{k-1}a_{k+1,1}e_k},$
we remove $a_{i,1}$ in $\r_{e_{n+1}}$ and $-a_{i,1}$ in $\L_{e_{n+1}}$ as well as
$\frac{i-3}{i-2}a_{i,1}$ and $\frac{3-i}{i-2}a_{i,1},(6\leq i\leq n)$ in $\r_{e_{n+2}}$ and $\L_{e_{n+2}},$ respectively. 
\end{itemize}
We obtain a Leibniz algebra $l_{n+2,1}$ given below:
\begin{equation}
\left\{
\begin{array}{l}
\displaystyle  \nonumber [e_1,e_{n+1}]=e_1+e_3,[e_3,e_{n+1}]=2e_3,
[e_4,e_{n+1}]=3(e_4-e_2),
 [e_{i},e_{n+1}]=(i-1)e_i,\\
\displaystyle [e_{n+1},e_{1}]=-e_1-e_3,[e_{n+1},e_{2}]=-4e_2,[e_{n+1},e_3]=-2e_3, [e_{n+1},e_4]=-e_2-3e_4,\\
\displaystyle [e_{n+1},e_{i}]=(1-i)e_i,[e_{1},e_{n+2}]=e_1,[e_{3},e_{n+2}]=e_3,[e_{4},e_{n+2}]=2(e_4-e_2),\\
\displaystyle [e_{i},e_{n+2}]=(i-2)e_i,(5\leq i\leq n,n\geq5),[e_{n+2},e_1]=-e_1,[e_{n+2},e_{2}]=-2e_2,\\
 \displaystyle[e_{n+2},e_{j}]=(2-j)e_j,(3\leq j\leq n).\end{array} 
\right.
\end{equation} 
 We summarize a result 
in the following theorem: 
 \begin{theorem}\label{LCodim2L4} There are four solvable
indecomposable left Leibniz algebras up to isomorphism with a codimension two nilradical
$\mathcal{L}^4,(n\geq4),$ which are given below:
\begin{equation}
\begin{array}{l}
\displaystyle  \nonumber (i)\,l_{n+2,1}: [e_1,e_{n+1}]=e_1+e_3,[e_3,e_{n+1}]=2e_3,
[e_4,e_{n+1}]=3(e_4-e_2),
 [e_{i},e_{n+1}]=(i-1)e_i,\\
\displaystyle [e_{n+1},e_{1}]=-e_1-e_3,[e_{n+1},e_{2}]=-4e_2,[e_{n+1},e_3]=-2e_3, [e_{n+1},e_4]=-e_2-3e_4,\\
\displaystyle [e_{n+1},e_{i}]=(1-i)e_i,[e_{1},e_{n+2}]=e_1,[e_{3},e_{n+2}]=e_3,[e_{4},e_{n+2}]=2(e_4-e_2),\\
\displaystyle [e_{i},e_{n+2}]=(i-2)e_i,(5\leq i\leq n,n\geq5),[e_{n+2},e_1]=-e_1,[e_{n+2},e_{2}]=-2e_2,\\
\displaystyle  [e_{n+2},e_{j}]=(2-j)e_j,(3\leq j\leq n),\\
\displaystyle  (ii)\,l_{6,2}: [e_1,e_5]=e_1,[e_3,e_5]=e_3,[e_4,e_5]=2(e_4-e_2),[e_{5},e_{1}]=-e_1,[e_5,e_2]=-2e_2,\\
\displaystyle [e_{5},e_3]=-e_3,
[e_{5},e_4]=-2e_4, [e_1,e_6]=e_1+be_3,[e_3,e_6]=e_1+be_3,[e_4,e_6]=(b+1)e_4-\\
\displaystyle(b+1)e_2,
[e_{6},e_1]=-e_1-be_3,[e_6,e_2]=-2(b+1)e_2,[e_{6},e_{3}]=-e_1-be_3,\\
\displaystyle [e_{6},e_4]=-(b+1)\left(e_2+e_4\right),(b\neq-1),\\
\displaystyle (iii)\, l_{6,3}: [e_1,e_5]=e_1+(c+1)e_3,[e_3,e_5]=ce_1+2e_3,[e_{4},e_5]=3\left(e_4-e_2\right),[e_{5},e_{1}]=-e_1-\\
\displaystyle (c+1)e_3,[e_5,e_2]=-2(c+2)e_2,[e_{5},e_3]=-ce_1-2e_3,[e_5,e_4]=(-2c-1)e_2-3e_4,\\
\displaystyle [e_1,e_6]=e_1,[e_3,e_6]=e_3, [e_{4},e_6]=2\left(e_4-e_2\right),[e_{6},e_1]=-e_1,[e_6,e_2]=-2e_2,[e_{6},e_{3}]=-e_3,\\
\displaystyle  [e_6,e_4]=-2e_4,(c\neq-2),\\
\displaystyle  (iv)\,l_{6,4}:[e_1,e_5]=e_1,[e_3,e_5]=e_3,[e_{4},e_5]=2\left(e_4-e_2\right),[e_{5},e_{1}]=-e_1,[e_5,e_2]=-2e_2,\\
\displaystyle [e_{5},e_3]=-e_3,[e_5,e_4]=-2e_4,[e_1,e_6]=be_1+(1-b)e_3,[e_3,e_6]=e_1,[e_{4},e_6]=b\left(e_4-e_2\right),\\
\displaystyle [e_{6},e_1]=-be_1+(b-1)e_3,[e_6,e_2]=-2e_2,[e_{6},e_{3}]=-e_1,[e_6,e_4]=(b-2)e_2-be_4,(b\neq0).
\end{array} 
\end{equation} 
\end{theorem}

\end{document}